\documentclass{compositio}
\usepackage{amssymb,mathrsfs}
\usepackage{amsmath}
\usepackage[normalem]{ulem}
\usepackage{csquotes}
\usepackage{amssymb,amstext,titletoc,fancyhdr,color,xcolor,calc,graphicx}
\usepackage{tikz-cd}
\usepackage{tikz}
\usepackage{mathtools}
\usepackage{rotating}
\usepackage{float}
\usepackage[toc,page]{appendix}
\theoremstyle{plain}
\newtheorem{theorem}{Theorem}[section]
\newtheorem{lemma}[theorem]{Lemma}
\newtheorem{proposition}[theorem]{Proposition}

\newtheorem{corollary}[theorem]{Corollary}

\theoremstyle{definition}
\newtheorem{definition}[theorem]{Definition}
\newtheorem{example}[theorem]{Example}

\theoremstyle{remark}
\newtheorem{remark}[theorem]{Remark}

\numberwithin{equation}{section}

\newenvironment{theoremprime}[1]
  {\renewcommand{\thetheorem}{\ref{#1}$'$}%
   \addtocounter{theorem}{-1}%
   \begin{theorem}}
  {\end{theorem}}

\def\E{{\mathop{\rm E}}}
\def\Im{\mathop{\rm Im}}
\def\Re{\mathop{\rm Re}}

\def\Hom{\mathop{\rm Hom}}

\newcommand{\beqnn}{\begin{equation}}
\newcommand{\eeqnn}{\end{equation}}
\newcommand{\eb}{\begin{enumerate}}
\newcommand{\ee}{\end{enumerate}}
\newcommand{\bbm}{\begin{bmatrix}}
\newcommand{\ebm}{\end{bmatrix}}
\newcommand{\bpm}{\begin{pmatrix}}
\newcommand{\epm}{\end{pmatrix}}

\newcommand{\bi}{\begin{itemize}}
\newcommand{\ei}{\end{itemize}}
\newcommand{\beq}{\begin{eqnarray*}}
\newcommand{\eeq}{\end{eqnarray*}}

\def\Cos{\mathop{\rm Cos}}
\def\Sin{\mathop{\rm Sin}}

\newcommand{\beqq}{\begin{eqnarray}}
\newcommand{\eeqq}{\end{eqnarray}}
\newcommand{\beqn}{\begin{eqnarray}}
\newcommand{\eeqn}{\end{eqnarray}}

\title[Holonomic  Bessel modules and generating functions]{Holonomic  Bessel modules and generating functions}
\author[Y. M. Chiang]{Yik Man Chiang}
\email{machiang@ust.hk}
\address{Department of Mathematics, The Hong Kong University of Science and Technology, Clear Water Bay, Kowloon, Hong Kong SAR.}
\author[A. Ching]{Avery Ching}
\email{aching001@dundee.ac.uk}
\address{Department of Mathematics, University of Dundee, Nethergate, Dundee DD1 4HN, Scotland, UK.}
\author[X. Lin]{Xiaoli Lin}
\email{xlinaw@connect.ust.hk}
\address{Department of Mathematics, The Hong Kong University of Science and Technology, Clear Water Bay, Kowloon, Hong Kong SAR.}
\thanks{
}

\dedication{}
\begin{document}

\begin{abstract}
We have solved a number of holonomic PDEs derived from the Bessel modules which are related to the generating functions of classical Bessel functions and the difference Bessel functions recently discovered by Bohner and Cuchta. This $D$-module approach both unifies and extends generating functions of the classical and the difference Bessel functions. It shows that the algebraic structures of the  Bessel modules and related modules determine the possible formats of Bessel's generating functions studied in this article. As a consequence of these $D$-modules structures, a 
 number of new recursion formulae, integral representations and new difference Bessel polynomials have been discovered. The key ingredients of our argument involve new transmutation formulae related to the Bessel modules and the construction of $D$-linear maps between different appropriately constructed submodules. This work can be viewed as $D$-module approach to Truesdell\rq{}s $F$-equation theory specialised to Bessel functions. The framework presented in this article can be applied to other  special functions.
\end{abstract}

\maketitle
\setcounter{tocdepth}{2}
\tableofcontents

%\vfill\eject
\section{Introduction} Generating functions play important roles in both pure and applied mathematics \cite{FS_2009, Wilf_2006}.  In particular, generating functions of Bessel functions play pivotal roles in the mathematical formulation of electromagnetic wave and acoustic scattering problems, see for examples,  \cite{Bowman_et_al_1987, Dunster_2013}. Other application includes the study of Rogers-Ramanujan identities \cite{Ehrenpreis_1990, Ehrenpreis_1993}.
%\rk{Recently, there are also complex analytic studies of the PDEs satisfied by generating functions of some classical special functions \cite{Hu_Yang_2009, Hu_Yang_2010, Hu_Li_2017}.} 
In this paper, we present a systematic study of generating functions of Bessel functions based on the Bessel module \eqref{E:Bessel_mod_intro} defined below. This is the first of a series of papers originates from our study of generating functions of classical special functions using $D$-module methodology.   It is known that  $D$-modules is an efficient language in exhibiting algebraic structures and in carrying out explicit computation of identities of special functions in general, see for examples \cite{Chyzak_2000, Chyzak_Salvy_1998, Ehrenpreis_1990, Ehrenpreis_1993, kashiwara, Kashiwara_Kawai_1981, Kashiwara_Schapira_1997,   Mansour_Schork_2016, Paule_Schorn_1995, Wilf_Zeilberger_1992, Zeilberger_1990}.  We demonstrate that this is indeed the case that the $D$-modules approach not only allows us to better understand some of the classical formulae about special functions but also to derive their new difference analogues. This work can be considered as a continuation of the previous works of \cite{Lommel, Lommel_1871, Nielson, Sonine,  Schlfafli_1873,  Truesdell_1947, Truesdell_1948} and to a less extent of \cite{Burchnall_1953, Ehrenpreis_1990, Ehrenpreis_1993} by algebraic methods.
A noteworthy feature of this study is that the main results are first obtained in the $D$-modules framework before their $D$-modules structures are being manifested in complex domains differently. Thus the study unifies the classical Bessel functions and the recently discovered difference Bessel functions amongst other possible manifestations. The fact that the main results have been derived entirely within the appropriate $D$-modules that is subject to different manifestations in the complex plane is a vindication of the viewpoint from umbral calculus discovered centuries ago.

\subsection{Bessel\rq{}s modules}
We extend Bessel\rq{}s classical generating function\footnote{It was known to Hansen in 1845 and earlier to Jacobi in the 1836 in a modified form \cite[\S 2.1]{Watson1944}.}
\begin{equation}\label{E:bessel_classical_gf_0}
		e^{\frac{x}{2}\big(t-\frac1t\big)}=\footnote{The series converges absolutely in each compact subset of $\mathbb{C}\times\mathbb{C}\backslash\{0\}$.} \sum_{n=-\infty}^\infty J_{n}(x)\, t^n
	\end{equation}
 for the Bessel coefficients $(J_n(x))$  which denotes the bilateral sequence of \textit{Bessel functions of the first kind} of integer orders, to the case of non-integer orders $(J_{\nu+n}(x))$ where $\nu\not=0$, amongst other results obtained, including their difference analogues, by studying
% in \eqref{E:bessel_classical_gf_1} below and their new difference analogues. We have also made a detailed study of classical generating functions for Bessel functions of half-integer orders and their difference analogues.  This half-integer orders generating function study also includes the classical Bessel polynomials and their new difference analogue.
the quotient module 
	\begin{equation}\label{E:Bessel_mod_intro}
		\mathcal{B}_\nu:= \frac{\mathcal{A}_2}
		{\mathcal{A}_2(X_1\partial_1+(\nu+X_2\partial_2)-X_1X_2)+
				\mathcal{A}_2(X_1\partial_1-(\nu+X_2\partial_2)+{X_1}/{X_2})},
	\end{equation}be called the \textit{Bessel module of order} $\nu\in \mathbb{C}$ in this article, where the $\mathcal{A}_2$ denotes the $D$-modules of the Weyl-algebra generated by  $\partial_i,\, X_i\ (j=1,\, 2)$ subject to the standard commutation relations 
	\begin{equation}\label{E:comm}
		[\partial_j, X_k]=\delta_{j, k}, \qquad [\partial_j, \partial_k]=0,
	\qquad [X_j, X_k]=0,\qquad \mbox{for } j,\, k=1,\, 2.
	\end{equation}
Here the two relations 
	\begin{equation}\label{E:2_elements}
		X_1\partial_1+\nu+X_2\partial_2-X_1X_2,
		\qquad 
		X_1\partial_1-(\nu+X_2\partial_2)+{X_1}/{X_2},
	\end{equation}
which define the Bessel module $\mathcal{B}_\nu$,  are abstractions of the well-known recursions \footnote{Lommel \cite[1868]{Lommel}.} 
%\begin{equation}\label{E:2_PDE_dd}
%		\begin{split}
%			& xJ'_{\nu}(x)-\nu J_{\nu}(x)=-xJ_{\nu+1}(x),
%			\\
%			& xJ'_{\nu}(x)+\nu J_{\nu}(x)=xJ_{\nu-1}(x)
%		\end{split}
%	\end{equation}
\begin{equation}\label{E:2_PDE_d_delta}
		\begin{split}
			& xJ'_{\nu+n}(x)-(\nu+n) J_{\nu+n}(x)=-xJ_{\nu+n+1}(x),
			\\
			& xJ'_{\nu+n}(x)+(\nu+n) J_{\nu+n}(x)=xJ_{\nu+n-1}(x),
		\end{split}
	\end{equation}about Bessel functions, known as Gauss\rq{} contiguous relations\footnote{They and other generalised ones will be called PDEs in a board sense in this article. The list of all holonomic PDEs are listed in Appendix \S\ref{SS:holo_modules}.} in special functions, which hold for every $n\in \mathbb{Z}$, where the $\partial_1$ and $\partial_2$ in \eqref{E:2_elements} play the roles of differentiation with respect to $x$ and shift of subscripts $n$ respectively.  As we shall see that the $\mathcal{B}_\nu$ 
is holonomic in the sense of Bernstein-Kashiwara-Sato's holonomic PDEs theory (see, e.g., \cite{coutinho, kashiwara}), and that solving $\mathcal{B}_0$ gives raise to the \eqref{E:bessel_classical_gf_0} as a special case of $\nu=0$. The two relations from \eqref{E:2_elements} forms a system of two PDEs from this viewpoint. Thus, we have derived
		\begin{equation}\label{E:bessel_classical_gf_1}
		t^{-\nu}e^{\frac{x}{2}\big(t-\frac1t\big)}\sim\footnote{The infinite sum is divergent, and the sign ``$\sim$" denotes \textit{Borel-resummation}; see the paragraph below the \eqref{gf-db-02}.} \sum_{n=-\infty}^\infty J_{\nu+n}(x)\, t^n,
	\end{equation} 
	for general $\nu\not=0$, and its difference analogue
		\begin{equation}\label{gf-db-02}
			e^{i\pi\nu}
			\frac{\sin(x-\nu)\pi}{\sin(\pi x)}\,
			t^{-\nu}\big[\frac{1}{2}(t-\frac{1}{t})+1\big]^{x}
			\sim\sum_{n=-\infty}^{\infty}
				J^{\Delta}_{n+\nu}(x)\,t^n,
		\end{equation}
%under the restriction $\Re (x)<\Re (\nu)$, 
for the \textit{difference Bessel functions} $(J_{\nu+n}^\Delta)$  discovered by Bohner and Cuchta  \cite{BC_2017} to be discussed below. Here  the ``$\sim$" signs used above denote the \textit{Borel-resummation} on the right-sides of the corresponding formulae since both the sequences $(J_{\nu+n})$  and the  $(J^\Delta_{\nu+n}$) are $1$-Gervey\footnote{See Theorem \ref{T:Bessel_Gevrey} and Theorem \ref{T:delta_Bessel_Gevrey}.}\footnote{For an up-to-date discussion about the \textit{resurgence theory}, please see \cite[Part II]{Mitschi_Sauzin_2016}, \cite{Shawyer_Watson_1994} and \cite{Delabaere_Pham_1999}. } in the sense that in each compact set, $|J_{\nu+n}(x)|\le C^nn!$ and $|J^\Delta_{\nu+n}(x)|\le D^nn!$ for some positive constants $C,\ D$, see Example \ref{D:borel}. 
%Under the additional assumption $ |(t-\frac{1}{t})|<2$, the \eqref{gf-db-02} becomes the equality
Both sides of \eqref{E:2_PDE_d_delta} and \eqref{gf-db-02} satisfy the systems of PDEs\footnote{We use the terminology ``PDEs\rq\rq{} in a generalised sense that may include difference or shift operators or simply members in the $\mathcal{A}_2$ in this paper.}
	\begin{equation}\label{E:PDE_gf_bessel}
%	\begin{tabular}{l}
		\begin{split}
			&y_x(x,t)+\big(1/t-t\big)/2\, y(x,t)=0,\\
			&ty_t(x,t)+\nu y(x,t)-{x}/{2}\,\big(1/t+t\big)\, y(x,t)=0,
%		  \footnote{The right-side of \eqref {E:bessel_classical_gf_1}  satisfies the PDEs formally.}
%	\end{tabular}
		\end{split}
	\end{equation}
and 
	\begin{equation}\label{E:PDE_gf_dBessel}
		\begin{split}
			&y(x+1, t)-y(x, t)+(1/t-t)/2 y(x, t)=0,\\
			&ty_t(x,t)+\nu y(x,t)-x(1/t+ t)/2y(x-1, t)=0
		\end{split}
	\end{equation}
respectively.

The method of derivation of the above results and other results in this paper is based on construction of $D$-linear maps from the Bessel modules $\mathcal{B}_\nu$ (and other $D$-modules) to appropriately \textit{manifested}  $D$-module analytic function spaces of two variables (e.g., $\mathcal{O}_{dd},\ \mathcal{O}_{\Delta d}$). Other quotient $D$-modules  related to the half-Bessel module $\mathcal{B}_{1/2}$ studied in this paper includes \textit{Bessel polynomial modules} $\Theta, \mathcal{Y}$ and \textit{Glaisher modules} $\mathcal{G}_\pm$ which also are listed in Appendix  \S\ref{SS:holo_modules}. This paper can be considered as a $D$-module study of Truesdell\rq{}s $F$-equation theory \cite{Truesdell_1948}, which summarised and extended research works about generating functions of predecessors from the nineteenth century, when applied to Bessel\rq{}s generating functions.\footnote{The authors only got to know Truesdell's work after obtaining the main results of this paper.}  However, unlike the analytic approach used in \cite{Truesdell_1948}, the existence and uniqueness of the $F$-equation theory are replaced by the holonomicity of the $D$-modules concerned.  
The paper aims to show that $D$-modules is not only a natural language but also an efficient computational tool to study special functions. We have built up our applications of $D$-modules that represent some of the simplest holonomic equations which give raise to the exponential, trigonometric functions, etc from scratch  in order to deal with the quotient modules studied in this paper. This is because we could reach our main results efficiently with our self-contained approach instead of utilising the full force of $D$-modules theory from, for examples, \cite{kashiwara, SST_2000}. This is especially the case when we derive difference analogues of the classical results. Closer to our approach but differ substantially in details were the pioneering works of Wilf and Zeilberger \cite{Zeilberger_1990, Wilf_Zeilberger_1992} in which the authors used holonomic $D$-modules methodology to derive hypergeometric type identities. Our approach is also different from the Lie algebraic approach  of generating functions  used by Weisner \cite{Weisner_1959}, see also \cite{Mcbride}. 
%Generating functions play important roles in both pure and applied mathematics \cite{FS_2009, Wilf_2006} in general.  The particular generating functions of Bessel functions play pivotal roles in the mathematical formulation of electromagnetic wave and acoustic scattering problems \cite{Bowman_et_al_1987, Dunster_2013}. \rk{Recently, there are also complex analytic studies of the PDEs satisfied by generating functions of some classical special functions \cite{Hu_Yang_2009, Hu_Yang_2010, Hu_Li_2017}.} We present a systematic study of generating functions of Bessel functions based on the Bessel module \eqref{E:Bessel_mod_intro} defined above.
% as the first example of this $D$-module study of special functions in this paper.
%This is the first of a series of papers originates from our study of generating functions of classical special functions using $D$-module methodology.   It is known that  $D$-modules is an efficient language in exhibiting algebraic structures and in carrying out explicit computation of identities of special functions in general, see for examples \cite{Chyzak_2000, kashiwara, Kashiwara_Kawai_1981, Kashiwara_Schapira_1997, Kashiwara_Schapira_2016, Paule_Schorn_1995, Wilf_Zeilberger_1992, Zeilberger_1990}.  We demonstrate that this is indeed the case that the $D$-modules approach not only allows us to explain some of the classical formulae about special functions but also to derive their new difference analogues. 
The choice to start with Bessel functions is a natural one in regard to them being ubiquity in many research areas and wide range of applications, see for examples \cite{AAR, EAM2, Lommel,  Wang_Guo1989, WW, Watson1944}. 

%We emphasis that most of the main results regarding Bessel's generating functions obtained in this article were first computed entirely using $D$-modules language before they being interpreted (i.e., via $D$-linear maps) in analytic functions spaces with different manifestations (i.e., operational conventions with usual differentiation or forward difference operator, etc.).

For easy of computation of the relevant generating functions, we add and subtract the two relations \eqref{E:2_elements} to obtain, respectively, the two alternative relations
	\begin{equation}\label{E:2_more_elements}
		2X_1\partial_1-X_1(X_2-1/X_2),\qquad 2\nu+2X_2\partial_2-X_1(X_2+1/X_2)
	\end{equation}
for which we can, without loss of generality, eliminate the term ``$X_1$\rq\rq{} altogether from the first relation  above. As an example with $\nu=0$, after a ``change of variables"
\footnote{This change of variables corresponds to the ``characteristic method\rq\rq{}.},
 the \eqref{E:2_more_elements} is rewritten as
	\[
		2\hat{X}_1\hat{\partial}_1-\hat{X}_1,\quad
		2\hat{X}_2\hat{\partial}_2+\hat{X}_2,
	\]where the $\{\hat{\partial}_1, \hat{X}_1, \hat{\partial}_2, \hat{X}_2\}$ satisfy  similar commutation relations as those in the  \eqref{E:comm}, and which when \textit{solved simultaneously}
%\footnote{The new algebra $\mathbb{C}\langle \hat{\partial}_1, \hat{\partial}_2, \hat{X}_1, \hat{X}_2\rangle$ is a $D$-modules.}
	 leads us to the \textit{solution}
	\begin{equation}\label{E:bessel_exp}
		\E \Big[\frac{X_1}{2}(X_2-\frac{1}{X_2})\Big],
	\end{equation}
where the $\E$ denotes the \textit{Weyl exponential} (see Example \ref{Eg:exp}). Right multiplication by $e^{\frac{x}{2}(x-1/x)}$ gives
the map $\mathcal{B}_0\to \mathcal{O}_{dd}$ where the $ \mathcal{O}_{dd}$ has the \textit{manifestation}, as it is called in this paper,
	\begin{equation}\label{E:manifestation_1}
		\begin{aligned}[l]
			&(\partial_1 f)(x, t)=f_x(x,t),\\
			&(X_1f)(x,t)=xf(x,t),
		\end{aligned}
		\qquad
		\begin{aligned}[l]
			&(\partial_2 f)(x, t)=f_t(x,t),\\
			&(X_2f)(x,t)=tf(x,t)
		\end{aligned}
	\end{equation}
	so that the $\mathcal{O}_{dd}$ becomes an $\mathcal{A}_2$-module, and the map itself is clearly well-defined.   It turns out that the system of PDEs \eqref{E:2_more_elements} become, when $\nu=0$,  the following system of PDEs
	\begin{equation}\label{E:PDE_gf_bessel_0}
%	\begin{tabular}{l}
		\begin{split}
			&y_x(x,t)+\big(1/t-t\big)/2\, y(x,t)=0,\\
			&ty_t(x,t)-x\,\big(1/t+t\big)/2\, y(x,t)=0,
%		  \footnote{The right-side of \eqref {E:bessel_classical_gf_1}  satisfies the PDEs formally.}
%	\end{tabular}
		\end{split}
	\end{equation}via an $A_2$-linear map, and that both sides of the \eqref{E:bessel_classical_gf_0} satisfy. Indeed, the $e^{\frac{x}{2}\big(t-\frac1t\big)}$ is an image of  \eqref{E:bessel_exp} via $\mathcal{A}_2/(\mathcal{A}_2\partial_1+\mathcal{A}_2\partial_2)\to \mathcal{O}_{dd}$. 
 On the other hand, we may equip the $\mathcal{O}^{\mathbb{Z}}$ consisting of sequences of analytic functions $(f_n)_{-\infty}^\infty$ all defined on an appropriate domain with a $D$-module structure by the \textit{manifestation}
\begin{equation}
	\begin{aligned}[l]\label{E:manifestation_2}
		&(\partial_1 f)_n(x)=f^\prime_n(x),\\
		&(X_1 f)_n(x)=xf_n(x),
	\end{aligned}
	\qquad
	\begin{aligned}[l]
		&(\partial_2 f)_n(x)=(n+1)f_{n+1}(x),\\
		&(X_2 f)_n(x)=f_{n-1}(x).
	\end{aligned}
\end{equation}It is easily verified that the $\mathcal{O}^{\mathbb{Z}}$ becomes an $\mathcal{A}_2$-module which is denoted by $\mathcal{O}^{\mathbb{Z}}_d$, so that right multiplication by $(J_n)$ shows that the map $\mathfrak{j}: \mathcal{B}_0\to \mathcal{O}^{\mathbb{Z}}_d$ is also $\mathcal{A}_2$-linear.  Moreover, we have the systems\footnote{According to Watson \cite[p. 18, p. 45]{Watson1944} these formulae were discovered by Bessel in 1824 and later the \eqref{E:2_PDE_d_delta}by Lommel in 1868 \cite{Lommel}  for general $\nu\not=0$.}
	\begin{equation}\label{E:2_PDE_d_delta_0}
		\begin{split}
			& xJ'_{n}(x)-n J_{n}(x)=-xJ_{n+1}(x),
			\\
			& xJ'_{n}(x)+n J_{n}(x)=xJ_{n-1}(x),
		\end{split}
	\end{equation}that lie in the kernel of the $\mathcal{A}_2$-linear map $\mathfrak{j}: \mathcal{B}_0\to \mathcal{O}^{\mathbb{Z}}_d$, which are special cases of the \eqref{E:2_PDE_d_delta} when $\nu=0$. Then following  diagram commutes
	\begin{equation}\label{E:commute-0}
			\begin{tikzcd} [row sep=large, column sep=huge]
			\mathcal{B}_{0} \arrow{r}{\mathfrak{j}} 
			 % \arrow[swap]{dr}{\mathfrak{g}} 
			\arrow[swap]{d}{\times \E [\frac{X_1}{2}(X_2-\frac{1}{X_2})]} 
			&   \mathcal{O}^{\mathbb{Z}}_d \arrow{d}{\mathfrak{z}}\\
 			 \widetilde{\mathcal{A}}_2
\arrow{r}{\times 1} &  \mathcal{O}_{dd},
			\end{tikzcd}
		\end{equation}where $\widetilde{\mathcal{A}}_2:= \overline{\mathcal{A}_2/\big[\mathcal{A}_2\partial_1+\mathcal{A}_2\partial_2\big]}$ is the completion of $\mathcal{A}_2/\big[\mathcal{A}_2\partial_1+\mathcal{A}_2\partial_2]$ and where the map $\mathfrak{z}$ is the $z$-transform, from which the classical formula \eqref{E:bessel_classical_gf_0} essentially follows because of the holonomicity of $\mathcal{B}_0$. A more general consideration gives the formula \eqref{E:bessel_classical_gf_1} for $\mathcal{B}_\nu$ (Theorem \ref{T:Bessel_gf}).

 The above description illustrates how the $D$-module approach allows us to see each generating function is a solution for the different holonomic $D$-modules studied (see the table in Appendix \ref{SS:holo_modules}).  They are obtained by solving the systems of holonomic PDEs from, for examples,  the \eqref{E:PDE_gf_dBessel}, \eqref{E:PDE_gf_bessel_0},  
\eqref{E:PDE_gf_dBessel_0},  \eqref{E:PDeltaE_neg_glaisher}, \eqref{E:delta_rev_bessel_pde} which are different manifestations inherited from the \textit{same} holonomic modules listed in Appendix \ref{SS:holo_modules}.

Before we discuss the system \eqref{E:PDE_gf_dBessel}, let us introduce the \textit{difference Bessel function} of order $\nu$ recently discovered by Bohner and Cuthta \cite{BC_2017}
	\begin{equation}\label{E:dBessel}
		J^{\Delta}_{\nu}(x):=\sum_{k=0}^{\infty} \frac{(-1)^k}{2^{\nu+2k}k!\Gamma(\nu+k+1)}(x)_{\nu+2k},
	\end{equation}where $(x)_\nu=\Gamma(x+1)/\Gamma(x+1-\nu)$ for an arbitrary $\nu$ from our $D$-module perspective. Bohner and Cuthta \cite{BC_2017} shows that the Newton series solves the \textit{difference Bessel equation}
	\begin{equation}\label{E:difference_bessel_eqn_0}
		x(x-1)\triangle^2 y(x-2)+x\triangle y(x-1)+x(x-1)y(x-2)-\nu^2y(x)=0.
	\end{equation}
The first two authors and Tsang  have shown in \cite{CCT3} that the function \eqref{E:dBessel} and the equation \eqref{E:difference_bessel_eqn_0} are manifestations of the Weyl-algebraic Bessel \eqref{E:pos_nu_map} and the Weyl-Bessel operator
	\[
		(X\partial)^2+ (X^2-\nu^2)
	\]
in the analytic function space $\mathcal{O}_\Delta$, where the subscript $\Delta$ indicates that the $\partial$ in $\mathbb{C}\langle\partial, X\rangle$ subjects to $[\partial, X]=\partial X-X\partial=1$, that is in $\mathcal{A}_1$, is interpreted as a forward difference operator instead, that is, $\partial f(x)=\Delta f(x)=f(x+1)-f(x)$ and $Xf(x)=xf(x-1)$ (see Example \ref{Eg:delta_Bessel_fn}.). Indeed, one recovers the classical Bessel equation
		\begin{equation}\label{E:Bessel_eqn}
		x^2y^{\prime\prime}(x)+xy^\prime(x)+(x^2-\nu^2)y(x)=0.
	\end{equation}and classical Bessel function with the well-known manifestation  $\partial f(x)= f^\prime (x),\ Xf(x)=xf(x)$.

Historically, both Sonine \cite{Sonine} and Nielsen \cite{Nielson}  (see also \cite[\S3.9]{Watson1944}) amongst others from the second half of nineteenth century  studied Bessel functions from the system of partial differential equations \eqref{E:2_PDE_d_delta} instead of \eqref{E:Bessel_eqn}.  Indeed, Nielson distinguished the solutions to the system \eqref{E:2_PDE_d_delta} called \textit{cylindrical functions} against those from the \eqref{E:Bessel_eqn}. These earlier studies were in line with Truesdell's  ``$F$-equation theory"  \cite{Truesdell_1948} as noted by Truesdell himself, that we have already mentioned.

An advantage of our $D$-modules approach to the generating function \eqref{E:bessel_classical_gf_1} is that 
we could modify the $\mathcal{A}_2$-module structures \eqref{E:manifestation_1} and \eqref{E:manifestation_2}, which represent one such \textit{manifestation} that we have mentioned above,  into different ones, say to 
	\begin{equation}\label{E:O_delta_d_0}
		(\partial_1f)(x,t)=\Delta_x f(x, t)=f(x+1,t)-f(x,t),\quad
		(X_1f)(x,t)=xf(x-1,t)
	\end{equation}
	and
	\[
		(\partial_1f)_n(x,t)=\Delta_x f_n(x, t)=f_n(x+1,t)-f_n(x,t),\quad
		(X_1f)_n(x,t)=xf_n(x-1,t)
	\]
	on $\mathcal{O}_{\Delta d}$ and $\mathcal{O}^{\mathbb{Z}}_\Delta$ respectively,
	while keeping the defining properties of the corresponding $\partial_2,\, X_2$ in \eqref{E:manifestation_1} and \eqref{E:manifestation_2} respectively unchanged. 	As a result, we arrive at the commutative diagram
	\begin{equation*}%\label{E:commute-1}
			  \begin{tikzcd} [row sep=huge, column sep=huge]
					    \mathcal{B}_0 \arrow{r}{\mathfrak{j}_\Delta} 
					    \arrow[swap]{d}{\times\E\big[\frac{X_1}{2}\big(X_2-\frac{1}{X_2}\big)\big]}
					    %\arrow[swap]{dr}{\mathfrak{g}_\Delta} 
					    & \mathcal{O}^{\mathbb{Z}}_\Delta 
					    \arrow{d}{\mathfrak{z}_\Delta} \\
     						\widetilde{A}_2 \arrow{r}{\times 1}& \mathcal{O}_{\Delta d}
				  \end{tikzcd}
				\end{equation*} similar to the \eqref{E:commute-0} for this new $\mathcal{A}_2$-module manifestation, which implies the new generating function
	\begin{equation}\label{gf-db-01}
		\Big[\frac{1}{2}(t-\frac{1}{t})+1\Big]^{x}=\sum_{n=-\infty}^{\infty}
		J^{\vartriangle}_{n}(x)\, t^n
		\end{equation} 
		when $\nu=0$.  The full statements will be given in Theorem \ref{T:delta_bessel_gf} and Theorem \ref{T:delta_bessel_0_gf} respectively. Here, both sides of \eqref{gf-db-02} satisfy the system of PDEs
		\begin{equation}\label{E:PDE_gf_dBessel_0}
		\begin{split}
			&y(x+1, t)-y(x, t)+(1/t-t)/2 y(x, t)=0,\\
			&ty_t(x,t)-x(1/t+ t)/2y(x-1, t)=0
		\end{split}
	\end{equation}as images of the \eqref{E:2_more_elements}, with $\nu=0$, which are the difference analogues of those in \eqref{E:PDE_gf_bessel_0} in our $D$-modules interpretation. They are the special cases of \eqref{E:PDE_gf_dBessel} when $\nu=0$. The  $(J^\triangle_{n}(x))$
in  \eqref{gf-db-01} are precisely the difference Bessel functions $(J^\triangle_{\nu+n}(x))$ discovered by Bohner and Cuchta \cite{BC_2017}. The reason that we can derive the infinite sum representation of the generating function on the right-hand side of \eqref{gf-db-01} is because of the following formulae
	\begin{align*}
				x\Delta J^{\Delta}_{n}(x-1)+n J^{\Delta}_{n}(x)-
		xJ^{\Delta}_{n-1}(x-1)&=0,%\label{E:delta_bessel_PDE_1}
		\\
				x\Delta J^{\Delta}_{n}(x-1)-\nu J^{\Delta}_{n}(x)+
		xJ^{\Delta}_{n+1}(x-1)&=0,%\label{E:delta_bessel_PDE_2}
%\\
	%			2\nu J_{\nu}^\Delta(x)-xJ_{\nu-1}^\Delta(x-1)-xJ_{\nu+1}^\Delta(x-1)&=0.\label{E:delta_bessel_3_term}
			\end{align*}derived by brute-force verification in \cite{BC_2017} are manifestation of the \eqref{E:2_elements} in the kernel of the $\mathcal{A}_2$-linear map $\mathcal{B}_0\to \mathcal{O}^{\mathbb{Z}}_\Delta$ as similar to the case of the  \eqref{E:2_PDE_d_delta}. 

If we modify the the above commutator to arbitrary step size $h$, i.e., $\partial f(x)=\triangle_h f(x)=\frac1h\big(f(x+h)-f(x)\big)$ and $Xf(x)=xf(x-h)$, then the generating function  \eqref{gf-db-01} is replaced by
	\begin{equation}\label{E:gf-db_h-01}
		\Big[\frac{h}{2}(t-\frac{1}{t})+1\Big]^{x/h}
		=\sum_{n=-\infty}^{\infty}
		J^{\vartriangle_h}_{n}(x)\, t^n.
	\end{equation}Clearly, we recover the classical formula \eqref{E:bessel_classical_gf_0} from the \eqref{E:gf-db_h-01} after letting $h\to 0$ on both sides of the above formula. See Theorem \ref{T:h_limit} in the Appendix A for more details.  By a similar consideration, the formulae \eqref{E:2_PDE_d_delta} become
	\begin{equation}\label{P:delta_bessel_PDE}
		\begin{split}
			&x\triangle \mathscr{C}^{\Delta}_{\nu+n}(x-1)+(\nu+n) \mathscr{C} ^{\Delta}_{\nu+n}(x)-
		x \mathscr{C} ^{\Delta}_{\nu+n-1}(x-1)=0,%\label{E:delta_bessel_PDE_1}
		\\
			&x\triangle \mathscr{C} ^{\Delta}_{\nu+n}(x-1)-(\nu+n) \mathscr{C} ^{\Delta}_{\nu+n}(x)+
		x \mathscr{C}^{\Delta}_{\nu+n+1}(x-1)=0.%\label{E:delta_bessel_PDE_2}
\\
%				2 &(n+\nu)J_{\nu+n}^\Delta(x)-xJ_{\nu+n-1}^\Delta(x-1)-xJ_{\nu+n+1}^\Delta(x-1)=0.\label{E:delta_bessel_3_term}
		\end{split}
	\end{equation}we the difference manifestation of $\mathcal{A}_2$. If we choose $\mathscr{C}_\nu(x)=J^\Delta_\nu(x)$ in the above two formulae for $J^\Delta_\nu(x)$, then we recover the two corresponding formulae that Bohner and Cuthta derived in \cite{BC_2017}  .  Combining the PDEs \eqref{E:PDE_gf_dBessel} and \eqref{P:delta_bessel_PDE} and a  commutative diagram similar to the last commutative diagram above implies the new generating function \eqref{gf-db-02}.

The difference Bessel functions are not discussed in Nikifornov \textit{et al}. We mention that the difference Bessel functions have not been discussed in recent major works on difference special functions such as \cite{Nikiforov_Suslov_1986} and \cite{Nikiforov_1991}, which mostly focus on orthogonal polynomials. 

%%%%%%%%%%%%%%%%%%%%%%%%%%%%%%%%%%%%%%%%%%%%%
%\vfill\eject 
Before we discuss similar formulae that we have obtained for the new difference analogues of Bessel polynomials and Glaisher's trigonometric type generating functions below, let us first return to the discussion mentioned in the first paragraph about the aspect of our Weyl-algebraic methodology in \textit{solving}\footnote{See Definition \ref{D:soln}.} the two ``PDEs\rq\rq{} \eqref{E:2_elements} in the context of the Bessel module $\mathcal{B}_\nu$.  In fact, the process of a priori computation for the ``generating functions" greatly simplifies the entire process of 
finding the generating function that are classically written within  suitably defined analytic function spaces. This a priori computation is therefore \textit{no less significant} than the conventional approach done in a complex domain.

\medskip

As a by-product of \eqref{gf-db-02} of ``extracting the residue" from the formula \eqref{gf-db-02}, we obtain the integral representation 
	\[
		J^{\Delta}_{\nu}(x)=e^{i\pi\nu} \dfrac{\sin(x-\nu)\pi}{\sin(\pi x)}
		\dfrac{1}{2\pi i}\int_\infty^{(0+)}
		t^{-\nu-1} \Big[ \frac{1}{2}(t-\frac{1}{t})+1\Big]^x dt,\quad \Re(x)<\Re(\nu)
	\]
	where the integration path is a Hankel-type contour (Theorem \ref{T:integral rep of db}). This integral representation is the difference analogue of the classical Schl\"afli-Sonine integral of Corollary \ref{C:Sonine} \cite[p. 176]{Watson1944}. Thus, the classical Schl\"afli-Sonine integral can be considered as ``residue extraction" from the formula \eqref{E:bessel_classical_gf_1}.
\medskip

The rest of this paper will focus on two main variants of the Bessel module of order half $\mathcal{B}_{1/2}$. They are related to classical Glaisher's trigonometric type generating functions and the classical Bessel polynomials. As a result we have also obtained their difference analogues. The solving of these variants of $\mathcal{B}_{1/2}$ (Theorem \ref{T:bessel_poly_gf_map}, Theorem \ref{T:bessel_poly_delta_gf}, Theorem \ref{T:neg_glaisher_cosine}, and Theorem \ref{T:pos_glaisher_sine})
are more complicated in the sense that certain \textit{quadratic extensions} on the base Weyl-algebra are needed.

\subsection{Glaisher's modules}

The treatment of Glaisher's Poisson type generating functions	\cite[p. 140]{Watson1944}
\[
		\sqrt{\frac{2}{\pi x}} \cos\sqrt{x^2-2xt}
		=\sum_{n=0}^\infty J_{n-\frac{1}{2}}(x)\, \frac{t^n}{n!},
%\quad
%		\sqrt{\frac{2}{\pi x}} \sin\sqrt{x^2+2xt}
%		=\sum_{n=0}^\infty J_{-n+\frac{1}{2}}(x)\, \frac{t^n}{n!},
\] which is valid for $2|t|<|x|$, requires a change of characteristics variables 
%\begin{equation}\label{E:classical_glaisher}
%		\begin{array}{ll}
%			\Xi_1=\partial_1, & W_1=X_1,\\
%	 		\Xi_2={1}/{X_2}, & W_2=X_2\partial_2 X_2
%		\end{array}
%	\end{equation}
to \eqref{E:2_elements} with $\nu=-\frac12$ different from that for obtaining the \eqref{E:bessel_exp}, so that the alternative $\mathcal{A}_2$-module, called the negative Glaisher's module is given by
	\[
		\mathcal{G}_{-}=
		\dfrac{\mathcal{A}_{2}}
		{\mathcal{A}_{2}
	\big(W_1 \Xi_1+(W_2\Xi_2-\frac{1}{2})- W_1\Xi_2^{-1}\big)
			+\mathcal{A}_{2}\big(W_1 \Xi_1-(W_2 \Xi_2-\frac{1}{2})+ W_1 \Xi_2\big)},
	\]in which the $\Xi_1, W_1$ and $\Xi_2,\, W_2$ play the roles of $\partial_1,\, X_1$ and $\partial_2,\, X_2$ respectively, as defined in \eqref{E:comm}.

	As an example, after a ``further change of variables" and solving the resulting PDEs finally results in the following difference analogue of Glaisher's generating functions for difference Bessel functions with half-integer orders (Theorem \ref{T:Delta_glaisher_gf}):
%	\begin{equation}\label{E:commute-half besssel-2}
%	\begin{equation}\label{E: gf-Delta_glaisher_cos}
	\[
		\sqrt{\frac{2}{\pi}}\dfrac{e^{-i\pi x}} {2i \sin(\pi x)}
		\int_{-\infty}^{(0+)} \frac{e^{\lambda}
		(-\lambda)^{-x-1}\lambda^{-\frac{1}{2}} 
		 \cos\sqrt{\lambda ^2-2\lambda t}}{\Gamma(-x)}\,  d\lambda
		=\sum_{n=0}^{\infty}J^{\Delta}_{n-\frac{1}{2}}(x)\frac{t^n}{n!},
%	\end{equation}
%	\]and
	%\begin{equation}\label{E: gf-Delta_glaisher_sin}
%	\[
%		\sqrt{\frac{2}{\pi}}\dfrac{e^{-i\pi x}} {2i \sin(\pi x)}
%		\int_{-\infty}^{(0+)} \frac{e^{\lambda}
%		(-\lambda)^{-x-1}\lambda^{-\frac{1}{2}} 
%		 \sin\sqrt{\lambda ^2+2\lambda t}}{\Gamma(-x)}\,  d\lambda
%		=\sum_{n=0}^{\infty}J^{\Delta}_{-n+\frac{1}{2}}(x)\frac{t^n}{n!}
%	\end{equation}
	\]
valid on the whole $\mathbb{C}\times\mathbb{C}$. 	In particular, the generating function is a solution to the system of delay-differential equations\footnote{These equations are generally called PDEs in a generalised sense in this paper.}  (Theorem \ref{T:difference_glaisher_gf})
			\begin{equation}\label{E:PDeltaE_neg_glaisher}
				\begin{split}
				& tf_{tt}(x,\, t)+\big(x+1/2\big)f_t(x,\, t)-xf_t(x-1,\, t)-xf(x-1,\, t)=0,\\
				&  tf_{t}(x,\, t)-xf_t(x-1,\, t)+xf(x-1,\, t)-\big(x+1/2\big) f(x,\, t)=0.
				\end{split}
			\end{equation}
\medskip	

\subsection{Bessel polynomial modules}
Next we consider half Bessel-module $\mathcal{B}_{1/2}$ which is more natural  to treat the classical reverse polynomials $\theta_n(x)$ instead of the $y_n(x)=x^n\theta_n(1/x)$. If $X\partial$ is replaced by $X\partial-X-1/2$ in \eqref{E:2_elements}, then the  $\mathcal{B}_{1/2}$ is replaced by the \textit{reverse Bessel polynomial module}
	\[
%		\begin{equation}\label{E:Theta}
\Theta=\dfrac{\mathcal{A}_2}{\mathcal{A}_2[\partial_1-1+X_1/\partial_2]+\mathcal{A}_2[X_1\partial_1-2X_2\partial_2-1-X_1+\partial_2]}.
%	\end{equation}
	\]such that the two generators $\partial_1-1+X_1/\partial_2$ and $X_1\partial_1-2X_2\partial_2-1-X_1+\partial_2$ would generate a third element
	$X_1\partial_1^2-2(X_1+X_2\partial_2)\partial_1+2X_2\partial_2$ which is a Weyl-algebraic analogue of the classical Bessel polynomial equation. In particular, we have derived the difference equation (Theorem \ref{T:difference_Bessel_Poly_DD})
	\begin{equation}
		(x-2n)\,\theta_n^\Delta(x+1)-4(x-n)\,\theta_n^\Delta(x)+3x\, \theta_n^\Delta (x-1)=0,
	\end{equation} which is satisfied by \textit{difference reverse  Bessel (Newton) polynomial of degree} $n$
%	\begin{equation}\label{E:delta_Bessel_poly}
	\[
		\theta^\Delta_n(x) :=\sum_{k=0}^n \frac{(n+k)!}{2^k (n-k)!\, k!}\, (x)_{n-k}.
	\]	
	In a similar fashion, the two generators $\partial_1-1+X_1/\partial_2$ and  $X_1\partial_1-2X_2\partial_2-1-X_1+\partial_2$ implies two delay-difference equations for  difference reversed Bessel polynomials for each $n$ (Theorem \ref{T:difference_Bessel_Poly_PDE}):
	\[
		\theta^\Delta_n(x+1)-2\theta^\Delta_n(x)+x\theta^\Delta_{n-1}(x-1)=0,
	\]
	and
	\[
		 \theta^\Delta_{n+1}(x)+(x-2n-1)\theta^\Delta_n(x)-2x\, \theta^\Delta_n(x-1)=0.
	\]
A ``characteristic change of variables" to the module $\Theta$ happens to be $\rho^2=1-2X_2$ and we formulate a ``quadratic extension of $\Theta$" in Theorem \ref{T:rev_bessel_poly_gf}
	\[
			\Theta(\rho):=\dfrac{\mathcal{A}_2(\rho)}{\mathcal{A}_2(\rho)[-\rho^2\partial_2^2
			+3\partial_2+X_1^2]
			+\mathcal{A}_2(\rho)[X_1\partial_1+\rho^2\partial_2-1-X_1]}.
		\]so that the commutative diagram of $\mathcal{A}_2$-linear maps enable us to derive a generating function for the difference reverse Bessel polynomials
%		\begin{equation}\label{E:bessel_difference_poly_gf}
	\[
			\frac{e^{-i\pi x}}{2\pi i\sin\pi x\Gamma(-x)}\int_{-\infty}^{(0+)}e^\lambda(-\lambda)^{-x-1} \frac{1}{\sqrt{1-2t}} \exp\big[\lambda(1-\sqrt{1-2t})\big]\,d\lambda=\sum_{n=0}^\infty
			\theta_n^\Delta (x)\, \frac{t^n}{n!}
%		\end{equation}
	\]that is valid  in $\mathbb{C}\times B(0,\, \frac12)$ in Theorem \ref{T:bessel_poly_delta_gf}. The generating function for the difference reverse Bessel polynomials satisfies the system of delay-differential equations (Theorem \ref{E:delta_bessel_poly_PDE})
	\begin{equation}\label{E:delta_rev_bessel_pde}
		\begin{split}
			&f_t(x+1,\, t)-2f_t(x,\, t)+xf(x-1,\, t)=0,\\
			&(1-2t) f_t(x,t)+(x-1)f(x, t)-2xf(x-1, t)=0. %\label{E:delta_rev_bessel_pde_2}
		\end{split}
		\end{equation}

This paper is organised as follows. Section  \ref{S:preliminaries}  explains some examples of  
 $D$-modules structures on certain spaces of analytic functions  and spaces of sequences of analytic functions  $\mathcal{O}^{\mathbb{Z}}$, as well as  $D$-linear maps between them. A list summarising these entities is given in Appendix \ref{SS:holo_modules}. The remaining parts of the \S\ref{S:preliminaries} gives a rudiment of holonomic $D$-modules theory needed in this article. A kind of generalised Borel resummation is also introduced. \S\ref{S:bessel} introduces the Bessel modules $\mathcal{B}_\nu$ for arbitrary $\nu$ and the method of characteristic in solving  a system of PDEs. Before discussing about the Bessel modules, we also discuss analogues of ``elementary functions" such as the Weyl-exponentials, Weyl-trigonometry, and Weyl-Bessel. 
 This section also contains the statements and proofs of generating functions for the classical Bessel functions and difference Bessel functions, plus some immediate applications of these results to yield difference analogue of  the classical  Schl\"afli-Sonine integral representation from classical Bessel functions to difference Bessel functions. In section \S\ref{S:half_bessel_I} we modify the Bessel module $\mathcal{B}_{1/2}$ to the ``\textit{reverse Bessel module}" $\Theta$ by the introducing an integrating factor. We work out a characteristic for $\Theta$  as well as  a variant $\mathcal{Y}$ corresponding to the usual Bessel polynomials $y_n$. These characteristics will be applied to find (Poisson type)  generating functions in due course. As applications of the $D$-modules approach, we derive, in a \textit{uniform way}, the well-known three-term recursion formulae for Bessel polynomials as well as the corresponding ones for the difference reverse Bessel polynomials.  Section \S\ref{S:Newton} introduces the \textit{Newton transform},
 a kind of generalised Mellin transform but with very different purpose as a $D$-linear map  from $\mathcal{O}_d$ to $\mathcal{O}_\Delta$, namely the $D$-modules of analytic functions endowed with a  forward difference operator. The Newton transformation serves as an effective tool to obtain generating functions for difference Bessel functions. Utilising results obtained from \S\ref{S:half_bessel_I}, we 
  give detailed proofs of generating functions for the reverse Bessel polynomials as well as their difference analogues in   \S\ref{SS:delta_reverse_bessel_poly}. Section \S\ref{S:half_bessel_II} take the half-Bessel modules from \S\ref{S:half_bessel_I} a step further to study trigonometric-type generating functions for the Bessel functions first obtained by Glaisher \cite{Watson1944}. Before the concluding section \S\ref{S:conclusion},  \S\ref{S:discussion}
 discusses our main results against previous results of Truesdell amongst others with some historical notes. In particular, it focus the philosophy that Truesdell pursued, and similarly of Wilf and Zeilberger, to find general methods of discovery of what kinds of formulae should exist rather than by ad hoc manipulations of formulae, a viewpoint shared by this study. 
Appendix \ref{A:proofs} contains proofs of various Theorems and Propositions stated but not fully verified in all aspects in the main text. Appendix \ref{Append:arbitrary} extends the study of the manifestation of the Bessel modules $\mathcal{B}_\nu$ to $\mathcal{O}_{\Delta_h}$, that is finite difference with arbitrary step size $h$. In the case of $\mathcal{B}_0$, the generating function derived unifies the case when $h=1$ for the difference Bessel functions and the classical Bessel functions when $h\to 0$.  This presents a further unification of the generating functions of the classical Bessel functions and that of difference Bessel functions of arbitrary step sizes.  Appendix \S\ref{A:list} gives lists of $D$-modules used, transmutation formulae, the systems of PDEs encountered for various generating functions defined in their appropriate analytic functions space of two variables. 
 
\section{Preliminaries}\label{S:preliminaries}
\subsection{Linear maps between $D$-modules} 

\begin{definition}[\cite{CCT3}]  Let $n\in \mathbb{N}.$ The Weyl-algebra $\mathcal{A}_n$ is the $\mathbb{C}$-algebra with $2n$ generators 
	$X_1, \ldots, X_n, \partial_1, \ldots, \partial_n$ 
	subject to the relations
	\[
		[X_i, X_j]=0;\quad[\partial_i, \partial_j]=0; \quad[\partial_i, X_j]=\delta_{ij}.
	\]
\end{definition}
In this paper, we denote $\mathcal{A}_1=\mathbb{C} \langle X, \partial\rangle$ or by $\mathcal{A}$  when this is clear from the context.
\medskip

We gather here preliminary material about  $D$-module that we shall employ to exhibit a unified theory for the classical Bessel functions and generating functions for the recently discovered difference Bessel functions by Bohner and Cuchta  \cite{BC_2017} in the next sections.  More specifically, amongst the generating functions studied include	
%	\begin{enumerate}
	\begin{itemize}
		\item extensions of generating functions of various kinds for the classical Bessel functions of arbitrary order $\nu$;
%		\item a new quasi-generating function of interpolation type for classical Bessel functions;
		\item new generating functions for difference Bessel functions \cite{BC_2017}; 
		\item new generating functions of Glaisher (Poisson)-type for difference Bessel functions;
		\item new generating functions for new difference Bessel polynomials.
%		\item a new generating function of interpolation type for difference Bessel functions.
%	\end{enumerate}
	\end{itemize}	
\medskip

The primary concern of this paper is on ${D}$-linear maps defined on the Bessel modules and their induced generating functions to be introduced in the following two sections. Hence some generating functions found for difference Bessel functions may not converge in the same way as those of their classical counterparts.  
%according to their algebraic and singularity structures. Thus we adopt standard terminology from Poincar\'e's theory of asymptotic expansions \cite[Chap. 2]{Copson_1965} to describe those divergent infinite sums derived from the study of the corresponding generating functions. 

%\begin{definition} A sequence of function $\phi_n(t)$ is called an asymptotic sequence as $x\to x_0$ if there is a neighbourhood of $x_0$ such that none of the functions vanish except at the point $x_0$ and if
%	\[
%		\phi_{n+1}=o(\phi_n(x)),\quad x\to x_0.
%	\]We say that the formal series
%	\[
%		\sum_{-\infty}^\infty a_n\phi(t)
%	\]is an \textit{asymptotic expansion} of $f(t)$ in Poincar\'e's sense, with respect to the asymptotic sequence $\{\phi_n(t)\}$ if, for every integer $m$,
%	\[
%		f(t)-\sum_{-(m-1)}^{m-1}a_n\phi_n(t)=\phi_m(t),\quad t\to t_0.
%	\]
%\end{definition} 
\medskip

The principle of obtaining these generating functions behind is based on the finite dimensionality  of the space of $D$-linear maps from  $\mathcal{B}_\nu$ to different $\mathcal{A}_2$-modules. The first step is the introduction of some examples of $\mathcal{A}_n$-modules and linear maps between them which are frequently used in this article.

\begin{definition} We denote by $\mathcal{A}(1/X)$  the Weyl-algebra $\mathbb{C}\langle\partial,X\rangle (1/X)=
\mathbb{C}\langle\partial,X,{1}/{X}\rangle$ modulo the aforementioned relations.
\end{definition}
\medskip

\begin{example}\label{Eg:O_deleted_d}
Let $\mathcal{O}$ be the space of analytic functions on a certain domain. Then the  $\mathcal{O}$  becomes a left $\mathcal{A}$-module denoted by $\mathcal{O}_d$, if one endows it with the structure
	\[
		\begin{array}{l}
		(\partial f)(t)=f^\prime (t),\\
		(Xf)(t)=tf(t)\quad\mbox{for all }f\in\mathcal{O}\mbox{ and }t\mbox{ in this domain}.
		\end{array}
\]Similarly, let $\mathcal{O}_2$ be the space of analytic functions on a certain domain in $\mathbb{C}\times \mathbb{C}$. Then the   $\mathcal{O}_2$  becomes a left $\mathcal{A}_2$-module denoted by $\mathcal{O}_{dd}$, if one endows it with the structure
	\begin{equation}\label{E:O_dd_endow}
			\begin{split}
				(\partial_1 f)(x,\, t) &=f_x (x,\, t), \quad (X_1f)(x,\, t)=xf(x,\, t),\\
				(\partial_2 f)(x,\, t) &=f_t(x,\, t), \quad (X_2f)(x,\, t)=tf(x,\, t).
			\end{split}
	\end{equation}
\end{example}
\medskip

\begin{example}\label{Eg:O_delta}
Let $\mathcal{O}$ be the space of analytic functions. Then the $\mathcal{O}$ becomes a left $\mathcal{A}$-module, denoted by $\mathcal{O}_\Delta$, if one endows  it 
with the structure
	\[
		\begin{split}
		&(\partial f)(t) =f(t+1)-f(t), \\
		&(Xf)(t)=tf(t-1).
		\end{split}
	\]
\end{example}
\medskip

\begin{example}\label{Eg:two-seq-numbers} Let $\mathbb{C}^\mathbb{Z}$ be the space of all functions $\mathbb{Z}\to\mathbb{C}$ (bilateral sequences). Let 
	\[
		a=(a_n)=(\cdots, a_{-2},\, a_{-1},\, a_0,\, a_1,\, a_2,\, \cdots)
	\]denote a bilateral sequence where $a_k\in\mathbb{C}$.  Then the $\mathbb{C}^\mathbb{Z}$ becomes a left $\mathcal{A}$-module by defining
\[
\begin{array}{l}
(\partial a)_n=(n+1)a_{n+1}\\
(Xa)_n=a_{n-1}\quad\mbox{ for all bilateral sequences }(a_n).
\end{array}
\]
Various growth conditions can be imposed to such bilateral sequences. 
\end{example}

\subsubsection{$z$-transform I}
\begin{theorem}[\textbf{(Generating function of numeric sequences)}]\label{T:z-transform-1} 
	Suppose that $\mathbb{C}^\mathbb{Z}$ consists only of sequences $\{a_n\}$ satisfying that
	\[
		0\le\limsup_{n\to\infty}{|a_{-n}|^{1/n}}<\frac{1}{\limsup_{n\to\infty}{|a_n|^{1/n}}}\le+\infty,
	\]
	and let $\mathcal{O}$ be the space of analytic functions defined on a certain annulus centred at $0$. Then the \textit{z-transform}
			\[
		\begin{array}{rcl}
		\mathfrak{z}:\mathbb{C}^\mathbb{Z}&\longrightarrow  &\mathcal{O}_d,\\
		(a_n)& \longmapsto & \displaystyle\sum_{n=-\infty}^\infty a_nt^n.
		\end{array}
	\]
	is left $\mathcal{A}$-linear.
\end{theorem}
\medskip

\begin{proof} The fact that the map $\mathfrak{z}$ is $\mathbb{C}$-linear is obvious. The verification  that both 
	\[
		\mathfrak{z}\, \partial (a_n)=\partial\,  \mathfrak{z} (a_n),
		\quad \mbox{and}\quad 
		\mathfrak{z}\, X (a_n)=X\, \mathfrak{z} (a_n)
	\] hold are  routine, with the understanding that the maps $\partial$ and $X$ are interpreted appropriately in $\mathbb{C}^\mathbb{Z}$ and $\mathcal{O}_d$ respectively. Hence the map $\mathfrak{z}$ is $\mathcal{A}$-linear. Clearly, it is both injective and surjective.
\end{proof}
\medskip

In fact, Theorem~\ref{T:z-transform-1} can be extended to Theorem~\ref{T:z-transform-1}$'$ below, which handles a larger class of sequences $\{a_n\}$ so that $\sum_{n=-\infty}^\infty a_nt^n$ may not converge.
\medskip

\begin{definition}[\textbf{(Borel resummation)}]\label{D:borel} Suppose that $\{a_n\}\in\mathbb{C}^{\mathbb{Z}}$ satisfies both
	\begin{equation}\label{E:G1}
		\limsup_{n\to\infty}{|a_n|^{1/n}}<+\infty \quad \mbox{and} \quad \limsup_{n\to\infty}{\left|\frac{a_{-n}}{n!}\right|^{1/n}}<+\infty.
	\end{equation}
 Then the \textit{Borel resummation} of the Laurent series $\sum_{n=-\infty}^\infty {a_nt^n}$ is defined by taking the Laplace transform of the Borel transform of its principal part, i.e.,
	\begin{equation}\label{E:BL}
	\int_0^{+\infty}{e^{-st}\left(\sum_{n=0}^{\infty}{\frac{a_{-n-1}}{n!}}s^n\right)\,ds} + \sum_{n=0}^\infty {a_nt^n}.
	\end{equation}
\end{definition}
\medskip

\begin{remark}\label{R:gervey}
A power series or Laurent series whose sequence of coefficients $\{a_n\}$ satisfy that \[\limsup_{n\to\infty}{\left|\frac{a_n}{n!}\right|^{1/n}}<+\infty\] is said to be \textit{$1$-Gevrey}. Note that the Borel resummation of a $1$-Gevrey Laurent series always exists as an analytic function, and the Borel resummation of a convergent Laurent series is just its ordinary sum, which is again an analytic function. We refer the reader to \cite[Part II]{Mitschi_Sauzin_2016}, \cite{Shawyer_Watson_1994} and \cite{Delabaere_Pham_1999} for general discussions of $1$-Gevrey or $m$-Gevrey series and their Borel-resummation.
\end{remark}
%\medskip

\subsubsection{$z$-transform II}
\begin{theoremprime}{T:z-transform-1}[\textbf{(Generating function of numeric sequences)}]
	Suppose that $\mathbb{C}^\mathbb{Z}$ consists only of sequences $\{a_n\}$ satisfying both
	\[
		\limsup_{n\to\infty}{|a_n|^{1/n}}<+\infty \quad \mbox{and} \quad \limsup_{n\to\infty}{\left|\frac{a_{-n}}{n!}\right|^{1/n}}<+\infty,
	\]
	and let $\mathcal{O}$ be the space of analytic functions defined on a certain region. Then the \textit{z-transform}
			\[
		\begin{array}{rcl}
		\mathfrak{z}:\mathbb{C}^\mathbb{Z}&\longrightarrow  &\mathcal{O}_d,\\
		(a_n)& \longmapsto & \mathfrak{B} \displaystyle\sum_{n=-\infty}^\infty a _nt^n,
		\end{array}
	\]where the $\mathfrak{B} \sum_{n=-\infty}^\infty a _nt^n$ denotes the Borel-resummation of $\sum_{n=-\infty}^\infty a _nt^n$, is left $\mathcal{A}$-linear.
\end{theoremprime}
\medskip

\begin{proof} Since the Borel and Laplace transforms are both left $\mathcal{A}$-linear, the Borel resummation is left $\mathcal{A}$-linear also.
\end{proof}
\medskip

\begin{example}\label{Eg:O_delta_d} Let $\mathcal{O}_2$ be a space of  analytic functions in two variables. Then $\mathcal{O}_2$ becomes a left $\mathcal{A}_2$-module by
	\begin{equation}\label{E:O_delta_d}
%		\begin{split}
%			\partial_1f(x,\, t) =xf_x(x,\, t), \quad X_1 f(x,\, t)=xf(x,\, t);\\
%			\partial_2f(x,\, t)&=f(x,\, t+1)-f(x,\, t),\quad X_2f(x,\, t)=tf(x,\, t-1).
%		\end{split}
		\begin{array}{ll}
			\partial_1f(x,\, t)=f(x+1,\, t)-f(x,\, t),   & X_1f(x,\, t)=xf(x-1,\, t);\\
			\partial_2f(x,\, t) =f_t(x,\, t), 		& X_2 f(x,\, t)=tf(x,\, t).
		\end{array}	
	\end{equation}
Such a left $\mathcal{A}_2$-module is denoted by $\mathcal{O}_{\Delta d}$.
\end{example}
\medskip

\begin{remark} In addition to the above notation $\mathcal{O}_{\Delta d}$, we shall liberally adopt notations such as $\mathcal{O}_{d d}$, $\mathcal{O}_{d \Delta}$, $\mathcal{O}_{\Delta\Delta}$  to represent different left $\mathcal{A}_2$-module structures endowed on $\mathcal{O}_2$ with respect to the corresponding operators.
\end{remark}
\medskip

We next extend the idea of space of bilateral sequences of numbers $\mathbb{C}^\mathbb{Z}$ in Example \ref{Eg:two-seq-numbers} to the space of bilateral sequences of analytic functions $\mathcal{O}^\mathbb{Z}$.
 \medskip
 
\begin{example}\label{Eg:two-seq-functions} Let $\mathcal{O}^\mathbb{Z}$ be the space of all  bilateral sequences of analytic functions with appropriately imposed growth restriction. Let 
	\[
		(f_n)=(\cdots, \, f_{-1},\, f_0,\, f_1,\,  \cdots)
	\]denote a bilateral infinite sequence of analytic functions.  Then $\mathcal{O}^\mathbb{Z}$ becomes a left $\mathcal{A}_2$-module by
	\[
		\begin{array}{ll}
		(\partial_1 f)_n(x)=f^\prime_n(x), &  (X_1 f)_n(x)=xf_n(x)\\
		(\partial_2 f)_n(x)=(n+1)f_{n+1}(x),& (X_2 f)_n(x)=f_{n-1}(x),
		\end{array}
	\]for all $n$ and for all $x$. This left $\mathcal{A}_2$-module is denoted by $\mathcal{O}_d^\mathbb{Z}$.
\end{example}
\medskip

\begin{example}\label{Eg:two-seq-functions_2} Let $\mathcal{O}_\Delta^\mathbb{Z}$ be the space of all  bilateral sequences of analytic functions with an appropriately imposed growth restriction. Let 
	\[
		(f_n)=(\cdots, \, f_{-1},\, f_0,\, f_1,\,  \cdots)
	\]denote a bilateral infinite sequence of $f_k\in\mathcal{O}$.  Then $\mathcal{O}^\mathbb{Z}$ becomes a left $\mathcal{A}_2$-module, by
	\begin{equation}\label{E:two-seq-functions_2}
		\begin{array}{ll}
		(\partial_1 f)_n(x)=\Delta f_n(x)=f_n(x+1)-f_n(x), &  (X_1 f)_n(x)=xf_n(x-1)\\
		(\partial_2 f)_n(x)=(n+1)f_{n+1}(x),& (X_2 f)_n(x)=f_{n-1}(x),
		\end{array}
	\end{equation}for all bilateral sequences of analytic functions $(f_n)$.  This left $\mathcal{A}_2$-module is denoted by $\mathcal{O}_\Delta^\mathbb{Z}$.
\end{example}
\medskip

Similar to  Theorem \ref{T:z-transform-1}, we have the following:
\medskip

%\subsubsection{$z$-transform III}
\begin{theorem}[\textbf{(Generating function of sequences  in $\mathcal{O}_d$)}]%\label{T:Z-trans} 
\label{T:z-transform-2} 
The following \textit{z-transform}
	\begin{equation}\label{E:z-transform-2} 
		\begin{array}{rcl}
		\mathfrak{z}:\mathcal{O}^\mathbb{Z}_d &\longrightarrow &\mathcal{O}_{dd},\\
		(f_n) &\longmapsto&\mathfrak{B}\displaystyle\sum_{n=-\infty}^\infty f_n(x)\, t^n
		\end{array}
	\end{equation}is left $\mathcal{A}_2$-linear.
\end{theorem}

\begin{theorem}[\textbf{(Generating function of  sequences  in $\mathcal{O}_\Delta$)}]\label{T:z-transform-3} 
The following \textit{z-transform}
	\begin{equation}\label{E:z-transform-3} 
		\begin{array}{rcl}
		\mathfrak{z}_\triangle:\mathcal{O}^\mathbb{Z}_\Delta &\longrightarrow &\mathcal{O}_{\Delta d},\\
		(f_n) &\longmapsto&\mathfrak{B}\displaystyle\sum_{n=-\infty}^\infty f_n(x)\, t^n
		\end{array}
	\end{equation}is left $\mathcal{A}_2$-linear.	
\end{theorem}
\medskip

One needs to make a slight modification of defining rules for $\mathcal{A}$-module in order to study one-sided (ordinary) sequences. The motivation of this rule will be explained in Remark~\ref{R:poisson} after we introduce the Poisson transform.

\begin{example}\label{Eg:seq-functions_poisson}
Let $\mathbb{C}^\mathbb{N}$ be the space of sequences 
	\[	
		a=(a_n)=(a_1,\, a_2,\, a_3,\, \cdots)
	\] of complex numbers. It is equipped with the structure of a left $\mathcal{A}$-module by defining
\[
\begin{array}{l}
(Xa)_n=na_{n-1}\\
(\partial a)_n=a_{n+1}\quad\mbox{for all sequences }(a_n).
\end{array}
\]
Similarly, let $\mathcal{O}^{\mathbb{N}_0}$ be the space of all sequences of analytic functions with appropriately imposed growth restriction. Let 
	\[
		(f_n)=(f_{0},\, f_1,\, f_2,\,  \cdots)
	\]denote an infinite sequence of $f_k\in\mathcal{O}$.  Then $\mathcal{O}^{\mathbb{N}_0}$ becomes a left $\mathcal{A}_2$-module by
	\begin{equation}\label{E:seq-functions_poisson}
		\begin{array}{ll}
		(\partial_1 f)_n(x)=f^\prime_n(x), &  (X_1 f)_n(x)=xf_n(x)\\
		(\partial_2 f)_n(x)=f_{n+1}(x),& (X_2 f)_n(x)=nf_{n-1}(x),
		\end{array}
	\end{equation}for all sequences of analytic functions $(f_n)$. This left $\mathcal{A}_2$-module is usually denoted as $\mathcal{O}_d^{\mathbb{N}_0}$.
\end{example}

\begin{example}\label{Eg:seq-functions_poisson_delta}
 Let $\mathcal{O}^{\mathbb{N}_0}$ be the space of all sequences of analytic functions with appropriately imposed growth restriction. Let 
	\[
		(f_n)=(f_0,\, f_1,\, f_2,\,  \cdots)
	\]denote an infinite sequence of $f_k\in\mathcal{O}$.  Then $\mathcal{O}^{\mathbb{N}_0}$ becomes a left $\mathcal{A}_2$-module by
	\begin{equation}\label{E:seq-functions_poisson_delta}
		\begin{array}{ll}
		(\partial_1 f)_n(x)=f_n(x+1)-f_n(x), &  (X_1 f)_n(x)=xf_n(x-1)\\
		(\partial_2 f)_n(x)=f_{n+1}(x),& (X_2 f)_n(x)=nf_{n-1}(x),
		\end{array}
	\end{equation}for all  sequences of analytic functions $(f_n)$. This left $\mathcal{A}_2$-module is usually denoted as $\mathcal{O}_\Delta^{\mathbb{N}_0}$.
\end{example}
\medskip

We first demonstrate how we use the theory set up so far to derive a classical generating function for Bessel functions from the Bessel module $\mathcal{B}_\nu$ defined in Definition \ref{E:bessel_mod_gf} which is isomorphic to Definition
 \ref{E:bessel_mod}.
%\medskip

A table summarising all the $D$-module structures relevant to this article is included in the Appendix \ref{SS:common_D_list}.

\subsection{Integrability criteria}\label{SS:integrable}

This section deals with the idea of holonomic modules, which can be found in \cite{coutinho}. A quick revision is included here so that this article is self-contained.
For each multi-index $\alpha=(\alpha_1,\ldots,\alpha_n)\in\mathbb{N}^n$, $X_1^{\alpha_1}\cdots X_n^{\alpha_n}$ will be abbreviated as $X^\alpha$. A similar abbreviation is adopted for the $\partial$'s. It is also a standard notation that $|\alpha|=|(\alpha_1,\ldots,\alpha_n)|=\sum_j\alpha_j$.

\medskip

\begin{definition}
The \textit{Bernstein filtration} of $\mathcal{A}_n$ is the filtration $F^0\mathcal{A}_n\subset F^1\mathcal{A}_1\subset\cdots\subset\mathcal{A}_n$ defined by
\[
F^k\mathcal{A}_n=\mbox{the }\mathbb{C}\mbox{-vector space spanned by }\{X^\alpha\partial^\beta:|\alpha|+|\beta|\leq k\}.
\]
\end{definition}

\medskip

\begin{remark}
The Bernstein filtration enjoys the following properties.
\begin{itemize}
\item
$F^i\mathcal{A}_n\cdot F^j\mathcal{A}_n\subset F^{i+j}\mathcal{A}_n$,
\item
$[F^i\mathcal{A}_n,F^j\mathcal{A}_n]\subset F^{i+j-1}\mathcal{A}_n$,
\item
the $k^{\footnotesize\mbox{th}}-$graded piece $\mathrm{Gr}^k\mathcal{A}_n=F^k\mathcal{A}_n/F^{k+1}\mathcal{A}_n$ has finite $\mathbb{C}$-dimension for all $k$.
\end{itemize}
In fact, in the following treatment, the Bernstein filtration can be replaced by any other filtration having the these properties.
\end{remark}

\medskip

\begin{lemma}
The graded ring of the Bernstein filtration is a commutative $\mathbb{C}$-algebra. Indeed,
\[
\mathrm{Gr}\,\mathcal{A}_n=\bigoplus_k\mathrm{Gr}^k\mathcal{A}_n\cong\mathbb{C}[X_1,\ldots,X_n,\partial_1,\ldots,\partial_n].
\]
\end{lemma}

\medskip

\begin{definition}
Let $M$ be a left $\mathcal{A}_n$-module. A \textit{good filtration} of $M$ is a filtration $\Gamma$ which satisfies
\begin{itemize}
\item
$F^i\mathcal{A}_n\cdot \Gamma^jM\subset\Gamma^{i+j}M$,
\item
there exists $J\in\mathbb{N}$ such that
$F^i\mathcal{A}_n\cdot \Gamma^jM=\Gamma^{i+j}M$ for all $i$, $j>J$.
\end{itemize}
\end{definition}

\medskip

\begin{lemma}
Let $M$ be a finitely generated left $\mathcal{A}_n$-module and $\Gamma$ be a good filtration of $M$. Then the graded ring of $\Gamma$, that is,
\[
\mathrm{Gr}\,M=\bigoplus_k{\Gamma^kM/\Gamma^{k+1}M},
\]
is a commutative algebra over the commutative ring $\mathrm{Gr}\,\mathcal{A}_n$.
\end{lemma}

\medskip

\begin{theorem}[(Hilbert)]
Let $M$ be a finitely generated left $\mathcal{A}_n$-module equipped with a good filtration $\Gamma$. Then there exists a polynomial $P$ such that

		\[
			P(k)=\dim_{\mathbb{C}}\Gamma^kM\quad\mbox{for every sufficiently large }k.
		\]
\end{theorem}

\medskip

\begin{definition}\label{D:hilbert}
The polynomial $P$ in the last theorem is called the \textit{Hilbert polynomial} of $M$. (It turns out that $P$ is independent of the choice of the good filtration chosen.). If $P$ is of degree $d$ and its leading coefficient is $m/d!$, then we call $d$ the \textit{dimension} of $M$ and $m$ the \textit{multiplicity} of $M$.
\end{definition}

\medskip

\begin{example}
Consider $\mathcal{A}_n$ as a left $\mathcal{A}_n$-module and the Bernstein filtration is a good one. Then,
\[
\dim_{\mathbb{C}}\Gamma^k\mathcal{A}_n=\binom{k+2n}{2n}
\]
so that the dimension of $\mathcal{A}_n$ is $2n$ and its multiplicity is $1$.
\end{example}

\medskip

\begin{example}\label{E:prime_integrable_eg}
Consider
\[
M=\mathcal{A}_n/\mathcal{A}_n\partial_1+\cdots+\mathcal{A}_n\partial_n,
\]
which inherits a good filtration $\Gamma$ from the Bernstein filtration. Then,
\[
\dim_{\mathbb{C}}\Gamma^kM=\binom{k+n}{n}
\]
so that the dimension of $M$ is $n$ and its multiplicity is $1$.
\end{example}
\medskip

\begin{example}\label{E:prime_integrable_eg_2}
Consider
\[
M=\mathcal{A}_2/\mathcal{A}_2\partial_1^2+\mathcal{A}_2\partial_2,
\]
which inherits a good filtration $\Gamma$ from the Bernstein filtration. Then,
\[
\dim_{\mathbb{C}}\Gamma^kM=(k+1)^2
\]
so that the dimension of $M$ and its multiplicity are both  equal to $2$.
\end{example}

\medskip

\begin{theorem}[(\textbf{Bernstein inequality})]\label{T:bernstein}
For each non-trivial finitely generated left $\mathcal{A}_n$-module $M$, its dimension $d(M)$ satisfies
\[
n\leq d(M)\leq 2n.
\]
\end{theorem}

\medskip

\begin{definition}
A finitely generated left $\mathcal{A}_n$-module is \textit{holonomic} if it is nontrivial and its dimension is $n$.
\end{definition}

\medskip

\begin{example}
$\mathcal{A}_n$ is not a holonomic left $\mathcal{A}_n$-module. However, $M=\mathcal{A}_n/(\mathcal{A}_n\partial_1+\cdots+\mathcal{A}_n\partial_n)$ is a holonomic left $\mathcal{A}_n$-module.
\end{example}

\medskip

\begin{example}[(\textbf{Integrability})] \label{EG:integrable}
Let $p(X_1,X_2)$, $q(X_1,X_2)\in\mathbb{C}[X_1,X_2]$ be polynomials in two variables. Denote their formal partial derivatives by subscripts. Consider the left $\mathcal{A}_2$-module
\[
M=\dfrac{\mathcal{A}_2}{\mathcal{A}_2(\partial_1+p(X_1,X_2))+\mathcal{A}_2(\partial_2+q(X_1,X_2))}.
\]
\begin{itemize}
\item If $p_2(X_1,X_2)\neq q_1(X_1,X_2)$, then $M$ 
is trivial and is therefore not holonomic.
\item If $p_2(X_1,X_2)=q_1(X_1,X_2)$, then $M$ can be analyzed in a way similar to the previous examples, and it turns out that $M$ is holonomic with its (Hilbert) dimension two and multiplicity one in this case as described in Example \ref{E:prime_integrable_eg}.
\end{itemize}
The condition $p_2(X_1,X_2)=q_1(X_1,X_2)$ is known as an \textit{integrability condition}. Integrability conditions give rise to holonomic $\mathcal{A}_n$-modules, which are important because their ``solution space" (see below) has finite $\mathbb{C}$-dimension. Such a $D$-module thus ensembles the basic features of an ODE. The examples of $D$-modules which we study in this article are going to be holonomic.
\end{example}

\subsection{Banach algebra and Poincar\'e-Perron theory} 
In order to study the convergence issue of Bessel's generating functions that are results of the $z$-transforms introduced in the last subsection,  we extend the classical Poincar\'e-Perron theory from sequences of complex numbers  to sequences in Banach algebra \cite{Yosida}. In our applications  we only need to consider sequences in Banach algebra satisfying second order linear difference equations. However, our adaption can clearly be extended to higher order difference equations.
\medskip

\begin{theorem}[(\bf{Poincar\'e-Perron theorem in Banach algebras})] \label{T:PP}
Let ($\mathbb{B},\|\cdot\|$) be a Banach algebra. Suppose that the $a$, $b\in\mathbb{B}$ and $(p_n)$, $(q_n)$ are sequences in $\mathbb{B}$ converging to zero as $n\to\infty$. If the characteristic equation of the associated difference equation
	\begin{equation}\label{E:PP}
		y_{n+2}+ay_{n+1}+by_n=0,
	\end{equation}
of the  difference equation
	\begin{equation}\label{E:PPP}
y_{n+2}+(a+p_n)y_{n+1}+(b+q_n)y_n=0,
	\end{equation}
has roots $\lambda_1,\, \lambda_2$ distinct  in norms, then for each solution $y_n$ to \eqref{E:PPP}, the limit
	\[
		\lim_{n\to\infty} \frac{\|y_{n+1}\|}{\| y_n\|}=\|\lambda_j\|,
	\]holds, for some $j=1,\, 2$.
\end{theorem}

\medskip
Proofs of  Poincar\'e's original statement  can be found from either Gel'fond \cite[\S5.5.2]{Gelfond} or Milne-Thomson \cite[XVII]{MT}. 
The proof of the above theorem is a straightforward adaptation  in the
 proof of Poincar\'e's original statement with $\mathbb{B}$ in place of $\mathbb{C}$.
Since the proof of Poincar\'e's theorem in full generality is long and the its adaption in Banach algebra $\mathbb{B}$ is almost verbatim except for the aforementioned replacement of ``$|\cdot |$\rq\rq{} by ``$\|\cdot \|$\rq\rq{} in all the computation and estimates,
 one can  verify that a shorter proof of Poincar\'e's theorem available for third order linear difference equations
discussed in \cite{MT} can easily be adopted to the Banach algebra $\mathbb{B}$ setting.
 Indeed, there are different adaptations of Poincar\'e-Perron theory in the literature. For example, the author in  \cite{Schafke} considered objects that assume values in certain abelian groups in place of 
 complex numbers. Although the asymptotic of the classical Bessel function $J_{\nu+n}$ as $n\to\infty$ is well-known, a corresponding result for the difference Bessel function $J^\Delta_{\nu+n}$ is unknown and appears highly non-trivial. The adaption of   Poincar\'e's theorem in a Banach algebra setting can easily produce the needed, though rough, asymptotic for the generating function of difference Bessel functions in to be stated in Theorem \ref{T:delta_Bessel_Gevrey}.

%\medskip

\section{Bessel modules and Bessel functions}\label{S:bessel}

In this section, we investigate the generating functions of the Bessel functions. We will construct an important holonomic left $\mathcal{A}_2$-module, called the Bessel module, from which we obtain a pair of PDEs that the generating function satisfy.
\medskip

\begin{definition}[\cite{CCT3}]\label{D:soln}  Let $M, N$ be left $\mathcal{A}_n$-modules. A \textit{solution} of $M$ in $N$ is a left $\mathcal{A}_n$-linear maps from $M$ to $N$. In particular, the set of \textit{all solutions} of $M$ in $N$ is denoted by 
	\[
		\Hom (M,  N).
	\]
\end{definition}

%\rk{
%\begin{definition}[\cite{CCT3} (Remainder Map)] Let $\mathcal{A}$ be equipped with an order. Fix $L, K \in \mathcal{A}$ in which $K $ is a monic first order operator. For each 
%	$s\in F^0 \mathcal{A}$, there exist unique $Q \in \mathcal{A}$ and 
%	$r \in F^0 \mathcal{A}$ such that 
%	$$Ls=QK+r.$$
%end{definition}
%}
\medskip

\begin{theorem}[\cite{CCT3}] Let $\mathcal{D}=\langle X\partial, X\rangle$ be the subalgebra of $\mathcal{A}$ generated by $X\partial$ and $X$. Then there exists an element $S=\sum_{k=0}^{\infty} a_kX^k$  in the $X$-adic completion $\mathbb{C}[[X]]$ so that the map
	\[
		\mathcal{D}/\mathcal{D}L
		\stackrel{\times S}{\longrightarrow}
	\mathcal{D}/\mathcal{D}(X\partial-\lambda)
	\]is left $\mathcal{A}$-linear, i.e., it gives a solution of $\mathcal{D}/\mathcal{D}L$ in $\mathcal{D}/\mathcal{D}(X\partial-\lambda)$.
\end{theorem}
\medskip

The element $S$ in the above theorem is a Weyl-algebraic analogue of a \textit{Frobenius series expansion} of a solution of $L$. 
In fact, each left $\mathcal{A}$-linear map $\mathcal{A}/\mathcal{A}L \longrightarrow \mathcal{A}/\mathcal{A}K$ induces a map between solution sets 
	\[
		\Hom (\mathcal{A}/\mathcal{A}K,  N)\longrightarrow
		\Hom (\mathcal{A}/\mathcal{A}L,  N).
	\]
Thus finding solutions of a differential operator $L$ in terms of solutions of another differential operator $K$ is equivalent to identifying an $\mathcal{A}$-linear map 
	\[
		\mathcal{A}/\mathcal{A}L
		\stackrel{\times S }{\longrightarrow}
		\mathcal{A}/\mathcal{A}K.
	\]
 
For instance, finding ``power series solutions" of a differential operator $L$ is equivalent to a left $\mathcal{A}$-linear map $\mathcal{A}/\mathcal{A}L\to\mathcal{A}/\mathcal{A}\partial$; finding a ``Frobenius series solution" of $L$ is equivalent to a left $\mathcal{A}$-linear map $\mathcal{A}/\mathcal{A}L\to\mathcal{A}/\mathcal{A}(X\partial-\nu)$; and so on. We refer to \cite{CCT3} for the details. We illustrate the above discussion with the following examples that will be needed for the rest of this paper.
\medskip

\subsection{Weyl series solutions}\label{SS:d-mod}

\begin{example}[(\textbf{Weyl-exponential})]\label{Eg:exp}
	Let $a\in \mathbb{C}$.  A solution of $L:=\partial-a$, i.e. a solution of $\mathcal{A}/\mathcal{A}(\partial-a)$ in any other $D$-module, is called a \textit{Weyl-exponential}. In this example, we consider solutions of $\mathcal{A}/\mathcal{A}(\partial-a)$ in $\mathcal{A}/\mathcal{A}\partial$, that is, we look for solutions as ``power series\rq\rq{} of $X$, which are called \textit{Weyl power series} to the Weyl exponential.  For this purpose we compute the element $S=\sum_k c_kX^k$ which gives rise to a 
left $\mathcal{A}$-linear map
	\[
		\mathcal{A}/\mathcal{A}(\partial-a)
		\stackrel{\times S }{\longrightarrow}
		\overline{\mathcal{A}/\mathcal{A}\partial}
	\]
guaranteed by the study in \cite[\S4.1]{CCT3}. For each $k\ge 1$, since $\partial X^k=X^k \partial+k X^{k-1}$, we have
that 
%$\sum_{k=0}^\infty c_k X^k$ is a solution (map), then
	\[
		\begin{split}
		0= (\partial-a )\left(\sum_{k=0}^\infty  c_k X^k\right) 
		 &= \sum_{k=0}^\infty c_k X^k \partial+\sum_{k=1}^\infty c_k k X^{k-1}-a\sum_{k=0}^\infty c_k X^k\\
	 	&= \sum_{k=0}^\infty (c_{k+1}(k+1)-a c_k) X^k. \quad \qquad \mod \mathcal{A}\partial
		\end{split}
	\]
The above formula gives recurrence of $c_k$ from the relation
	\[
		kc_k -ac_{k-1}=0,\qquad k\ge 1.
	\]
By choosing $c_0=1,$ we obtain $c_{k}=a^k/k!$. Thus we denote
	\begin{equation}\label{E:exp}
		S(X)=\E(aX):=\sum_{k=0}^{\infty} \frac{a^k}{k!} X^k
	\end{equation}
to be the \textit{Weyl exponential series}.
\end{example}
\medskip

\begin{example}\label{Eg:trigo}
(\textbf{Weyl Sine and Cosine}) The \textit{Weyl-sine} and the \textit{Weyl-cosine} are solutions of $L:=\partial^2+1$, i.e., solutions of the $D$-module $\mathcal{A}/\mathcal{A}(\partial^2+1)$. To express them as ``Weyl power series", i.e., in terms of powers of $X$, we again let $K=\partial$,  and compute the element $S =\sum c_{k} X^k$ which yields the left $\mathcal{A}$
	\begin{equation}\label{E:X^2+1}
		\mathcal{A}/\mathcal{A}(\partial^2+1)
		\stackrel{\times S }{\longrightarrow}
		\overline{\mathcal{A}/\mathcal{A}\partial},
	\end{equation}
which are solutions of $\mathcal{A}/\mathcal{A}(\partial^2+1)$ in  $\mathcal{A}/\mathcal{A}\partial$.
	Since $\partial^2+1=(\partial-i)(\partial+i)$, we apply Example \ref{Eg:exp} with the choice $a=-i$ from the Example \ref{Eg:exp}. Then it follows immediately that
	\[
		(\partial-i)(\partial+i) \E(-iX)=0\qquad \mod \partial
	\]
	and since 	$\partial^2+1=(\partial+i)(\partial-i)$ so that 
	\[
		(\partial+i)(\partial-i) \E(iX)=0\qquad \mod \partial,
	\]where we have chosen $a=i$ from the Example \ref{Eg:exp}. Thus both $\E(iX)$ and $\E(-iX)$ in $\mathcal{A}$ are possible choices of $S$ in \eqref{E:X^2+1}. Hence both 
	\begin{equation}\label{E:sine-map}
		\Sin (X):=\frac{1}{2i}\big(\E(iX)-\E(-iX)\big)
	\end{equation}
	and 
	\begin{equation}\label{E:cosine-map}
		\Cos (X):=\frac{1}{2}\big(\E(iX)+\E(-iX)\big)
	\end{equation}
	are different choices of elements of $\mathcal{A}$, called the \textit{Weyl-Sine and Cosine series}, that can also serve for the $S$ in \eqref{E:X^2+1}. 
corresponding 

 Alternatively, we may compute the Weyl Sine and Cosine series $S =\sum c_{k} X^k$ \textit{directly} by
	\[
		\begin{split}
		0&=(\partial-i)(\partial+i)\sum_{k=0}^\infty c_{k} X^k\\
		&=\sum_{k=0}^\infty c_{k} X^k \partial^2+2\sum_{k=1}^\infty c_{k} kX^{k-1}\partial
	+\sum_{k=2}^\infty c_{k}k(k-1) X^{k-2}+\sum_{k=0}^\infty c_{k} X^k\\
	&=\sum_{k=0}^\infty c_{k+2}(k+1)(k+2) X^{k}+\sum_{k=0}^\infty c_{k} X^k.\qquad\mod\partial
		\end{split}
	\]That is,
	\[
		c_{k+2}=\frac{-c_{k}}{(k+1)(k+2)},\quad k\ge 0.
	\]
	Thus, if $c_0=0, c_{1}=1$, we obtain Weyl Sine series
	$$\Sin X=X-\frac{X^3}{3!}+\frac{X^5}{5!}-\cdots+
	\frac{(-1)^{k+1}}{(2k-1)!}X^{2k-1}+\cdots.$$
	If $c_0=1, c_{1}=0$, we obtain Weyl (power) Cosine series
	$$\Cos X=1-\frac{X^2}{2!}+\frac{X^4}{4!}-\cdots+
	\frac{(-1)^{k}}{(2k)!}X^{2k}+\cdots.$$
Hence they serve as $S$ with ``appropriately imposed initial conditions".
\end{example} 
%\medskip
\medskip

We now introduce one of the main subjects of study in this paper.

\begin{example}[(\textbf{Weyl Bessel of order $\nu$})]\label{Eg:bessel_map} Let $\nu\in\mathbb{C}$ such that $2\nu$ is not an integer, and let
	\begin{equation}\label{E:bessel-eqn-map}
		L=(X\partial)^2+X^2-\nu^2.
	\end{equation}
Let $\mathcal{D}=\langle X\partial, X\rangle$ be the subalgebra of $\mathcal{A}$ generated by $X\partial $ and $X$. Then the \textit{Weyl Bessel of order $\nu$} is a solution of $L$ in a left $\mathcal{D}$-module $N$ , i.e., a solution of $\mathcal{D}/\mathcal{D}L$ in $N$. To find the ``Frobenius series expansion" of this Weyl Bessel, we look for a solution of $\mathcal{D}/\mathcal{D}L$ in the $D$-module $\overline{\mathcal{D}/\mathcal{D}(X\partial-\nu)}$, i.e., we find the element $S=\sum_{k=0}^\infty c_kX^k$ that gives the left $\mathcal{A}$-linear map
	\begin{equation}\label{E:bessel_map}
		\mathcal{D}/\mathcal{D}L
		\stackrel{\times S}{\longrightarrow}
	\overline{\mathcal{D}/\mathcal{D}(X\partial-\nu)}.
	\end{equation}
	Since
	\[
		(X\partial+\nu)(X\partial-\nu)X^k
		=X^k[(X\partial+\nu)+2k](X\partial-\nu)+(2\nu k+k^2)X^k,
	\]holds for every $k\ge 0$ (as well as trivially when $k=0$), we have
	\[
		\begin{split}
		0&=L\sum_{k=0}^\infty c_{k} X^k = \sum_{k=0}^\infty c_{k}[(X\partial+\nu)(X\partial-\nu)+X^2]X^k\\
		&=\sum_{k=0}^\infty  c_k  X^k[(X\partial+\nu)+2k](X\partial-\nu) + \sum_{k=1}^\infty{c_kk(2\nu +k)X^k}+\sum_{k=0}^\infty c_k X^{k+2} \\
		&=\sum _{k=1}^\infty c_k  k(2\nu+k)X^k+\sum _{k=0}^\infty c_kX^{k+2} \hfill \mod X\partial-\nu\\
		&= c_1(1+2\nu)X+\sum _{k=2}^\infty (c_k k (2\nu+k)+ c_{k-2})X^{k}, \hfill \mod X\partial-\nu
		\end{split}
	\]which implies that $c_1=c_3=c_5=\cdots= 0$, and by choosing $c_0 \Gamma(\nu+1)=2^{-\nu},$ we must have
	\[
		c_{2k}=\frac{(-1)^k }{2^{\nu+2k} k!\Gamma(\nu+k+1)},\quad k\ge 0.
	\]This shows that
	\[
		S=\sum_{k=0}^{\infty} \frac{(-1)^kX^{2k}}{2^{\nu+2k}k!\Gamma(\nu+k+1)}
	\]
can be chosen as such a Weyl series $S$ in \eqref{E:bessel_map} that serves as a solution for \eqref{E:bessel_map}. We denote this specific Weyl Bessel series by $\mathcal{J}_\nu$ as solution in $\mathcal{D}/\mathcal{D}(X\partial-\nu)$, or more generally, in a left $\mathcal{A}$-module $N$, which can be considered as the Weyl algebraic version of the classical Bessel $J_\nu$. We caution that the reader not to confuse the series form of $\mathcal{J}_\nu$ in contrast to the original Bessel function of order $J_\nu$ since there is no factor ``$X^\nu$" that would appear there unless when the $\nu$ reduces to an integer.
%We also denote the general left $\mathcal{D}$-linear map in \eqref{E:bessel_map} by the notation $\mathfrak{C}_\nu$ which serves as the Weyl algebraic version of the general solution in the left $\mathcal{D}$-module $\mathcal{O}$, or more generally in a left $\mathcal{A}$-module, to the classical notation $\mathscr{C}_\nu$.
\end{example}
\medskip

\subsection{Transmutation formulae}\label{SS:transmutation_1}

After defining the Weyl Bessel as a solution of $L=L_\nu:=(X\partial)^2+X^2-\nu^2$ in the last subsection, we observe the following transmutation formulae for $L$. They can be verified directly, but they do not appear to be found immediately amongst the formulae in Derezi\'nski and Majewski \cite{Derezinski2014,Derezinski2020,Der_Maj}\footnote{Please see \S\ref{S:conclusion} Conclusion for discussion related to the work of Infeld and Hull \cite{Infeld_Hull_1951}.}. As we shall see that they are instrumental in the development of Bessel module to be defined  in the next subsection.
\medskip

\begin{proposition}\label{P:bessel_transmutation}
For each $\nu\in\mathbb{C}$, we have
\begin{align}
	[(X\partial)^2+X^2-(\nu-1)^2]\left(\partial+\dfrac{\nu}{X}\right)&=\left(\partial+\dfrac{\nu-2}{X}\right)[(X\partial)^2+X^2-\nu^2],\label{E:bessel_transmutation_1}\\
[(X\partial)^2+X^2-(\nu+1)^2]\left(\partial-\dfrac{\nu}{X}\right)&=\left(\partial-\dfrac{\nu+2}{X}\right)[(X\partial)^2+X^2-\nu^2],\label{E:bessel_transmutation_2}
\end{align}
%Here the second transmutation formula above can be obtained from the first one by changing the signs of $\nu$. We immediately deduce from these transmutation formulae the following left $\mathcal{A}$-linear maps.
which induce the following left $\mathcal{A}$-linear maps
\[
\begin{array}{rcl}\mathcal{A}/\mathcal{A}((X\partial)^2+X^2-(\nu-1)^2)&\stackrel{\times(\partial+\frac{\nu}{X})}{\longrightarrow}&\mathcal{A}/\mathcal{A}((X\partial)^2+X^2-\nu^2),\\
\mathcal{A}/\mathcal{A}((X\partial)^2+X^2-(\nu+1)^2)&\stackrel{\times(\partial-\frac{\nu}{X})}{\longrightarrow}&\mathcal{A}/\mathcal{A}((X\partial)^2+X^2-\nu^2).
\end{array}
\]
\end{proposition}
\medskip

%Here the second transmutation formula above can be obtained from the first one by changing the signs of $\nu$. We immediately deduce from these transmutation formulae the following left $\mathcal{A}$-linear maps.
\begin{corollary}\label{C:bessel_formulae}
Let  $\nu\in\mathbb{C}$, the classical Bessel functions satisfy the following recurrence relations.
	\begin{align}
		J'_\nu(x)+\dfrac{\nu J_\nu(x)}{x}&=J_{\nu-1}(x),\label{E:bessel_formula_1}\\
		J'_\nu(x)-\dfrac{\nu J_\nu(x)}{x}&=-J_{\nu+1}(x).\label{E:bessel_formula_2}
	\end{align}
\end{corollary}
\medskip

\begin{proof} This proof serves to provide an alternative to the classical proof found in most literature. 
It is a direct consequence from the previous Proposition with the $D$-module structure on $\mathcal{O}$ introduced from the Example \ref{Eg:O_deleted_d}. Let $J_\nu$ be the classical Bessel function to $L_\nu$ with the asymptotic behaviour $\lim_{x\to 0^+}{J_\nu(x)}/{x^\nu}={1}/{2^\nu\Gamma(\nu+1)}$ \cite[p. 43]{Watson1944}. Then one applies the both sides of \eqref{E:bessel_transmutation_1}.  It implies that  $(\partial+\nu/X)J_\nu=CJ_{\nu-1}+DJ_{-\nu+1}$ holds for some constants $C, D$. One can easily verify that the $C=1,\ D=0$ from the asymptotic  $J^\prime_\nu(x)/x^{\nu-1}\sim 1/2^{\nu-1}\Gamma(\nu)$ \cite[p. 43]{Watson1944}. This verifies the \eqref{E:bessel_formula_1}. The verification of \eqref{E:bessel_formula_2} is similar.

%the choice that $\mathscr{C}_\nu=J_\nu$, where the $J_\nu$ is the classically defined as the solution of $(X\partial)^2+X^2-\nu^2$ with $\partial f(x)=f^\prime(x)$ and $Xf(x)=xf(x)$. But then the  $J_\nu(x)$ has asymptotic $\lim_{x\to 0^+}{J_\nu(x)}/{x^\nu}={1}/{2^\nu\Gamma(\nu+1)}$.
\end{proof}
\medskip

\begin{corollary}\label{C:Bessel_trans}
Let $\nu\in\mathbb{C}$ and a left $\mathcal{D}$-module $N$ be given. Then there exists a sequence $\mathrm{(}\mathfrak{J}_{\nu+n}\mathrm{)}$ in $N^\mathbb{Z}$ and for each $n$ the $\mathfrak{J}_{\nu+n}$ is a solution to $L_n=(X\partial)^2+X^2-(\nu+n)^2$ in $N$, such that we have
	\begin{align}
		X\partial\, \mathfrak{J}_{\nu+n} +(\nu+n) \mathfrak{J}_{\nu+n} - X\mathfrak{J}_{\nu+n-1}=0,\label{E:Bessel_trans_1}
		\\
		X\partial\, \mathfrak{J}_{\nu+n} -(\nu+n) \mathfrak{J}_{\nu+n} + X\mathfrak{J}_{\nu+n+1}=0.\label{E:Bessel_trans_2}
	\end{align}
\end{corollary}
\medskip

\begin{remark} Reader are cautioned not to confuse the notation $\mathfrak{J}_\nu$ introduced in the last Corollary with notation $\mathcal{J}_\nu$ that denotes the Weyl-Bessel series introduced in Example \ref{Eg:bessel_map} that is defined as a solution with respect to the \eqref{E:bessel_map}.

\end{remark}
\medskip

\begin{proof} Recall that $L_\nu=(X\partial)^2+X^2-\nu^2$. Let us replace $\nu$ by $\nu+n$ in \eqref{E:bessel_transmutation_1}. Then it is clear that the right-side of \eqref{E:bessel_transmutation_1} annihilates $\mathfrak{J}_{\nu+n}$.  However the expression  $(\partial+\nu/X)\mathfrak{J}_{\nu+n}=(X\partial+\nu)/X\, (X\mathfrak{J}_\nu$) on the left-side must be annihilated by $L_{\nu+n-1}$. Hence $X\partial\, \mathfrak{J}_{\nu+n} +(\nu+n) \mathfrak{J}_{\nu+n} = \pm X\mathfrak{J}_{\nu+n-1}$. This gives the \eqref{E:Bessel_trans_1} except for the ``$\mp$" in front of the third term $ X\mathfrak{J}_{\nu+n-1}$. We have chosen the minus sign over the plus sign in accordance with the corresponding formula for the classical Bessel function $J_\nu$. This establishes the \eqref{E:Bessel_trans_1}. The proof for \eqref{E:Bessel_trans_2} is similar.
\end{proof}
\medskip

\begin{corollary} [(\textbf{Three-term recurrence})] \label{C:Bessel_3term}
Let $\nu\in\mathbb{C}$ and a left $\mathcal{D}$-module $N$ be given. Then there exists a sequence $\mathrm{(}\mathfrak{J}_{\nu+n}\mathrm{)}$ such that for each $n$ the $\mathfrak{J}_{\nu+n}$ is a solution to $L_n=(X\partial)^2+X^2-(\nu+n)^2$ in $N$, such that we have
	\begin{equation}\label{E:Bessel_3term}
			2(\nu+n)\,\mathfrak{J}_{\nu+n}-X\mathfrak{J}_{\nu+n-1}-X\mathfrak{J}_{\nu+n+1}=0
	\end{equation}
as an element in $N$.
\end{corollary}

%\begin{corollary}\label{C:bessel_gauge}
%For each $\nu\in\mathbb{C}$, the following are well-defined left $\mathcal{A}$-linear maps.
%\[
%\begin{array}{rcl}
%\mathcal{A}/\mathcal{A}((X\partial)^2+X^2-(\nu+1)^2)&\stackrel{\times(\partial-\frac{\nu}{X})}{\longrightarrow}&\mathcal{A}/\mathcal{A}((X\partial)^2+X^2-\nu^2),\\ \mathcal{A}/\mathcal{A}((X\partial)^2+X^2-(\nu-1)^2)&\stackrel{\times(\partial+\frac{\nu}{X})}{\longrightarrow}&\mathcal{A}/\mathcal{A}((X\partial)^2+X^2-\nu^2).
%\end{array}
%\]
%\end{corollary}
\medskip

%\rk{We illustrate here an immediate application of the above Corollary \ref{C:Bessel_trans} to classical Bessel functions $J_\nu$. We shall discuss how the same $\mathcal{A}$-linear maps can also lead to the corresponding formulae for difference Bessel functions \rk{in due course}. }
%\medskip

%\rk{\begin{corollary}\label{C:bessel_formulae}
%For each $\nu\in\mathbb{C}$, the classical Bessel functions $J_\nu$ satisfy the following recurrence relations.
%	\begin{align*}
%		J'_\nu(x)-\dfrac{\nu J_\nu(x)}{x}&=-J_{\nu+1}(x),\\
%		J'_\nu(x)+\dfrac{\nu J_\nu(x)}{x}&=J_{\nu-1}(x).\\
%	\end{align*}
%\end{corollary}
%\medskip
%}

%\rk{\begin{proof}
%It is a direct consequence of the Corollary \ref{C:Bessel_trans} as the  $J_\nu$ is classically defined as the solution of $(X\partial)^2+X^2-\nu^2$ with $\partial f(x)=f^\prime(x)$ and $Xf(x)=xf(x)$. But then the  $J_\nu(x)$ has asymptotic $\lim_{x\to 0^+}{J_\nu(x)}/{x^\nu}={1}/{2^\nu\Gamma(\nu+1)}$.
%\end{proof}
%}
%\medskip

%\rk{Reader can easily obtain, by similar procedures as illustrated in the above Corollary, the corresponding formulae for the Bessel functions of the second-kind $Y_\nu$, Hankel functions $H_\nu^{(1)},\, H_\nu^{(2)}$ and the Bessel functions of the third-kind $I_\nu,\, K_\nu$ as  seen, for example, on p. 66, p. 74 and p. 79 in  Watson \cite{Watson1944} respectively.}
%\medskip

\subsection{Bessel modules}

Recall that  the Weyl algebra $\mathcal{A}_2$ is the free $\mathbb{C}$-algebra $\mathbb{C}\langle \partial_1,\partial_2, X_1, X_2\rangle$ subject to 
\[
\partial_i\partial_j-\partial_j\partial_i=0,\quad X_iX_j-X_jX_i=0,\quad\partial_iX_j-X_j\partial_i=\delta_{ij},
\]
for $i,j=1,2$. Then we define

\begin{definition}\label{D:A_2_mod} We denote the algebra obtained by adjoining the new elements $1/X_1,\, 1/X_2$ to the Weyl algebra $\mathcal{A}_2$ by
	\[
		\mathcal{A}_2(1/X_1,\, 1/X_2)
		:=\mathbb{C}\langle\partial_1,\partial_2, X_1, \dfrac{1}{X_1},X_2,\dfrac{1}		{X_2}\rangle.
	\]
\end{definition}

\medskip

	Let $\nu\in\mathbb{C}$. Motivated by Corollary \ref{C:Bessel_trans}, we now consider the left $\mathcal{A}_2$-module generanted by the following elements 
	\begin{equation}\label{E:bessel_PDE}
%			\begin{split}
			X_1\partial_1+(X_2\partial_2+\nu)-X_1X_2,\qquad
			X_1\partial_1-(X_2\partial_2+\nu)+\frac{X_1}{X_2}.
%			\end{split}
		\end{equation}
%and their difference
%	\begin{equation}\label{E:bessel_3term_recursion}
%			X_2\partial_2+\nu-\frac12 X_1\big(X_2+\frac{1}{X_2}\big)
%	\end{equation}
% all belong to $\mathcal {A}_{2}(1/X_1,\, 1/X_2)$. The last element \eqref{E:bessel_3term_recursion} naturally gives raise to the well-known three-term recursion formula
%	\[
%		J_{\nu +1}(x)+J_{\nu-1}(x) = \frac{2\nu}{x} J_\nu(x).
%	\]
\medskip

With the \eqref{E:bessel_PDE} in mind, we define
\medskip

\begin{definition}[(\textbf{Bessel  module of order $\nu$})]\label{D:bessel_mod}
       Let $\nu\in\mathbb{C}$. The left $\mathcal {A}_{2}$-module\footnote{The $\mathcal{B}_\nu$ is in fact a quotient module. We have left out the word ``quotient" when no confusion can arise.} 
       		\begin{equation}\label{E:bessel_mod}
			\mathcal{B}_\nu=\frac{\mathcal{A}_2}
			{\mathcal{A}_2(X_1\partial_1+(\nu+X_2\partial_2)-X_1X_2)+
			\mathcal{A}_2(X_1\partial_1-(\nu+X_2\partial_2)+{X_1}/{X_2})}
		\end{equation} is called the \textit{Bessel module of order} $\nu$.
\end{definition}
\medskip

First we check that the $\mathcal{B}_\nu$ is indeed holonomic.
\medskip

\begin{proposition}\label{P:bessel_holonomic} Let $\nu\in\mathbb{C}$. The Bessel module 
$\mathcal{B}_\nu$ is a left $\mathcal{A}_2$-module
\footnote{We omit the description that $\mathcal{B}_\nu$ is a left $\mathcal{A}_2$-quotient module for the sake of simplicity.} 
with dimension $2$ and multiplicity  $1$. In particular, it is  holonomic.
\end{proposition}
\smallskip

\begin{proof} The  two generators from Definition \ref{D:bessel_mod}  that generate the Bessel module $\mathcal{B}_\nu$ satisfy the integrability criterion stated in  Example \ref{E:prime_integrable_eg}. Hence  the Bessel module has dimension $2$ and multiplicity $1$.
\end{proof}
\medskip

The following theorem shows that one can recover from $\mathcal{B}_\nu$ the ODE satisfied by Bessel functions.
\medskip
 
 \begin{proposition}[(\textbf{Bessel ODE module of order $\nu$})] \label{P:bessel_ode_mod}

The map
\[
	\dfrac{\mathcal{A}_2}{\mathcal{A}_2[(X_1\partial_1)^2+X_1^2-(\nu+X_2\partial_2)^2]}\longrightarrow\mathcal{B}_\nu
\]
is a well-defined left $\mathcal{A}_2$-linear surjection.
\end{proposition}
\medskip

\begin{proof} Recall that $\mathcal{B}_\nu$ is generated by the two elements
\[
X_1\partial_1+X_2\partial_2-X_1X_2+\nu,\qquad X_1\partial_1-X_2\partial_2+\frac{X_1}{X_2}-\nu
\]
as in \eqref{E:bessel_PDE}, so as an element in $\mathcal{B}_\nu$, we have
\begin{align*}
	(X_1\partial_1)^2+X_1^2-(\nu+X_2\partial_2)^2 &= (X_1\partial_1)^2+X_1^2+(\nu+X_2\partial_2)(X_1\partial_1-X_1X_2) \\
	&= (X_1\partial_1)(X_1\partial_1+\nu+X_2\partial_2)+(X_1X_2)(\frac{X_1}{X_2}-\nu-\partial_2X_2) \\
	&= (X_1\partial_1)(X_1X_2)+(X_1X_2)(X_2\partial_2-X_1\partial_1-\partial_2X_2)\\
	&= (X_1\partial_1)(X_1X_2)+(X_1X_2)(-\partial_1X_1)\\
	&= 0.
\end{align*}
Hence the map is well defined.
\end{proof}
\medskip

%Thus the Bessel ODE module of order $\nu$ is more general than the Bessel module of order $\nu$. In fact, it consists of both the $\mathcal{B}_\nu$ and $\mathcal{B}_{-\nu}$.  When we drop the ``$\partial_2X_2$" from $(X_1\partial_1)^2+X_1^2-(\nu+X_2\partial_2)^2$, we also recover the Bessel operator $L$ from Example \ref{Eg:bessel_map}.

We may rephrase the above result in the following form.
\medskip

\begin{proposition} 
Let $\nu\in\mathbb{C}$ and $M$ be a left $\mathcal{A}$-module. If the sequence $(f_n)$ is a solution of $\mathcal{B}_\nu$ in $M^\mathbb{Z}$, then for each $n\in\mathbb{Z}$, $f_n$ is a solution of $(X\partial)^2+X^2-(\nu+n)^2$ in $M$. In other words, the element $(X_1\partial_1)^2+X_1^2-(\nu+X_2\partial_2)^2$ is equivalent to $0$ in $\mathcal{B}_\nu$.
\end{proposition}
\medskip

\subsection{Characteristics of Bessel modules}

We will see later that the Bessel module $\mathcal{B}_\nu$ gives rise to PDEs satisfied by the generating functions of the Bessel functions $(J_{\nu+n})$ in \S\ref{SS:classical_gf} as well as of the recently discovered difference Bessel functions $(J^\Delta_{\nu+n})$ in \S\ref{SS:delta_gf} below. The following theorem is an algebraic analogue of the \textit{method of characteristics}, which will help in solving such PDEs.
\medskip

\begin{theorem}\label{T:bessel_gen_map_2} Let $\nu\in\mathbb{C}$ and $\mathcal{B}_\nu$ be defined above. Then there exists a left $\mathcal{A}_2(1/X_{j=1,2})$-linear map $S$
	\begin{equation}\label{E:bessel_gen_map_1}
		\mathcal{B}_\nu
		\xrightarrow 
		{\times S}
		\overline{
		\mathcal{A}_2/\big[\mathcal{A}_2\partial_1+\mathcal{A}_2(X_2\partial_2+\nu)\big]}.
	\end{equation} In fact, we can take
	\begin{equation}\label{E:bessel_gen_map_2}
			S=\E\Big[\frac{X_1}{2}\Big(X_2-\frac{1}{X_2}\Big)\Big],
	\end{equation}where $\E$ denotes the Weyl exponential series introduced in Example \ref{E:exp}.
\end{theorem}
\smallskip

\begin{proof}  Define new symbols $\Theta_1, \Theta_2, Y_1, Y_2$ by
\[
	Y_{1}=X_{1}X_{2},\qquad Y_{2}=\frac{X_{1}}{X_{2}},\qquad
	2\Theta_{1}=\frac{\partial_1}{X_2}+\frac{\partial_{2}}{X_{1}},\qquad
	2\Theta_{2}=X_{2}\partial_{1}-\frac{X^2_{2}\partial_{2}}{X_{1}}.
\]

Direct verification yields 
	\begin{equation}\label{E:computation_0}
		[\Theta_{1}, Y_{1}]=1, \quad [\Theta_{2}, Y_{2}]=1,
		\quad
		[\Theta_{1}, \Theta_{2}]=0,\quad [Y_1, Y_{2}]=0,
		\quad [\Theta_i,\, Y_j]=0,\ i\not=j.
	\end{equation}
\medskip
The elements in \eqref{E:bessel_PDE} become
	\[
		2Y_1\Theta_{1}-Y_1+{\nu}\qquad \mbox{and}\qquad 2Y_2\Theta_{2}+Y_2-{\nu}. 
	\]

%Then the PDEs \eqref{E:bessel_PDE} become
%	\[
%		\begin{split}
%		L_1: &=2\Theta_{1}-1+\frac{\nu}{Y_1} ,\\ 
%		L_2: &=2\Theta_{2}+1-\frac{\nu}{Y_2}. 
%		\end{split}
%	\]
 
 It follows from the Weyl exponential introduced in Example \ref{Eg:exp} for $a=\nu/2$ that
 		\[
		\begin{split}
			(2Y_1\Theta_1-Y_1+\nu)\, \E(Y_1/2) 
			&= 2Y_1\E(Y_1/2)\Theta_1 + \nu\E(Y_1/2)\\
			&= \E(Y_1/2)(2Y_1\Theta_1+\nu)\\
			&=0. \qquad\qquad\quad  \mod \mathcal{A}_2 (Y_1\Theta_1+\nu/2)
		\end{split}
		\]
	In general we have $S=F(Y_2)\cdot \E(Y_1/2)$ for some factor $F(Y_2)$ depending on $Y_2$ only. Next we require $F(Y_2)$ to satisfy that
	\[
			(2Y_2\Theta_2+Y_2-\nu)\,F(Y_2)\cdot \E(Y_1/2) =0.
			\qquad\mod \mathcal{A}_2(Y_2\Theta_2-\nu/2)
		\]
		That is,
	\[
			(2Y_2\Theta_2+Y_2-\nu)\,F(Y_2)=0,\qquad\qquad \mod \mathcal{A}_2(Y_2\Theta_2-\nu/2).
	\]

It follows from a similar consideration as above that $F(Y_2)=c\E(-Y_2/2)$ for some complex number $c$.
Without loss of generality we may choose $c=1$ so that the above map becomes
%	\[
%	\overline{
%	\mathbb{C}[[Y_i\Theta_{i}, Y_{i}]]/\mathbb{C}][[Y_i\Theta_{i}, Y_{i}]]L_{i}} \xrightarrow{\times \E(\pm Y_i/2)}
%\overline{
%	\mathbb{C}[[Y_i\Theta_{i}, Y_{i}]]/\mathbb{C}[[Y_i\Theta_{i}, Y_{i}]](Y_i\Theta_i\pm \nu/2)}, \quad i=1,\, 2
%	\]are well-defined. 
% \medskip
%That is, apart from a non-zero constant, we have. because of $[\Theta_i,\, Y_j]=0,\ i\not=j$ and $[X_1,\, X_2]=0$ in \eqref{E:computation_0}, that
	\begin{equation}\label{E:bessel_gen_map_3}
		S:=\E(Y_1/2)\cdot \E(-Y_2/2) 
		=\sum_{k=0}^\infty  \frac{X_1^kX_2^k}{2^kk!} \cdot
		\sum_{\ell=0}^\infty \frac{(-1)^\ell X_1^\ell}{X_2^\ell 2^\ell \ell !}
		=\E\Big[\frac{X_1}{2}\Big(X_2-\frac{1}{X_2}\Big)\Big]
	\end{equation}
which	is the map  asserted in \eqref{E:bessel_gen_map_1}.
\end{proof}
\medskip

\begin{remark} We may rewrite  the product of two series in $S$  in \eqref{E:bessel_gen_map_3} to get
	\[
		\sum_{k=0}^{\infty} \frac{X^k_{1}X^k_{2}}{2^k k!} 
\sum_{\ell=0}^{\infty}\frac{(-1)^l X^\ell_{1}}{X^\ell_{2}2^\ell !}
=\sum_{n=-\infty}^{\infty}\left(\sum_{\ell=0}^{\infty}\frac{(-1)^\ell X^{n+2\ell}_{1}}{2^{n+2\ell} (n+\ell)! \ell!} \right) X_{2}^{n}.
	\]
The ``inner sum\rq\rq{}  
	\[
		\mathcal{J}_{n} (X)=\sum_{\ell =0}^{\infty}
		\frac{(-1)^\ell}{ (n+\ell)!l!} 
		\left(\frac{X}{2}\right)^{n+2\ell}
	\]
agrees with Example~\ref{Eg:bessel_map} after composing with the multiplication by $X^n$ from $\mathcal{A}/\mathcal{A}(X\partial-n)$ to $\mathcal{A}/\mathcal{A}\partial$. It gives a ``power series" of the Weyl Bessel of order $n$. Following the arguments used in Watson \cite[pp. 15-16]{Watson1944} that
	\begin{equation}\label{E:neg_Bessel}
		\mathcal{J}_{-n} (X)=(-1)^n \mathcal{J}_{n} (X)
	\end{equation}
	holds for each integer $n\ge 1$.
\end{remark}

%\section{Generating functions: preliminaries}  
\subsection{Generating function of classical Bessel functions}\label{SS:classical_gf}
As was explained by Watson \cite[Chap. 2., \S2,.1]{Watson1944} that the convergence of the infinite expansion of the generating function 
	\begin{equation}\label{E:bessel_classical_gf}
		e^{\frac{x}{2}\big(t-\frac1t\big)}=\sum_{n=-\infty}^\infty J_n(x)\, t^n
	\end{equation} 
of Bessel function is absolute in powers of $t$, and valid for all $x$ and $t\not=0$.   We shall take this expansion as the basis of our study of various extensions of this generating function for various types of Bessel functions in this section.

\medskip

\begin{example}[(\textbf{Frobenius series of classical Bessel functions of order $\nu$})]\label{Eg:Bessel} Let $\nu$ be a complex number such that $2\nu$ is not an integer, and consider the the Bessel operator $L=(X\partial)^2+X^2-\nu^2$ as in Example~\ref{Eg:bessel_map}. The classical Bessel function of order $\nu$, $J_\nu(x)$, is a solution of $\mathcal{D}/\mathcal{D}L$ in the $D$-module $\mathcal{O}_d$. To obtain a Frobenius series expansion of $J_\nu(x)$, we
 recall from Example \ref{Eg:bessel_map}  the element 
	\begin{equation}\label{E:pos_nu_map}
		S=\mathcal{J}_\nu(X)=\sum_{k=0}^{\infty} \frac{(-1)^kX^{2k}}{2^{\nu+2k}k!\Gamma(\nu+k+1)}
	\end{equation}
gives the composition map 
	 of left $\mathcal{A}$-linear maps 
	\begin{equation}\label{E:pos_bessel_comp_map}
		\mathcal{D}/\mathcal{D}L
		\stackrel{\times S}{\longrightarrow}
	\overline{\mathcal{D}/\mathcal{D}(X\partial-\nu)}
	\stackrel{\times x^{\nu}}{\longrightarrow} \mathcal{O}_d,
	\end{equation}
	and so
	\begin{equation}\label{E:pos_bessel-soln}
		J_\nu(x)=x^\nu \sum_{k=0}^{\infty} \frac{(-1)^kx^{2k}}{2^{\nu+2k}k!\Gamma(\nu+k+1)},
	\end{equation}
gives raise to the classical Bessel function of order $\nu$.
\medskip

On the other hand, if we replace the $\nu$ in \eqref{D:bessel_mod} by $-\nu$, and proceed in a similar manner for an analogous set up, then we would have obtained a  ``linearly-independent" solution of $\mathcal{D}/\mathcal{D}L$ 
in $\mathcal{O}_d$ given by
	\begin{equation}\label{E:neg_bessel_comp_map}
		\mathcal{D}/\mathcal{D}L
		\stackrel{\times S}{\longrightarrow}
	\overline{\mathcal{D}/\mathcal{D}(X\partial+\nu)}
	\stackrel{\times x^{-\nu}}{\longrightarrow} \mathcal{O}_d,
	\end{equation}
	where
	 \begin{equation}\label{E:neg_nu_map}
		S=\mathcal{J}_{-\nu}(X)=\sum_{k=0}^{\infty} \frac{(-1)^kX^{2k}}{2^{-\nu+2k}k!\Gamma(-\nu+k+1)}.
	\end{equation}
The \eqref{E:neg_bessel_comp_map} gives rise to the Frobenius series expansion
	\begin{equation}\label{E:neg_bessel-soln}
		J_{-\nu}(x)=x^{-\nu} \sum_{k=0}^{\infty} \frac{(-1)^kx^{2k}}{2^{-\nu+2k}k!\Gamma(-\nu+k+1)}.
	\end{equation}
\end{example}
\medskip

The following example illustrates how sequences of analytic functions are realized from $D$-modules.
\medskip

\begin{example}\label{Eg:bilateral_J} Let $\nu\in \mathbb{C}$ and consider the bilateral  sequence of classical Bessel functions $(J_{\nu+n})_n$. Then the maps
	\begin{equation}\label{E:j-map-bessel}
		\begin{array}{rcl}
		\mathfrak{j}_+:\mathcal{B}_\nu & \stackrel{\times (J_{\nu+n})}{\longrightarrow} &\mathcal{O}^\mathbb{Z}
		\end{array}
	\end{equation}
and
	\begin{equation}%\label{E:j-map-bessel}
		\begin{array}{rcl}
		\mathfrak{j}_-:\mathcal{B}_\nu & \stackrel{\times (J_{-\nu-n})}{\longrightarrow} &\mathcal{O}^\mathbb{Z}
		\end{array}
	\end{equation}  
are both left $\mathcal{A}_2$-linear.  

More generally, a general solution of $\mathcal{B}_\nu$ in $\mathcal{O}^\mathbb{Z}$ will be denoted by  $\times \mathrm{(}\mathscr{C}_{\nu+n}\mathrm{)}$. That is,
%be an arbitrary sequence of solution of $\mathcal{B}_\nu$ in $\mathcal{O}^\mathbb{Z}$. Suppose
	\begin{equation}\label{E:general-map-bessel}
		\begin{array}{rcl}
		\mathfrak{j}: \mathcal{B}_\nu & \stackrel{\times (\mathscr{C}_{\nu+n})}{\longrightarrow} &\mathcal{O}^\mathbb{Z}
		\end{array}
	\end{equation}
%\footnote{ Watson used $ \mathscr{C}_{\nu}(x)$ to denote an arbitrary solution to the system of  Bessel equation \eqref{E:Bessel_eqn}.}.
%Then the map
%	\begin{equation}\label{E:j-map-bessel-0}
%		\begin{array}{rcl}
%		\mathfrak{j}:\mathcal{B}_\nu & \stackrel{\times (\mathscr{C}_{\nu+n})}{\longrightarrow} &\mathcal{O}^\mathbb{Z}
%		1&\longmapsto& (\mathscr{C}_{\nu+n})
%		\end{array}
%	\end{equation}is left $\mathcal{A}_2$-linear.
\end{example}
\medskip

The following proposition and corollary are direct consequences of Corollary~\ref{C:Bessel_trans} with $M=\mathcal{O}_d$ and $\mathscr{C}_{\nu+n}$.
\medskip

\begin{proposition}[(\textbf{Differential-difference formulae})]\label{P:PDE_cylinderical_bessel} 
Let $\mathcal{B}_\nu\ \substack{
\times (\mathcal{C}_{\nu+n})\\ \longrightarrow} \ \mathcal{O}^\mathbb{Z}_d$  $\mathrm{(}\mathscr{C}_{\nu+n}\mathrm{)}$ be a solution of $\mathcal{B_\nu}$ in $\mathcal{O}^\mathbb{Z}_d$. Then the $\mathrm{(}\mathscr{C}_{\nu+n}\mathrm{)}$
 satisfies the following classical differential-difference formulae
%Let $\mathcal{O}_{d \Delta}$ be a $\mathcal{A}_2$-module as introduced in \eqref{Eg:two-seq-functions} and $\nu\in \mathbb{C}\backslash\{-\mathbb{N}\}$. Then the following two maps 
%		\[
%		\begin{array}{lrl}
%		&\displaystyle\frac{\mathcal{A}_2}
%		{\mathcal{A}_2(X_1\partial_1+(\nu+X_2\partial_2)-X_1X_2)}	
%		&\longrightarrow
%		\mathcal{O}^{\mathbb{Z}}\\
%		&&\\
%		&1 &\mapsto  \big(\mathcal{C}_{\nu+n}(x)\big),
%		\end{array}
%		\]
%		\[
%		\begin{array}{rrl}
%%		&\displaystyle \frac{\mathcal{A}_2}
%		{\mathcal{A}_2(X_1\partial_1-(\nu+X_2\partial_2)+X_1/X_2)}	
%		&\longrightarrow
%		\mathcal{O}^{\mathbb{Z}}\\
%		&&\\
%		&1 &\mapsto
%			 \big(\mathcal{C}_{\nu+n}(x)\big)
%		\end{array} 
%		\]are $\mathcal{A}_2$-linear maps. As a result, we have the (well-known) differential-difference formulae for the classical Bessel functions,
		\begin{equation}\label{E:any_bessel_recus_1}
				x\mathscr{C}_{\nu}^\prime(x)+\nu\mathscr{C}_{\nu}(x)-
		x\mathscr{C}_{\nu-1}(x)=0,
			\end{equation}
and
		\begin{equation}\label{E:any_bessel_recus_2}
			x\mathscr{C}^{\prime}_{\nu}(x)-\nu\mathscr{C}_{\nu+n}(x)+
		x\mathscr{C}_{\nu+1}(x)=0.
		\end{equation}
	\end{proposition}
	\medskip

The usage of the notation $\mathscr{C}_{\nu+n}$ follows Sonine's notation  (see \cite[p. 83]{Watson1944}) for functions that satisfy the system \eqref{E:any_bessel_recus_1} and \eqref{E:any_bessel_recus_2}.

Here is an application of the proposition to recover two well-known formulae about the classical Bessel functions. We shall  see how this will lead to new results for difference Bessel  functions and half-Bessel modules for difference reverse Bessel polynomials in \S\ref{SS:delta_reverse_bessel_poly}.
\medskip

One recovers the well-known formulae \eqref{E:bessel_formula_1} and \eqref{E:bessel_formula_2} by choosing $\mathscr{C}_\nu=J_\nu$ in the above Corollary.
\medskip

%Reader can easily obtain, by similar procedures as illustrated in the above Corollary, the corresponding formulae for the Bessel functions of the second-kind $Y_\nu$, Hankel functions $H_\nu^{(1)},\, H_\nu^{(2)}$ and the Bessel functions of the third-kind $I_\nu,\, K_\nu$ as  seen, for example,  in  Watson \cite[ pp. 66, 74, 79]{Watson1944} respectively.

The following classical three-term recursion formula is a consequence of the above differential-difference formulae \eqref{E:any_bessel_recus_1} and \eqref{E:any_bessel_recus_2}.
\medskip

\begin{corollary} [(\textbf{Three-term recurrence})]\label{C:any_bessel_3_term}
 Let $\mathcal{B}_\nu\ \substack{
\times (\mathcal{C}_{\nu+n})\\ \longrightarrow} \ \mathcal{O}^\mathbb{Z}_d$  $\mathrm{(}\mathscr{C}_{\nu+n}\mathrm{)}$ be a solution of $\mathcal{B_\nu}$ in $\mathcal{O}^\mathbb{Z}_d$.  Then
%Let $\mathcal{O}_{d \Delta}$ be a $\mathcal{A}_2$-module as introduced in \eqref{Eg:two-seq-functions} and $\nu\in \mathbb{C}\backslash\{-\mathbb{N}\}$. Then the following map 
%	\[
%		\begin{array}{rrl}
%		\dfrac{\mathcal{A}_2}
%		{\mathcal{A}_2[2(X_2\partial_2+\nu)-X_1X_2-X_1/X_2]}&	
%	\longrightarrow&
%	\mathcal{O}^{\mathbb{Z}}\\
%	 1&\mapsto & \big(\mathscr{C}_{\nu+n}(x)\big)
%		\end{array} 
%	\]is $\mathcal{A}_2-$linear. As a result, we have the (well-known) three-term recursion
		\begin{equation}\label{E:any_bessel_3_term}
			2\nu\mathscr{C}_{\nu}(x)-x\mathscr{C}_{\nu-1}(x)-x\mathscr{C}_{\nu+1}(x)=0.
		\end{equation}
\end{corollary}
\medskip

Now we come to the main result of this subsection, which is about the generating function of the classical Bessel functions $(J_{\nu+n}(x))$. It is a two-variable analytic solution of the holonomic system of PDEs obtained from the Bessel module $\mathcal{B}_\nu$ \eqref{E:bessel_mod}.
\medskip

\begin{theorem}[(\bf{Asymptotic behaviour})]\label{T:Bessel_Gevrey}
The sequence $(J_{n+\nu})$ is a uniformly 1-Gevrey compacta solution of $\mathcal{B}_\nu$ in $\mathcal{O}_d^\mathbb{Z}$.
\end{theorem}
\medskip

\begin{proof} It follows from the the Corollary \ref{C:any_bessel_3_term} that $J_{\nu+n}$  satisfies 
	\[
		y_{n+2}-\dfrac{2(n+\nu+1)}{x}y_{n+1}+y_n=0.
	\]
As a result, the modified sequence ($Y_n$) where $Y_n:=J_{\nu+n}/n!$ satisfies the equation
	\begin{equation}\label{E:modified_bessel_3_term}
		Y_{n+2}-\dfrac{2(n+\nu+1)}{(n+2)x}Y_{n+1}+\dfrac{1}{(n+2)(n+1)}Y_n=0.
	\end{equation}
Now one fixes a compact set $K\subset\mathbb{C}\backslash\{0\}$. The space of functions analytic on $K$ becomes a Banach algebra $\mathbb{B}_K$ equipped with the usual sup norm $\|\cdot\|_K$.Therefore, we may apply the generalized Poincar\'e-Perron theory stated in Theorem \ref{T:PP} on the Banach algebra $\mathbb{B}_K$ for the   asymptotic behaviour of ($J_{\nu+n}$) on the compact set $K$. Since the characteristic equation of \eqref{E:modified_bessel_3_term}  is given by
	\[
		\lambda^2-2\lambda/x=0,
	\]so that its roots are $0$ and $2/x$ whose norms are obviously different. 
Consequently 	
	\begin{equation}\label{E:uniform_gervey}
		\|J_{n+\nu}/n!\|_K=O((2/d(K,0))^n).
	\end{equation}That is, according to the Remark \ref{R:gervey}, the sequence $(J_{n+\nu})$ is $1$-Gevrey.
\end{proof}
\medskip

\begin{remark}
	\begin{enumerate}
		\item It is instructive to compare the estimate \eqref{E:uniform_gervey}
 with the well-known asymptotic
	\[
		J_{\nu+n}(x)\sim \frac{e^{-\nu}}{\sqrt{2\pi n}}\Big(\frac{ex}{2n}\Big)^{\nu+n}
	\]as given, see for example, in \cite[\S5]{Gautschi}. An application of Stirling formula to \eqref{E:uniform_gervey} shows that we have the same order of growths from the two asymptotic estimates.  \footnote{In fact, this can be seen from the elementary growth estimate \cite[p. 225]{Watson1944}
	\[
		J_\nu(x)\approx \exp\big\{\nu+\nu\log(x/2) -(\nu+\frac12)\log \nu\big\}
	\cdot \big[1/\sqrt{2\pi}+\frac{c_1}{\nu}+\frac{c_2}{\nu^2}+\cdots\big]
	\]holds for a fixed $x$ as $\nu\to\infty$, 
so that although the ``analystic part\rq\rq{}, i.e., the ``$\sum_0^\infty$\rq\rq{} part of \eqref{E:gf_bessel} is convergent, while the ``principle part\rq\rq{}, or the ``$\sum_{-1}^{-\infty}$\rq\rq{} is divergent. }
		\item Similarly, one sees that the sequence $(J_{-n+\nu})$ is ``uniformly 1-Gevrey compacta" also. The conclusion remains valid if the standard Bessel functions are replaced by an arbitrary sequence $(\mathscr{C}_{n+\nu})$ as formulated above.
		\item It follows from the last remark item that the $z$-transform of the sequence $(J_{n+\nu})$, or the Borel resummation of $\sum_n J_{n+\nu}t^n$, is a well-defined analytic functions in two variables as defined in the Definition \ref{D:borel}.
		\item We also note that the conclusion of Theorem \ref{T:Bessel_Gevrey} remains valid if the uniform compacta norm is replaced by any other norm.
	\end{enumerate}
\end{remark}

\medskip

\begin{theorem}[(\textbf{Generating function of Bessel functions})]\label{T:Bessel_gf}  Let $\nu\in\mathbb{C}$. 
	\begin{enumerate}
		\item Let $(\mathscr{C}_{n+\nu})_n$ be a bilateral sequence of analytic functions which is a solution of the Bessel module $\mathcal{B}_\nu$ in $\mathcal{O}^\mathbb{Z}$. Then there exists a complex number $C_\nu$ such that
		\begin{equation}\label{E:gf_bessel}
				C_\nu\, t^{-\nu}\exp\Big[\dfrac{x}{2}\big(t-\dfrac{1}{t}\big)\Big]
				\sim\sum_{n=-\infty}^\infty \mathscr{C}_{\nu+n}(x)\, t^n,
		\end{equation}
	\end{enumerate}where the symbol ``$\sim$\rq\rq{} above means that the left-side is the Borel resummation\footnote{Borel resummation is defined in Definition \ref{D:borel}.} of the right-hand side.
	\item Moreover,  the holonomic system of PDEs \eqref{E:bessel_PDE} when realized in 
$\mathcal{O}_{dd}$ in Example \ref{Eg:O_deleted_d} defined by \eqref{E:O_dd_endow}  is given by
		\begin{equation}\label{E:PDE_bessel}
			y_x+(1/t-t)/2\, y=0,\quad
			\nu y+ty_t-{x}/2\,(1/t+t)\, y=0\footnote{See also the Appendix \ref{SS:PDE_list}.}.
		\end{equation}
\end{theorem}
\medskip

%\begin{remark} Note that the series on the left-hand side is divergent in general as it is $1$-Gevrey as defined in Remark \ref{R:gervey}. \footnote{In fact, this can be seen from the elementary growth estimate \cite[p. 225]{Watson1944}
%	\[
%		J_\nu(x)\approx \exp\big\{\nu+\nu\log(x/2) -(\nu+\frac12)\log \nu\big\}
%	\cdot \big[1/\sqrt{2\pi}+\frac{c_1}{\nu}+\frac{c_2}{\nu^2}+\cdots\big]
%	\]holds for a fixed $x$ as $\nu\to\infty$, 
%so that although the ``analystic part\rq\rq{}, i.e., the ``$\sum_0^\infty$\rq\rq{} part of %\eqref{E:gf_bessel} is convergent, while the ``principle part\rq\rq{}, or the ``$\sum_{-1}^{-\infty}$\rq\rq{} is divergent. }
%\end{remark}
%\medskip

To prove Theorem~\ref{T:Bessel_gf}, we need the following two lemmas, which serve the purpose of simplifying and solving the holonomic system of PDEs \eqref{E:bessel_PDE} in $\mathcal{O}_{dd}$.

\begin{lemma}\label{T:bessel_gf_mod}
The Bessel module $\mathcal{B}_\nu$ is isomorphic to
	\begin{equation}\label{E:bessel_mod_gf}
		\dfrac{\mathcal{A}_2}{\mathcal{A}_2[\partial_1+\frac{1}{2}({1}/{X_2}-X_2)]+\mathcal{A}_2[(\nu+X_2\partial_2)-\frac12 X_1(1/X_2+X_2)]}.%\longrightarrow\mathcal{B}_\nu.
	\end{equation}
\end{lemma}
\medskip

\begin{lemma}%[\textbf{(Classical appearance of the generating function)}]
\label{C:classical_app} The map
	\[
		\begin{array}{rcl}
\dfrac{\mathcal{A}_2}{\mathcal{A}_2[\partial_1+\frac{1}{2}({1}/{X_2}-X_2)]+\mathcal{A}_2[(\nu+X_2\partial_2)-\frac12 X_1(1/X_2+X_2)]}
		&\stackrel{\times t^{-\nu}\exp[\frac{x}{2}(t-\frac{1}{t})]}{\longrightarrow}
		&\mathcal{O}_{dd}
		\end{array}
	\]
	is well-defined and left $\mathcal{A}_2$-linear. In fact, this linear map is unique up to a complex multiple.
\end{lemma}
\medskip

\begin{proof} It is straightforward to verify that the function $ t^{-\nu}\exp[\frac{x}{2}(t-\frac{1}{t})]$ satisfies the system of PDEs
	\begin{equation}\label{E:bessel_PDE_fn}
		 \frac{\partial y}{\partial x} +\frac{1}{2}\big(\dfrac{1}{t}-t\big)y=0,\qquad
		 \nu y +t \frac{\partial y}{\partial t} -\frac{x}{2}\big(\dfrac{1}{t}+t\big)y=0.
	\end{equation}
	To see the uniqueness, for a general analytic function $y=y(x,\, t)$. The first equation yields $y(x,t)=g(t)\exp\big[(\dfrac{x}{2}(t-\dfrac{1}{t})\big]$ for some $g$. Substituting this $y$ into the second differential equation yields $g(t)=Ct^{-\nu}$ for some $C\in\mathbb{C}$.
\end{proof}
\medskip

\noindent\textit{Proof of Theorem~\ref{T:Bessel_gf}.}
	Recall the left $\mathcal{A}_2$-linear maps $\mathfrak{j}$ and $\mathfrak{z}$ as defined in \eqref{E:general-map-bessel} and \eqref{E:z-transform-2} respectively.
%		\begin{equation}\label{E:g-bessel-map_0}
%			\begin{array}{rrlcl}
%		\mathfrak{g}:&\mathcal{B}_\nu & \xrightarrow[]{\times \E [\frac{X_1}{2}(X_2-\frac{1}{X_2})]} 
%& \widetilde{\mathcal{A}}_2:= \overline{\mathcal{A}_2/\big[\mathcal{A}_2\partial_1+\mathcal{A}_2(X_2\partial_2+\nu)\big]} &
% \xrightarrow[]{\times t^{-\nu}}	\mathcal{O}_{dd}\\
%			\end{array}
%			\end{equation}
	We have the diagram
		\begin{equation}\label{E:commute-1}
			\begin{tikzcd} [row sep=large, column sep=huge]
			\mathcal{B}_{\nu} \arrow{r}{\mathfrak{j}} 
			 % \arrow[swap]{dr}{\mathfrak{g}} 
			\arrow[swap]{d}{\times \E [\frac{X_1}{2}(X_2-\frac{1}{X_2})]} 
			&   \mathcal{O}^{\mathbb{Z}} \arrow{d}{\mathfrak{z}}\\
 			 \widetilde{\mathcal{A}}_2 \arrow{r}{\times t^{-\nu}} &  \mathcal{O}_{d d}
			\end{tikzcd}
		\end{equation}
	in which $\widetilde{\mathcal{A}}_2:=\mathcal{A}_2/\big[\mathcal{A}_2\partial_1+\mathcal{A}_2(X_2\partial_2+\nu)\big]$. Here the left vertical map $\mathcal{B}_\nu\longrightarrow \widetilde{\mathcal{A}}_2$ of \eqref{E:commute-1} is given by \eqref{E:bessel_gen_map_2} from Theorem \ref{T:bessel_gen_map_2}. Since the above diagram commutes up to a complex multiple, the sum $\sum_n \mathscr{C}_{\nu+n}(x)\, t^n$, being the image of $1$ in $\mathcal{O}_{dd}$ via the top-right path, is also a formal solution to the system of PDEs \eqref{E:bessel_PDE}. Now the bottom-left path of \eqref{E:commute-1} gives $t^{-\nu}\exp\big[\dfrac{x}{2}(t-\dfrac{1}{t})\big]$ which is a solution to \eqref{E:bessel_PDE} because of Lemma~\ref{T:bessel_gf_mod} and Lemma~\ref{C:classical_app}. Since $\mathcal{B}_\nu$ is holonomic, the $\mathbb{C}$-dimension of the local solution space of \eqref{E:bessel_PDE} equals the multiplicity of $\mathcal{B}_\nu$, which is one. So, in summary, we may regard \eqref{E:gf_bessel} holds up to a complex scalar multiple $C_\nu$ which can be identically zero.			
\hfill\qed
\medskip

The next corollary determines the constant $C_\nu$ that appears in \eqref{E:gf_bessel} 
for various classes of Bessel functions. Moreover, a connection between these generating functions and Sonine integrals, see e.g. Watson \cite[Chap. VI]{Watson1944} is established.
\medskip

\begin{corollary}
			\begin{enumerate}
				\item  Let $\nu\in\mathbb{C}$. Then
					\begin{equation}\label{E:gf_J_nu}
				t^{-\nu}\exp\big[\dfrac{x}{2}\big(t-\dfrac{1}{t})\big]\sim\sum_{n=-\infty}^\infty J_{\nu+n}(x)\, t^n.
				\end{equation}
				\item  Let $\nu\in\mathbb{C}^\ast$. Then
					\begin{equation}\label{E:gf_Y_nu}
						e^{-i\pi/2}\,  t^{-\nu}\exp\big[\dfrac{x}{2}\big(t-\dfrac{1}{t})\big]\sim \sum_{n=-\infty}^\infty Y_{\nu+n}(x)\, t^n.
					\end{equation}
				\item  Let $\nu\in\mathbb{C}$. Then
					\begin{equation}\label{E:gf_I_nu}
						 t^{-\nu}\exp\big[\dfrac{x}{2}\big(t+\dfrac{1}{t})\big]\sim \sum_{n=-\infty}^\infty I_{\nu+n}(x)\, t^n.
					\end{equation}
				\item  Let $\nu\in\mathbb{C}^\ast$. Then
					\begin{equation}\label{E:gf_K_nu}
						i\pi \,  t^{-\nu}\exp\big[-\dfrac{x}{2}\big(t+\dfrac{1}{t})\big]\sim \sum_{n=-\infty}					^\infty K_{\nu+n}(x)\, t^n.
					\end{equation}
%			\item 
%				\begin{equation}
%					\textcolor{red}{C_2??}t^{-\nu}\exp\big[\dfrac{x}{2}\big(t-\dfrac{1}{t})\big]\textcolor{red}{\approx} \sum_{n=-\infty}^\infty H^{(1)}_{\nu+n}(x)\, t^n;
%			\end{equation}\textcolor{red}{over the sector ???}
%			\item 
%				\begin{equation}
%					\textcolor{red}{C_3??}t^{-\nu}\exp\big[\dfrac{x}{2}\big(t-\dfrac{1}{t})\big]=\sum_{n=-\infty}^\infty K_{\nu+n}(x)\, t^n.
%				\end{equation}
		\end{enumerate}
	where the notation ``$\sim$" denotes the Borel resummation as defined in Definition \ref{D:borel}.
\end{corollary}
\medskip

\begin{remark} We note that the \eqref{E:gf_J_nu} and \eqref{E:gf_I_nu} above become well-known equalities when $\nu$ is an integer. 
\end{remark}
\medskip

\begin{proof} \begin{enumerate}
	\item Suppose $\mathscr{C}_{\nu+n}=J_{\nu+n}$ in \eqref{E:gf_bessel}. We aim to show that $C_\nu=1$. In particular, we have
	\begin{equation}\label{E:gf_J_nu_1}
				C_\nu t^{-\nu-1}\exp\big[\dfrac{x}{2}\big(t-\dfrac{1}{t})\big]-\frac{J_{0}(x)}{t}\sim\sum_{n=-\infty \atop n\not=0}^{+\infty} J_{\nu+n}(x)\, t^{n-1}.
	\end{equation}
	That is, there is a $F\in \mathcal{O}_{dd}$\footnote{
	It is clear that
	\[
		\sum_{n=-\infty \atop n\not=0}^{+\infty} \frac{J_{\nu+n}(x)}{n}
		\, t^{n}
	\]is a formal primitive of the right-side of \eqref{E:gf_J_nu_1}.} such that
	\begin{equation}\label{E:gf_J_nu_2}
		C_\nu t^{-\nu-1}\exp\big[\dfrac{x}{2}\big(t-\dfrac{1}{t})\big]-\frac{J_{0}(x)}{t}
		=\mathfrak{z}\big( \partial_2(J_{\nu+n}/n)\big)=\partial_2 F
%				\sum_{n=-\infty \atop n\not=0}^{+\infty} J_{\nu+n}(x)\, t^{n-1}.
	\end{equation}where the $z$-transform given by the Borel-resummation is defined in \eqref{E:z-transform-2}, and the $\partial_2$'s are defined in \eqref{E:two-seq-functions_2}, where $\lim_{t\to -\infty}F(x, t)=0$ if $\Re(x)>0$. According to Theorem \ref{T:z-transform-2}, the Borel transform $\mathfrak{B}$ representing the $\mathfrak{z}$-transform is $\mathcal{A}$-linear, it follows that the left-side of \eqref{E:gf_J_nu_1} also has a primitive, which can be written as
	\[
		\int  \Big\{C_\nu t^{-\nu-1}\exp\big[\dfrac{x}{2}\big(t-\dfrac{1}{t})\big]-\frac{J_{\nu}(x)}{t}\Big\}\, dt.
	\]  	
Let $0<\delta< R$ be given and $|\arg x|<\pi/2$. Consider the contour 
		\[
			\Gamma_{R, \delta}=(Re^{-i \pi},\, \delta e^{-i \pi}) \cup \mathcal{C}_\delta\cup (\delta e^{i \pi},\, Re^{i \pi}),
		\]where $\mathcal{C}_\delta$ denotes the circle centred at the origin with radius $\delta$. Thus the contour $\Gamma_{R, \delta}$ can be considered a \textit{truncated Hankel\rq{}s contour} which emanates from $-\infty$ below the negative real-axis and then back to $-\infty$ above the negative real-axis after circulating the origin once in an anti-clockwise direction.

		Hence
	\[
%		\begin{split}
		C_\nu\int_{\Gamma_{R, \delta}} t^{-\nu-1}\exp\big[\dfrac{x}{2}\big(t-\dfrac{1}{t})\big]\, dt=\int_{\Gamma_{R, \delta}} \frac{J_{\nu}(x)}{t}\, dt
	\]
	
	We now take the Cauchy principal values  on both sides of this equation under the limits that $R\to+\infty$ and $\delta\to 0$. The right-side yields
		\[
			C_\nu\int_{\Gamma_{\infty, 0}} t^{-\nu-1}\exp\big[\dfrac{x}{2}\big(t-\dfrac{1}{t})\big]\, dt
		=2\pi i\, J_\nu (x).
	\]However, since the integral on the left-side converges to a Hankel contour, so we obtain
%		\begin{equation}\label{E:gf_J_nu_2}
		\[
			\int_{\Gamma_{R, \delta}} t^{-\nu-1}\exp\big[\dfrac{x}{2}\big(t-\dfrac{1}{t})\big]\, dt
			\longrightarrow
			\int_{-\infty}^{(0+)}  t^{-\nu-1}\exp\big[\dfrac{x}{2}\big(t-\dfrac{1}{t})\big]\, dt=2\pi i\, J_\nu(x)
%		\end{equation}
		\]as  $R\to+\infty$ and $\delta\to 0$, which actually equals to  $2\pi i J_\nu(x)$ from  the classical  Schl\"afli-Sonine integrals which is valid for $|\arg x|<\frac{\pi}{2}$, see e.g., Watson \cite[Chap. VI, p. 176, (2)]{Watson1944}. This proves that $C_\nu=1$ after comparing the above calculations. Alternatively, one can verify directly that the Hankel-type contour as the limit of $R\to 0,\, \delta\to 0$ in \eqref{E:gf_J_nu_2} satisfies the Bessel equation \eqref{E:Bessel_eqn} and the following asymptotic
		 \[
		 	\begin{split}
				\int_{-\infty}^{(0+)}  t^{-\nu-1}\exp\big[\dfrac{x}{2}\big(t-\dfrac{1}{t})\big]\, dt& \stackrel{t=\frac{2x}{u}}{
				=} \frac{x^\nu}{2^\nu} 
				\int_{-\infty}^{(0+)}  u^{-\nu-1}\exp\big[u-\dfrac{x^2}{4u}\big]\, du\\
				&\approx \frac{x^\nu}{2^\nu}\frac{2^\nu}{\Gamma(\nu+1)}=\frac{2^\nu}{\Gamma(\nu+1)}
			\end{split}
		\]when $x\to 0$. See also Watson \cite[pp. 175-176]{Watson1944}. Hence 
			\[
				\int_{-\infty}^{(0+)}  t^{-\nu-1}\exp\big[\dfrac{x}{2}\big(t-\dfrac{1}{t})\big]\, dt=2\pi i\, J_\nu(x).
			\]Therefore, we deduce $C_\nu=1$. This proves \eqref{E:gf_J_nu}.
	\item
Recall that 
	\[
		Y_{\nu+n}(x)=\frac{\cos (\nu+n)\pi\, J_{n+\nu}-J_{-\nu-n}}{\sin \nu\pi}
		=\cot\nu\pi\, J_{\nu+n}(x)-(-1)^n \csc\nu \pi\, J_{-\nu-n}(x).
	\]Since $\big(J_{-\nu-n}(x)\big)_n$ corresponds to PDEs  \eqref{E:bessel_PDE} with $\nu$ replaced by $-\nu$. Hence 
	\[
		t^{\nu}\exp\big[\dfrac{x}{2}\big(t-\dfrac{1}{t})\big]\sim\sum_{n=-\infty}^\infty J_{-\nu-n}(x)\, t^n.
	\]Thus
	\[
		\begin{split}
			\sum_{n=-\infty}^\infty Y_{\nu+n}(x)
			&\sim\sum_{n=-\infty}^\infty 
			\big[\cot \nu\pi J_{\nu+n} (x)-(-1)^n\csc\nu \pi J_{-\nu-n}(x)\big]\, t^n\\
			&=\sum_{n=-\infty}^\infty \cot \nu\pi J_{\nu+n}(x)\, t^n
			 - \csc\nu\pi \sum_{n=-\infty}^\infty  (-1)^nJ_{-\nu-n}(x) \, t^n\\
			 &=\cot \nu\pi\sum_{n=-\infty}^\infty  J_{\nu+n}(x)\, t^n
			 - \csc\nu\pi \sum_{n=-\infty}^\infty  J_{-\nu+n}(x) \, (-1/t)^n\\
			 &=\cot \nu\pi  t^{-\nu}\exp\big[\dfrac{x}{2}\big(t-\dfrac{1}{t})\big]
			 	-\csc\nu\pi  (-1/t)^\nu \exp\big[\dfrac{x}{2}\big(-\dfrac{1}{t}+t)\big]\\
			&=\big(\cot \nu\pi-(-1)^\nu \csc\nu\pi\big) t^{-\nu}\exp\big[\dfrac{x}{2}\big(t-\dfrac{1}{t})\big]\\
			&= e^{-i\pi/2} t^{-\nu}\exp\big[\dfrac{x}{2}\big(t-\dfrac{1}{t})\big]
		\end{split}
	\]as asserted.
	\item The proofs of (iii) and (iv) can be derived in a similar way as the proof of (ii) with the definitions $I_\nu(x)=J_\nu(ix)$ and $K_\nu(x)=\frac{\pi}{2}(I_{-\nu}(x)-I_\nu(x))/\sin\nu\pi$.
	\end{enumerate}
\end{proof}

\medskip

%\textcolor{red}{
%\begin{remark} Two alternative methods to prove the constant $C_\nu=1$ for the $(J_{\nu+n})$ will be discussed in the Appendix ??
%\end{remark}
%}
\medskip

By an argument using the Cauchy integral formula, we can now "extract" the $n$-th coefficient from the generating function \eqref{E:gf_J_nu} above to give a way to explain why the existence of the following integral representation of Bessel functions avoids using series expansions as in \cite[\S17.231]{WW}, \cite{Watson1944}, see also \cite[\S 10.10-11]{Derezinski2020}.
\medskip

\begin{corollary}[(\textbf{Schl\"afli-Sonine integrals})]\label{C:Sonine}
 Let $\nu\in\mathbb{C}$ be given and $|\arg x|<\displaystyle\frac{\pi}{2}$. Then we have
	\begin{equation}\label{E:Sonine}
		J_{\nu+n}(x)=\frac{1}{2\pi i}\int_{-\infty}^{(0+)}  t^{-\nu-1-n}\exp\big[\dfrac{x}{2}\big(t-\dfrac{1}{t})\big]\, dt
	\end{equation}
	holds for every $n\in \mathbb{Z}$.
\end{corollary}
\medskip

\medskip

\subsection{Generating function of difference Bessel functions}\label{SS:delta_gf}

	 Let us recall the Example \ref{Eg:O_delta} that  $\mathcal{O}_{\vartriangle}$  represents the space of analytic functions defined on a domain is a left $\mathcal{A}_1$-module  with
	\[
	(\partial f)(x)=f(x+1)-f(x), \qquad
	(Xf)(x)=xf(x-1),\quad f\in \mathcal{O}_{\vartriangle}.
	\]	
\medskip

We will occasionally use the well-known \textit{Newton polynomials} written in falling factorial notations: 
	\[
		(x)_0=1,\quad 
		(x)_n=
		\begin{cases} 
			x(x-1)\cdots (x-n+1), & n\ge 1,\\
			1/[(x+1)\cdots (x+n)] & n\le -1.
		\end{cases}
	\]For non-integer $\nu$, we adopt the notation
	\begin{equation}\label{E:fractional_Newton}
		(x)_\nu=\frac{\Gamma (x+1)}{\Gamma (x+1-\nu)}.
	\end{equation}
\medskip

\begin{example}[(\textbf{Difference Exponential function})]\label{Eg:delta_exp} 
Let $a\in\mathbb{C}$. We have seen in Example~\ref{Eg:exp} that the Weyl exponential is any solution of the $D$-module $\mathcal{A}/\mathcal{A}(\partial-a)$. Now in particular, the \textit{difference exponential function} $\exp_\Delta(\cdot;a)$ is defined as a solution of $\mathcal{A}/\mathcal{A}(\partial-a)$ in the $D$-module $\mathcal{O}_\Delta$, i.e.,
\[
	\mathcal{A}/\mathcal{A}(\partial-a) \stackrel{\times\exp_\Delta(\cdot;a)}{\longrightarrow} \mathcal{O}_{\Delta}.
\]
To obtain a binomial (or ``Newton") series expansion of $\exp_\Delta(\cdot;a)$, we factorize the above left $\mathcal{A}$-linear map as
\[
	\mathcal{A}/\mathcal{A}(\partial-a) \stackrel{\times S}{\longrightarrow} \mathcal{A}/\mathcal{A}\partial \stackrel{\times 1}{\longrightarrow} \mathcal{O}_{\Delta},
\]
in which $S=\sum{\frac{a_k}{k!}X^k}$ as in \eqref{E:exp}. We can illustrate the above by the 
commutative diagram
\begin{equation}
	\begin{tikzcd}%[row sep=large, column sep=large]
		\mathcal{A}/\mathcal{A}(\partial-a) \arrow{r}{\times S}
		\arrow[swap]{dr}{\times \exp_\Delta(\cdot;a)}
		& \mathcal{A}/\mathcal{A}\partial
		\arrow{d}{\times 1} \\
		& \mathcal{O}_{\Delta}
	\end{tikzcd}
	\end{equation}
As a result, we obtain
	\[
	      \exp_\Delta(x;a)=\sum_{k=0}^{\infty} \frac{a^k}{k!} x(x-1)\cdots(x-k+1).
	\]	
It is well-known that this binomial series converges absolutely to $(a+1)^x$ for any complex $x$ and any $|a|<1$, when an appropriate branch of logarithm is chosen.
\end{example}
\medskip

\begin{example}[(\textbf{Difference trigonometric functions})]\label{Eg:dtrigo} We continue the list of producing difference versions of trigonometric functions from Example \ref{Eg:trigo}, where the Weyl sine and the Weyl cosine are solutions of the $D$-module $\mathcal{A}/\mathcal{A}(\partial^2+1)$ in a $D$-module $N$. Now in particular, we define the \textit{difference sine function} and the \textit{difference cosine function} to be solutions of $\mathcal{A}/\mathcal{A}(\partial^2+1)$ in the $D$-module $\mathcal{O}_\Delta$. To obtain binomial series expansions of these two functions, we use the factorization
	\begin{equation}
		\mathcal{A}/\mathcal{A}(\partial^2+1) \stackrel{\times S}{\longrightarrow} \mathcal{A}/\mathcal{A}\partial \stackrel{\times 1}{\longrightarrow} \mathcal{O}_{\Delta}
	\end{equation} in which we respectively take $S=\sin X$ and $S=\cos X$ as in \eqref{E:sine-map} and \eqref{E:cosine-map}. This factorization can be illustrated by the 
commutative diagram
	\begin{equation}
	\begin{tikzcd}%[row sep=large, column sep=large]
		\mathcal{A}/\mathcal{A}(\partial^2+1) \arrow{r}{\times S}
		\arrow[swap]{dr}{\times \mathrm{Sin\, }X/\times \mathrm{Cos\, }X}
		& \mathcal{A}/\mathcal{A}\partial
		\arrow{d}{\times 1} \\
		& \mathcal{O}_{\Delta}
	\end{tikzcd}
	\end{equation} 
	Thus we obtain
	\[
		\sin_\Delta x:=\frac{1}{2i}[(1+i)^x-(1-i)^x]=\sum_{k=1}^{\infty}\frac{(-1)^{k+1}}{(2k-1)!}(x)_{2k-1}
	\]
	and 
	\[
	\cos_\Delta x:=\frac{1}{2}[(1+i)^x+(1-i)^x]=\sum_{k=0}^{\infty}\frac{(-1)^{k}}{(2k)!}(x)_{2k}.
	\]
	These two series converge absolutely if and only if either $\Re(x) > 0$ or $x = 0$, which  are unlike the power series of the usual sine and cosine functions.
\end{example}
\medskip

\begin{example}[(\textbf{Difference Bessel functions})] \label{Eg:delta_Bessel_fn}
Let $\mathcal{A}=\mathcal{A}_1$ as before. Recall that the Weyl Bessel of order $\nu$ \eqref{E:bessel_map} is any solution of the Bessel operator
	\[
		L=(X\partial)^2+X^2-\nu^2
	\]
in some $D$-module $N$, in particular, we define the \textit{difference Bessel function of order $\nu$}, denoted by $J^{\Delta}_{\nu}(x)$, to be a solution of $\mathcal{A}/\mathcal{A}L$ in the $D$-module $\mathcal{O}_\Delta$. Recall the factorization from \eqref{E:pos_bessel_comp_map}
	\begin{equation}\label{E:composition}
			\mathcal{A}/\mathcal{A}L\stackrel{\times S}{\longrightarrow} \mathcal{A}/\mathcal{A}(X\partial-\nu)\stackrel{\times (x)_\nu}{\longrightarrow}
			\mathcal{O}_\Delta,
	\end{equation}
	in which
	\[
		S=\mathcal{J}_\nu=\sum_{k=0}^{\infty} \frac{(-1)^kX^{2k}}{2^{\nu+2k}k!\Gamma(\nu+k+1)}.
	\]
%The corresponding commutative diagram can be given by
%\begin{equation}
%	\begin{tikzcd}%[row sep=large, column sep=large]
%		\mathcal{A}/\mathcal{A}L \arrow{r}{\times S}
%		\arrow[swap]{dr}{\times J^\Delta_\nu(X)}
%		& \mathcal{A}/\mathcal{A}(X\partial-\nu)
%		\arrow{d}{\times (x)_\nu} \\
%		& \mathcal{O}_{\Delta}
%	\end{tikzcd}
%	\end{equation} 
	
Thus we have obtained the Newton series expansion
	\begin{equation}\label{E:difference_bessel_fn}
		J^{\Delta}_{\nu}(x): =S((x)_\nu)
		=\sum_{k=0}^{\infty} \frac{(-1)^k}{2^{\nu+2k}k!\Gamma(\nu+k+1)}(x)_{\nu+2k}.
	\end{equation}
$J^{\Delta}_{\nu}(x)$ solves the second order difference equation \eqref{E:difference_bessel_eqn_0} or in its equivalent form
	\begin{equation}\label{E:difference_bessel_eqn}
		[(x+2)^2-\nu^2] y(x+2)-(x+2)(2x+3)y(x+1)+2(x+2)(x+1)y(x)=0,
	\end{equation}
	which is obtained from the action of the element $L=(X\partial)^2+(X^2-\nu^2)$ in $\mathcal{O}_\Delta$, see  Example \ref{Eg:O_delta}.

When $x\in \mathbb{Z}$, \eqref{E:difference_bessel_eqn} and \eqref{E:difference_bessel_fn} are precisely the difference Bessel equation and the difference Bessel function  $ J^{\Delta}_{\nu}(x)$ that appear to be first studied by Bohner and Cuchta \cite{BC_2017}\footnote{Bohner and Cuchta \cite{BC_2017} called $J^\Delta_\nu(x)$ the ``Bessel difference function\rq\rq{} and the equation \eqref{E:difference_bessel_eqn} the ``Bessel difference equation\rq\rq{} instead.}. Either by replacing $\nu$ in \eqref{E:difference_bessel_fn} by $-\nu$ or by considering \eqref{E:neg_nu_map}, we easily obtain
	\begin{equation}\label{E:difference_bessel_fn_2}
			J^{\Delta}_{-\nu}(x):=\sum_{k=0}^{\infty} \frac{(-1)^k}{2^{-\nu+2k}k!\Gamma(-\nu+k+1)}(x)_{-\nu+2k},
	\end{equation}
	so that $\{J^{\Delta}_{\nu}, J^{\Delta}_{-\nu}\}$ are $\mathbb{C}$-linearly independent solutions of \eqref{E:difference_bessel_eqn} in $\mathcal{O}_\Delta$, provided that  $2\nu\notin\mathbb{Z}$.
	\end{example}

\medskip
We now provide a statement about an half-plane convergence of the  $ J^{\Delta}_{\nu}(x)$.

\begin{proposition}\label{P:bessel_convergence} Let $\nu\in \mathbb{C}$ be arbitrary. Then the difference Bessel function  $ J^{\Delta}_{\nu}(x)$ \eqref{E:difference_bessel_fn} converges uniformly in each compact subset of $\Re(x)>0$.
\end{proposition}
\medskip

A justification of the above statement will be given in the Appendix \S\ref{SS:bessel_convergence}. However, it was demonstrated in \cite[p. 1569]{BC_2017} that although the series $ J^{\Delta}_{n}(x)$ diverges for $x<-1$, it converges at $x=-1$. Thus, a natural question is about the convergence in the strip $\{x: -1<\Re x<-\frac12\}$.
\medskip

Similar to Example \ref{Eg:bilateral}, we have

\begin{example}\label{Eg:bilateral} Let $\nu\in \mathbb{C}$ and consider the bilateral  sequence of classical Bessel functions $(J^\Delta_{\nu+n})_n$. Then the maps
	\begin{equation}\label{E:j-delta-map-bessel}
		\begin{array}{rcl}
		\mathfrak{j}_\Delta^+:\mathcal{B}_\nu & \stackrel{\times (J^\Delta_{\nu+n})}{\longrightarrow} &\mathcal{O}^\mathbb{Z}
		\end{array}
	\end{equation}
and
	\begin{equation}%\label{E:j-map-bessel}
		\begin{array}{rcl}
		\mathfrak{j}_\Delta^-:\mathcal{B}_\nu & \stackrel{\times (J^\Delta_{-\nu-n})}{\longrightarrow} &\mathcal{O}^\mathbb{Z}
		\end{array}
	\end{equation}  
are both left $\mathcal{A}_2$-linear.  
More generally, a general solution of $\mathcal{B}_\nu$ in $\mathcal{O}_\Delta^\mathbb{Z}$ will be  denoted by  $\times \mathrm{(}\mathscr{C}^\Delta_{\nu+n}\mathrm{)}$. That is,
%be an arbitrary sequence of solution of $\mathcal{B}_\nu$ in $\mathcal{O}^\mathbb{Z}$. Suppose
	\begin{equation}\label{E:general-delta-map-bessel}
		\begin{array}{rcl}
		\mathfrak{j}_\Delta: \mathcal{B}_\nu & \stackrel{\times (\mathscr{C}^\Delta_{\nu+n})}{\longrightarrow} &\mathcal{O}_\Delta^\mathbb{Z}.
		\end{array}
	\end{equation}
%\footnote{ Watson used $ \mathscr{C}_{\nu}(x)$ to denote an arbitrary solution to the system of  Bessel equation \eqref{E:Bessel_eqn}.}.
%Then the map
%	\begin{equation}\label{E:j-map-bessel-0}
%		\begin{array}{rcl}
%		\mathfrak{j}:\mathcal{B}_\nu & \stackrel{\times (\mathscr{C}_{\nu+n})}{\longrightarrow} &\mathcal{O}^\mathbb{Z}
%		1&\longmapsto& (\mathscr{C}_{\nu+n})
%		\end{array}
%	\end{equation}is left $\mathcal{A}_2$-linear.
\end{example}
\medskip

The following proposition and corollary are direct consequences of Corollary~\ref{C:Bessel_trans} with $N=\mathcal{O}_\Delta$ and $\mathscr{C}^\Delta_{\nu+n}$.

\begin{proposition}[(\bf{Delay-difference formulae})]\label{P:bilateral_Delta_PDE} 
Let $\nu\in\mathbb{C}$ and $\mathrm{(}\mathscr{C}^\Delta_{\nu+n}\mathrm{)}$ be an arbitrary (bilateral sequence of) solution of $\mathcal{B}_\nu\to \mathcal{O}_\Delta^\mathbb{Z}$ as defined in \eqref{E:general-delta-map-bessel}. Then the $\mathrm{(}\mathscr{C}^\Delta_{\nu+n}\mathrm{)}$
 satisfies
 % the following classical differential-difference formulae
%Let $\mathcal{O}_\Delta^\mathbb{Z}$ be endowed with the  $\mathcal{A}_2$-module structure described in Example~\ref{Eg:two-seq-functions_2}.  Let $\nu\in \mathbb{C}$, and let $(\mathscr{C}^\Delta_{\nu+n})_n$ denote an arbitrary bilateral sequence of analytic functions that satifies 
the following delay-difference formulae
	\begin{equation}\label{E:delta_bessel_PDE_a}
				x\Delta \mathscr{C}^{\Delta}_{\nu}(x-1)+\nu\mathscr{C}^{\Delta}_{\nu}(x)-
		x\mathscr{C}^{\Delta}_{\nu-1}(x-1)=0
			\end{equation}
and
		\begin{equation}\label{E:delta_bessel_PDE_b}
			x\Delta \mathscr{C}^{\Delta}_{\nu}(x-1)-\nu\mathscr{C}^{\Delta}_{\nu}(x)+
		x\mathscr{C}^{\Delta}_{\nu+1}(x-1)=0.
		\end{equation}
for each $n$. 
%Then each $\mathscr{C}^\Delta_{\nu+n}$ is a solution of the difference Bessel equation \eqref{E:difference_bessel_eqn} of order $\nu+n$, and the map
%	\begin{equation}\label{E:j-delta-map-bessel}
%		\begin{array}{rcl}
%		\mathfrak{j}_\Delta:\mathcal{B}_\nu &\stackrel{\times (\mathscr{C}_{\nu+n}^\Delta)}{\longrightarrow} &\mathcal{O}^\mathbb{Z}_\Delta
%		\end{array}
%	\end{equation}is left $\mathcal{A}_2$-linear.
\end{proposition}
\medskip

The following three-term recurrence is a direct consequence of Corollary~\ref{C:Bessel_3term}  with $N=\mathcal{O}_\Delta$, or directly from the above two corollary.
\medskip

\begin{proposition}[(\textbf{Three-term recurrence})] For every $\nu\in\mathbb{C}$ and every $x\in\mathbb{C}$, we have
		\begin{equation}\label{E:delta_bessel_3_term_a}
			2\nu\mathscr{C}_{\nu}^\Delta(x)-x\mathscr{C}_{\nu-1}^\Delta(x-1)-x\mathscr{C}_{\nu+1}^\Delta(x-1)=0.
		\end{equation}
\end{proposition}
%{\color{blue}
%\begin{proof}
%{\color{red}Given $\nu\in\mathbb{C}$, the element $2(X_2\partial_2+\nu)-X_1X_2-X_1/X_2$ is equivalent to $0$ in $\mathcal{B}_\nu$ as it is the difference of the two generators. Thus the left $\mathcal{A}_2$-linear map \eqref{E:j-delta-map-bessel} implies that
%	\[
%			2(n+\nu)\mathscr{C}_{\nu+n}^\Delta(x)-x\mathscr{C}_{\nu+n-1}^\Delta(x-1)-x\mathscr{C}_{\nu+n+1}^\Delta(x-1)=0
%	\]
%	for every $n\in\mathbb{Z}$ and $x\in\mathbb{C}$. So the result holds.}
%\end{proof}
%}
%\medskip

One deduces the following formulae originally derived by Bohner and and Cuchta \cite[Theorems 5-6, Corollary 7]{BC_2017} about $J_\nu^\Delta$ as  consequences of the above when $\mathscr{C}_{\nu+n}^\Delta(x)=J_{\nu+n}^\Delta(x)$ for all $n$.

\begin{corollary}[({\cite[Theorems 5-6, Corollary 7]{BC_2017}})]  For every $\nu\in\mathbb{C}$ and every $x\in\mathbb{C}$, we have
			\begin{align}
				x\Delta J^{\Delta}_{\nu}(x-1)+\nu J^{\Delta}_{\nu}(x)-
		xJ^{\Delta}_{\nu-1}(x-1)&=0,\label{E:delta_bessel_PDE_1}
		\\
				x\Delta J^{\Delta}_{\nu}(x-1)-\nu J^{\Delta}_{\nu}(x)+
		xJ^{\Delta}_{\nu+1}(x-1)&=0,\label{E:delta_bessel_PDE_2}
\\
				2\nu J_{\nu}^\Delta(x)-xJ_{\nu-1}^\Delta(x-1)-xJ_{\nu+1}^\Delta(x-1)&=0.\label{E:delta_bessel_3_term}
			\end{align}
\end{corollary}
\medskip

\medskip

\begin{proof} We provide an alternative proof based on the transmutation formulae from Proposition \ref{P:bessel_transmutation}. This more direct approach is similar to the derivation of the Corollary \ref{C:bessel_formulae}. Thus it suffices to prove the \eqref{E:delta_bessel_PDE_1}. Since the right-side of \eqref{E:bessel_transmutation_1} annihilates the $CJ^\Delta_\nu(x)$ for some constant $C$, and the fact that on the Newton series \eqref{E:difference_bessel_fn} gives
	\[
		\begin{split}
		(\partial+\nu/X)J^\Delta_\nu(x)
		&= 
		J_{\nu}^\Delta(x) -J_{\nu}^\Delta(x-1)+\nu J_{\nu}^\Delta(x+1)/(x+1)\\
		&= \frac{\Gamma(x+1)}{2^\nu\Gamma(\nu+1)\Gamma(x+2-\nu)}
		\Big( \frac2\nu+O\big(\frac1x\big)\Big)\\
		&=\frac{\Gamma(x+1)}{2^{\nu-1}\Gamma(\nu)\Gamma(x+2-\nu)}
		\Big( 1+O\big(\frac1x\big)
%		\frac{2}{x-\nu}-\frac{x-\nu}{2(2!) \nu(\nu+1)}-
%		\frac{(x-\nu-1)(x-\nu-2)}{2^2(2!) \nu(\nu+1)}+\cdots
\Big)\\
%		& \approx
%			\frac{C}{2^\nu\Gamma(\nu)\Gamma(-\nu)}\
%			\Big\{ -\frac{1}{\nu}+\frac{1}{2(2!)(\nu+1)}-\frac{\nu+2}{2^2(2!)}+\cdots\Big\}\\
%		&=\frac{C}{2^\nu\Gamma(\nu)\Gamma(1-\nu)}\Big\{1-\frac{\nu}{2(2!)(\nu+1)}+\frac{\nu(\nu+1)}{2^2(2!)}+\cdots\Big\}
		\end{split}
	\]as $x\approx 0$. On the other hand, according to the left-side of \eqref{E:bessel_transmutation_1}, the $(\partial+\nu/X)J^\Delta_\nu(x)$ must be equal to $CJ^\Delta_{\nu-1}(x)$ for some constant $C$. But 
		\[
			 J_{\nu-1}^\Delta(x)\approx \frac{\Gamma(x+1)}{2^{\nu-1}\Gamma(\nu)\Gamma(x+2-\nu)}		
		\]as $x\approx 0$. One immediately deduces that $C=1$ by comparing the above two asymptotes. This establishes the \eqref{E:delta_bessel_PDE_1}. The verification of  \eqref{E:delta_bessel_PDE_2} is similar. One can derive the  \eqref{E:delta_bessel_3_term} by subtracting the  \eqref{E:delta_bessel_PDE_1} and  \eqref{E:delta_bessel_PDE_2}.
		\end{proof}
\medskip

Now we establish that a difference analogue of the generating function $e^{\frac{x}{2}(t-\frac1t)}$ of the classical Bessel function for difference Bessel functions \eqref{E:bessel_classical_gf} is given by the function
	\[
		\big[\frac12\big(t-\frac1t\big)+1\big]^x
	\]and to show that it serves as a generating function to integer order difference Bessel function to \eqref{E:difference_bessel_eqn} in $\mathcal{O}_\Delta$.
\medskip

\begin{theorem}\label{T:delta_bessel_0_gf}
 Let $ |\frac{1}{2}(t-\frac{1}{t})|<1$ and $x\in\mathbb{C}$. Then
		\begin{equation}\label{gf-db-1}
		\Big[\frac{1}{2}(t-\frac{1}{t})+1\Big]^{x}=\sum_{n=-\infty}^{\infty}
		J^{\vartriangle}_{n}(x)\, t^n.
		\end{equation}
\end{theorem}
\medskip

\begin{proof} Since 	$| \frac{1}{2}(t-\frac{1}{t})|<1$, so the series
		\[
		\Big[\frac{1}{2}(t-\frac{1}{t})+1\Big]^{x}=\sum_{k=0}^{\infty} \binom{x}{k} 
		\big[ \frac{1}{2}(t-\frac{1}{t})\big]^{k}
		\]is  absolutely convergent. 
		Writing
		\[
		\Big[ \frac{1}{2}(t-\frac{1}{t})\Big]^{k}= \sum_{s=0}^{k}\binom{k}{s}
		(-\frac{1}{2t})^s (\frac{t}{2})^{k-s}.
		\]
		Then
		\[
		\Big[\frac{1}{2}(t-\frac{1}{t})+1\Big]^{x}=
		\sum_{k=0}^{\infty}(x)_{k} \sum_{s=0}^{k}\frac{1}{(k-s)! s!}
		(-\frac{1}{2t})^s (\frac{t}{2})^{k-s}.
		\]
		It is clear that the power of $t$ goes from $-\infty$ to $\infty$, as $k$ goes to infinity. Then we exchange the order of summations to arrive at
		\[
		\sum_{-\infty}^{\infty} t^n\sum_{s=0}^{\infty}(-1)^{s}
		\frac{(x)_{n+2s}}{s!(n+s)!2^{n+2s}}
		=\sum_{-\infty}^{\infty} J^{\Delta}_{n}(x)t^n.
		\]
	\end{proof}
\medskip

The series manipulation technique used in the above proof is in the same spirit of that for the classical Bessel generating function \eqref{E:bessel_classical_gf} presented in \cite[\S2.1]{Watson1944}. In order to deal with the generating function when $\nu\not=0$, we first need to study the asymptotic behaviour of $J^\Delta_{\nu+n}$ as $n\to\infty$. We apply the same technique of Banach algebra to $(J^\Delta_{n+\nu})$ as in the proof of Theorem \ref{T:Bessel_Gevrey} with suitable modification. 
\medskip

\begin{theorem}[(\bf{Asymptotic behaviour})]\label{T:delta_Bessel_Gevrey}
The sequence $(J^\Delta_{n+\nu})$ is a uniformly 1-Gevrey compacta solution of $\mathcal{B}_\nu$ in $\mathcal{O}_\Delta^\mathbb{Z}$.
\end{theorem}
\medskip

\begin{proof}
The proof makes use of the Newton transform $\mathfrak{N}$ which is described in Theorem \ref{T:Newton_trans}. Fix a compact $K\subset\mathbb{C}$ and for each $f\in\mathcal{O}_d$, one defines a norm by
\[
	|\| f |\| =\|\mathfrak{N}(f)\|_{L^\infty,K},
\]
where the right-side is the sup norm on the compact set $K$.
Then $(\mathcal{O}_d,|\|\cdot |\|)$ becomes a Banach algebra which can be verified easily.

The proof of Theorem \ref{T:Bessel_Gevrey} under this alternative norm shows that there exists $C>0$ such that
\[
	|\|J_{\nu+n} |\| \leq Cn!(\dfrac{2}{|\|x |\|})^n\quad\mbox{for large }n.
\]
Consequently, for large $n$,
\[
\|J^\Delta_{\nu+n}\|_{L^\infty,K}=\|\mathfrak{N}(J_{\nu+n})\|_{L^\infty,K}=|\|J_{\nu+n}|\|\leq Cn!(\dfrac{2}{|\|x|\|})^n.
\]
\end{proof}

\begin{remark}
Similarly, the sequence $(J^\Delta_{\nu-n})$ is also 1-Gevrey. The same is true if $J^\Delta$ is replaced by $\mathscr{C}^\Delta$. As a consequence, the $z$-transform of the sequence $(J^\Delta_{\nu+n})$ is a well-defined analytic function in two variables.
\end{remark}
\medskip

\begin{theorem}\label{T:difference_Bessel_gf} Let $\nu\in\mathbb{C}$. 
	\begin{enumerate}
		\item Let $(\mathscr{C}^\Delta_{n+\nu})_n$ be a bilateral sequence of analytic functions which is a solution of the Bessel module $\mathcal{B}_\nu$ in $\mathcal{O}_{\Delta}^\mathbb{Z}$. Then there exists a $1$-periodic function $C_\nu$ such that
	\begin{equation}\label{E:delta_bessel_gf}
				C_\nu(x)\, t^{-\nu}\big[\dfrac{1}{2}(t-\dfrac{1}{t})+1\big]^x
				\sim\sum_{n=-\infty}^\infty \mathscr{C}_{\nu+n}^\Delta (x)\, t^n,
			\end{equation}
			where the symbol $\sim$ means that the left-hand side is the Borel resummation \footnote{Definition \ref{D:borel}} of the right-hand side whenever it diverges.
		\item Moreover,  the manifestation of the holonomic system of PDEs \eqref{E:bessel_PDE} in $\mathcal{O}_{\Delta d}$ in Example \ref{Eg:O_delta_d}  		defined by \eqref{E:O_delta_d}  is given by
	\begin{equation}\label{E:PDE_dbessel}
		y(x+1, t)-y(x, t)=
		\frac{1}{2}(t-\frac{1}{t}) y(x, t),\quad
		\nu y(x,t) +ty_t(x,t)=\frac{x}{2}\big(t+\frac{1}{t}\big)y(x-1, t).
	\end{equation}
	\end{enumerate}
\end{theorem}
\medskip

To prove the above theorem, we need the following lemma, which serves the purpose of solving the holonomic system of PDEs that arise from $\mathcal{B}_\nu$ in $\mathcal{O}_{\Delta d}$.
\begin{lemma}%[\textbf{(Classical appearance of the generating function)}]
\label{E:classical_app_2}
	\[
		\begin{array}{rcl}
\dfrac{\mathcal{A}_2}{\mathcal{A}_2[\partial_1+\frac{1}{2}(1/{X_2}-X_2)]+\mathcal{A}_2[2(\nu+X_2\partial_2)-X_1(X_2+1/{X_2})]}&\stackrel{\times t^{-\nu}[\frac{1}{2}(t-\frac{1}{t})+1]^x }{\longrightarrow}
&\mathcal{O}_{\Delta d}
		\end{array}
	\]is a left $\mathcal{A}_2$-linear map.
\end{lemma}

\begin{proof} It follows from the two generators of the $\mathcal{B}_\nu$ defined above that we solve the system of PDEs
	\[
		 \partial_1 y+\frac{1}{2}(\dfrac{1}{X_2}-X_2)y=0,\qquad
		 (\nu+ X_2\partial_2)y-\frac{X_1}{2}(\dfrac{1}{X_2}+X_2)y=0
	\]for an analytic function $y=y(x,\, t)$ for the current choices of $\partial_j,\, X_j,\ j=1,\, 2$. That is,
	\begin{equation}\label{E:PDE_dBessel_gf}
		y(x+1, t)-y(x, t)=
		\frac{1}{2}(t-\frac{1}{t}) y(x, t),\qquad
		 \nu y(x,t) +t\frac{\partial y}{dt}(x,t)=\frac{x}{2}( t+\frac{1}{t})y(x-1, t).
	\end{equation}
	The first equation yields $y(x,t)=g(t)\big[\dfrac{1}{2}(t-\dfrac{1}{t})+1\big]^x$ for some $g$. Substituting this $y$ into the second differential-difference equation yields $g(t)=C_\nu t^{-\nu}$ for some $C_\nu\not=0$.
\end{proof}

\noindent\textit{Proof of Theorem~\ref{T:difference_Bessel_gf}.} Recall the maps $\mathfrak{j}_\Delta$ and $\mathfrak{z}_\Delta$ as defined in \eqref{E:general-delta-map-bessel} and \eqref{E:z-transform-3} respectively. We have the diagram
			\begin{equation}%\label{E:commute-1}
			  \begin{tikzcd} [row sep=huge, column sep=huge]
					    \mathcal{B}_\nu \arrow{r}{\mathfrak{j}_\Delta} 
					    \arrow[swap]{d}{\times\E\big[\frac{X_1}{2}\big(X_2-\frac{1}{X_2}\big)\big]}
					    %\arrow[swap]{dr}{\mathfrak{g}_\Delta} 
					    & \mathcal{O}^{\mathbb{Z}}_\Delta 
					    \arrow{d}{\mathfrak{z}_\Delta} \\
     						\widetilde{A}_2 \arrow{r}{\times t^{-\nu}}& \mathcal{O}_{\Delta d}
				  \end{tikzcd}
				\end{equation}
in which $\widetilde{A}_2:=\overline{\mathcal{A}_2/\big[\mathcal{A}_2\partial_1+\mathcal{A}_2(X_2\partial_2+\nu)\big]}$. Since the above diagram commutes up to a periodic multiple, the sum $\sum_n \mathscr{C}_{\nu+n}^\Delta (x)\, t^n$, being the image of $1$ in $\mathcal{O}_{\Delta d}$ via the top right path, is also a formal solution to the system of PDEs \eqref{E:bessel_PDE}. Now $t^{-\nu}\big[\dfrac{1}{2}(t-\dfrac{1}{t})+1\big]^x $ is a solution to \eqref{E:bessel_PDE} because of Lemma~\ref{E:classical_app_2}. Since $\mathcal{B}_\nu$ is holonomic, the $\mathbb{C}$-dimension of the local solution space of \eqref{E:bessel_PDE} equals the multiplicity of $\mathcal{B}_\nu$, which is one. So \eqref{E:delta_bessel_gf} holds up to a complex $1$-periodic function $C_\nu(x)$.		
\hfill\qed
\medskip

\begin{theorem}\label{T:delta_bessel_gf}
	\begin{enumerate}
		\item 	Let $\Re(x)<\Re(\nu)$. Then
\begin{equation}\label{gf-db-2}
e^{i\pi\nu}
\frac{\sin(x-\nu)\pi}{\sin(\pi x)}\,
t^{-\nu}\big[\frac{1}{2}(t-\frac{1}{t})+1\big]^{x}
{\color{blue}\sim} \sum_{n=-\infty}^{\infty}
J^{\Delta}_{n+\nu}(x)\,t^n,
\end{equation}
where the notation ``$\sim$" means that the left-side is the Borel-resummation of the right-side.
	\item Moreover, the the generating function from part \textrm{(i)} satisfies the system of PDE \eqref{E:PDE_dBessel_gf}\footnote{See also the Appendix \ref{SS:PDE_list}.}
	 \end{enumerate}
\end{theorem}
\medskip

\begin{proof}
From Theorem~\ref{T:difference_Bessel_gf}, we have
	\[
		C_{\nu}(x)t^{-\nu}\big[\dfrac{1}{2}\big(t-\dfrac{1}{t})+1\big]^x\sim\sum_{n=-\infty}^\infty J^\Delta_{\nu+n}(x)\, t^n.
	\]
	Subtracting $J_\nu(x)$ from both sides and then dividing both sides by $t$, we obtain
	\begin{equation}\label{Eq:integral3}
	C_{\nu}(x)t^{-\nu-1}\big[\dfrac{1}{2}\big(t-\dfrac{1}{t})+1\big]^x-\frac{J_\nu(x)}{t}\sim\sum_{n=-\infty,n\ne 0}^\infty J^\Delta_{\nu+n}(x)\, t^{n-1}.
	\end{equation}
	Now note that the right-side of \eqref{Eq:integral3} has an anti-$\partial_2$ (e.g.,  $\partial_2(\sum_{n=-\infty,n\ne 0}^\infty \frac{J^\Delta_{\nu+n}(x)}{n}\, t^n)$). Since the Borel resummation is $\mathcal{A}$-linear, it follows that the LHS of \eqref{Eq:integral3} has an antiderivative. 
	
	Now let $0<\delta< R$ be given and $|\arg x|<\pi/2$, and consider the contour 
		\[
			\Gamma_{R, \delta}=(Re^{-i \pi},\, \delta e^{-i \pi}) \cup \mathcal{C}_\delta\cup (\delta e^{i \pi},\, Re^{i \pi}),
		\]where $\mathcal{C}_\delta$ denotes the circle centred at the origin with radius $\delta$. Thus the contour $\Gamma_{R, \delta}$ can be considered a truncated Hankel\rq{}s contour which emanates from $-\infty$ below the negative real-axis and then back to $-\infty$ above the negative real-axis after circulating the origin once in an anti-clockwise direction. By Cauchy integral theorem, the integral of the left-side of \eqref{Eq:integral3} along $\Gamma_{R, \delta}$ is zero, i.e.,
	\begin{equation}\label{Eq:integral4}
		C_{\nu}(x)\int_{\Gamma_{R, \delta}} t^{-\nu-1}\big[\dfrac{1}{2}\big(t-\dfrac{1}{t})+1\big]^x\, dt = \int_{\Gamma_{R, \delta}}\frac{J^\Delta_\nu(x)}{t} \,dt = 2\pi i J^\Delta_\nu(x).
	\end{equation}
	By \eqref{E:dBessel integral rep} which is an analytic result to be obtained in the next subsection, we have
		\[
			C_\nu(x)=e^{i\pi\nu}\frac{\sin(x-\nu)\pi}{\sin(\pi x)}.
		\]
\end{proof}
\medskip

We remark that when $\nu=0$ in the above theorem,  one can recover the infinite sum in \eqref{E:delta_bessel_gf} from the generating function $\big[\dfrac{1}{2}(t-\dfrac{1}{t})+1\big]^x$ by Theorem \ref{T:delta_bessel_0_gf}, thus showing that $C_0(x)\equiv 1$.

\subsection{Integral representations}
Now we aim to find an integral representation of the difference Bessel functions that is analogous to the Sonine integral representation in Corollary~\ref{C:Sonine}.	
\medskip
	
	\begin{theorem}\label{T:solutions of db equ}
		Let $\nu\in \mathbb{C}$ and $\Re (x)<\Re (\nu)$. Let $U$ be a Hankel-type contour in the $t$-plane that starts from, $+\infty$ above the real-axis, circles around the origin in the counter-clockwise direction, and returns to $+\infty$ below the real-axis, and in particular
		\[
		t^{-\nu} \big[ \frac{1}{2}(t-\frac{1}{t})+1\big]^{x}\bigg\vert_{\partial U}
		=0.
		\]
		 Then the integral
	\[
		y_\nu(x)=\int_{U} t^{-\nu-1} \big[ \frac{1}{2}(t-\frac{1}{t})+1\big]^{x}\, dt
	\]
	\begin{enumerate}
	\item 	satisfies the delay-difference equations  \eqref{E:delta_bessel_PDE_1} and  \eqref{E:delta_bessel_PDE_2} above,
		\item and, in particular, the integral solves the  Bessel difference equation of order $\nu$ \eqref{E:difference_bessel_eqn}:
	\[
	x(x-1)\triangle^2 y(x-2)+x\triangle y(x-1)+x(x-1)y(x-2)-\nu^2y(x)=0.
	\]
	\end{enumerate}
	\end{theorem}
\medskip

	\begin{proof} We first note that because of $\Re (x)>\Re (-\nu)$, so that integral is a well-defined function of $x$.
		\begin{enumerate}
		\item We rewite the \eqref{E:delta_bessel_PDE_1} and \eqref{E:delta_bessel_PDE_2} in the equivalent forms
		\[
		\begin{split}
			&2\triangle y_{\nu}(x)-y_{\nu-1}(x)+y_{\nu+1}(x)=0,\\
			&2\nu y_{\nu}(x)-xy_{\nu-1}(x-1)-xy_{\nu+1}(x-1)=0.
		\end{split}
		\]
%		Let $y_{\nu}(x)=\int_{U} t^{-\nu-1} \big[ \frac{1}{2}(t-\frac{1}{t})+1\big]^{x} dt.$
	Then, it is easy to see that
		\[
		2\triangle 	y_{\nu}(x)
		=\int_{U}( t^{-\nu} -t^{-\nu-2})
		\big[ \frac{1}{2}(t-\frac{1}{t})+1\big]^{x} dt
		=y_{\nu-1}(x)-y_{\nu+1}(x)
		\]holds. To verify the second equation, we substitute the $y_n$ into the equation and apply integration-by-parts once to get
		\[
		\begin{split}
		xy_{\nu-1}(x-1)+xy_{\nu+1}(x-1)&=
		x\int_{U}( t^{-\nu} +t^{-\nu-2})
		\big[ \frac{1}{2}(t-\frac{1}{t})+1\big]^{x-1} dt\\
		&=2\int_{U} t^{-\nu}
		d\big[ \frac{1}{2}(t-\frac{1}{t})+1\big]^{x} \\
		&=-2\int_{U} 
		\big[ \frac{1}{2}(t-\frac{1}{t})+1\big]^{x}  dt^{-\nu}\\
		&=2\nu y_{\nu}(x),
		\end{split}
		\]as required
		\item 
				The proof to this part actually follows directly from the part (i) and Proposition \ref{P:bessel_ode_mod}. However, we offer a direct verification.
Applying $\triangle$ to the  \eqref{E:delta_bessel_PDE_1} and replacing $\nu$ by $\nu-1$ in the  \eqref{E:delta_bessel_PDE_2} respectively yield
		\[
		\begin{split}
		&x\triangle^2 y_{\nu}(x-1)+\triangle y_{\nu}(x)-x\triangle y_{\nu-1}(x-1)
		-y_{\nu-1}(x)+\nu \triangle y_{\nu}(x)=0,\\
		&x\triangle y_{\nu-1}(x-1)+xy_{\nu}(x-1)-(\nu-1) y_{\nu-1}(x)=0.
		\end{split}
		\]
		Adding the above two equations yields
		\[
		x\triangle^2 y_{\nu}(x-1)+\triangle y_{\nu}(x)+xy_{\nu}(x-1)-\nu y_{\nu-1}(x)+\nu \triangle y_{\nu}(x)=0.
		\]
		
		Replacing $x$ by $x-1$, and multiplying $x$ throughout both sides of the above equation yield
		\[
		x(x-1)\triangle^2 y_{\nu}(x-2)+x\triangle y_{\nu}(x-1)+x(x-1)y_{\nu}(x-2)-x\nu y_{\nu-1}(x-1)+x\nu \triangle y_{\nu}(x-1)=0.
		\]
		Multiplying   by $\nu$ on both sides of \eqref{E:delta_bessel_PDE_1} and substitute the resulting equation into the last expression yield precisely the  Bessel difference equation of order $\nu$ \eqref{E:difference_bessel_eqn}.
%		Hence, the Equation \eqref{PDEs for J_nu} is equivalent to Bessel difference equation of order $\nu$.		
		\end{enumerate}
	\end{proof}
	\medskip

	\begin{theorem}[(\textbf{Difference Schl\"afli-Sonine integrals})]\label{T:integral rep of db}
		Let $\Re(x)<\Re(\nu)$.  	Let $U$ be an Hankel-type contour parametrised by $t$ that starts at $+\infty$, above the real-axis, circles around the origin in the counter-clockwise direction, and returns to $+\infty$ below the real-axis, and such that $|2t+t^2|\geq1$ on $U$. Then
		\begin{equation}\label{E:dBessel integral rep}
		J^{\Delta}_{\nu}(x)=e^{i\pi\nu} \dfrac{\sin(x-\nu)\pi}{\sin(\pi x)}
		\dfrac{1}{2\pi i}\int_\infty^{(0+)}
		t^{-\nu-1} \Big[ \frac{1}{2}(t-\frac{1}{t})+1\Big]^x dt
		\end{equation}
\end{theorem}
\begin{proof} The integrand has branch points $0,\, -1\pm\sqrt{2}$ and $ \infty $ when $\nu , x$ are not integers. We pick the branch cuts  from $[-1-\sqrt{2},\, 0]$ and $[-1+\sqrt{2},\, +\infty]$ on the real-axis. Since $|2t+t^2|\geq1$ on $U$, the integrand allows an absolutely convergent expansion 
	\[
	\begin{split}
     	t^{-\nu-1}\big[\frac{1}{2}(t-\frac{1}{t})+1\big]^{x}
     	&
	=t^{-\nu-1}(1+\frac{t}{2})^x\big[1-\frac{1}{2t(1+t/2)}\big]^x\\
	&=t^{-\nu-1}(1+\frac{t}{2})^x\sum_{k=0}^{\infty}\binom{x}{k}(-1)^k [\frac{1}{2t(1+t/2)}]^k.\\
	\end{split}
	\]
	Observe that the \textit{lowest} term of \eqref{E:difference_bessel_fn} has
		\[
			\frac{2^{-\nu}\Gamma(x+1)}{\Gamma(\nu+1)\Gamma(x+1-\nu)}
			\approx \frac{1}{2^\nu \Gamma(\nu+1)\Gamma(1-\nu)},
		\]as $x\approx 0$. 
		
		Next we consider the lowest term of
 \[
 \begin{split}
 &\dfrac{1}{2\pi i}\int_{\infty}^{(0+)}
 t^{-\nu-1} \big[ \frac{1}{2}(t-\frac{1}{t})+1\big]^x dt\\
 &=\dfrac{1}{2\pi i}\int_{\infty}^{(0+)}t^{-\nu-1}(1+\frac{t}{2})^x
 \sum_{k=0}^{\infty}\binom{x}{k}(-1)^k [\frac{1}{2t(1+t/2)}]^k.\\
 \end{split}
 \]
That is, 
\[
\dfrac{1}{2\pi i}\int_{\infty}^{(0+)}t^{-\nu-1}(1+\frac{t}{2})^xdt.
\] 
Let $\frac{t}{2}=e^{i\pi} u$ in above integration. Then we have 
\[
\dfrac{1}{2\pi i}e^{-\nu\pi i} 2^{-\nu}
\int_{-\infty}^{(0+)}u^{-\nu-1}(1-u)^{x}\, du.
\]
We recall a classical result about beta function, see for example \cite[p. 104 (12)]{Wang_Guo1989}, in an Hankel-type contour which happens to match exactly with the change of variable $t=2e^{i\pi} u$ above, namely that
	\[
		\dfrac{\Gamma(p+q+1)}{\Gamma(p+1)\Gamma(q+1)}
		=\dfrac{1}{2\pi i} \int_{-\infty}^{(0+)}
		t^{-p-1}(1-t)^{-q-1}dt,\quad \Re(p+q+1)>0
	\]holds. Hence we obtain 
\[
\dfrac{1}{2\pi i}e^{-\nu\pi i} 2^{-\nu}
\int_{-\infty}^{(0+)}u^{-\nu-1}(1-u)^{x}du
=e^{-\nu\pi i} 2^{-\nu}\dfrac{\Gamma(\nu-x)}{\Gamma(\nu+1)\Gamma(-x)}
\]
which is 
valid for $\Re(x-\nu)<0$. Since 
\[
\dfrac{\Gamma(\nu-x)}{\Gamma(-x)}
=\dfrac{\sin(-x)\pi\,\Gamma(x+1)}{\sin(\nu-x)\pi \,\Gamma(x+1-\nu)}.
\]
Hence we obtain
\[
e^{i\pi\nu} \dfrac{\sin(x-\nu)\pi}{\sin(\pi x)}
\dfrac{1}{2\pi i}\int_{\infty}^{(0+)}
t^{-\nu-1}  (1+\frac{t}{2})^xdt
=\frac{2^{-\nu}\Gamma(x+1)}{\Gamma(\nu+1)\Gamma(x+1-\nu)}
	\approx  \frac{2^{-\nu} }{\Gamma(\nu+1)\Gamma(1-\nu)},
\]as $x\approx 0$ which matches  the coefficient of the lowest term of the other side. This implies that 
both sides of \eqref{E:dBessel integral rep} are the same and \eqref{E:dBessel integral rep} holds. \end{proof}
\medskip

We can modify the Hankel-type contour $U$ in the integral considered above to obtain a
 second linearly independent solution to the difference Bessel equation  \eqref{E:difference_bessel_eqn}.
 \medskip

\begin{theorem}\label{T:integral rep of db -nu}
	Let $\Re(x)>\Re(\nu+1)$. Let $U^\prime$ be an Hankel-type contour parametrised by $s=-1/t$ where $t$ parametrised by the contour $U$ as is defined in Theorem \ref{T:integral rep of db}.  Then
	\begin{equation}\label{E:dBessel integral rep-nu}
	J^{\Delta}_{-\nu}(x)=e^{i\pi} e^{-2i\pi\nu} \dfrac{\sin(x+\nu)\pi}{\sin\pi x}
	\dfrac{1}{2\pi i}\int_{U^\prime}
	s^{-\nu-1} \Big[ \frac{1}{2}(s-\frac{1}{s})+1\Big]^x ds.
	\end{equation}
\end{theorem}
\medskip

\begin{proof}Notice that both ends of the contour $U^\prime$ would be at the origin which corresponds to $U$ at $+\infty$. The restriction $\Re(x)>\Re(\nu+1)$ ensures that the integral converges at the origin. Substituting $t=-1/s$ into the contour integral yields
%	\begin{equation}\label{E:ad_hoc_1}
	\[
		\frac{1}{2\pi i}\int_U (-t)^{\nu-1}  \Big[ \frac{1}{2}(t-\frac{1}{t})+1\Big]^x dt
		=\frac{-e^{i\nu\pi}}{2\pi i}\int_U t^{\nu-1} \Big[ \frac{1}{2}(t-\frac{1}{t})+1\Big]^x dt
	\]where we need to be more restrictive under the extra temporary assumption  that $\Re(x)<\Re(-\nu)$  holds in order to ensure that the integral converges at $+\infty$.
Then the proof of the \eqref{E:dBessel integral rep-nu} follows closely that of the last theorem. One then considers the \textit{lowest term} from the integral and to compare it with that of the $J^\Delta_{-\nu}(x)$ in \eqref{E:difference_bessel_fn_2}.
We can now remove the extra assumption $\Re(x)<\Re(-\nu)$.
\end{proof}

\section{Half-Bessel modules I: Bessel polynomials}\label{S:half_bessel_I}

We recall the formula
	\begin{equation}\label{E:K_bessel_poly}
		K_{n+\frac12}(x)=\sqrt{\frac{\pi}{2}}\, e^{-x} x^{-n-\frac12} \theta_n(x)
	\end{equation}that connects the modified Bessel function (or Macdonald function) of order $n+\frac12$ and the \textit{reversed Bessel polynomials} $\theta_n(x)$.
%%%%%%%%%%%%%%%%%%%%%%%%%%%%%%%%%%%%%%%%%%%%%%%%%%
%	\footnote{We shall see later that it is more natural from our $D-$modulus perspective to work with the reversed Bessel polynomials $\theta_n(x)$ than the Bessel polynomials $y_n(x)=x^{n}\theta_n(1/x)$. Obviously one can derive all relevant formulae for $y_n(x)$ from those of $\theta_n(x)$ and their difference counterparts.}
%%%%%%%%%%%%%%%%%%%%%%%%%%%%%%%%%%%%%%%%%%%%%%%%%%
	The modified Bessel functions $I_\nu(x)$ and $K_\nu(x)$ are essentially related to $J_\nu(ix)$ and $Y_\nu(ix)$  respectively. We refer to Watson \cite{Watson1944} for their precise definitions so that both $I_\nu(x)$ and $K_\nu(x)$ are real on the real axis. We first modify the adopt the transmutation formulae for the Bessel modules from subsection  \S\ref{SS:transmutation_1} to half-Bessel modules that better suit our purpose of discussion in this section. Thus, it is more natural to discuss the reverse Bessel polynomials $\theta_n(x)$ than the Bessel polynomials $y_n(x)=x^{n}\theta_n(1/x)$ from our ${D}$-modules viewpoint. So the discussion of the $y_n(x)$ and its difference analogue will be postponed to the end of this section.

\subsection{Transmutation formulae}

Replacing  $\partial$ by $-i\partial$ and  $X$  by $iX$, we still have $[-i\partial,\, iX]=1$. So the transmutation formulae in Proposition \ref{P:bessel_transmutation} 
becomes

\begin{lemma}\label{L:Kbessel_transmutation}
For each $\nu\in\mathbb{C}$,
\begin{align}\label{E:Kbessel_transmutation}
[(X\partial)^2-(X^2+(\nu+1)^2)]\left(\partial-\dfrac{\nu}{X}\right)&=\left(\partial-\dfrac{\nu+2}{X}\right)[(X\partial)^2-(X^2+\nu^2)],\\
[(X\partial)^2-(X^2+(\nu-1)^2)]\left(\partial+\dfrac{\nu}{X}\right)&=\left(\partial+\dfrac{\nu-2}{X}\right)[(X\partial)^2-(X^2+\nu^2)],
\end{align}
which induce the left $\mathcal{A}$-linear maps
\[
\begin{array}{rcl}
\mathcal{A}/\mathcal{A}((X\partial)^2-(X^2+(\nu+1)^2))&\stackrel{\times(\partial-\frac{\nu}{X})}{\longrightarrow}&\mathcal{A}/\mathcal{A}((X\partial)^2-(X^2+\nu^2)),\\
\mathcal{A}/\mathcal{A}((X\partial)^2-(X^2+(\nu-1)^2))&\stackrel{\times(\partial+\frac{\nu}{X})}{\longrightarrow}&\mathcal{A}/\mathcal{A}((X\partial)^2-(X^2+\nu^2)).
\end{array}
\]
\end{lemma}
\medskip

As a result we deduce from this corollary the following well-known formulae with the help of the asymptotic of $K_\nu(x)$ as $x\to 0$. 

\begin{corollary}
For each $\nu\in\mathbb{C}$, the modified Bessel functions $K_\nu(x)$ satisfies the following recurrence relations.
\begin{align*}
	K'_\nu(x)-\dfrac{\nu K_\nu(x)}{x}&=-K_{\nu+1}(x),\\
	K'_\nu(x)+\dfrac{\nu K_\nu(x)}{x}&=-K_{\nu-1}(x).
\end{align*}
\end{corollary}
\medskip

    The formula \eqref{E:K_bessel_poly} suggests the following ``change of variable" of $X\partial $ to $X\partial-X-\nu$ and keep the $X$ unchanged that the transmutation formulae in Lemma \ref{L:Kbessel_transmutation} become
     
\begin{proposition}[(\textbf{Transmutation formulae})]
\label{P:bessel_poly_transmutation_1} For each $\nu\in\mathbb{C}$, we have
%	\begin{equation}\label{E:bessel_poly_transmutation_1} 
		\begin{align}
			\Big(X\partial^2-2(\nu+\frac12+X)\partial & +2(\nu +\frac12)\Big)(X\partial-X-2\nu)\notag\\
			&=\big(X\partial-X-2(\nu+\frac12)\big)\Big(X\partial^2-2(\nu-\frac12+X)\partial+2(\nu -\frac12)\Big),\label{E:bessel_poly_transmutation_1} \\ 
			\Big(X\partial^2-2(\nu-\frac12+X)\partial &+2(\nu-\frac12)\Big)\frac1X(\partial-1)\notag\\
			&=\frac1X (\partial-1-\frac1X)\Big(X\partial^2-2(\nu+\frac12+X)\partial+2(\nu+\frac12)\Big),\label{E:bessel_poly_transmutation_2} 
		\end{align}
%	\end{equation}
which induce the left $\mathcal{A}$-linear maps
\[
\begin{array}{rcl}
\mathcal{A}/\mathcal{A}\Big(X\partial^2-2(\nu+\frac12+X)\partial +2(\nu +\frac12)\Big)&\stackrel{\times(X\partial-X-2\nu)}{\longrightarrow}&\mathcal{A}/\mathcal{A}\Big(X\partial^2-2(\nu-\frac12+X)\partial+2(\nu -\frac12)\Big),\\
\mathcal{A}/\mathcal{A}\Big(X\partial^2-2(\nu-\frac12+X)\partial +2(\nu -\frac12)\Big)&\stackrel{\times\frac{1}{X}(\partial-1)}{\longrightarrow}&\mathcal{A}/\mathcal{A}\Big(X\partial^2-2(\nu+\frac12+X)\partial+2(\nu +\frac12)\Big).
\end{array}
\]
\end{proposition}

%Similarly, we have

%\begin{corollary}\label{C:bessel_poly_gauge}
%For each $\nu\in\mathbb{C}$, the following are well-defined left $\mathcal{A}$-linear maps:
%\[
%\begin{array}{rcl}
%\mathcal{A}/\mathcal{A}\Big(X\partial^2-2(\nu+\frac12+X)\partial +2(\nu +\frac12)\Big)&\stackrel{\times(X\partial-X-2\nu)}{\longrightarrow}&\mathcal{A}/\mathcal{A}\Big(X\partial^2-2(\nu-\frac12+X)\partial+2(\nu -\frac12)\Big),\\
%\mathcal{A}/\mathcal{A}\Big(X\partial^2+2(\nu-\frac12-X)\partial -2(\nu -\frac12)\Big)&\stackrel{\times(X\partial-X+2\nu)}{\longrightarrow}&\mathcal{A}/\mathcal{A}\Big(X\partial^2+2(\nu+\frac12-X)\partial-2(\nu +\frac12)\Big).
%\end{array}
%\]
%\end{corollary}
\medskip

\subsection{Reverse Bessel polynomial modules} \label{SS:reverse_bessel_poly_mod}
	
\subsubsection{Reverse Bessel polynomial operator}\label{SSS:bessel_poly_oper}
	Some of Burchnall\rq{}s formulae found in \cite{Burchnall_1953} are applications of the Weyl-algebraic formulae to be derived in this section. They are special cases in the manifestation of $D$-modules  $\mathcal{O}_d$. We shall change the  manifestation of $D$-modules to  $\mathcal{O}_\Delta$ when we treat the difference Bessel polynomials later.  We are ready to  set $\nu=n+\frac12$ and $\nu=n-\frac12$ for $n\in \mathbb{Z}$ in \eqref{E:bessel_poly_transmutation_1}  and \eqref{E:bessel_poly_transmutation_2} respectively to obtain
	
\begin{corollary}\label{C:bessel_poly_gauge_2}
For each $n\in\mathbb{Z}$, the following are well-defined left $\mathcal{A}$-linear maps.
\[
\begin{array}{rcl}
	\mathcal{A}/\mathcal{A}\big(X\partial^2-2(n+1+X)\partial +2(n+1)\big)&\stackrel{\times(X\partial-X-2n-1)}{\longrightarrow}&\mathcal{A}/\mathcal{A} \big(X\partial^2-2(n+X)\partial +2n\big),\\
	\mathcal{A}/\mathcal{A}\big(X\partial^2-2(n-1+X)\partial +2(n-1)\big) & \stackrel{\times\frac{1}{X}(\partial-1)}{\longrightarrow} & \mathcal{A}/\mathcal{A} \big(X\partial^2-2(n+X)\partial +2n\big).
\end{array}
\]
\end{corollary}

		Let $\Theta_0=1$ and consider for each $n\in\mathbb{Z}$ the gauge transformations defined by right multiplication by
		\[
			G_n :=-(X\partial-X-2n-1),\quad n\in\mathbb{Z}.  \mod X\partial
		\]
Then we define
		\begin{equation}\label{E:Theta_N}
			\Theta_n(X)= G_{n-1}\, \cdots  G_1\, G_0\, \Theta_0(X)\mod X\partial
		\end{equation}
to be the \textit{reversed Bessel polynomial of degree} $n$. 
\medskip

\begin{theorem}\label{T:theta_bessel_poly_map} For each $n\in\mathbb{N}\cup \{0\}$, the map
	\[
		\mathcal{A}/\mathcal{A}\big(X\partial^2-2(n+X)\partial +2n)\big)\stackrel{\times
		\Theta_n(X)}{\longrightarrow}\mathcal{A}/\mathcal{A} \big(X\partial^2-2X\partial \big)
	\]
	where $\Theta_n(X)$ is defined in \eqref{E:Theta_N} is  a well-defined left $\mathcal{A}$-linear map.
\end{theorem}
\medskip

	In particular, we have
	\begin{equation}
		\begin{split}
			\Theta_1(X)&=G_0\, \Theta_0 (X) = (-X\partial+X+1)=X+1\quad \mod X\partial\\ 
			\Theta_2(X)&= G_1\, G_0\, \Theta_0(x) = 
			(-X\partial +X+3)\, (X+1)=X^2+3X+3,\quad \mod X\partial\\
			\Theta_3(X)&=G_2\, G_1\, G_0 \, \Theta_0(x)
			(-X\partial +x-5) X^2+3X+3=X^3+6X^2+15X+15,\quad \mod X\partial\\
			\cdots & \cdots
		\end{split}
	\end{equation}
\medskip
			
\begin{theorem}\label{T:theta_bessel_poly_map_2}  Let $n\in\mathbb{N}_0$  and $\Theta_n(X)= G_{n-1} \cdots  G_1 G_0 \Theta_0(X)$. Then we have
	\begin{equation}
		\Theta_n(X)=\sum_{k=0}^n \frac{(n+k)!}{2^k (n-k)!\, k!}X^{n-k},
	\end{equation}
and the map
		\begin{equation}
		\mathcal{A}/\mathcal{A}\big(X\partial^2-2(n+X)\partial +2n)\big)
\stackrel{\times\Theta_n(X)}	{\longrightarrow} \mathcal{A}/ \mathcal{A}\partial
	\end{equation}
is left $\mathcal{A}$-linear.
\end{theorem}
\medskip

On the other hand, let $\Theta_0=1$ and consider the gauge transformation $H$ defined by a left multiplication by the element $-\frac1X(\partial-1)$. Then we have
		\begin{equation}%\label{E:Theta_N}
			1=\Theta_0(X)= H^{n}(\Theta_n(X))=(H \, \cdots \, H)(\Theta_n(X)).
		\end{equation}
\medskip

That is,
	\begin{equation}
		\begin{split}
			\Theta_0(X)&=H(\Theta_1 (X)) = 1/X(1-\partial) (X+1)\quad \mod \mathcal{A}\partial\\ 
			\Theta_0(X)&=H^2(\Theta_2(X)) =\big[ 1/X(1-\partial)\big]^2( X^2+3X+3),\quad \mod \mathcal{A}\partial\\
			\Theta_0(X)&=H^3(\Theta_3(X))=\big[ 1/X(1-\partial)\big]^3(X^3+6X^2+15X+15),\quad \mod \mathcal{A}\partial\\
			\cdots & \cdots
		\end{split}
	\end{equation}
\medskip

\begin{corollary}\label{C:Besselpoly_trans}
For each $n$, let $\Theta_n$ be defined as above, which is a solution of $X\partial^2-2(n+X)\partial+2n$ in the left $\mathcal{A}$-module $\mathbb{C}[X]$. Then
\begin{align*}
	(X\partial-X-2n-1)\Theta_n + \Theta_{n+1}=0,\\
	(\partial-1) \Theta_n + X\Theta_{n-1}=0.
\end{align*}
as elements in $M$.
\end{corollary}
\medskip

\begin{proof} We only proof the first expression. The second expression can be dealt with similarly. Let $\nu=n+\frac12$ in \eqref{E:bessel_poly_transmutation_1}. Then we note that the right-side of \eqref{E:bessel_poly_transmutation_1} annihilates the $\Theta_N(X)$, while the left-side of \eqref{E:bessel_poly_transmutation_1} therefore implies that 
	\[
		(X\partial-X-(2n+1))\Theta_n(X)=C\Theta_{n+1}(X),\mod \mathcal{A}\partial
	\]for some constant $C$. A simple comparison of the highest terms on both sides yields that $C=-1$. This proves the desired result.
\end{proof}
\medskip

We immediately deduce from the above Corollary the following recursion.

\begin{corollary} [(\textbf{Three-term recurrence})] \label{C:Besselpoly_3term}
For each $n$, let $\Theta_n$ be defined as above, which is a solution of $X\partial^2-2(n+X)\partial+2n$ in the left $\mathcal{A}$-module $\mathbb{C}[X]$. Then
		\[
			\Theta_{n+1}-(2n+1)\Theta_n-X^2\Theta_{n-1}=0
		\]
as an element in $M$.
\end{corollary}
\medskip

We next show that the reverse Bessel polynomial equation \cite[p. 8]{Grosswald} or more generally the \textit{half-order Bessel polynomial operator}
		\[
			L_n=X\partial^2-2(n+X)\partial +2n
		\]
		can be studied as a special case of the \textit{half-Bessel module} below. The construction of this module is motivated by Corollary~\ref{C:Besselpoly_trans} and the $\mathcal{A}_2$-module structure endowed on $\mathcal{O}^{\mathbb{N}_0}$
		as in Example~\ref{Eg:seq-functions_poisson}.

\begin{definition}\label{D:Theta} We define the \textit{reverse Bessel polynomial module} to be the
modified half-Bessel module to be the left $\mathcal{A}_2$-module
	\begin{equation}\label{E:Theta}
\Theta=\dfrac{\mathcal{A}_2}{\mathcal{A}_2(\partial_1-1+{X_1}/{\partial_2})+\mathcal{A}_2(X_1\partial_1-2X_2\partial_2-1-X_1+\partial_2)}.
	\end{equation}
\end{definition}
\medskip

We first verify that the Bessel polynomial module $\Theta$ defined above is holomonic.
\medskip

\begin{theorem} \label{T:B_poly_mod_holonomic}
The left $\mathcal{A}_2$-module $\Theta$ has dimension two and multiplicity two. In particular, $\Theta$ is holonomic.
\end{theorem}
\medskip

\begin{proof}
Multiply  $\partial_2$ throughout the first generator from \eqref{E:Theta} yields
\[
		\partial_2\partial_1-\partial_2+X_1
	\]
before eliminating $\partial_1$ from the second generator from the \eqref{E:Theta}
	\[
		X_1\partial_1-2X_2\partial_2-1-X_1+\partial_2.
	\]This yields
	\begin{equation}\label{E:B_poly_3th_generator}
		\begin{split}
			\partial_2\partial_1-\partial_2+{X_1}&=
			\partial_2\big(\frac{2X_2}{X_1}\partial_2-\frac{1}{X_2}\partial_2+1+1/X_1\big)-\partial_2+X_1\\
		&=((2X_1-1)\partial_2^2+3\partial_2+X_1^2)/X_1
		\end{split}	
	\end{equation}	
Thus we can replace the original two generators by 
	\[
		((2X_1-1)\partial_2^2+3\partial_2+X_1^2)/X_1,
		\quad
		\partial_1-(2X_2+1)/X_1\partial_2-1-1/X_1.
	\]This renders the elements in the $\Gamma^k\Theta$ are spanned by $\{X_1^aX_2^b: a+b\le k\}$ and 
$\{X_1^{a\rq{}}X_2^{b\rq{}}\partial_2: a+b\le k-1\}$. Hence the $\Theta$ shares the same Hilbert polynomial with that studied in Example \ref{E:prime_integrable_eg_2} so that 
%Define a connection $\nabla$ in the trivial line bundle on $\mathbb{C}\times\mathbb{C}$ by having matrix relative to the standard frame
%	\[
%		\omega=[x_1x_2-x_2]dx_1+\frac{1}{2}[2x_1-x_1^2]dx_2.
%\]
%Verify that $d\omega-\omega\wedge\omega=0$ so that $\nabla$ is a flat connection. It induces a holonomic $\mathcal{A}_2-$module.
\[
\dim_{\mathbb{C}}\mathrm{Gr}^kM=k^2.
\]It follows that the $\Theta$ is holonomic, hence dimension two and with multiplicity two.
\end{proof}
\medskip

Eliminating $\partial_1$ again from the above two relations of $\mathcal{A}_2$ from the Definition $\Theta$ \eqref{D:Theta} yields the following element of $\mathcal{A}_2$-module. 
\medskip

\begin{proposition}[(\textbf{Three-term recursion})] \label{P:Weyl_mod_half_bessel_3term}
 The natural map
	\[
\dfrac{\mathcal{A}_2}{\mathcal{A}_2[(2X_1-1)\partial_2^2+3\partial_2+X_1^2]}\longrightarrow\Theta
	\]
is a well-defined left $\mathcal{A}_2$-linear surjection.
\end{proposition}
\medskip

The following records the relationship between the half-Bessel module $\Theta$ defined above and the second order half-Bessel operator $L_n$.
\medskip

\begin{proposition}[(\textbf{Weyl-algebraic reverse Bessel-polynomial sequence of $\Theta$})]\label{P:Weyl_Bessel_Poly_Eqn} The natural map
\[
\dfrac{\mathcal{A}_2}{\mathcal{A}_2[X_1\partial_1^2-2(X_1+X_2\partial_2)\partial_1+2X_2\partial_2]}\longrightarrow\Theta
\]
is a well-defined left $\mathcal{A}_2$-linear surjection.
\end{proposition}
\medskip

The derivation of this result is given in the Appendix \ref{A:Weyl_Bessel_Poly_Eqn}.

\begin{proof} See the Appendix.
\end{proof}

\subsubsection{Classical reverse Bessel polynomials} 
To study Poisson generating functions, i.e., generating functions of the form $\displaystyle\sum_\mathbb{N} a_n {t^n}/{n!}$, one uses the $D$-module structure endowed on spaces $\mathbb{C}^{\mathbb{N}_0}$ and $\mathcal{O}^{\mathbb{N}_0}$ of (one-sided) sequences as in Example~\ref{Eg:seq-functions_poisson}. The purpose of the section is to illustrate that many classically known defining formulae about the reverse Bessel polynomials can be derived and presented in a coherent manner as consequences of manifestation of  the $D$-module structures discussed in the last section. We shall show that the same $D$-module structure with a different manifestation would allow us to reach an analogous theory for difference reverse Bessel polynomials in the next subsection. Thus except for the transmutation formulae derived from above, the remaining formulae that we derive here are well-known in the classical Bessel polynomial literature, e.g., \cite{Grosswald}. Reader who are looking for new formulae about difference Bessel polynomials may  proceed directly to the next section.

We deduce from Theorem \ref{T:theta_bessel_poly_map_2} the expected result:

\begin{proposition} Let $n\in \mathbb{N}_0$. Then the left $\mathcal{A}$-linear map
		\[
			\begin{array}{rclll}
			\mathcal{A}/\mathcal{A}\big(X\partial^2-2(n+X)\partial +2n)\big) 
			&\stackrel{\times\Theta_n(X)}{\longrightarrow} 
			& \mathcal{A}/ \mathcal{A}\partial
			&\stackrel{\times 1}{\longrightarrow} 
			&\mathcal{O}%,\\
%			1      & \mapsto    & \Theta_n{\color{blue}(X)} &  \mapsto  &\theta_n(x)
			\end{array}
		\]
		sends $1$ to the classical reversed Bessel polynomial
		\[
			\theta_n(x)=\sum_{k=0}^n \frac{(n+k)!}{2^k (n-k)!\, k!}\, x^{n-k}.
		\]
\end{proposition}
\medskip

\begin{example}\label{Eg:bilateral_theta} Consider the  sequence of classical Bessel polynomials $(\theta_n)$. Then the maps
	\begin{equation}\label{E:j-map-bessel_poly}
		\begin{array}{rcl}
		\mathfrak{j}:\Theta & \stackrel{\times (\theta_n)}{\longrightarrow} &\mathcal{O}^{\mathbb{N}_0}
		\end{array}
	\end{equation}
is well-defined and left $\mathcal{A}_2$-linear.  
\end{example}
\medskip

The following differential-difference equations and three-term recurrence, as well as Corollaries~\ref{C:Besselpoly_trans} and \ref{C:Besselpoly_3term} about classical Bessel polynomials below present solution of the module $\Theta$ in $\mathcal{O}_d$. They follow from Example \ref{E:j-map-bessel_poly}, hence giving new proofs to these classical results. We shall see how solutions of the module $\Theta$ in $\mathcal{O}_\Delta$ gives their difference analogues below.
	
\begin{theorem}[(\textbf{differential-difference equations} \cite{Grosswald}, p. 19)]\label{T:differential-differential}
The reverse classical Bessel polynomials $\{\theta_n(x)\}$ satisfy the following (well-known) formulae
	\begin{equation}\label{E:bessel_poly_trans}
	\begin{split}
		&\theta_n^\prime(x)-\theta_n(x)+x\theta_{n-1}(x)=0,\\
		&x\theta_n^\prime(x)-(x+2n+1)\theta_n(x)+\theta_{n+1}(x)=0
	\end{split}
	\end{equation}
for all $n\ge 1$ and $x$. 
%The generating function $f(x,\, t)$ for the reverse classical Bessel polynomials satisfies the system of partial differential equations
%	\begin{equation}\label{E:bessel_poly_PDE}
%		\begin{split}
%			&f_{xt}(x,\, t)-f_t(x,\, t)+xf(x, \, t)=0,\\
%			&xf_x (x,\, t)+(1-2t)f_t(x,\, t)-(1+x)f(x,\, t)=0
%		\end{split}
%	\end{equation}
%for all $x$ and $t$.
\end{theorem}

\begin{proof} This follows from Corollary \ref{C:Besselpoly_trans} with the manifestation of analytic function space $\mathcal{O}_d$ as a $\mathcal{A}_1$-module  given by the Example \ref{Eg:O_deleted_d}. 
\end{proof}
%{\color{blue}
%\begin{proof}
%Construct the fundamental left $\mathcal{A}_2$-linear maps:
%	\begin{equation}\label{E:bessel_poly_seq}
%			\begin{array}{rcl}
%		\Theta & \stackrel{\times(\theta_n(x))}{\longrightarrow} & \mathcal{O}^{\mathbb{N}_0}
%			\end{array}
%	\end{equation}
%	and
%	\begin{equation}\label{E:bessel_poly_gf}
%			\begin{array}{rcl}
%		\Theta & \stackrel{\times f(x,\, t)}{\longrightarrow} & \mathcal{O}_{d d}
%			\end{array}
%	\end{equation}
%	Since $\partial_1-1+X_1/\partial_2$ and $X_1\partial_1-2X_2\partial_2-1-X_1+\partial_2$ are generators of $\Theta$, they are both equivalent to $0$ in $\Theta$. Thus \eqref{E:bessel_poly_trans} follows from the left $\mathcal{A}_2$-linear map \eqref{E:bessel_poly_seq}, and \eqref{E:bessel_poly_PDE} follows from the left $\mathcal{A}_2$-linear map \eqref{E:bessel_poly_gf}.
%\end{proof}
%}
\medskip

%It follows from Proposition \ref{P:Weyl_mod_half_bessel_3term} and the structure of the above $\mathcal{A}_2-$module the well-known three-term recursion formula for reversed Bessel polynomials.
%\medskip

\begin{proposition}[(\textbf{Three-term recurrence} \cite{Grosswald}, p. 18)]  The sequence $(\theta_n)$ satisfies the (well-known) three-term recurrence formula
	\begin{equation}\label{E:bessel_poly_3term}
		 \theta_{n+2}(x)-(2n+3)\theta_{n+1}(x)-x^2\theta_n(x)=0
	\end{equation}
	for every $n\in\mathbb{N}_0$ and every $x$.
\end{proposition}

\begin{proof} This follows from the last Corollary or directly from the  Corollary \ref{C:Besselpoly_3term}  with the manifestation of the $\mathcal{A}_1$-module $\mathcal{O}_d$ as given by the Example \ref{Eg:O_deleted_d}. 
\end{proof}
\medskip

%\begin{theorem}[(\textbf{Generating function})] \label{T:gf_bessel_poly} Let $f(x,\, t)$ be a solution of $\Theta$ in $\mathcal{O}_{dd}$. Then it satisfies the system of partial differential equations
%	\begin{equation}\label{E:bessel_poly_PDE}
%		\begin{split}
%			&f_{xt}(x,\, t)-f_t(x,\, t)+xf(x, \, t)=0,\\
%			&xf_x (x,\, t)+(1-2t)f_t(x,\, t)-(1+x)f(x,\, t)=0
%		\end{split}
%	\end{equation}
%for all $x$ and $t$.
%\end{theorem}

%\begin{proof} This follows from applying the two relations from \eqref{E:delta_Bessel_poly_2_more_elements} to $f(x,\, t)$ in the $D$-module $\mathcal{O}_{dd}$. 
%\end{proof}
\medskip

%\begin{proof} It is straightforward to check that any function of two variables $f(x,t)$ in
%\begin{equation}\label{E:gf-map-bessel_poly}
%		\begin{array}{rcl}
%		\mathfrak{j}:\Theta & \stackrel{\times f(x,t)}{\longrightarrow} &\mathcal{O}_{dd}
%		\end{array}
%	\end{equation}
%is left $\mathcal{A}_2$-linear. Moreover the $f(x,t)$ must satisfy \eqref{E:bessel_poly_PDE}.
%\end{proof}
%\medskip

Recently, there are also complex analytic studies of the PDEs satisfied by generating functions of some classical special functions, including the Bessel polynomials \cite{Hu_Yang_2009, Hu_Li_2017}.

\subsubsection{Difference reverse Bessel polynomials}	
Now we consider the left $\mathcal{A}_2$-module structure of $\mathcal{O}_\Delta^{\mathbb{N}_0}$ as in Example~\ref{Eg:seq-functions_poisson_delta}. Recalling that $\Theta_n$ is a solution of $L_n:=X\partial^2-2(n+X)\partial+2n$, we immediately obtain the following\medskip

\begin{theorem}[(\textbf{Difference equation})]\label{T:difference_Bessel_Poly_DD} For each integer $n\ge 0$, let $\theta^\Delta_n$ be the \textit{difference reverse Bessel polynomial of degree} $n$ defined by
	\begin{equation}\label{E:delta_Bessel_poly}
		\theta^\Delta_n(x) :=\sum_{k=0}^n \frac{(n+k)!}{2^k (n-k)!\, k!}\, (x)_{n-k}.
	\end{equation}
Then for each integer $n\ge 0$ and each $x$, we have
	\begin{equation}\label{E:delta_Bessel_poly_eqn}
		(x-2n)\,\theta_n^\Delta(x+1)-4(x-n)\,\theta_n^\Delta(x)+3x\, \theta_n^\Delta (x-1)=0.
	\end{equation}
\end{theorem}
\medskip

The following delay-differential relations and three-term recurrence are direct consequences of the module $\Theta$, as well as Corollaries~\ref{C:Besselpoly_trans} and \ref{C:Besselpoly_3term} with $N=\mathcal{O}_\Delta$.
\medskip

\begin{theorem}[(\textbf{Delay-difference relations})]\label{T:difference_Bessel_Poly_PDE}
The sequence of difference reverse Bessel polynomials $(\theta^\Delta_n)$ satisfy the following two delay-difference equations
	\begin{equation}\label{E:delta_bessel_poly_trans}
	\begin{split}
		\theta^\Delta_n(x+1)-2\theta^\Delta_n(x)+x\theta^\Delta_{n-1}(x-1)=0, \\
		 \theta^\Delta_{n+1}(x)+(x-2n-1)\theta^\Delta_n(x)-2x\, \theta^\Delta_n(x-1)=0,
		\end{split}
	\end{equation}
	for every $n\ge 1$ and every $x$. 
\end{theorem}

\begin{proof} This follows from Corollary \ref{C:Besselpoly_trans} with the manifestation of the analytic function space $\mathcal{O}_\Delta$ as a $\mathcal{A}_1$-module  given by the Example \ref{Eg:O_delta}. 
\end{proof}

%\medskip
 
 It follows from Proposition \ref{P:Weyl_mod_half_bessel_3term} and the structure of the above $\mathcal{A}_2$-module the new three-term recursion formula for the reverse difference Bessel polynomials.
%\medskip

\begin{theorem}[(\textbf{Three-term recurrence})] \label{T:3-term}
The sequence of difference reversed Bessel polynomials  $(\theta^\Delta_n)$ satisfies the following recurrence relation
		\begin{equation}\label{E:delta_reverse_bessel_poly_3term}
			 \theta^\Delta_{n+2}(x)-(2n+3)\theta^\Delta_{n+1}(x)-x(x-1)\theta^\Delta_n(x-2)=0
		\end{equation}
		for every $n\in\mathbb{N}_0$ and every $x$.
\end{theorem}

\begin{proof} This follows from the last Corollary or directly from the  Corollary \ref{C:Besselpoly_3term}  with the manifestation of the analytic function space $\mathcal{O}_d$ as a $\mathcal{A}_1$-module  given by the Example \ref{Eg:O_delta}. 
\end{proof}
%\medskip
\subsection{Characteristics of modified half-Bessel modules} 

We recall from \eqref{E:Theta} the definition of the half-Bessel module. We may change one of the two generators from the half-Bessel module  \eqref{E:Theta}  by $(2X_2-1)\partial_2^2+3\partial_2+X_1^2$, as suggested by Corollary \ref{C:Besselpoly_3term}.
\medskip

\begin{theorem} The half-Bessel module $\Theta$ is isomorphic to the left $\mathcal{A}_2$-module
\[
\dfrac{\mathcal{A}_2}{\mathcal{A}_2[(2X_2-1)\partial_2^2
	+3\partial_2+X_1^2]
	+\mathcal{A}_2[X_1\partial_1+(1-2X_2)\partial_2-1-X_1]	}.
\]
\end{theorem}
\medskip

\begin{theorem}
The map
\[
\dfrac{\mathcal{A}_2}{\mathcal{A}_2[(2X_2-1)\partial_2^2
	+3\partial_2+X_1^2]
	+\mathcal{A}_2[X_1\partial_1+(1-2X_2)\partial_2-1-X_1]	}
\longrightarrow \Theta
\]
is a left $\mathcal{A}_2$-linear isomorphism.
\end{theorem}
\medskip

\begin{theorem}\label{T:bessel_poly_gf_map} Let 
	\[
		\mathcal{A}_{2}(\rho):=\frac{\mathcal{A}_2( \rho, {1}/{\rho})}{\langle\rho^2-1+2X_2\rangle}
	\] and let $S$ be the element
		\begin{equation}\label{E:bessel_poly_gf_map}
		S = C_1\rho^{-1}\E[X_1(1-\rho)]+C_2\rho^{-1}\E[X_1(1+\rho)],
		\end{equation}
		for every choice of $C_1,\, C_2\in\mathbb{C}$, where the $\E$ is the Weyl exponential defined in Example \ref{Eg:exp}. Then the map
		\begin{equation}
		   \begin{split}
			\Theta(\rho):=\dfrac{\mathcal{A}_2(\rho)}{\mathcal{A}_2(\rho)[(2X_2-1)\partial_2^2
			+3\partial_2+X_1^2]
			+\mathcal{A}_2(\rho)[X_1\partial_1+(1-2X_2)\partial_2-1-X_1]	}\\
		 \stackrel{\times S}{\longrightarrow}
		\overline{ \dfrac{\mathcal{A}_2(\rho)}
		{\mathcal{A}_2(\rho)\partial_1+\mathcal{A}_2(\rho)\partial_2}}
		\end{split}
		\end{equation}
		is left $\mathcal{A}_{2}(\rho)$-linear.
\end{theorem}
\medskip
\medskip

\begin{proof} Let
	\begin{equation}\label{E:char05}
		\Lambda_{1}=\partial_1
		\quad \Lambda_{2}=\rho\partial_{2},\quad Z_{1}=X_{1},\quad Z_{2}=-\rho,
	\end{equation}where $\rho^2=1-2X_2$. Since
	\[
		\partial_2\rho-\rho\partial_2=-\frac1\rho, 
	\]
we deduce from direct computations that
	\begin{equation}\label{E:computation_1}
		[\Lambda_{1}, Z_{1}]=1, \quad [\Lambda_{2}, Z_{2}]=1,
		\quad
		[\Lambda_{1}, \Lambda_{2}]=0,\quad [Z_1, Z_{2}]=0,
		\quad [\Lambda_i,\, Z_j]=0,\quad i\not=j
	\end{equation}hold.
\medskip

It is routine to check that
	\begin{equation}\label{E:char_2}
		(2X_2-1)\partial_2^2
	+3\partial_2+X_1^2
	=-\Lambda_2^2-\frac{2}{Z_2}\Lambda_2+Z_1^2.
	\end{equation}
	Let us make use of the characteristic change of elements by  \eqref{E:char05}  in the second generator $X_1\partial_1+(1-2X_2)\partial_2-1-X_1$  to obtain
	\[
		Z_1\Lambda_1-Z_2 \Lambda_2-1-Z_1
		=Z_1(\Lambda_1-1)-\Lambda_2Z_2.
	\]

	Let $\widehat{\mathcal{A}_2}=\langle \Lambda_1,\, \Lambda_2,\, Z_1,\, Z_2\rangle$ and	
	\[
			\widehat{\Theta}:= \frac{\widehat{\mathcal{A}_2}}
			{\widehat{\mathcal{A}_2}[(Z_2\Lambda_2^2+2\Lambda_2-Z_2Z_1^2]+\widehat{\mathcal{A}_2}[(Z_1(\Lambda_1-1)-\Lambda_2Z_2)]}
	\]Note that the above quotient $\widehat{\mathcal{A}_2}$-module $\widehat{\Theta}$ is holonomic with dimension and multiplicity both equal to $2$, see e.g., Example \ref{E:prime_integrable_eg_2}.	
	
		In order to find the map
		\[
			\widehat{\Theta}\ 
			\stackrel{\times \widehat S}{\longrightarrow}\overline{
	 \frac{\widehat{\mathcal{A}_2}}
			{{\widehat{\mathcal{A}_2}}\Lambda_1+{\widehat{\mathcal{A}_2}}\Lambda_2}},
		\]we
	left multiply the expression \eqref{E:char_2} by $Z_2$.  Further simplification yield
	\[
		\begin{split}
			(Z_2\Lambda_2)\Lambda_2+2\Lambda_2-Z_2Z_1^2
			&=\Lambda_2^2 Z_2-Z_2Z_1^2\\
			&=(\Lambda_2^2-Z_1^2)Z_2\\
			&=(\Lambda_2+Z_1)(\Lambda_2-Z_1)Z_2.
		\end{split}
		\] 
		
	Since the two first-order operators above commute, so we deduce that
	\[
		\frac{1}{Z_2} F(Z_1) E(Z_1Z_2)+\frac{1}{Z_2} G (Z_1) E(-Z_1Z_2)
	\]is the general solution to the above operator
	for some $F(Z_1)$ and $G(Z_1)$. This is a sum of two solutions which also conform with our knowledge that the $\widehat{\Theta}$ has multiplicity $2$ discussed above. 
	
	We now make use of the second generator of $\widehat{\Theta}$ to determine the exact forms of the solutions $F$ and $G$.  We first note that 
	\begin{align}
		& \Lambda_1 \E(Z_1Z_2)=Z_{2}\E(Z_1Z_2)+\E(Z_1Z_2)\Lambda_1,\label{E:product_rule_1}\\
		& \Lambda_2 \E(Z_1Z_2)=Z_1\E(Z_1Z_2)+\E(Z_1Z_2)\Lambda_2.\label{E:product_rule_2}
	\end{align}It follows from \eqref{E:product_rule_1} and \eqref{E:product_rule_2} that
	\[
		\begin{split}
			&\big(Z_1(\Lambda_1-1)-\Lambda_2Z_2\big)
			\frac{1}{Z_2} F(Z_1) \E(Z_1Z_2)\\
			&=Z_1(\Lambda_1-1)F(Z_1) \E(Z_1Z_2) - \Lambda_2 F(Z_1) \E(Z_1Z_2)\\
			&=(Z_1/Z_2)(\Lambda_1-1)F(Z_1) \E(Z_1Z_2) - F(Z_1) Z_1 \E(Z_1Z_2)
				-F(Z_1)\E(Z_1Z_2)\Lambda_2\\
			&=(Z_1/Z_2)(\Lambda_1-1)F(Z_1) \E(Z_1Z_2) - F(Z_1) Z_1 \E(Z_1Z_2)\quad \mod \widehat{\mathcal{A}}_2 \Lambda_2\\
			&=(Z_1/Z_2)\big(\E(Z_1Z_2)\Lambda_1F(Z_1)-F(Z_1)\Lambda_1\E(Z_1Z_2)\big)
				-(Z_1/Z_2)F(Z_1)E(Z_1Z_2)\\
			&\qquad- F(Z_1) Z_1 \E(Z_1Z_2)\\
			&=(Z_1/Z_2)\E(Z_1Z_2)\Lambda_1F(Z_1)-(Z_1/Z_2)F(Z_1)Z_2\E(Z_1Z_2)
				-(Z_1/Z_2)F(Z_1)E(Z_1Z_2)\\
			&\qquad - F(Z_1) Z_1 \E(Z_1Z_2)\\
			&=(Z_1/Z_2)\E(Z_1Z_2)\Lambda_1F(Z_1)-(Z_1/Z_2)F(Z_1)E(Z_1Z_2)\\
			&=(Z_1/Z_2)\E(Z_1Z_2)(\Lambda_1-1)F(Z_1)\\
%			&=(Z_1/Z_2)(\Lambda_1-1-Z_2)F(Z_1) \sum_{n=0}^\infty 
%				\frac{Z_1^nZ_2^n}{n!}\quad \mod \Lambda_2\\
%			&=(Z_1/Z_2)  \sum_{n=0}^\infty \frac{(nZ_1^{n-1}+Z_1^n\Lambda_1)F(Z_1)Z_2^n-F(Z_1)Z_1^nZ_2^n-F(Z_1)Z_1^nZ_2^{n+1}}{n!}\\
	%		&=(Z_1/Z_2)  \sum_{n=0}^\infty \frac{(\Lambda_1 -1)F(Z_1) Z_1^nZ_2^n}{n!}\quad \mod \Lambda_2\\				
			&=0\quad \mod \widehat{\mathcal{A}}_2 \Lambda_1+\widehat{\mathcal{A}}_2 \Lambda_2,
		\end{split}
	\]where the last equal sign holds provided that
		\[
			(\Lambda_1 -1)F(Z_1)=0\quad \mod \widehat{\mathcal{A}}_2 \Lambda_1.
		\]But this equation could hold when we have chosen $F(Z_1)=-\E(Z_1)$, see Example \ref{Eg:exp}, except perhaps for a constant multiple. 

It remains to consider the following and a similar calculation yields 
		\[
		\begin{split}
			&\big(Z_1(\Lambda_1-1)-\Lambda_2Z_2\big)
			\frac{1}{Z_2} G(Z_1) \E(-Z_1Z_2)\\
			&={Z_1}/{Z_2}\Big(\E(-Z_1Z_2)\Lambda_1G(Z_1)-Z_2\E(-Z_1Z_2)G(Z_1)\Big)
			-{Z_1}/{Z_2}G(Z_1) \E(-Z_1Z_2)\\
			&\qquad -\Big(\E(-Z_1Z_2)\Lambda_2G(Z_1)-Z_1\E(-Z_1Z_2)G(Z_1)\Big)\qquad
			\mod \widehat{\mathcal{A}}_2 \Lambda_1+\widehat{\mathcal{A}}_2 \Lambda_2\\
			&=\E(-Z_1Z_2)/Z_2\Big(Z_1(\Lambda_1-1)+Z_2\Lambda_2\Big)G(Z_1)
			\qquad
			\mod \widehat{\mathcal{A}}_2 \Lambda_1+\widehat{\mathcal{A}}_2 \Lambda_2\\
			&=(1/Z_2)\E(-Z_1Z_2) Z_1(\Lambda_1-1)G(Z_1)\\
			&=0 \qquad \mod \widehat{\mathcal{A}}_2 \Lambda_1+\widehat{\mathcal{A}}_2 \Lambda_2,
		\end{split}
		\]where the last equal sign would hold if
			\[
				G(Z_1)=\E(Z_1) \quad \mod \widehat{\mathcal{A}}_2\Lambda_1,
			\]
		which may differ by a constant multiple. This establishes the \eqref{E:bessel_poly_gf_map} and hence completes the proof.
\end{proof}

\subsection{Bessel polynomials module}

In this subsection, we define the \textit{Weyl Bessel polynomial of degree} $n\ge 0$ by
\begin{equation}
Y_n(X)=X^n\Theta_n(1/X)=\sum_{k=0}^{n}\frac{(n+k)!}{(n-k)! k! 2^k}X^k.
\end{equation}
Right multiplication by $Y_n(X)$ gives  a left $\mathcal{A}$-linear map
\[
\mathcal{A}/\mathcal{A}(X^2\partial^2+(2X+2)\partial-n(n+1)
\stackrel{\times Y_n(X)\ \ }{\longrightarrow}\overline{\mathcal{A}/\mathcal{A}\partial}.
\]
\medskip

We may further define when $n=-1$, then $Y_{-1}(X)=1$.
\medskip

We deduce from Corollary \ref{C:Besselpoly_trans} the followings.

\begin{corollary}\label{C:Bpoly_trans}
For each $n$, let $Y_n$ be defined as above, which is a solution of $X^2\partial^2+(2X+2)\partial-n(n+1)$ in a certain left $\mathcal{A}$-module $\mathbb{C}[X]$. Then
\begin{align}
	(X^2\partial+(n+1)X+1)Y_n - Y_{n+1}=0,\label{bessel_poly_recursion_1 in a A}\\
	(X^2\partial-nX+1)Y_n - Y_{n-1}=0.\label{bessel_poly_recursion_2 in a A}
\end{align}
as elements in $N$.
\end{corollary}
\medskip

\begin{proof}
Replacing $X$ by $1/X$ and $\partial$ by $-X^2\partial$ in Corollary~\ref{C:Besselpoly_trans}, we obtain
\begin{align*}
	(-X\partial-1/X-2n-1)\Theta_n(1/X) + \Theta_{n+1}(1/X)=0,\\
	(-X^2\partial-1) \Theta_n(1/X) + (1/X)\Theta_{n-1}(1/X)=0.
\end{align*}
Left-multiplying both sides of the first equation by $X^{n+1}$ and both sides of the second equation by $X^n$, we obtain
\begin{align*}
	(-X^2\partial-1-(n+1)X)X^n\Theta_n(1/X) + X^{n+1}\Theta_{n+1}(1/X)=0,\\
	(-X^2\partial-1+nX) X^n\Theta_n(1/X) + X^{n-1}\Theta_{n-1}(1/X)=0,
\end{align*}
and hence the result follows.
\end{proof}
\medskip

\begin{corollary} [(\textbf{Three-term recurrence})] \label{C:Bpoly_3term}
For each $n$, let $Y_n$ be defined as above, which is a solution of $X^2\partial^2+(2X+2)\partial-n(n+1)$ in a certain left $\mathcal{A}$-module $C[x]$. Then
		\begin{equation}\label{bessel_poly_recursion_3 in a A}
			Y_{n+1}-(2n+1)XY_n-Y_{n-1}=0
		\end{equation}
as an element in $N$.
\end{corollary}

We introduce $(\partial_2f)_n=f_{n+1}$ and $ (X_2 f)_n=nf_{n-1}$ to constrruct an $\mathcal{A}_2$-module, such that 
	\[
  [\partial_1, X_1]	=1, \quad   [\partial_2, X_2]	=1, \quad [X_1, X_2]=0, \quad  [\partial_1, \partial_2]=0.
	\]

Replacing $n$ by $n-1$ in \eqref{bessel_poly_recursion_1 in a A} and representing the resulting formula in terms of $X_1, \partial_i, i=1, 2$ yield
	\begin{equation}\label{E:dBessel_1st_gen}
X_{1}^2\partial_{1}+X_{1}X_{2}\partial_{2}+1-\partial_{2}.
	\end{equation}
Similarly, the equation \eqref{bessel_poly_recursion_2 in a A} can be rewritten as 
\[
X_{1}^2\partial_{1}-X_{1}X_{2}\partial_{2}+1-\partial_{2}^{-1}.
\]
Right multiplying $\partial_2$ to the above formula gives 
	\begin{equation}\label{E:dBessel_2nd_gen}
X_{1}^2\partial_{1}\partial_{2}-X_{1}X_{2}\partial^2_{2}+\partial_{2}-1.
	\end{equation}

Thus the \textit{Bessel polynomial module} is given by

\begin{definition}\label{D:y-1}
	The \textit{Bessel polynomial module} is the left $\mathcal{A}_2$-module defined by
\begin{equation}\label{E:y-1}
\mathcal{Y}=\dfrac{\mathcal{A}_{2}}
{\mathcal{A}_{2}
	(X_{1}^2\partial_{1}\partial_{2}-X_{1}X_{2}\partial^2_{2}+\partial_{2}-1)
	+\mathcal{A}_{2}
	(X_{1}^2\partial_{1}+X_{1}X_{2}\partial_{2}+1-\partial_{2})}.
\end{equation}
\end{definition}
%\medskip

Similarly, we can represent \eqref{bessel_poly_recursion_3 in a A} in $\mathcal{A}_2$ by
\begin{equation}\label{bessel_poly_recursion_3 in a A becomes}
\partial_2^2-2X_1X_2\partial_2^2-X_1\partial_2-1.
\end{equation}
%\medskip

\begin{proposition}
The map 
\[
\dfrac{\mathcal{A}_2}{\mathcal{A}_2(X_1^2\partial_1^2+(2X_1+2)\partial_1-X^2_2\partial_2^2)}
\longrightarrow \mathcal{Y}
\]
is a well-defined left $\mathcal{A}_2$-linear surjection.
\end{proposition}
%\medskip

\subsubsection{Classical Bessel polynomials}

\begin{proposition} Let $n\in \mathbb{N}_0$. Then the classical Bessel polynomial
	\[
	y_{n}(x)=\sum_{k=0}^{n}\frac{(n+k)!}{2^k (n-k)!\, k!}\, x^{k},
	\]for $n\ge 0$ and $y_{-1}(x)=1$,
	is the image of $1$ via the left $\mathcal{A}$-linear map
	\[
	\begin{array}{rclll}
	\mathcal{A}/\mathcal{A}\big(X^2\partial^2+2(X+2)\partial -n(n+1)\big) 
	&\stackrel{\times Y_{n}(X)}{\longrightarrow} 
	& \mathcal{A}/ \mathcal{A}\partial
	&\stackrel{\times 1}{\longrightarrow} 
	&\mathcal{O}%,\\
%	1      & \mapsto    & Y_{n}(X) &  \mapsto  &y_{n}(x)
	\end{array}
	\]
\end{proposition}
\medskip

The following PDEs (delay-differential relations) and three-term recurrence are direct consequences of the module $\mathcal{Y}$, as well as Corollaries~\ref{C:Bpoly_trans} and \ref{C:Bpoly_3term} with $N=\mathcal{O}_d$.
\medskip

\begin{theorem}[(\textbf{Delay-differential equations})]\label{T:differential-differential-classical bessel poly}
	The classical Bessel polynomials $\{y_n(x)\}$ \cite[p. 19]{Grosswald} satisfy the following (well-known) formulae
	\begin{equation}\label{E:classical_bessel_poly_trans}
	\begin{split}
	&x^2y'_{n-1}(x)-[(n-1)x-1]y_{n-1}(x)-y_{n-2}(x)=0,\\
	&x^2y'_{n-1}(x)-y_{n}(x)+(nx+1)y_{n-1}(x)=0.
	\end{split}
	\end{equation}
	for all $n$ and all $x$.
\end{theorem}
%\medskip

\begin{proof} This follows from the two generators \eqref{E:dBessel_1st_gen} and \eqref{E:dBessel_2nd_gen} of the Bessel polynomial module $\mathcal{Y}$ in the Definition \ref{D:y-1}, follows by the manifestation of the analytic function space $\mathcal{O}_d$ as a $\mathcal{A}_1$-module  given by the Example \ref{Eg:O_deleted_d}. 
\end{proof}

%{\color{blue}	
%\begin{proof}
%Construct the fundamental left $\mathcal{A}_2$-linear maps:
%	\begin{equation}\label{E:classical_bessel_poly_seq}
%			\begin{array}{rcl}
%		\mathcal{Y} & \stackrel{\times(y_{n-1}(x))}{\longrightarrow} & \mathcal{O}^{\mathbb{N}_0}
%			\end{array}
%	\end{equation}
%	and
%	\begin{equation}\label{E:classical_bessel_poly_gf}
%			\begin{array}{rcl}
%		\mathcal{Y} & \stackrel{\times f(x,\, t)}{\longrightarrow} & \mathcal{O}_{dd}
%			\end{array}
%	\end{equation}
%	Since $X_{1}^2\partial_{1}\partial_{2}-X_{1}X_{2}\partial^2_{2}+\partial_{2}-1$ and $X_{1}^2\partial_{1}-\partial_{2}+X_{1}X_{2}\partial_{2}+1$ are generators of $\mathcal{Y}$, they are both equivalent to $0$ in $\mathcal{Y}$. Thus \eqref{E:classical_bessel_poly_trans} follows from the left $\mathcal{A}_2$-linear map \eqref{E:classical_bessel_poly_seq}, where $\mathcal{O}^{\mathbb{N}_0}$ has the left $\mathcal{A}_2$-module structure as in Example \ref{Eg:seq-functions_poisson}; and \eqref{E:classical_bessel_poly_PDE} follows from the left $\mathcal{A}_2$-linear map \eqref{E:classical_bessel_poly_gf}, where $\mathcal{O}_{d d}$ has the left $\mathcal{A}_2$-module structure as in \eqref{E:O_dd_endow}.
%\end{proof}
%}

%\medskip
%
%The difference of the two generators of $\mathcal{Y}$ above yields the element
%\[
%(1-2X_1X_2)\partial_2^2-X_1\partial_2-1.
%\]
%As a result, we have following proposition.
%\medskip

\begin{proposition}[(\textbf{Three-term recurrence} \cite{Grosswald}, {p. 18})] The $(y_n)$ satisfies the (well-known) three-term recurrence formula
	\begin{equation}\label{E:classical_bessel_poly_3term}
	y_{n+1}(x)-(2n+1)xy_{\color{blue}n}(x)-y_{n-1}(x)=0.
	\end{equation}
	or every $n\in\mathbb{N}_0$ and every $x$.
\end{proposition}

\begin{proof} This either follows from the last corollary or directly from the Corollary \ref{C:Bpoly_3term} with the manifestation of the analytic function space $\mathcal{O}_d$ as a $\mathcal{A}_1$-module given by the Example  \ref{Eg:O_deleted_d}. 
\end{proof}

\subsubsection{Difference Bessel polynomials}
\begin{proposition} Let $n\in \mathbb{N}_0$. Then the difference Bessel polynomial
	\[
	y^{\Delta}_{n}(x):=\sum_{k=0}^{n}\frac{(n+k)!}{2^k (n-k)!\, k!}\, (x)_k,
	\]where we define $y^\Delta_{-1}(x):=1$, is the image of $1$ via the left $\mathcal{A}$-linear map
	\[
	\begin{array}{rclll}
	\mathcal{A}/\mathcal{A}\big(X^2\partial^2+2(X+2)\partial -n(n+1)\big) 
	&\stackrel{\times Y_{n}(X)}{\longrightarrow} 
	& \mathcal{A}/ \mathcal{A}\partial
	&\stackrel{\times 1}{\longrightarrow} 
	&\mathcal{O}_\Delta.%,\\
%	1      & \mapsto    & Y_{n}(X) &  \mapsto  &y^{\Delta}_{n}(x).
	\end{array}
	\]
\end{proposition}
\medskip

The following delay-differential relations and three-term recurrence are direct consequences of the module $\mathcal{Y}$, as well as Corollaries~\ref{C:Bpoly_trans} and \ref{C:Bpoly_3term} with $M=\mathcal{O}_\Delta$.
\medskip

\begin{theorem}[(\textbf{Delay-difference equations})]\label{T:classical_Bessel_Poly_DD}  
The difference Bessel polynomials $(y_n)$ satisfy the following two delay-difference equations
	\begin{equation}\label{E:delta_classical_bessel_poly_trans}
	\begin{split}
x(x-1) \Delta y^{\Delta}_{n-1}(x-2)
-(nx-x)y^{\Delta}_{n-1}(x-1)+y^{\Delta}_{n-1}(x)-
y^{\Delta}_{n-2}(x)=0, \\
x(x-1) \Delta y^{\Delta}_{n-1}(x-2)-y^{\Delta}_{n}(x)+
nxy^{\Delta}_{n-1}(x-1)+y^{\Delta}_{n-1}(x)=0
	\end{split}
	\end{equation}
	for every $n$ and every $x$. 
%Their generating function $f(x,\, t)$ solves the following system of delay-differential equations
%	\begin{equation}\label{E:delta_classical_bessel_poly_PDE}
%	\begin{split}
%	&x(x-1)[f_t(x-1,\, t)-f_t(x-2, \, t)]-xtf_{tt}(x-1,\, t)+f_t(x,\, t)-f(x, \, t)=0,\\
%	&x(x-1)[f(x-1,\, t)-f(x-2, \, t)]-f_{t}(x,\, t)+xtf_t(x-1,\, t)+f(x, \, t)=0.
%	\end{split}
%	\end{equation}
\end{theorem}

\begin{proof} 

\end{proof}
%{\color{blue}
%\begin{proof}
%Construct the fundamental left $\mathcal{A}_2$-linear maps:
%	\begin{equation}\label{E:delta_classical_bessel_poly_seq}
%			\begin{array}{rcl}
%		\mathcal{Y} & \stackrel{\times(y^\Delta_{n-1}(x))}{\longrightarrow} & \mathcal{O}^{\mathbb{N}_0}
%			\end{array}
%	\end{equation}
%	and
%	\begin{equation}\label{E:delta_classical_bessel_poly_gf}
%			\begin{array}{rcl}
%		\mathcal{Y} & \stackrel{\times f(x,\, t)}{\longrightarrow} & \mathcal{O}_{\Delta d}
%			\end{array}
%	\end{equation}
%	Since $X_{1}^2\partial_{1}\partial_{2}-X_{1}X_{2}\partial^2_{2}+\partial_{2}-1$ and $X_{1}^2\partial_{1}-\partial_{2}+X_{1}X_{2}\partial_{2}+1$ are generators of $\mathcal{Y}$, they are both equivalent to $0$ in $\mathcal{Y}$. Thus \eqref{E:delta_classical_bessel_poly_trans} follows from the left $\mathcal{A}_2$-linear map \eqref{E:delta_classical_bessel_poly_seq}, where $\mathcal{O}^{\mathbb{N}_0}$ has the left $\mathcal{A}_2$-module structure as in Example \ref{Eg:seq-functions_poisson}; and \eqref{E:delta_classical_bessel_poly_PDE} follows from the left $\mathcal{A}_2$-linear map \eqref{E:delta_classical_bessel_poly_gf}, where $\mathcal{O}_{d d}$ has the left $\mathcal{A}_2$-module structure as in \eqref{Eg:two-seq-functions}.
%\end{proof}
%}

\begin{proposition}[(\textbf{Three-term recurrence})] \label{P:3-term_bessel_poly} The difference Bessel polynomials $(y^\Delta_n)$ satisfy the three-term recurrence formula
	\begin{equation}\label{E:delta_bessel_poly_3term}
		y^{\Delta}_{n+1}(x)-x(2n+1)y^{\Delta}_{n}(x-1)-y^{\Delta}_{n-1}(x)=0
	\end{equation}
for every $n$ and $x$.
\end{proposition}
\medskip

\subsection{Characteristic of Bessel polynomial module}
We change one of the two generators from the  defining Bessel polynomial module \eqref{E:y-1} by the element from the three-term recursion \eqref{bessel_poly_recursion_3 in a A becomes} to obtain following theorem.
\begin{theorem}\label{T: iso bp-n-1} 
The natural map
\[
\dfrac{\mathcal{A}_{2}}
{\mathcal{A}_{2}
	((1-2X_1X_2)\partial_2^2-X_1\partial_2-1)
	+\mathcal{A}_{2}
	(X_{1}^2\partial_{1}-\partial_{2}+X_{1}X_{2}\partial_{2}+1)}
\longrightarrow \mathcal{Y}
\]
is a left $\mathcal{A}_2$-linear isomorphism.
\end{theorem}
\medskip

\begin{theorem}\label{T:symbol gen yn-1}
	Let 
		\[
			\mathcal{A}_{2}(\eta):=\dfrac{\mathcal{A}_2\langle\eta\rangle}{\langle\eta^2-1+2X_1X_2\rangle},
		\]
	\[
	\mathcal{Y}(\eta)=\dfrac{\mathcal{A}_{2}(\eta)}
	{\mathcal{A}_{2}(\eta)(\eta^2\partial_2^2-X_1\partial_2-1)
		+\mathcal{A}_{2}(\eta)
		(X_{1}^2\partial_{1}-\partial_{2}+X_{1}X_{2}\partial_{2}+1)}
	\]
	and let
	\[
			S:=C_1\,\E(\frac{1-\eta}{X_1})+C_2\,\E(\frac{1+\eta}{X_1})
	\]
	for every choice of $C_1,\, C_2\in\mathbb{C}$, where $\E$ denote the Weyl exponential introduced in Example \ref{Eg:exp}.
	Then $\mathcal{Y}(\eta)$ is a left $\mathcal{A}_{2}(\eta)$-module and the map
	\begin{equation}\label{eta gen for bp}
	\mathcal{Y}(\eta)
	\xrightarrow[]{\times S}
	\overline{\mathcal{A}_{2}(\eta)/(\mathcal{A}_{2}(\eta)\partial_1+
		\mathcal{A}_{2}(\eta)\partial_2)},
	\end{equation}
	is left $\mathcal{A}_{2}(\eta)$-linear.
\end{theorem}
\medskip

Since the proof of Theorem \ref{T:symbol gen yn-1} is similar to that of Theorem \ref{T:bessel_poly_gf_map}, so we place it in the Appendix \ref{A:BP_chara}.

\section{Newton transformations} \label{S:Newton}
To handle more sophisticated generating functions from Bessel modules in relation to difference calculus, we would like to introduce a transformation that allows us to ``transform\rq\rq{} entities with respect to differential operators to difference operators under different sittings. An origin of the transform can be traced to spectral analysis and integral forms of gamma functions, see e.g., \cite[\S 12.22]{WW}:
	\[
	\Gamma(x)=-\frac{1}{2i \sin\pi x}\int_{\infty}^{(0+)}
	e^{-\lambda}(-\lambda)^{x-1} d\lambda\quad (|\arg(-\lambda)|<\pi)
	\]which is valid on  $\mathbb{C}$ except at negative integers including the origin, where the upper and lower limits denote the standard Hankel contour which starts from $+\infty$ above the real-axis, then circles around the origin in the counter-clockwise direction before returns to $+\infty$ below the real-axis. If we replace $x$ by $-x$, then the integral becomes
	
\begin{equation}\label{E: Gamma}
\Gamma(-x)=\frac{e^{-i\pi x}}{2i \sin\pi x}\int_{-\infty}^{(0+)}
e^\lambda(-\lambda)^{-x-1} d\lambda\quad (|\arg(\lambda)|<\pi)
\end{equation}
where the integral contour starts at $-\infty$ before the negative real-axis, then circle around the origin in the counter-clockwise direction, before returns to $-\infty$ above negative real-axis. 
 \medskip
 
\begin{theorem}[(\textbf{Newton transformations})]\label{T:Newton_trans}
Let $\mathcal{O}_\Delta$ be the space of entire functions endowed with the structure of a left $\mathcal{A}$-module as in Example \ref{Eg:O_delta_d}, and $\mathcal{O}_d$ be the space of entire functions $f$ with growth rate
\[
	|f(x)\, e^{mx}|\to 0\quad\mbox{ as }\quad x\to-\infty
\]
for some $m<1$, endowed with the structure of a left $\mathcal{A}$-module as in Example \ref{Eg:O_deleted_d}.
Then the map
	\begin{equation}\label{E:newton_trans}
\begin{array}{rcl}
	\mathfrak{N}: \mathcal{O}_d&\longrightarrow&\mathcal{O}_\Delta\\
f&\mapsto&\dfrac{e^{-i\pi x}}{2i\sin(\pi x) \Gamma(-x)}\displaystyle\int_{-\infty}^{(0+)}\, f(s)\, e^s(-s)^{-x-1}\,ds
\end{array}
	\end{equation}
is left $\mathcal{A}$-linear. This map is called the Newton transformation.
\end{theorem}
\medskip

\begin{proof}
We just need to verify that $\mathfrak{N}(\partial f) = \partial(\mathfrak{N}f)$ and $\mathfrak{N}(X f) = X(\mathfrak{N}f)$ as follows.
\begin{align*}
	\mathfrak{N}(\partial f)(x) &= \dfrac{e^{-i\pi x}}{2i\sin(\pi x) \Gamma(-x)}\displaystyle\int_{-\infty}^{(0+)}\, f'(s)\, e^s(-s)^{-x-1}\,ds \\
	&= -\dfrac{e^{-i\pi x}}{2i\sin(\pi x) \Gamma(-x)}\displaystyle\int_{-\infty}^{(0+)}\, f(s)[e^s(-s)^{-x-1}+(x+1)e^s(-s)^{-x-2}]\,ds \\
	&=(\mathfrak{N}f)(x+1) - (\mathfrak{N}f)(x) = \partial(\mathfrak{N}f)(x);\\
	\mathfrak{N}(X f)(x) &= \dfrac{e^{-i\pi x}}{2i\sin(\pi x) \Gamma(-x)}\displaystyle\int_{-\infty}^{(0+)}\, sf(s)\, e^s(-s)^{-x-1}\,ds \\
	&= -\dfrac{e^{-i\pi x}}{2i\sin(\pi x) \Gamma(-x)}\displaystyle\int_{-\infty}^{(0+)}\, f(s)\, e^s(-s)^{-x}\,ds \\
	&= x(\mathfrak{N}f)(x-1) = X(\mathfrak{N}f)(x).
\end{align*}
\end{proof}
\medskip

\begin{theorem}
The Newton transformation defined above is an injection.
\end{theorem}
\begin{proof}
It follows from the injectivity of the Mellin transformation and the observation that the integral above is nothing but the Mellin transformation of $f$ times the exponential function.
\end{proof}
\medskip

\begin{example} Let us revisit the Weyl exponential $\E(aX)$ from Example \ref{Eg:exp} derived when finding a $\mathcal{A}$-linear map as a solution $S$ to
			\[
		\mathcal{A}/\mathcal{A}(\partial-a)
		\stackrel{\times S }{\longrightarrow}
		\mathcal{A}/\mathcal{A}\partial.
	\]
Since it is known that $\mathcal{A}/\mathcal{A}(\partial-a) \xrightarrow{\times \E(aX)} \mathcal{A}/\partial\mathcal{A} \xrightarrow{\times 1}\mathcal{O}_d$ can be given by the classical function $e^{ax}$ where $\mathcal{O}_d$ endowed with the $\mathcal{A}$-module structure described in Example \ref{Eg:O_deleted_d}. Then the Newton transformation above completes the following commutative diagram
			\[
%				\begin{equation}\label{E:commute-1}
				 	  \begin{tikzcd}[row sep=large, column sep=large]
					    \mathcal{D}/\mathcal{D}(\partial-a)  \arrow{r}{\times e^{ax}} 
					    \arrow[swap]{dr}{\times \exp_\Delta(a;\,X)} 
					    & \mathcal{O}_d 
					    \arrow{d}{\mathfrak{N}} \\
     						& \mathcal{O}_{\Delta}
				  \end{tikzcd}
%			\end{equation}
		\]
where the space of analytic functions $\mathcal{O}_\Delta$ is endowed with a $\mathcal{A}$-module structure defined in Example \ref{Eg:O_delta}. It is straightforward to verify that
	\begin{equation}\label{E:newton_bases}
		\mathfrak{N}: x^n\ \ \longmapsto\ \  x(x-1)\cdots (x-n+1),
	\end{equation}for every integer $n$, so that the above commutative diagram gives the difference exponential function
	\[
	      \exp_\Delta(a;\, x):=(a+1)^x=\sum_{k=0}^{\infty} \frac{a^k}{k!}\, x(x-1)\cdots(x-k+1).
	\]	
converges absolutely for any complex $x$, if $|a|<1$ directly. The $\exp_\Delta$ was already derived in Example \ref{Eg:delta_exp}.
\end{example}
    Moreover, it is easy to see that
    \begin{equation}\label{E:newton_bases_2}
		\mathfrak{N}: x^\nu\ \  \longmapsto\ \  \frac{\Gamma(x+1)}{\Gamma(x+1-\nu)}= (x)_\nu,
	\end{equation}for every $\nu\in\mathbb{C}$.
\medskip

We shall apply similar rationale as the above example to derive generating functions for certain half-Bessel
 modules with respect to difference operator to be considered in subsequent sections. In fact, similar transformations already appeared in the literature in different applications spread over many decades. See for example \cite{BHPR_2019}, in integral form \cite[p. 5 (25)]{Truesdell_1948} for Laguerre polynomials, and again in \cite[p. 401]{Gessel_2003} in ad hoc manners.
 
\begin{remark}
Variants of the Newton transformation are possible. For example, if $\mathcal{O}_d$ in Theorem \ref{T:Newton_trans} is replaced by the space of entire functions $f$ with growth rate such that
\[
	|f(x)\, e^{mx}|\to 0\mbox{ as }x\to\pm i\infty,
\]
holds for every $ m\in\mathbb{C}$, then one replaces the Hankel-type contour in the above Newton transformation by, for example,
%	\[
	%	f\mapsto\dfrac{e^{-i\pi x}}{2i\sin(\pi x)\Gamma(-x)}\displaystyle\int_{-i\infty}^{+i\infty}f(s)\, e^s\, (-s)^{-x-1}\,ds.
%	\]
Then the conclusion of Theorem \ref{T:Newton_trans} remains valid. 
\end{remark}

\section{Poisson-type generating functions for (difference) Bessel polynomials}

We are ready to apply the Theorem \ref{T:bessel_poly_gf_map} derived from the last section to recover the well-known (Poisson-type) generating function for classical Bessel polynomials $\theta_n(x)$ and to derive new \textit{difference Bessel polynomials} $\theta_n^\Delta(x)$ in this section.

\medskip

%\begin{example}
%Let $\mathbb{C}^\mathbb{N}$ and $\mathcal{O}^\mathbb{N}$  be the space of sequences of complex numbers and sequence of analytic functions. Then they are equipped with the structure of $\mathcal{A}-$module by
%begin{equation}\label{E:poisson_module}
%\begin{array}{l}
%(Xa)_n=na_{n-1}\\
%(\partial a)_n=a_{n+1}\quad\mbox{for all sequence }(a_n).
%\end{array}
%\end{equation}
%\end{example}
%\medskip

\begin{theorem}[(\textbf{Poisson transformation I})] Let $\mathcal{O}_d^{\mathbb{N}_0}$ denote the space of sequences of analytic functions $(f_n(x))_0^\infty$ with appropriately restricted growth rate be endowed with the $\mathcal{A}_2$-module structure defined in Example \ref{Eg:seq-functions_poisson}. Let $\mathcal{O}_{dd}$ denote the space of analytic functions be endowed with the $\mathcal{A}_2$-module structure as in the Theorem \ref{T:Bessel_gf}. Then the map
\begin{equation}\label{E:poisson}
\begin{array}{rcl}
	\mathfrak{p}:\mathcal{O}^{\mathbb{N}_0}&\to&\mathcal{O}_{dd}
\\
(f_n)&\mapsto&\displaystyle \sum_{n=0}^\infty f_n\dfrac{t^n}{n!},
\end{array}
\end{equation}
called the ``Poisson transformation" is a left  $\mathcal{A}_2$-linear map. 
%It becomes a left $\mathcal{A}_1$-linear map if we have considered $\mathfrak{p}:\mathcal{\mathbb{C}}^{\mathbb{N}_0}\to\mathcal{O}$ on the space of sequences of constants $(a_n)$ instead.
\end{theorem}
\medskip

\begin{proof} This is a straightforward verification.
\end{proof}
\medskip

For finite difference operator $\Delta$, we have the following analogue.

\medskip

\begin{theorem}[(\textbf{Poisson transformation II})] Let $\mathcal{O}_\Delta^{\mathbb{N}_0}$ denote the space of bilateral sequences of analytic functions $(f_n(x))$ with appropriately restricted growth rate be endowed with the $\mathcal{A}_2$-module structure defined in Example \ref{Eg:seq-functions_poisson_delta}. Let $\mathcal{O}_{\Delta d}$ denote the space of analytic functions be endowed with the $\mathcal{A}_2$-module structure as in the Theorem \ref{T:Bessel_gf}. Then the map
\begin{equation}\label{E:poisson_2}
\begin{array}{rcl}
	\mathfrak{p}_\Delta:\mathcal{O}^{\mathbb{N}_0}&\to&\mathcal{O}_{\Delta d}
\\
(f_n)&\mapsto&\displaystyle \sum_{n=0}^\infty f_n\dfrac{t^n}{n!},
\end{array}
\end{equation}
called the ``difference Poisson transformation" is a left  $\mathcal{A}_2$-linear map.
% It becomes a left $\mathcal{A}_1$-linear map if we have considered $\mathfrak{p}_\Delta:\mathcal{\mathbb{C}}^{\mathbb{N}_0}\to\mathcal{O}$ on the space of sequences of constants $(a_n)$ instead.
\end{theorem}
\medskip

\begin{remark}\label{R:poisson}
The left $\mathcal{A}$-module structure of $\mathbb{C}^\mathbb{Z}$ as in Example~\ref{Eg:two-seq-numbers} is specifically designed so that the z-transform $\mathfrak{z}:\mathbb{C}^\mathbb{Z}\to\mathcal{O}_d$ in Theorem~\ref{T:z-transform-1} is left $\mathcal{A}$-linear. In a similar way, the left $\mathcal{A}$-module structure of $\mathbb{C}^\mathbb{N}$ as in Example~\ref{Eg:seq-functions_poisson} and Example~\ref{Eg:seq-functions_poisson_delta} is specifically designed so that the above Poisson transforms are left $\mathcal{A}$-linear.
\end{remark}
\medskip

\subsection{Classical reverse Bessel polynomials $\theta_n(x)$}

\begin{example}\label{Eg:Bessel_mod_no_root}
Let $\mathcal{A}_{2}(\rho):=\dfrac{\mathcal{A}_2\langle \rho, {1}/{\rho}\rangle}{\langle\rho^2-1+2X_2\rangle}$ be defined as in Theorem \ref{T:bessel_poly_gf_map}, and
	\[
			\Theta (\rho):=\dfrac{\mathcal{A}_2(\rho)}{\mathcal{A}_2(\rho)(-\rho^2\partial_2^2
			+3\partial_2+X_1^2)
			+\mathcal{A}_2(\rho)(X_1\partial_1+\rho^2\partial_2-1-X_1)}.
		\]
Then $\Theta(\rho)$ is a left $\mathcal{A}_{2}(\rho)$-module. On the other hand, the space $\mathcal{O}_2$ of two-variable analytic functions with appropriate growth rate endowed with
		\begin{equation}\label{E:O_dd_endow_1}
			\begin{split}
				(\partial_1 f)(x, t)) &=f_x (x, t), \quad (X_1f)(x, t)=xf(x, t),\\
				(\partial_2 f)(x, t) &=f_t (x, t), \quad (X_2f)(x, t)=tf(x, t)
			\end{split}
		\end{equation}
		and by spectra analysis, $\rho^2f$ can be defined so that
\[
			(\rho^2 f)(x,t) = (1-2t)f(x,t),
	\]
$\mathcal{O}_2$ is a left $\mathcal{A}_2(\rho)$-module again denoted by $\mathcal{O}_{dd}$, and the same space $\mathcal{O}_2$ endowed with \eqref{E:O_delta_d}
	\begin{equation}\label{E:O_delta_d_2}
		\begin{array}{ll}
			\partial_1f(x,\, t)=f(x+1,\, t)-f(x,\, t),   & X_1f(x,\, t)=xf(x-1,\, t);\\
			\partial_2f(x,\, t) =f_t(x,\, t), 		& X_2 f(x,\, t)=tf(x,\, t),
		\end{array}	
	\end{equation} and 
		\[
			(\rho^2 f)(x, t)=(1-2t)f(x, t),
		\]
is a left $\mathcal{A}_2(\rho)$-module again denoted by $\mathcal{O}_{\Delta d}$; and the space $\mathcal{O}^{\mathbb{N}_0}$ of all sequences of analytic functions endowed with \eqref{E:seq-functions_poisson} and $\rho^2f$ can be defined so that
		\begin{equation}\label{E:O_dd_endow_2}
			(\rho^2 f)_n = f_n-2nf_{n-1}
	\end{equation}
	is also a left $\mathcal{A}_2(\rho)$-module.
\end{example}

We are ready to state our first main result in this subsection for which the generating function of the classical reverse Bessel polynomials derived by Burchnall \cite{Burchnall_1953} is a special case.
\medskip

\begin{theorem}\label{T:general_rev_bessel_poly_gf} 
	\begin{enumerate} 
		\item Let $(\vartheta_n)$ be a sequence of analytic functions which is a solution of the reverse Bessel polynomial module $\Theta$ in $\mathcal{O}_d^{\mathbb{N}_0}$. Then there exist complex constants $C_1,\, C_2$ such that
			\begin{equation}\label{E:gen_rev_bessel_poly_gf}
				\frac{1}{\sqrt{1-2t}}\Big\{ C_1\exp\big[x(1-\sqrt{1-2t})\big]
				+C_2 \exp\big[x(1+\sqrt{1-2t})\big]\Big\}
				=\sum_{n=0}^\infty
				\vartheta_n(x)\, \frac{t^n}{n!}
			\end{equation}
			which converges uniformly in each compact subset of $\mathbb{C}\times B(0,\, \frac12)$.
		\item Moreover, the generating function in the part $\mathrm{(i)}$ above satisfies the system of partial differential equations
	\begin{equation}\label{E:bessel_poly_PDE}
		\begin{split}
			&f_{xt}(x,\, t)-f_t(x,\, t)+xf(x, \, t)=0,\\
			&xf_x (x,\, t)+(1-2t)f_t(x,\, t)-(1+x)f(x,\, t)=0
		\end{split}
	\end{equation}
for all $x$ and $t$.\footnote{See also the Appendix \ref{SS:PDE_list}.}. 
	\end{enumerate}
\end{theorem}
\medskip
	
We skip the proof to this theorem since it is very similar to that of the following special case.
\medskip
 
\begin{theorem}[(\cite{Grosswald}, p. 42)]\label{T:rev_bessel_poly_gf} 
	\begin{enumerate}
		\item The reverse Bessel polynomials $(\theta_n)$ have the generating function
		\begin{equation}\label{E:rev_bessel_poly_gf}
			\frac{1}{\sqrt{1-2t}} \exp\big[x(1-\sqrt{1-2t})\big]=\sum_{n=0}^\infty
			\theta_n(x)\, \frac{t^n}{n!}
		\end{equation}
		which converges uniformly in each compact subset of $\mathbb{C}\times B(0,\, \frac12)$.
		\item Moreover, the generating function in the part $\mathrm{(i)}$ above satisfies the system of partial differential equations \eqref{E:bessel_poly_PDE}\footnote{See also the Appendix \ref{SS:PDE_list}.}.
	\end{enumerate}
\end{theorem}

\begin{proof} 
Recall the map $\mathfrak{p}$ as defined in \eqref{E:poisson} and the Example \ref{Eg:Bessel_mod_no_root} just discussed above that $\Theta(\rho)$ is a $\mathcal{A}_2(\rho)$-modules above. Then we have the diagram
	\begin{equation}\label{E:commute-2}
				  \begin{tikzcd} [row sep=huge, column sep=huge]
					    \Theta(\rho) \arrow{r}{\times(\theta_n(x))} 
					    \arrow[swap]{d}{\times \{C_1\rho^{-1}\E[X_1(1-\rho)]+C_2\rho^{-1}\E[X_1(1+\rho)]\}}
					    %\arrow[swap]{dr}{\mathfrak{g}} 
					    & \mathcal{O}^{\mathbb{N}_0} 
					    \arrow{d}{\mathfrak{p}} \\
     						\widetilde{A}_2(\rho) \arrow{r}{\times 1}& \mathcal{O}_{d d}
				  \end{tikzcd}
	\end{equation}
in which $\widetilde{A}_2(\rho):=\overline{\mathcal{A}_2(\rho)/\big[\mathcal{A}_2(\rho)\partial_1+\mathcal{A}_2(\rho)\partial_2\big]}$. Since the above diagram commutes, the sum $\sum_{n=0}^\infty\theta_n (x)\, t^n/n!$ being the image of $1$ in $\mathcal{O}_{dd}$ via the top-right path, is also a solution to the system of differential-difference equations
		\begin{equation}\label{E:bessel-poly-PDEs}
			\begin{split}
				&\theta_n^\prime(x)-\theta_n(x)+x\theta_{n-1}(x)=0,\\
				&x\theta_n^\prime(x)-(x+2n+1)\theta_n(x)+\theta_{n-1}(x)=0.
			\end{split}
		\end{equation}
		mentioned in the Theorem \ref{T:differential-differential}. Now the function $\frac{1}{\sqrt{1-2t}} \exp\big[x(1-\sqrt{1-2t})\big]$ is a solution to the system of PDEs \eqref{E:bessel_poly_PDE}
	\begin{align}
			&f_{xt}(x,\, t)-f_t(x,\, t)+xf(x, \, t)=0,\label{E:rev_bessel_poly_PDE_1}
			\\
			&xf_x (x,\, t)+ (1-2t)f_t(x,\, t)-(1+x)f(x,\, t)=0.\label{E:rev_bessel_poly_PDE_2}
	\end{align}
Since the PDEs \eqref{E:rev_bessel_poly_PDE_1} and \eqref{E:rev_bessel_poly_PDE_2} are images of the two generators of the Bessel module $\Theta$ in Definition \ref{D:Theta}, so the dimension of the local solution space of the PDEs coincides with the multiplicity two as stated in Theorem \ref{T:B_poly_mod_holonomic}.The Theorem \ref{T:bessel_poly_gf_map} directly verifies that the solution space of the PDEs has $\mathbb{C}$-dimension two by solving the explicit solution as a linear combination of $\rho^{-1}\E[X_1(1-\rho)]$ and $\rho^{-1}\E[X_1(1+\rho)]$ of $\Theta(\rho)$ in $\tilde{A}_2(\rho)$ that effects the bottom-left path in the commutative diagram \eqref{E:commute-2}. 
So
		\[
			\frac{1}{\sqrt{1-2t}}\big\{C_1 \exp\big[x(1-\sqrt{1-2t})\big]+C_2\exp\big[x(1+\sqrt{1-2t})\big]\big\}=\sum_{n=0}^\infty
			\theta_n(x)\, \frac{t^n}{n!}
		\]
	for complex scalars $C_1,\, C_2$. Substituting $t=0$ into the above equation yields $C_1=1$ and $C_2=0$. One way to verify the convergence of \eqref{E:rev_bessel_poly_gf} is to combine \eqref{E:rev_bessel_poly_PDE_1} and \eqref{E:rev_bessel_poly_PDE_2} into a single ODE. Subtracting the equation \eqref{E:rev_bessel_poly_PDE_1} from \eqref{E:rev_bessel_poly_PDE_2} after it being differentiated partially with respect to $t$ yields the ODE
		 \[
		 	(1-2t)f_{tt}(x,\, t)-3f_t(x,\, t)-x^2f(x,\, t)=0.
		\]For each fixed, but otherwise arbitrary, $x$ the ODE has a regular singularity at $t=\frac12$. Hence the infinite sum in \eqref{E:rev_bessel_poly_gf} converges uniformly in each compact subset of $\mathbb{C}\times B(0,\, \frac12)$. This completes the proof.	
\end{proof}
\medskip

\subsection{Difference reverse Bessel polynomials $\theta^\Delta_n(x)$} \label{SS:delta_reverse_bessel_poly}
\begin{example}\label{Eg:BP_delta_J} Let $\Theta(\rho)$ and 
 $(\theta^\Delta_n(x))$ to denote the Bessel polynomial module defined in Example \ref{Eg:Bessel_mod_no_root} and infinite sequence of difference Bessel polynomials respectively.  Suppose that $\mathcal{O}^{\mathbb{N}_0}=\mathcal{O}^{\mathbb{N}_0}_\Delta$ is endowed with the $\mathcal{A}_2$-module structure as defined in Example \ref{Eg:seq-functions_poisson_delta}.
 Then the map
	\begin{equation}\label{E:BP_delta_J}
		\begin{array}{rcl}
		\mathfrak{j}_\Delta:\Theta(\rho) & \stackrel{\times (\theta^\Delta_{n})}{\longrightarrow} &\mathcal{O}^{\mathbb{N}_0}_\Delta
		\end{array}
	\end{equation}  
is left $\mathcal{A}_2$-linear.  
\end{example}
\medskip

We are now ready to make use of the Example \ref{Eg:Bessel_mod_no_root} and to apply 
Newton's transform $\mathfrak{N}$ to ``map" the classical generating for the reverse Bessel polynomials derived  in the last subsection to the difference reverse Bessel polynomials.
\medskip

\begin{theorem}\label{T:bessel_poly_delta_gf}
	\begin{enumerate}
		\item The difference reverse Bessel polynomials have the generating function
		\begin{equation}\label{E:bessel_difference_poly_gf}
			\frac{e^{-i\pi x}}{2i\sin\pi x\Gamma(-x)}\int_{-\infty}^{(0+)}e^\lambda(-\lambda)^{-x-1} \frac{1}{\sqrt{1-2t}} \exp\big[\lambda(1-\sqrt{1-2t})\big]\,d\lambda=\sum_{n=0}^\infty
			\theta_n^\Delta (x)\, \frac{t^n}{n!},
		\end{equation}
		where the infinite series converges uniformly in each compact subset of $\mathbb{C}\times B(0,\, \frac12)$.
		\item Moreover, the generating function in part $\mathrm{(i)}$ above satisfies the delay-differential equations\footnote{See also the Appendix \ref{SS:PDE_list}.}
			\begin{equation}\label{E:delta_bessel_poly_PDE}
				\begin{split}
					&f_t(x+1,\, t)-2f_t(x,\, t)+xf(x-1,\, t)=0,\\
					&(1-2t)f_t(x,\, t)+(x-1)f(x,\, t)-2xf(x-1,\, t)=0.
				\end{split}
			\end{equation}
	\end{enumerate}
\end{theorem}
\medskip

\begin{proof} We first draw on the simple fact from the Example \ref{Eg:Bessel_mod_no_root}
that $\Theta (\rho)$ is a $\mathcal{A}_2(\rho)$-module and from the Example \ref{Eg:BP_delta_J}  that the map $\mathfrak{j}_\Delta:\Theta(\rho)  \stackrel{\times (\theta^\Delta_{n})}{\longrightarrow} \mathcal{O}^{\mathbb{N}_0}_\Delta
$ is $\mathcal{A}_2(\rho)$-linear. Then
the map $\mathfrak{p}_\Delta: \mathcal{O}^{\mathbb{N}_0} \longrightarrow\mathcal{O}_{\Delta d}$  as defined in \eqref{E:poisson_2}  and their composition is also  an $\mathcal{A}_2(\rho)$-linear. Thus $\mathfrak{p}_\Delta\circ \mathfrak{j}_\Delta$ acting on the identity yields the right-side of \eqref{E:bessel_difference_poly_gf}. Moreover, we have the commutative diagram
\begin{equation}\label{E:commute-dbessel-2}
			\begin{tikzcd}  [row sep=large, column sep=large] %[row sep=huge, column sep=huge]
			& \Theta (\rho) \arrow[swap]{dl}{\times\rho^{-1}\E[X_1(1-\rho)]} \arrow{d}{\mathfrak{j}} \arrow{dr}{\mathfrak{j}_\Delta} \\
			\widetilde{\mathcal{A}}_2(\rho) \arrow{dr}{\times 1} & \mathcal{O}_d^{\mathbb{N}_0} \arrow{d} {\mathfrak{p}}
& \mathcal{O}_\Delta^{\mathbb{N}_0} \arrow{d}{\mathfrak{p}_\Delta}\\
 			&  \mathcal{O}_{d d} \arrow{r}{\mathfrak{N}} &  \mathcal{O}_{\Delta d} 
		\end{tikzcd}
		\end{equation}
		in which $\widetilde{\mathcal{A}}_2(\rho):=\overline{\mathcal{A}_2(\rho)/\big[\mathcal{A}_2(\rho)\partial_1+\mathcal{A}_2(\rho)\partial_2\big]}$, and the growth of the analytic functions in $\mathcal{O}_{d d}$ and $\mathcal{O}_{\Delta d}$ are suitably restricted.  Let
	\[
		y(x,t)=\int_{-\infty}^{(0+)}e^\lambda(-\lambda)^{-x-1} \frac{1}{\sqrt{1-2t}} \exp\big[\lambda(1-\sqrt{1-2t})\big]\,d\lambda.
	\]
	Since $\frac{1}{\sqrt{1-2t}} \exp\big[x(1-\sqrt{1-2t})\big]$ is a solution to the system of PDEs \eqref{E:bessel_poly_PDE} according to the leftmost path of the above diagram, this $y(x,\, t)$ satisfies
	\[
		\begin{split}
			&[\partial_2\partial_1-\partial_2+X_1]y(x,\, t)=0,\\
			&\big[X_1\partial_2+(1-2X_2)\partial_2-1-X_1\big] y(x,\, t)=0,
		\end{split}
	\]by Theorem~\ref{T:Newton_trans}
in the analytic function space $\mathcal{O}_{\Delta d}$ endowed with a $\mathcal{A}_2$-module structure given by the Example \ref{Eg:O_delta_d}. Hence the $y(x,\, t)$ satisfies the system of delay-differential equations \eqref{E:delta_bessel_poly_PDE}. We instead offer an alternative approach by direct verification that the $y(x,\, t)$ to satisfies the equivalent system\footnote{This would make the verification shorter.}
	\[
		\begin{split}
					&[(2X_1-1)\partial_2^2+3\partial_2+X_1^2]y(x,\, t)=0,\\
					&\big[ X_1\partial_1+(1-2X_2)\partial_2-1-X_1\big]y(x,\, t)=0,
		\end{split}
	\]where the first expression above can be found in Proposition \ref{P:Weyl_mod_half_bessel_3term}. 
It is straightforward to note that
	\[
	%\begin{split}
		\partial_2 y=\frac{1}{\Gamma(-x)}\int_{-\infty}^{(+0)}e^\lambda(-\lambda)^{-x-1}
	\big[ (1-2t)^{-1}+\lambda(1-2t)^{-\frac{1}{2}}\big]
 		\frac{1}{\sqrt{1-2t}} \exp\big[\lambda(1-\sqrt{1-2t})\big]\,d\lambda
%\end{split}
	\]
and 
\[
\begin{split}
\partial^2_2y &=\frac{1}{\Gamma(-x)}\int_{-\infty}^{(+0)}e^\lambda(-\lambda)^{-x-1}
\big[ 3(1-2t)^{-2} +3\lambda(1-2t)^{-\frac{3}{2}}
+\lambda^{2}(1-2t)^{-1}\big]\\
&\quad\times 
\frac{1}{\sqrt{1-2t}} \exp\big[\lambda(1-\sqrt{1-2t})\big]\,d\lambda\\
\end{split}
\]
so that
\[
\begin{split}
(2X_2-1)\partial_2^2y
	=&\frac{1}{\Gamma(-x)}\int_{-\infty}^{(+0)}e^\lambda(-\lambda)^{-x-1}
		\big[ -3(1-2t)^{-1}-3\lambda(1-2t)^{-\frac{1}{2}}
-\lambda^{2}\big]\\
	&\quad \times\frac{1}{\sqrt{1-2t}} \exp\big[\lambda(1-\sqrt{1-2t})\big]\,d\lambda\\
&=(-3\partial_2-X_1^2) y.
\end{split}
\]
Hence
	\[
		\big[(2X_2-1)\partial_2^2+3\partial_2+X_1^2\big]y=0.
	\]
On the other hand, we have
\[
X_1\partial_1y=\frac{1}{\Gamma(-x)}\int_{-\infty}^{(+0)}e^\lambda(-\lambda)^{-x-1}
(x-\lambda)
\frac{1}{\sqrt{1-2t}} \exp\big[\lambda(1-\sqrt{1-2t})\big]\,d\lambda\\
\]and
	\[
\begin{split}
\big[(1-2X_2)\partial_2-1-X_1\big]
y& =\frac{1}{\Gamma(-x)}\int_{-\infty}^{(+0)}e^\lambda(-\lambda)^{-x}
(1-\sqrt{1-2t})
\frac{1}{\sqrt{1-2t}} \exp\big[\lambda(1-\sqrt{1-2t})\big]\,d\lambda\\
=&\frac{1}{\Gamma(-x)}\int_{-\infty}^{(+0)}e^\lambda(-\lambda)^{-x}
\frac{1}{\sqrt{1-2t}} 
d\exp\big[\lambda(1-\sqrt{1-2t})\big].
\end{split}
\]
Integration-by-parts of the above improper integral over an Hankel-type contour yields\[
\frac{1}{\Gamma(-x)}e^\lambda(-\lambda)^{-x}
\frac{1}{\sqrt{1-2t}} \exp\big[\lambda(1-\sqrt{1-2t})\big]\Big\vert_{-\infty}^{(0+)}
=0
\]so that
\[
\begin{split}
\big[(1-2X_2)\partial_2-1-X_1\big] y=&-\frac{1}{\Gamma(-x)}\int_{-\infty}^{(+0)} \frac{1}{\sqrt{1-2t}} \exp\big[\lambda(1-\sqrt{1-2t})\big]\,d(e^\lambda(-\lambda)^{-x})\\
=&-X_1\partial_1\frac{1}{\Gamma(-x)}\int_{-\infty}^{(+0)}e^\lambda(-\lambda)^{-x-1} \frac{1}{\sqrt{1-2t}} \exp\big[\lambda(1-\sqrt{1-2t})\big]\,d\lambda.
\end{split}
\]Hence
	\[
		\big[ X_1\partial_1+(1-2X_2)\partial_2-1-X_1\big]y=0
	\]as desired. 
	That is, $y(x,\, t)$ satisfies
		\begin{align}
			&(2t-1)f_{tt}(x,\, t)+3f_t(x, t)+x(x-1)f(x-2, t)=0, \label{E:delta_rev_bessel_PDE_1}\\
			&(1-2t) f_t(x,t)+(x-1)f(x, t)-2xf(x-1, t)=0.\label{E:delta_rev_bessel_PDE_2}
		\end{align}
	Hence $y(x,\, t)$ satisfies the original system of delay-differential equations \eqref{E:delta_bessel_poly_PDE}.

Let us return to the right-half of the commutative diagram above. We notice that the rightmost path of the commutative diagram \eqref{E:commute-dbessel-2}, the sum $\sum_{n=0}^\infty\theta_n^\Delta (x)\, t^n/n!$ is a (formal sum) solution to the system of delay-difference equations
		\begin{align*}
			\theta^\Delta_n(x+1)-2\theta^\Delta_n(x)+x\theta^\Delta_{n-1}(x-1)&=0, \\
		 \theta^\Delta_{n+1}(x)+(x-2n-3)\theta^\Delta_n(x)-2x\, \theta^\Delta_n(x-1)&=0
	\end{align*}
	for each $n$, which were 
	 derived in the Theorem \ref{T:difference_Bessel_Poly_DD}.

		Let us determine the region of convergence of \eqref{E:bessel_difference_poly_gf}. First we differentiate the equation \eqref{E:delta_rev_bessel_PDE_2} with respect to $t$. This yields
		\begin{equation}
			\label{E:delta_rev_bessel_PDE_4}
				(1-2t)f_{tt}(x,\, t)+(x-3)f_t(x,\, t)-2xf_t(x-1,\, t)=0.
		\end{equation}
Replace $x$ by $x-1$ in equation \eqref{E:delta_rev_bessel_PDE_2} and multiply the resulting equation by $x$ throughout yield
		\begin{equation}\label{E:delta_rev_bessel_PDE_5}
			x(1-2t)f_{t}(x-1,t) +x(x-2)f(x-1, t)-2x(x-1)f(x-2,\, t)=0.
		\end{equation}
Multiply the equation \eqref{E:delta_rev_bessel_PDE_1} throughout by $2$ and added to the \eqref{E:delta_rev_bessel_PDE_5} yield
		\begin{equation}\label{E:delta_rev_bessel_PDE_6}
			2(2t-1)f_{tt}(x,t) +6 f_t(x,t) +x(1-2t) f_t(x-1, t)+x(x-2) f(x-1, t)=0.
		\end{equation}
Multiply the equation  \eqref{E:delta_rev_bessel_PDE_4} throughout by $1-2t$ yields
		\begin{equation}\label{E:delta_rev_bessel_PDE_8}
			(1-2t)^2f_{tt}(x,t)+(1-2t)(x-3)f_t(x,t)-2(1-2t)x f_t(x-1,t)=0.
		\end{equation}
Multiply the equation \eqref{E:delta_rev_bessel_PDE_6} by $2$ and add the resulting equation to \eqref{E:delta_rev_bessel_PDE_8} yield
		\begin{equation}\label{E:delta_rev_bessel_PDE_10}
			[(1-2t)^2+4(2t-1)]f_{tt}(x,t)+[(1-2t)(x-3)+12]f_t(x,t)+2x(x-2)f(x-1,t)=0.
		\end{equation}

Now multiply the equation \eqref{E:delta_rev_bessel_PDE_2} throughout by the factor $x-2$ yields
		\begin{equation}\label{E:delta_rev_bessel_PDE_12}
			(x-2)(1-2t)f_t(x,t)+(x-2)(x-1)f(x,t)-2x(x-2)f(x-1,t)=0.
		\end{equation}

It is now clear that the result of adding \eqref{E:delta_rev_bessel_PDE_10} and \eqref{E:delta_rev_bessel_PDE_12} yields, after simplification, the equation
		\begin{equation}
			\label{E:delta_rev_bessel_PDE_14}
				(2t-1)(2t+3) f_{tt}(x,t)+[(1-2t)(2x-5)+12] f_t(x,t)+(x-1)(x-2)f(x,t)=0.
		\end{equation}
		For a fixed $x$, the above equation \eqref{E:delta_rev_bessel_PDE_14} is a second order linear ODE with two finite regular singularities and $t=\frac12$ is the one closest to the origin $t=0$ which is an ordinary point. Hence we conclude that the expansion in \eqref{E:bessel_difference_poly_gf} converges uniformly in each compact subset of $\mathbb{C}\times B(0,\, \frac12)$ as asserted.
 This completes the proof of the theorem.
\end{proof}

\subsection{Classical Bessel polynomials $y_n(x)$}
\begin{example}
Let $\mathcal{A}_{2}(\eta):=\dfrac{\mathcal{A}_2\langle\eta\rangle}{\langle\eta^2-1+2X_1X_2\rangle}$ be as defined in Theorem \ref{T:symbol gen yn-1}. Then $\mathcal{Y}(\eta)$ as defined in Theorem \ref{T:symbol gen yn-1} is a left $\mathcal{A}_{2}(\eta)$-module. Other examples of $\mathcal{A}_{2}(\eta)$-modules include $\mathcal{O}_{dd}$, $\mathcal{O}_{\Delta d}$, $\mathcal{O}^{\mathbb{N}_0}$ and so on.
\end{example}

In much the same spirit of the derivation of the generating function for the classical reverse Bessel polynomials in Theorem \ref{T:rev_bessel_poly_gf}, we have the corresponding theorems for the generating function of the classical Bessel polynomials $(y_n(x))$ \cite{Krall_Frink_1949, Grosswald}.
\medskip

\begin{theorem}\label{T:general_bessel_poly_gf}
	\begin{enumerate} 
		\item Let $(\mathscr{Y}_n)$ be a sequence of analytic functions which is a solution of the  Bessel polynomial module $\mathcal{Y}$ in $\mathcal{O}_d^{\mathbb{N}_0}$. Then there exist complex constants $C_1,\, C_2$ such that
		\begin{equation}\label{E:gen_bessel_poly_gf}
		C_1\exp\left(\frac{1-(1-2xt)^{1/2}}{x}\right)
		+C_2 \exp\left(\frac{1+(1-2xt)^{1/2}}{x}\right)
		=\sum_{n=0}^\infty
		\mathscr{Y}_{n-1}(x)\, \frac{t^n}{n!},
		\end{equation}
		which holds for  $|t|<\left|\frac{1}{2x}\right|$, with an appropriate choice of $\mathscr{Y}_{-1}$.
		\item Moreover, the holonomic system of PDEs of $\mathcal{Y}$ from the Definition \ref{D:y-1} when manifested in $\mathcal{O}_{dd}$ is given by 
	\begin{equation}\label{E:classical_bessel_poly_PDE}
	\begin{split}
	&x^2f_{xt}(x,\, t)	-xtf_{tt}(x,\, t)+f_t(x,\,t)-f(x,\,t)=0,\\
	&x^2f_x(x,\, t)-f_t(x,\, t)+xtf_t(x,\, t)+f(x,\,t)=0,
	\end{split}
	\end{equation}for which both sides of \eqref{E:gen_bessel_poly_gf} satisfy.
	\end{enumerate}
\end{theorem}
\medskip

\begin{theorem}\label{T:Bessel p_gf_1}
	\begin{enumerate}
		\item  The classical Bessel polynomials have the generating function
		\begin{equation}\label{E:bessel p_gf-1}
		\exp\left(\frac{1-(1-2xt)^{1/2}}{x}\right)
		=\sum_{n=0}^\infty y_{n-1}(x)\, \frac{t^n}{n!}
		\end{equation}
		which holds for $|t|<\left|\frac{1}{2x}\right|$.
		\item Moreover, the generating function from part \textrm{(i)} satisfies the PDE system \eqref{E:classical_bessel_poly_PDE}\footnote{See also the Appendix \ref{SS:PDE_list}.}.
	\end{enumerate}
\end{theorem}

\begin{proof}
Recall the Poisson transform $\mathfrak{p}$ as defined in \eqref{E:poisson}. Then we have the commutative diagram
		\begin{equation}\label{E:commute-diff bess pol}
		\begin{tikzcd} [row sep=huge, column sep=huge]
		\mathcal{Y}(\eta) \arrow{r}{\times(y_{n-1}(x))} 
		\arrow[swap]{d}{\E(\frac{1-\eta}{X_1})}
		%\arrow[swap]{dr}{\mathfrak{g}_p} 
		& \mathcal{O}^{\mathbb{N}_0}_d 
		\arrow{d}{\mathfrak{p}} \\
		\widetilde{A}_2(\eta) \arrow{r}{\times 1}& \mathcal{O}_{d d}
		\end{tikzcd}
		\end{equation}
in which $\widetilde{\mathcal{A}_2}(\eta):=\overline{
			\mathcal{A}_2(\eta)/\big[\mathcal{A}_2(\eta)\partial_1+\mathcal{A}_2(\eta)\partial_2\big] }$ and the analytic functions in $\mathcal{O}_{d d}$ are suitably restricted. Due to the similarity of the proof of Theorem  \ref{T:rev_bessel_poly_gf}, the remaining of the proof will be completed in the Appendix \S\ref{A:BP_gf_1}.
			\end{proof}
\medskip

\subsection{Difference Bessel polynomials $y^\Delta_n(x)$} \label{S:detal_BP_gf_1}

\begin{theorem}\label{T:delta Bessel_gf_1} 
	\begin{enumerate}
		\item
The difference Bessel polynomials  $(y^\Delta_n(x))$ have the Poisson-type generating function
		\begin{equation}\label{E:bessel dp_gf-1}
		\displaystyle\frac{e^{-i\pi x}}{2i \sin \pi x\Gamma(-x)}\int_{-\infty}^{(0+)}
		e^{\lambda+\frac{1-\sqrt{1-2\lambda t}}{\lambda} }
		(-\lambda)^{-x-1}d\lambda
		=\sum_{n=0}^\infty y^{\Delta}_{n-1}(x)\, \frac{t^n}{n!},
		\end{equation}
		which converges uniformly in each compact subset of $\mathbb{C}\times B(0,\frac12)$.
		\item Moreover, the generating function from part $\mathrm{(i)}$ above satisfies the PDEs (delay-differential equations)\footnote{See also the Appendix \ref{SS:PDE_list}.}
		\begin{equation}\label{E:delta_classical_bessel_poly_PDE}
			\begin{split}
				&x(x-1)[f_t(x-1,\, t)-f_t(x-2, \, t)]-xtf_{tt}(x-1,\, t)+f_t(x,\, t)-f(x, \, t)=0,\\
				&x(x-1)[f(x-1,\, t)-f(x-2, \, t)]-f_{t}(x,\, t)+xtf_t(x-1,\, t)+f(x, \, t)=0.
			\end{split}
		\end{equation}
	\end{enumerate}
\end{theorem}
\medskip

\begin{proof} Recall the maps $\mathfrak{p}$ and $\mathfrak{p}_\Delta$ as defined in \eqref{E:poisson} and \eqref{E:poisson_2} respectively. Then we have the diagram 
		\begin{equation}\label{E:commute-dbp}
			\begin{tikzcd}  [row sep=large, column sep=large] %[row sep=huge, column sep=huge]
			& \mathcal{Y}(\eta) \arrow[swap]{dl}{\E(\frac{1-\eta}{X_1})} \arrow{d}{\times(y_{n-1}(x))} \arrow{dr}{\times(y^\Delta_{n-1}(x))} \\
			\widetilde{\mathcal{A}}_2(\eta) \arrow{dr}{\times 1} & \mathcal{O}^{\mathbb{N}_0} \arrow{d} {\mathfrak{p}} 	&   \mathcal{O}^{\mathbb{N}_0} \arrow{d}{\mathfrak{p}_\Delta} \\
			& \mathcal{O}_{d d} \arrow{r}{\times\mathfrak{N}} &  \mathcal{O}_{\Delta d} 
			\end{tikzcd}
			\end{equation}
		in which $\widetilde{\mathcal{A}}_2(\eta):=\overline{\mathcal{A}_2/\big[\mathcal{A}_2(\eta)\partial_1+\mathcal{A}_2(\eta)\partial_2\big]}$ and the analytic functions in $\mathcal{O}_{\Delta d}$ have suitably restricted growth.	The remaining of the proof, being similar to that of the Theorem \ref{T:bessel_poly_delta_gf}, will be completed in the Appendix \S \ref{A:detal_BP_gf_1}.
\end{proof}

\section{Half-Bessel modules II: Glaisher's generating functions}\label{S:half_bessel_II}

In a number of publications between the years 1873-1881 (see e.g., \cite{Glaisher1881}), the derivations of trigonometric-type generating functions for Bessel functions 
	\[
			\begin{split}
		&\sqrt{\frac{2}{\pi x}} \cos\sqrt{x^2-2xt}
		=\sum_{n=0}^\infty J_{n-\frac{1}{2}}(x)\, \frac{t^n}{n!},\\
		&\sqrt{\frac{2}{\pi x}} \sin\sqrt{x^2+2xt}
		=\sum_{n=0}^\infty J_{-n+\frac{1}{2}}(x)\, \frac{t^n}{n!}
		\end{split}
	\] obtained by Glaisher are known to be ``formidable" to compute as noted by Watson \cite[p. 144]{Watson1944}. We shall show in this section that how a suitable ``characteristic change of variables" from the generators of the Bessel module of order $n+\frac12$ in Definition \ref{D:bessel_mod} can simplify Glaisher\rq{}s computation \cite[pp. 774-781]{Glaisher1881}. 
Moreover, in addition to recovering and extending Glaisher's results for classical Bessel functions, we also derive analogous generating functions for the difference Bessel functions in this section. 
Analogous generating functions for $q$-Bessel functions will be investigated in a separate publication. On the other hand, Glaisher's results have been superseded by Lommel's results from a different perspective, see also \cite[p. 144]{Watson1944}. Weisner \cite{Weisner_1959} re-derived Lommel's results with Lie-algebraic method, whose idea  can also be found in \cite{Mcbride}. We comment that no explicit Lie algebraic  method has been employed in our approach.
	
\subsection{Characteristics of Glaisher modules}
\subsubsection{Negative Glaisher module}

	The characteristic change of variables we shall apply are
	\begin{equation}\label{E:glaisher_char_1}
		\begin{array}{ll}
			\Xi_1=\partial_1, & W_1=X_1\\
	 		\Xi_2={1}/{X_2}, & W_2=X_2\partial_2 X_2.
		\end{array}
	\end{equation}
%	\textcolor{red}{
%		\begin{equation}\label{E:glaisher_char_1}
%		\begin{array}{ll}
%			\Xi_1=\pm \partial_1, & W_1=\pm X_1\\
%	 		\Xi_2={1}/{X_2}, & W_2=X_2\partial_2 X_2.
%		\end{array}
%	\end{equation}
%}
	Then one can easily verify that
	\[
		[\pm \Xi_{1}, \pm W_{1}]=1, \quad [\Xi_{2}, W_{2}]=1,
		\quad
		[\Xi_{1}, \Xi_{2}]=0,\quad [W_1, W_{2}]=0,
		\quad [\Xi_i,\, W_j]=0,\ i\not=j
	\]hold.
	
	Substitute the above characteristic change of variables into \eqref{E:bessel_PDE} and choose  
	$\nu=-\frac{1}{2}$  yield
%	\begin{equation}%\label{E:PDEs_neg_glaisher}
	\begin{align}
		&W_1 \Xi_1+(W_2\Xi_2-\frac{1}{2})-W_1\Xi_2^{-1},\label{E:PDEs_neg_glaisher_1}\\
	&W_1 \Xi_1-(W_2 \Xi_2-\frac{1}{2})+W_1 \Xi_2.\label{E:PDEs_neg_glaisher_2}
	\end{align}
%	\end{equation}
%	\textcolor{red}{
%		\begin{equation}\label{E:PDEs_half_bessel}
%		\begin{split}
%			&W_1 \Xi_1+(W_2\Xi_2\pm\frac{1}{2})\pm W_1\Xi_2^{-1},\\
%			&W_1 \Xi_1-(W_2 \Xi_2\pm\frac{1}{2})\mp W_1 \Xi_2.
%		\end{split}
%		\end{equation}
%	}
\medskip

\begin{definition}\label{D:neg_glaisher_module}
		Let $W_i$ and $\Xi_i$ be defined as in \eqref{E:glaisher_char_1}. The
		\begin{equation}\label{E:neg_glaisher_module}
			\mathcal{G}_{-}=
		\dfrac{\mathcal{A}_{2}}
		{\mathcal{A}_{2}
	\big(W_1 \Xi_1+(W_2\Xi_2-\frac{1}{2})- W_1\Xi_2^{-1}\big)
			+\mathcal{A}_{2}\big(W_1 \Xi_1-(W_2 \Xi_2-\frac{1}{2})+ W_1 \Xi_2\big)}
		\end{equation}
		is called the \textit{negative Glaisher module}.
\end{definition}
\medskip

\begin{theorem}[\textbf{(Integrability)}]\label{T:neg_glaisher_holonomic} 
The negative Glaisher module $\mathcal{G}_{-}$ is a left $\mathcal{A}_2$-module that has dimension two and multiplicity two. In particular, it is holonomic.
%The negative Glaisher module is an holonomic $\mathcal{A}_2$-module. In particular, the space of all left $\mathcal{A}_2$-linear maps $\mathcal{G}_{-}\to\mathcal{O}_2$ has  $\mathbb{C}$-dimension two and multiplicity two.
\end{theorem}
\medskip
\begin{proof} Let us effect the following ``change of variables"
	\[
		\widetilde{W}_2=\Xi_2,\quad
		\widetilde{\Xi}_2=-W_2
	\]
	in the two generators of the negative Glaisher module \eqref{E:neg_glaisher_module}. Note that $[\widetilde{\Xi}_2,\, \widetilde{W}_2]=1$.
	This yields
		\[
	\begin{split}
		&W_1 \Xi_1+(\widetilde{W}_2\widetilde{\Xi}_2-\frac{1}{2})-W_1\widetilde{W}_2^{-1},\\
	&W_1 \Xi_1-(\widetilde{W}_2\widetilde{\Xi}_2-\frac{1}{2})+W_1 \widetilde{W}_2.
	\end{split}
	\]
Subtracting and adding the above two relations yield
	\[
		\begin{split}
			& \widetilde{\Xi}_2+\frac12\Big(\frac{1}{\widetilde{W}_2}+W_1+\frac{W_1}{\widetilde{W}_2^2}\Big)\\
			& \Xi_1+\frac12\Big(\widetilde{W}_2-\frac{1}{\widetilde{W}_2}\Big).
		\end{split}
	\]One immediately verifies that the two relations satisfy the \textit{integrability condition} discussed from Example \ref{EG:integrable}.
	
	A disadvantage of the above argument is although the ``change of variables" method preserves holonomicity, there is no guarantee that the multiplicity of the negative Glaisher's module remains unchanged. Indeed, we could follow the argument used in the proof of Theorem \ref{T:B_poly_mod_holonomic} to conclude directly that the dimension and multiplicity of the module are both two. We skip the details.
\end{proof} 
\medskip

We immediately obtain from the negative Glaisher module \eqref{E:neg_glaisher_module}
the following surjection.

	\begin{proposition}
		The natural map
		\[
		\dfrac{\mathcal{A}_2}
		{\mathcal{A}_2
			[(W_1\Xi_1)^2+W_1^2-(W_2 \Xi_2
			-\frac{1}{2})^2]}
		\rightarrow \mathcal{B}_{-\frac{1}{2}}
		\]
		is a well-defined $\mathcal{A}_2$-linear surjection. 
	\end{proposition}
\medskip

We omit its routine verification. 
\medskip

\begin{theorem}\label{T:neg_glaisher_cosine}
		Let $\mathcal{G}_{-}$ be the negative Glaisher module. Set $\omega_-^2=W_1^2- 2W_1W_2$ and
		\[
		\mathcal{A}_{2}(\omega_-)=\langle  \Xi_1, \Xi_2, 
		 W_1, W_2,  \omega_-^{\pm 1}\rangle.
		\]
		Then the
			\[
				S=C_1\Cos \omega_-+ C_2\Sin \omega_-
			\] for each choice of  complex constants $C_1, C_2$, where $\Sin$ and $\Cos$ are the Weyl sine and the Weyl cosine defined in \eqref{E:sine-map} and \eqref{E:cosine-map}, is an element of $\mathcal{A}_2(\omega_-)$, such that the map
		\[
			\mathcal{G}_{-}
		\xrightarrow 
		{\times S}
		\overline{\mathcal{A}_{2}(\omega_-)/[\mathcal{A}_{2}(\omega_-)
			(W_1\Xi_1+\frac{1}{2})
			+\mathcal{A}_{2}(\omega_-)\Xi_2]}
		\]
is well-defined and left $\mathcal{A}_2$-linear.
	\end{theorem}
\medskip
\medskip

\begin{proof} Suppose 
			\[
				\delta_1=\Xi_1\, \omega_--\omega_-\,\Xi_1,
			\qquad
				\delta_2=\Xi_2\, \omega_--\omega_-\, \Xi_2
			\]
			and 
			\begin{equation}\label{E:glaisher_char_3}
				\delta_1\, \omega=\omega_-\,\delta_1,\quad  \omega_-\delta_2=\delta_2\omega_-,
				\quad [\omega_-,\,  W_1]=0,\quad [\omega_-,\, W_2]=0.
			\end{equation}			
		Then it is easy to compute
		\[
		\Xi_1\, \omega_-^2-\omega_-^2\,\Xi_1=2(W_1-W_2), \quad
		\Xi_2\, \omega_-^2-\omega_-^2\, \Xi_2=-2W_1.
		\]
Hence
		\begin{equation}\label{E:glaisher_char_4}
		\delta_1=	\Xi_1\, \omega_- -\omega_-\, \Xi_1=\frac{W_1-W_2}{\omega_-},
			\quad 
			\delta_2=\Xi_2\, \omega_--\omega_-\, \Xi_2=-\frac{W_1}{\omega_-}.
		\end{equation}
	Subtracting the two elements   \eqref{E:PDEs_neg_glaisher_1} and  \eqref{E:PDEs_neg_glaisher_2} yields the expression
		\[
			-2(W_2\Xi_2-\frac12\big)+W_1\Xi_2+W_2/\Xi_2.
		\]Left multiplication to the above expression throughout by $\Xi_2$ on the left yields
		\[
			\Xi_2[(W_1-2W_2)\Xi_2+1+W_1/\Xi_2]=(W_1-2W_2)\Xi_2^2-\Xi_2+W_1.
		\]
		Right multiplication of the above expression by $W_1$ throughout yields
		\[
			W_1(W_1-2W_2)\Xi_2^2-W_1\Xi_2+W_1^2=\omega^2\, \Xi_2^2-W_1\Xi_2+W_1^2.
		\]
		Substituting the second expression from \eqref{E:glaisher_char_4} into the above expression and utilising the \eqref{E:glaisher_char_3} lead to
		\[
			(\omega_-\, \Xi_2)^2+W_1^2
			=W_1^2 \big[\big(\frac{\omega_-\, \Xi_2}{W_1}\big)^2+1\big]=W_1^2 (\Xi^2+1).
		\]Let us denote
			\begin{equation}\label{E:PDEs_neg_glaisher_3} 
				\Xi=-\frac{\omega_-\,  \Xi_2}{W_1}.
			\end{equation}It follows from the second expression of \eqref{E:glaisher_char_4} that
			\[
				[\Xi,\,  \omega_-]=1.
			\]
Therefore, it follows from the above commutator and \eqref{E:PDEs_neg_glaisher_2} that it is appropriate to set
		\begin{align}
			\aleph: &= \Xi_1+\Big(1-\frac{W_2}{W_1}\Big)\Xi_2,\label{E:PDEs_neg_glaisher_4} \\
			W: &= W_1.\label{E:PDEs_neg_glaisher_5} 
		\end{align} Then it is easy to check that
		\[ 
			[\aleph, W]=1,\quad
			[\Xi,\,  \omega_-]=1,\quad
			[\aleph, \Xi]=0,\quad
			[\aleph, \omega_-]=0,\quad
			[\Xi, W]=0,\quad
			[\omega_-, W]=0
		\]hold. Hence we may consider
		\[
			\widehat{\mathcal{A}}_2=\mathbb{C}\langle \aleph , \Xi, W, \omega_-\rangle,
		\]
and to rewrite the negative Glaisher module \eqref{E:neg_glaisher_module} in the form with the above ``change of variables\rq\rq{}:
		\[
			\widehat{\mathcal{G}}_- :=\frac{\widehat{\mathcal{A}}_2}
			{\widehat{\mathcal{A}}_2(W\aleph+\frac12)+\widehat{\mathcal{A}}_2 (\Xi^2+1)}.
		\]
	We observe from Example \ref{E:prime_integrable_eg_2} that the $\widehat{\mathcal{G}}_-$ is a holonomic module. In particular, its dimension and multiplicity are both equal to $2$. On the other hand, it follows from Example \ref{Eg:trigo} that there exists a left $\widehat{\mathcal{A}_2}$-linear map
	\[
			 \widehat{\mathcal{A}}_2/
			 \widehat{\mathcal{A}}_2 (\Xi^2+1) \ \xrightarrow {\times S}\
			\overline{ \widehat{\mathcal{A}}_2/
			 \widehat{\mathcal{A}}_2 \Xi}
	\]
	 of the form
			\[
				S= C_1\Cos (\omega_-)+C_2\Sin (\omega_-)
			\]for some complex numbers  $C_1,\, C_2$.

			Hence a solution map $S$ for
		\[
			\widehat{\mathcal{G}}_- 
			\ \xrightarrow {\times S}\
				\overline{\widehat{\mathcal{A}}_2/\big[
			\widehat{\mathcal{A}}_2(W\aleph+\frac12) +\widehat{\mathcal{A}}_2\Xi\big]}
		\]must assume the form
		\[
			S=f_1(W)\Cos (\omega_-)+f_2(W)\Sin (\omega_-),
		\]where the maps $f_1, f_2$ depends on $W$ only. 
		
		Without loss of generality, we may consider,  $f_1$ only. As a result,
	\[
	\begin{split}
	0=&\big(W\aleph+\dfrac12)f_1(W)\,\mathrm{Cos} \,\omega_-
	\ \mod \widehat{\mathcal{A}_2}(W\aleph+\dfrac12)\\
	=&\mathrm{Cos} \,\omega_-(W\aleph+\dfrac12\big)f_1(W)
	\ \mod  \widehat{\mathcal{A}_2}(W\aleph+\dfrac12)\\
	=&W\aleph f_1(W),
	\end{split}
	\]
	which implies $\aleph f_1(W)=0$, Hence  $f_1(W)$ reduces to a constant. Similarly $f_2(W)$ also reduces to a constant.
	
	Hence it follows from \eqref{E:PDEs_neg_glaisher_3}, \eqref{E:PDEs_neg_glaisher_4} and \eqref{E:PDEs_neg_glaisher_5} that we have derived the asserted map $S$
	\[
	\mathcal{G}_-(\omega_-)
	\xrightarrow 
	{\times (C_1\,\mathrm{Cos} \,\omega_-+ C_2\,\mathrm{Sin} \,\omega_-)}
	\overline{\mathcal{A}_2(\omega_-)/
		[\mathcal{A}_2(\omega_-)(W_1\Xi_1+1/2)+
		\mathcal{A}_2(\omega_-)\Xi_2}].
	\]This completes the proof.
	\end{proof}
\medskip

\subsubsection{Positive Glaisher module}

	The characteristic change of variables we shall apply are specific sign changes of  those of \eqref{E:glaisher_char_1}
	\begin{equation}\label{E:glaisher_char_2}
		\begin{array}{ll}
			\Xi_1=-\partial_1, & W_1=-X_1,\\
	 		\Xi_2={1}/{X_2}, & W_2=X_2\partial_2 X_2
		\end{array}
	\end{equation}with $\nu=\frac12$ 
	in  \eqref{E:bessel_PDE} and choose  
			\begin{equation}\label{E:PDEs+half_bessel}
	\begin{split}
		&W_1 \Xi_1+(W_2\Xi_2-\frac{1}{2})+W_1\Xi_2^{-1},\\
	&W_1 \Xi_1-(W_2 \Xi_2-\frac{1}{2})-W_1 \Xi_2.
	\end{split}
	\end{equation}

\begin{definition}\label{D:pos_glaisher_module} Let $W_i$ and $\Xi_i$ be defined as in \eqref{E:glaisher_char_2}. Then
	\begin{equation}\label{E:pos_glaisher_module}
		\mathcal{G}_{+}=
		\dfrac{\mathcal{A}_{2}}
		{\mathcal{A}_{2}
	\big(W_1 \Xi_1+(W_2\Xi_2-\frac{1}{2})+W_1\Xi_2^{-1}\big)
			+\mathcal{A}_{2}\big(W_1 \Xi_1-(W_2 \Xi_2-\frac{1}{2})-W_1 \Xi_2\big)}
		\end{equation}
		a \textit{positive  Glaisher module}.
\end{definition}
\medskip

We obviously have the following companion theorem for the negative Glaisher module.

\begin{theorem}\label{T:pos_glaisher_holonomic} The positive Glaisher module  $\mathcal{G}_{+}$ is  a left $\mathcal{A}_2$-module that has dimension two and multiplicity two. In particular, it is holonomic.
\end{theorem}
\medskip

We state without giving details the theorem
	\begin{theorem}\label{T:pos_glaisher_sine}
		Let $\mathcal{G}_{+}$ be defined in Definition 
		\ref{D:pos_glaisher_module}. Set $\omega_+^2=W_1^2+2W_1W_2$ and
		\[
		\mathcal{A}_{2}(\omega_+)=\langle  -\Xi_1, \Xi_2, 
		 -W_1, W_2,  \omega_+^{\pm 1}\rangle.
		\]
		Then the expression
			\[
				S=C_1\Cos \omega_++ C_2\Sin \omega_+,
			\]
			where $C_1, C_2$ are constants, $\Sin$ and $\Cos$ are the Weyl sine and the Weyl cosine defined in \eqref{E:sine-map} and \eqref{E:cosine-map} respectively, is an element of $\mathcal{A}_2(\omega_+)$,  such that the map
		\[
		\overline{\mathcal{G}_{+}}
		\xrightarrow 
		{\times S}
		\overline{\mathcal{A}_{2}(\omega_+)/[\mathcal{A}_{2}(\omega_+)
			(W_1\Xi_1+\frac{1}{2})
			+\mathcal{A}_{2}(\omega_+)\Xi_2]}
		\]
		is left $\mathcal{A}_{2}(\omega_+)$-linear.
	\end{theorem}

\subsection{Glaisher's generating functions}

Combining the discussion from the previous subsections concerning the positive and negative Glaisher-modules, which are modified half-Bessel modules, yields the followings.

\begin{theorem}\label{T:classical_glaisher_gf}
	\begin{enumerate}
		\item Let $(\mathscr{C}_{n+\nu})_n$ be a bilateral sequence of analytic functions which is a solution of the Glaisher modules $\mathcal{G}_\mp$ in $\mathcal{O}^\mathbb{Z}$. Then there exist complex numbers $C_1,\, C_2$ such that
		\begin{equation}\label{E: general half bessel_gf}
			x^{-\frac12} 
			\big[C_1\cos\sqrt{x^2\pm 2xt}+C_2\sin \sqrt{x^2\pm 2xt}
\big]
			=\sum_{n=0}^\infty \mathscr{C}_{\mp n\pm \frac{1}{2}}(x)\, \frac{t^n}{n!},
		\end{equation}where the convergence of	the infinite sum in \eqref{E: general half bessel_gf} is uniform in each compact subset of $\mathbb{G}:=\{(x,\, t)\in \mathbb{C}^\dagger\times\mathbb{C}: 2|t|<|x|\}$, $\mathbb{C}^\dagger:=\mathbb{C}\backslash\{x:x\le 0\}$.
	\item Moreover, the holonomic systems of PDEs that generates the Glaisher's modules $\mathcal{G}_{-}$ \eqref{E:neg_glaisher_module} and $\mathcal{G}_{+}$ \eqref{E:pos_glaisher_module} with manifestation in $\mathcal{O}_{dd}$  as defined in \eqref{E:O_dd_endow} from Example \ref{Eg:O_deleted_d} are given, respectively,  by\footnote{See also the Appendix \ref{SS:PDE_list} for the list of holonomic PDE systems.}
			\begin{align}	
			&xf_{xt}(x,\, t)+tf_{tt}(x,\, t)+\frac12 f_t(x,\, t)-xf(x, t)=0,\label{E:glaisher_PDE_1}\\
			&xf_x(x,\, t)+(x-t)f_t(x,\, t)+ \frac12 f(x,\, t)=0\label{E:glaisher_PDE_2}
			\end{align}
		and
			\begin{align}
				&xf_{xt}(x,\, t)+tf_{tt}(x,\, t)+\frac12 f_t(x,\, t)+xf(x, t)=0,\label{E:glaisher_PDE_6}
				\\
				&xf_x(x,\, t)-(x+t)f_t(x,\, t)+\frac12 f(x,\, t)=0.\label{E:glaisher_PDE_7}
		\end{align}
	\end{enumerate}
\end{theorem}
\medskip

\begin{proof}
Recall the map $\mathfrak{p}$ as defined in \eqref{E:poisson}. Then we have the diagram 
		\begin{equation*}%\label{E:commute-glaisher}
			\begin{tikzcd} [row sep=large, column sep=large]
			\mathcal{G}_{\pm} \arrow{r}{\times(\mathscr{C}_{\pm n\mp\frac12})} 
			\arrow[swap]{d}{\Cos\omega_\pm\slash\Sin\omega_\pm} 
			&   \mathcal{O}^{\mathbb{N}_0} \arrow{d}{\mathfrak{p}}\\
 			 \widetilde{\mathcal{A}_2}(\omega_\pm)
\arrow{r}{\times x^{-\frac12}} &  \mathcal{O}_{d d}
			\end{tikzcd}
		\end{equation*}
in which 
	\[
		\widetilde{\mathcal{A}_2}(\omega_\pm):=\overline{\displaystyle
		{\mathcal{A}_{2}(\omega_\pm)}/{\big[\mathcal{A}_{2}(\omega_\pm)(W_1\Xi_1+\frac{1}{2})
			+\mathcal{A}_{2}(\omega_\pm)\Xi_2\big]}},
	\]where the left-vertical map given in terms of 
	Weyl cosine and sine $\Cos \omega_\pm\slash \Sin \omega_\pm$ which are defined as in Example \ref{Eg:trigo}, is guaranteed by the Theorem \ref{T:neg_glaisher_cosine}. We also note that the analytic functions in $\mathcal{O}_{dd}$ are suitably restricted. 

We first consider the case of negative Glaisher's module $\mathcal{G}_-$. Since $\omega_-^2 =W_1^2-2W_1W_2$, the lower left path of the above commutative diagram shows that
  		\begin{equation}\label{E:glaisher_fn}
			x^{-\frac12}
			\big(C_1\cos\sqrt{x^2- 2xt}+C_2\sin \sqrt{x^2- 2xt}
\big)
		\end{equation}solves the PDEs \eqref{E:PDEs_neg_glaisher_1} and \eqref{E:PDEs_neg_glaisher_2}
when manifested in the $\mathcal{A}_2$-module $\mathcal{O}_{dd}$ endowed with the interpretation in the Example \ref{Eg:O_deleted_d}, i.e., 
		\begin{align*}
%			\begin{split}
				&xf_{xt}(x,\, t)+tf_{tt}(x,\, t)+\frac12 f_t(x,\, t)-xf(x, t)=0,%\label{E:glaisher_PDE_1}
				\\
				&xf_x(x,\, t)+(x-t)f_t(x,\, t)+\frac12 f(x,\, t)=0%\label{E:glaisher_PDE_2}
%			\end{split}
		\end{align*}
	which are precisely the PDEs \eqref{E:glaisher_PDE_1} and \eqref{E:glaisher_PDE_2}.
On the other hand, via the top-right path of the diagram, the sum 
		\[
			\sum_{n=0}^\infty J_{n-\frac{1}{2}}(x)\,\frac{ t^n}{n!}
		\]is a solution to the system of PDEs \eqref{E:PDEs_neg_glaisher_1} and \eqref{E:PDEs_neg_glaisher_2}.
		Since we know from Theorem \ref{T:neg_glaisher_holonomic} that the $\mathcal{G}_-$ is holonomic with both dimension and multiplicity equal to two, so the dimension of the local solution space about $t=0$ for the PDEs   \eqref{E:PDEs_neg_glaisher_1} and \eqref{E:PDEs_neg_glaisher_2} is therefore two. We conclude that the formula \eqref{E: general half bessel_gf} holds. We next investigate the convergence region of the series \eqref{E: general half bessel_gf}.

Differentiating the equation \eqref{E:glaisher_PDE_2} with respect to $t$ yields
		\begin{equation}\label{E:glaisher_PDE_3}
			xf_{xt}(x,\, t)+(x-t)f_{tt}(x,\, t)-\frac12f_t(x,\, t)=0.
		\end{equation}
		Substituting the equation \eqref{E:glaisher_PDE_1} into equation \eqref{E:glaisher_PDE_3} yields the new equation
			\begin{equation}\label{E:glaisher_PDE_4}				
				(x-2t)f_{tt}(x,\, t) -f_t(x, t)+xf(x,\, t)=0.
			\end{equation}The equation is a second order linear ODE for a fixed $x$ and it has an ordinary point at $t=0$ and a regular singularity at $t={x}/{2}$. Hence the expansion \eqref{E: general half bessel_gf}  converges uniformly in each compact subset of $\mathbb{C}\times \mathbb{C}$ where $2|t|<|x|$.

The handling of the second case when $\omega_+=W_1^2+2W_1W_2$ when $\nu=\frac12$ was dealt with in the last subsection which is similar to that of the case $\mathcal{G}_-$ 
and is therefore omitted. We only note that the PDEs \eqref{E:glaisher_PDE_1} and \eqref{E:glaisher_PDE_2} are replaced by
		\begin{align*}
				&xf_{xt}(x,\, t)+tf_{tt}(x,\, t)+\frac12 f_t(x,\, t)+xf(x, t)=0,%\label{E:glaisher_PDE_6}
				\\
				&xf_x(x,\, t)-(x+t)f_t(x,\, t)+\frac12 f(x,\, t)=0%\label{E:glaisher_PDE_7}
		\end{align*}which are precisely the \eqref{E:glaisher_PDE_6} and \eqref{E:glaisher_PDE_7}
		 A similar procedure that has led to the equation \eqref{E:glaisher_PDE_4} 
		applied to the above equations \eqref{E:glaisher_PDE_6} and  \eqref{E:glaisher_PDE_7} yields
			 \begin{equation}\label{E:glaisher_PDE_8}				
				(x+2t)f_{tt}(x,\, t) + f_t(x,\, t)+xf(x,\, t)=0.
			\end{equation}This equation is similar to the equation \eqref{E:glaisher_PDE_4} which also has an ordinary point at $t=0$ and a regular singular point at $t=-x/2$. So we omit the details. \footnote{It is unclear to the authors of the reason why Watson indicated in \cite[p. 140]{Watson1944} that the restriction $2|t|<|x|$ in the above theorem only applies to the case $\nu=\frac12$, that is to $\mathcal{G}_+$ instead to both since $t=x/2$ is a common regular singularity for both of the differential equations  \eqref{E:glaisher_PDE_4}  and \eqref{E:glaisher_PDE_8}. }

% It becomes clear that  the $\widetilde{A}_2(\omega)$ defined in \eqref{E:glaisher-map} is  holonomic since it is  \textit{integrable} as explained in 

%Example \ref{EG:integrable}. It follows from Bernstein\rq{}s Theorem \ref{T:bernstein}  that the $\deg(\mathcal{G}_\pm)=2$. Hence \eqref{E: general half bessel_gf} holds  apart from two complex constant multiples which can easily be incorporated into the constants $C_1,\, C_2$. 

The above analysis shows that the commutative diagram \eqref{E:commute-glaisher} holds.
This completes the proof.
\end{proof}
\medskip

%\medskip

One can deduce the following Glaisher-type generating functions for other classical Bessel functions from the above Theorem.
 
\begin{corollary} \label{C:Graisher_gf}

Let $\mathbb{G}:=\{2|t|<|x|:\ (x,\, t)\in \mathbb{C}^\dagger\times\mathbb{C}\}$ where $\mathbb{C}^\dagger:=\mathbb{C}\backslash\{x:x\le 0\}$. Then the following series converge uniformly in any compact subset of $\mathbb{G}$.\footnote{Modified forms of the (i) and (ii) can be found from \cite[p.  439]{AS_1964} for $j_n(x)=\sqrt{\pi/2x}J_{n+\frac12}(x)$ and $y_n(x)=\sqrt{\pi/2x}Y_{n+\frac12}(x)$.} 
	\begin{enumerate}
		\item 
		\begin{equation}\label{E: first half bessel_gf}
			\begin{split}
		&\sqrt{\frac{2}{\pi x}} \cos(x^2-2xt)^{\frac{1}{2}}
		=\sum_{n=0}^\infty J_{n-\frac{1}{2}}(x)\, \frac{t^n}{n!},\\
		&\sqrt{\frac{2}{\pi x}} \sin(x^2+2xt)^{\frac{1}{2}}
		=\sum_{n=0}^\infty J_{-n+\frac{1}{2}}(x)\, \frac{t^n}{n!};
		\end{split}		\end{equation}
		
	\item
	\begin{equation}\label{E: second half bessel gf}
	\begin{split}
	&\sqrt{\frac{2}{\pi x}}\sin(x^2-2xt)^{\frac{1}{2}}
	=\sum_{n=0}^{\infty} Y_{n-\frac{1}{2}}(x)\,\frac{t^n}{n!},\\
	-&\sqrt{\frac{2}{\pi x}} \cos(x^2+2xt)^{\frac{1}{2}}
	=\sum_{n=0}^{\infty} Y_{-n+\frac{1}{2}}(x)\,\frac{t^n}{n!};
	\end{split}
	\end{equation}
	\item
	\begin{equation}\label{E: first modify half bessel gf}
	\begin{split}
	&\sqrt{\frac{2}{\pi x}}\cos(-x^2-2xt)^{\frac{1}{2}}=
	\sum_{n=0}^{\infty} I_{n-\frac{1}{2}}(x)\, \frac{t^n}{n!},\\
	-&i\sqrt{\frac{2}{\pi x}}\sin(-x^2-2xt)^{\frac{1}{2}}=
	\sum_{n=0}^{\infty} I_{-n+\frac{1}{2}}(x)\, \frac{t^n}{n!};
	\end{split}
	\end{equation}
	
\item 
\begin{equation}\label{E: second modify half bessel gf}
\sqrt{\frac{\pi}{2 x}}e^{i (-x^2+2xt)^{1/2}}=
\sum_{n=0}^{\infty} K_{n-\frac{1}{2}}(x)\, \frac{t^n}{n!}=
\sum_{n=0}^{\infty} K_{-n+\frac{1}{2}}(x)\, \frac{t^n}{n!}.
\end{equation}
	\end{enumerate}

\end{corollary}
\medskip

	\begin{proof}
		\begin{enumerate}
		\item  By choosing $t=0$ in equation \eqref{E: general half bessel_gf}, we immediately have 
		\[
		x^{-\frac{1}{2}} \big[ C_1\cos x +C_2\sin x\big]=
		\mathscr{C}_{-\frac{1}{2}}(x)
		\] 
		and 
		\[
		x^{-\frac{1}{2}} \big[ C_1\sin x +C_2\cos x\big]
		=\mathscr{C}_{\frac{1}{2}}(x).
		\] 
		
		Since 
		\[
			J_{-\frac{1}{2}}(x)=  \sqrt{2/\pi} x^{-\frac{1}{2}} \cos x,
			\qquad
		J_{\frac{1}{2}}(x)=  \sqrt{2/\pi} x^{-\frac{1}{2}} \sin x.
		\]
		Hence $C_1=C_2=\sqrt{2/\pi}$ and we obtain  \eqref{E: first half bessel_gf}.
		 
		\item  
		Applying the well-known formulae
		\[
			Y_{-\frac{1}{2}}(x)=  \sqrt{2/\pi} x^{-\frac{1}{2}} \sin x,
			\qquad
			Y_{\frac{1}{2}}(x)=  -\sqrt{2/\pi} x^{-\frac{1}{2}} \cos x
		\]
	to $\mathscr{C}_{-\frac{1}{2}}(x)$ and 
	$\mathscr{C}_{\frac{1}{2}}(x)$,
	 we obtain \eqref{E: second half bessel gf} immediately.

\item We have
	\begin{align*}
	\sum_{n=0}^{\infty} I_{n-\frac{1}{2}}(x)\, \frac{t^n}{n!} &= 
	\begin{cases}
	\sum_{n=0}^{\infty}(e^{-\frac{ i\pi}{2}(n-\frac{1}{2}) }
	J_{n-\frac{1}{2}}(xe^{\frac{i\pi}{2}}))\, \frac{t^n}{n!},
	& -\pi <\arg x\leqslant \frac{1}{2}\pi\\
	\sum_{n=0}^{\infty}(e^{\frac{3}{2}i  \pi(n-\frac{1}{2})} J_{n-\frac{1}{2}}(xe^{\frac{-3i\pi}{2}}))\, \frac{t^n}{n!},
	& \frac{\pi }{2}<\arg x\leqslant \pi,
	\end{cases} \\
	&= 
	\begin{cases}
	e^{\frac{ i\pi}{4} }\sum_{n=0}^{\infty}
	J_{n-\frac{1}{2}}(xe^{\frac{i\pi}{2}}) \, \frac{(te^{-\frac{i\pi}{2}})^n}{n!},
	& -\pi <\arg x\leqslant \frac{1}{2}\pi,\\
	e^{-\frac{3}{4}i  \pi} \sum_{n=0}^{\infty}
	 J_{n-\frac{1}{2}}(xe^{\frac{-3i\pi}{2}})\, \frac{(te^{\frac{3i\pi}{2}})^n}{n!},
	& \frac{\pi }{2}<\arg x\leqslant \pi.
	\end{cases} \\
	&= 
	\begin{cases}
	\sqrt{2/\pi}x^{-\frac{1}{2}}
	\cos(-x^2-2xt)^{\frac{1}{2}} ,
	& -\pi <\arg x\leqslant \frac{1}{2}\pi,\\
	\sqrt{2/\pi}x^{-\frac{1}{2}}
	\cos(-x^2-2xt)^{\frac{1}{2}}
	& \frac{\pi }{2}<\arg x\leqslant \pi.
	\end{cases}
	\end{align*}
	from (i). Similarly, we have 
	\[
	\sum_{n=0}^{\infty} I_{-n+\frac{1}{2}}(x)\, \frac{t^n}{n!}= 
	\begin{cases}
	e^{-\frac{ i\pi}{4} }\sum_{n=0}^{\infty}
	J_{-n+\frac{1}{2}}(xe^{\frac{i\pi}{2}}) 
	\, \frac{(te^{\frac{i\pi}{2}})^n}{n!},
	& -\pi <\arg x\leqslant \frac{1}{2}\pi,\\
	e^{\frac{3}{4}i  \pi} \sum_{n=0}^{\infty}
	J_{-n+\frac{1}{2}}(xe^{\frac{-3i\pi}{2}})
	\, \frac{(te^{\frac{-3i\pi}{2}})^n}{n!},
	& \frac{\pi }{2}<\arg x\leqslant \pi.
	\end{cases}
	\]
We can now substitute the generating function for $J_{-n+\frac{1}{2}}(x)$ derived from (i) above to obtain the desired result.
		\item 
		Recall the relationships
		\[
		K_{n-\frac{1}{2}}(x)=\frac{\pi}{2} (-1)^n 
		\big[  I_{n-\frac{1}{2}}(x) -I_{-n+\frac{1}{2}}(x) \big],
		\]
		and 	
		\[
		K_{-n+\frac{1}{2}}(x)=\frac{\pi}{2} (-1)^n 
		\big[  I_{n-\frac{1}{2}}(x) -I_{-n+\frac{1}{2}}(x) \big]
		\]
	and	we obtain the corresponding generating functions.
		\end{enumerate}
	\end{proof}

\subsection{Difference Glaisher's generating functions}

\begin{theorem}\label{T:difference_glaisher_gf}
	\begin{enumerate}
		\item
		The difference Bessel functions have the Glaisher's generating function
			\begin{equation}\label{E:delta_glaisher_gf}
				\begin{split}
			\dfrac{e^{-i\pi x}}{2i\sin(\pi x)} \int_{-\infty}^{(0+)} 
			&	\lambda^{-\frac12} \frac{\big[C_1(x)\sin \sqrt{\lambda^2\mp 2\lambda t} +C_2(x)\cos \sqrt{\lambda^2\mp 2\lambda t}\big]}{\Gamma(-x)}  e^\lambda (-\lambda)^{-x-1}\, d\lambda\\
			&=\sum_{n=0}^\infty \mathscr{C}^\Delta_{\pm n\mp \frac{1}{2}}(x)\, \frac{t^n}{n!},
				\end{split}
			\end{equation}where the infinite sum converges uniformly in each compact subset of $\mathbb{C}\times\mathbb{C}$.
		\item The holonomic systems of PDEs (delay-differential equations) that generates Glaisher's modules $\mathcal{G}_{-}$ \eqref{E:neg_glaisher_module} and $\mathcal{G}_{+}$ \eqref{E:pos_glaisher_module} with manifestation in $\mathcal{O}_{\Delta d}$  as defined by \eqref{E:O_delta_d} from Example \ref{Eg:O_delta_d} are given, respectively,  by\footnote{See also the Appendix \ref{SS:PDE_list}.}
			\begin{align}
				& tf_{tt}(x,\, t)+\big(x+\frac12\big)f_t(x,\, t)-xf_t(x-1,\, t)-xf(x-1,\, t)=0,\label{E:PDeltaE_neg_glaisher_-1}\\
				&  tf_{t}(x,\, t)-xf_t(x-1,\, t)+xf(x-1,\, t)-\big(x+\frac12\big) f(x,\, t)=0\label{E:PDeltaE_neg_glaisher_-2}
			\end{align}
		and
		\begin{align}
				& tf_{tt}(x,\, t)+\big(x+\frac12\big)f_t(x,\, t)-xf_t(x-1,\, t)+xf(x-1,\, t)=0,\label{E:PDeltaE_pos_glaisher_+1}\\
				&  tf_{t}(x,\, t)+xf_t(x-1,\, t)+xf(x-1,\, t)-\big(x+\frac12\big) f(x,\, t)=0.\label{E:PDeltaE_pos_glaisher_+2}
		\end{align}
	\end{enumerate}
\end{theorem}

\begin{proof}
Recall the maps $\mathfrak{p}$, $\mathfrak{p}_\Delta$ and $\mathfrak{N}$ as defined in \eqref{E:poisson}, \eqref{E:poisson_2} and \eqref{E:newton_trans} respectively. Then we have the diagram
	\begin{equation}\label{E:commute-half besssel-2}
		\begin{tikzcd}  [row sep=large, column sep=large] %[row sep=huge, column sep=huge]
			& \mathcal{G}_{\pm}(\omega_\pm) \arrow[swap]{dl}{\Cos\omega\slash\Sin\omega} \arrow{d}{\times(\mathscr{C}_{\pm n\mp  \frac12})} \arrow{dr}{\times(\mathscr{C}^\Delta_{\pm n\mp \frac12})}
\\
			\widetilde{\mathcal{A}_2}(\omega_\pm) \arrow{dr}{\times x^{-\frac12}} &   \mathcal{O}^{\mathbb{N}_0} \arrow{d}{\mathfrak{p}} & \mathcal{O}^{\mathbb{N}_0} \arrow{d}{\mathfrak{p}_\Delta} \\
 			& \mathcal{O}_{d d} \arrow{r}{\times\mathfrak{N}} &  \mathcal{O}_{\Delta d} 
		\end{tikzcd}
\end{equation}
in which $\widetilde{\mathcal{A}_2}(\omega_\pm):=\overline{\displaystyle\frac{\mathcal{A}_{2}(\omega_\pm)}{\big[\mathcal{A}_{2}(\omega_\pm)(W_1\Xi_1+\frac{1}{2})
			+\mathcal{A}_{2}(\omega_\pm)\Xi_2\big]}}$, the Weyl cosine and sine $\Cos \omega_\pm\slash \Sin \omega_\pm$ are as defined in Example \ref{Eg:trigo},  and the analytic functions in $\mathcal{O}_{\Delta d}$ are suitably restricted.

		Again, it suffices to consider the case for the negative Glaisher module $\mathcal{G}_-$ for the proof of $\mathcal{G}_+$ being similar. Since $\omega_-^2 =W_1^2-2W_1W_2$, the leftmost path in the above commutative diagram asserts that the analytic function from \eqref{E:glaisher_fn}
		\begin{equation}\label{E:glaisher_fn_2}
			x^{-\frac12}
			\big(C_1\cos\sqrt{x^2- 2xt}+C_2\sin \sqrt{x^2- 2xt}
\big)
		\end{equation}
		satisfies the system of PDEs \eqref{E:glaisher_PDE_1} and \eqref{E:glaisher_PDE_2}. A further composition of this analytic function under the Newton transformation \eqref{E:newton_trans} gives
		\begin{equation}\label{E:Glaisher_dintegral_soln}
	f_-(x,\, t)= \int_{-\infty}^{(0+)} 
				\lambda^{-\frac12} \frac{\big[C_1(x)\sin \sqrt{\lambda^2- 2\lambda t}+C_2(x)\cos \sqrt{\lambda^2-2\lambda t}\big]}{\Gamma(-x)}  e^\lambda (-\lambda)^{-x-1}\, d\lambda,
	\end{equation}
	which satisfies the system of delay-differential equations \eqref{E:PDEs_neg_glaisher_1} and \eqref{E:PDEs_neg_glaisher_2} by Theorem~\ref{T:Newton_trans}, i.e.,
			\begin{align}
				& tf_{tt}(x,\, t)+\big(x+\frac12\big)f_t(x,\, t)-xf_t(x-1,\, t)-xf(x-1,\, t)=0,\label{E:PDeltaE_neg_glaisher_1}\\
				&  tf_{t}(x,\, t)-xf_t(x-1,\, t)+xf(x-1,\, t)-\big(x+\frac12\big) f(x,\, t)=0.\label{E:PDeltaE_neg_glaisher_2} 
			\end{align}Without loss of generality, we may assume that the $1$-periodic factor ${e^{-i\pi x}}/{2i\sin(\pi x)}$ is included in the unknown $1$-periodic functions $C_1(x),\, C_2(x)$ in the following computation. The handling of the second case when $\omega_+=W_1^2+2W_1W_2$ when $\nu=\frac12$ is similar and is therefore omitted.

On the other hand, the rightmost path in the above commutative diagram shows that the sum $\sum_{n=0}^\infty \mathscr{C}^\Delta_{n-\frac{1}{2}}(x)\, t^n/{n!}$ is a solution to the system of PDEs \eqref{E:PDeltaE_neg_glaisher_1} and  \eqref{E:PDeltaE_neg_glaisher_2}.  It is clear that  the $\mathcal{G}_\pm(\omega_\pm)$ defined above is holonomic with multiplicity two. Hence \eqref{E:delta_glaisher_gf} holds apart from two $1$-periodic functions in $x$ which can easily be incorporated into the constants $C_1,\, C_2$. 
\medskip

 We now derive a third order differential equation from the equations \eqref{E:PDeltaE_neg_glaisher_1} and \eqref{E:PDeltaE_neg_glaisher_2}
in order to determine the range of convergence of the \eqref{E:delta_glaisher_gf}. Adding and subtracting the equations \eqref{E:PDeltaE_neg_glaisher_1} and \eqref{E:PDeltaE_neg_glaisher_2} yields the equations
	\begin{equation}
		\label{E:PDeltaE_neg_glaisher_3}
			tf_{tt}(x,\, t)+\big(x+t+\frac12\big)f_t(x,\, t)-2x f_t(x-1,\, t)-\big(x+\frac12\big)f(x,\, t)=0
	\end{equation}
	and 
	\begin{equation}
		\label{E:PDeltaE_neg_glaisher_4}
			tf_{tt}(x,\, t)+\big(x-t+\frac12\big)f_t(x,\, t)-2xf(x-1,\, t)+\big(x+\frac12\big)f(x,\, t)=0
	\end{equation}
	respectively. Differentiating \eqref{E:PDeltaE_neg_glaisher_4} with respect to $t$ yields
	\begin{equation}
		\label{E:PDeltaE_neg_glaisher_5}
			tf_{ttt}(x,\, t)+\big(x-t+\frac32\big) f_{tt}(x,\, t)+\big(x-\frac12\big)f_t(x,\, t) -2xf_t(x-1,\, t)=0.
	\end{equation}Subtracting \eqref{E:PDeltaE_neg_glaisher_3} from \eqref{E:PDeltaE_neg_glaisher_5} yields the equation
	\begin{equation}
		\label{E:PDeltaE_neg_glaisher_6}
			tf_{ttt}(x,\, t)+\big(x-2t+\frac32\big) f_{tt}(x,\, t)-(t+1)f_t(x,\, t)+\big(x+\frac12\big)f(x,\, t)=0,
	\end{equation}
	which is a third order linear ordinary differential equation in $t$ for each fixed $x$. Clearly, the equation \eqref{E:PDeltaE_neg_glaisher_6} has only one finite regular singularity at $t=0$, so that any power series solution of it must converge unformly in each compact subset of $\mathbb{C}\times\mathbb{C}$.  Substituting a Frobenius solution 
$f(x,\cdot)=\sum_{k=0}^\infty a_kx^{\sigma+k}$ into the third-order equation \eqref{E:PDeltaE_neg_glaisher_6} further confirming that $\sigma=0,\, 1$ to be the only integer indicial roots. Hence the \eqref{E:delta_glaisher_gf} follows.
	
	The handling of the $\mathcal{G}_\frac12$ is similar to that for $\mathcal{G}_{-\frac12}$ and is therefore omitted.  This completes the proof.
\end{proof}
\medskip

\begin{remark} We offer a direct verification that the \eqref{E:Glaisher_dintegral_soln} solves the delay-differential equations \eqref{E:PDeltaE_neg_glaisher_1} and \eqref{E:PDeltaE_neg_glaisher_2}. Since the computation of the sine and cosine functions are similar, so it suffice to consider only the special case with cosine kernel. Moreover, the $1$-periodic function $C_1(x)$ can be factored out in the computation below since it is independent of $t$ and so it is also omitted. Let
		\[
			y(x, t)=\displaystyle\int_{-\infty}^{(0+)} e^{\lambda}
			\dfrac{\cos(\lambda^2-2\lambda t)^{\frac{1}{2}}}
			{\Gamma(-x)	(-\lambda)^{x+1}\lambda^\frac{1}{2}} d\lambda.
			\]We now show that the $y(x,\, t)$ is a solution to the  delay-differential equation \eqref{E:PDeltaE_neg_glaisher_2}. We split the computation of this  equation into two parts. We substitute this $y$ into the third and the fourth terms from \eqref{E:PDeltaE_neg_glaisher_2}. This yields
			\begin{equation}\label{E:three-terms}
%				\begin{split}
			xy(x, t)-xy(x-1, t)+\frac{1}{2}y(x, t)\\
			=\displaystyle\frac{1}{\Gamma(-x)}
			\int_{-\infty}^{(0+)} (x-\lambda+\frac{1}{2})e^{\lambda}
			(-\lambda)^{-x-1} 
			\cos(\lambda^2-2\lambda t)^{\frac{1}{2}}
			\lambda^{-\frac12}d\lambda.
%			\end{split}
			\end{equation}We now substitute the $y(x,\, t)$ into the first two terms in \eqref{E:PDeltaE_neg_glaisher_2} before performing an integration-by-parts once yield
			\[
  \begin{split}
	&x\frac{dy(x-1, t)}{dt}-t\frac{dy(x, t)}{dt}
			=\displaystyle\frac{1}{\Gamma(-x)}
			\int_{-\infty}^{(0+)} e^{\lambda}
			(-\lambda)^{-x-1} \lambda^{\frac{1}{2}} \,
			d[-\cos(\lambda^2-2\lambda t)^{\frac{1}{2}}]\\
			&=-e^{\lambda}
			\frac{(-\lambda)^{-x-1} }{\Gamma(-x)}\lambda^{\frac{1}{2}} 
			\cos(\lambda^2-2\lambda t)^{\frac{1}{2}}
			\Big|_{-\infty}^{(0+)}
			+\displaystyle\frac{1}{\Gamma(-x)}
			\int_{-\infty}^{(0+)} 
			\cos(\lambda^2-2\lambda t)^{\frac{1}{2}}
			d(e^{\lambda}
			(-\lambda)^{-x-1} \lambda^{\frac{1}{2}} ).
			\end{split}
\]Since $\Re(x)<-\frac{1}{2}$, then the limit of the first term above vanishes. That is, 
\[
-e^{\lambda}
\frac{(-\lambda)^{-x-1} }{\Gamma(-x)}\lambda^{\frac{1}{2}} 
\cos(\lambda^2-2\lambda t)^{\frac{1}{2}} \Big|_{-\infty}^{(0+)}=0.
\]Thus the above expression becomes, after this integration-by-parts, 
	\[
	\begin{split}
	&x\frac{dy(x-1, t)}{dt}-t\frac{dy(x, t)}{dt}\\
	&=\displaystyle\frac{1}{\Gamma(-x)}
	\int_{-\infty}^{(0+)} 
	\cos(\lambda^2-2\lambda t)^{\frac{1}{2}}
	e^{\lambda}(-\lambda)^{-x-1}\lambda^{-\frac{1}{2}}
	(\lambda-x-\frac{1}{2}) d\lambda\\
%	&=-(xy(x, t)-xy(x-1, t)+\frac{1}{2}y(x, t))
	\end{split}
	\]which is precisely the negative of the expression above \eqref{E:three-terms}. This verifies that $y(x,\, t)$ solves the first equation from \eqref{E:PDeltaE_neg_glaisher_2}. The verification of $y(x, t)$ to the  \eqref{E:PDeltaE_neg_glaisher_1} is similar and is therefore omitted.
\end{remark}
\medskip

The following difference analogues for the Glaisher-type generating functions \eqref{E: first half bessel_gf} for difference Bessel functions $J_{n+\frac12}^\Delta(x)$ are amongst the special cases described in the above theorem.
\medskip

	\begin{theorem}\label{T:Delta_glaisher_gf} The series
	\begin{enumerate}
		\item
		\begin{equation}\label{E: gf-Delta_glaisher_cos}
	\sqrt{\frac{2}{\pi}}\dfrac{e^{-i\pi x}} {2i \sin(\pi x)}
	\int_{-\infty}^{(0+)} 
	\frac{e^{\lambda}
	(-\lambda)^{-x-1}\lambda^{-\frac{1}{2}} 
	 \cos\sqrt{\lambda ^2-2\lambda t}}{\Gamma(-x)}\,  d\lambda
	=\sum_{n=0}^{\infty}J^{\Delta}_{n-\frac{1}{2}}(x)\frac{t^n}{n!},
	\end{equation}
			\item and
	\begin{equation}\label{E: gf-Delta_glaisher_sin}
	\sqrt{\frac{2}{\pi}}\dfrac{e^{-i\pi x}} {2i \sin(\pi x)}
	\int_{-\infty}^{(0+)} 
	\frac{e^{\lambda}
	(-\lambda)^{-x-1}\lambda^{-\frac{1}{2}} 
	 \sin\sqrt{\lambda ^2+2\lambda t}}{\Gamma(-x)}\,  d\lambda
	=\sum_{n=0}^{\infty}J^{\Delta}_{-n+\frac{1}{2}}(x)\frac{t^n}{n!}
	\end{equation}
		\end{enumerate}	
converge uniformly in each compact subset of $ \mathbb{C}\times\mathbb{C}$.
	\end{theorem}
\medskip

It may be possible to derive the two generating functions of the above theorem from the Theorem \ref{T:difference_glaisher_gf}, but we choose a more direct method by applying the Newton transform Theorem \ref{T:Newton_trans} to Glaisher\rq{}s formulae directly.
\medskip

\begin{proof}   
Due to the similarity of the two relations \eqref{E: gf-Delta_glaisher_cos} and \eqref{E: gf-Delta_glaisher_sin}, so it suffices to only establish the \eqref{E: gf-Delta_glaisher_cos}. 

Let 
	\[
		L=(X\partial)^2+(X^2-\nu^2),
	\]
and to recall that a $\mathcal{D}-$linear map \eqref{E:neg_nu_map}
	\[
		S_{-\nu}=\sum_{k=0}^{\infty} \frac{(-1)^kX^{2k}}{2^{\nu+2k}k!\Gamma(-\nu+k+1)}
	\]
	in the following diagram
\begin{equation}\label{E:commute-glaisher}
		\begin{tikzcd}  [row sep=large, column sep=large] %[row sep=huge, column sep=huge]
			\displaystyle\frac{\mathbb{C}[[X\partial, X]]}{\mathbb{C}[[X\partial, X]]L} \arrow{r}{\times S_{-\nu}} 
%			\arrow[swap]{d}{\Cos\omega\slash\Sin\omega} 
			&  \displaystyle\frac{\mathbb{C}[[X\partial, X]]}{\mathbb{C}[[X\partial, X]](X\partial+\nu)} 
			\arrow{d}{\times 1} 
			\arrow{dr}{\times 1}\\
 %			 \widetilde{\mathcal{A}_2}(\omega)\arrow{r}{\times x^{-\frac12}} 
 &  \mathcal{O}_{d } \arrow{r}{\times\mathfrak{N}} &  \mathcal{O}_{\Delta} 
		\end{tikzcd}
	\end{equation}where the $\mathfrak{N}$ again denote the Newton transformation. Since
\eqref{E:newton_bases_2}. Let us apply a truncated  Newton transformation $\mathfrak{N}_R$ of \eqref{E:newton_trans} which has the truncated Hankel-type contour $\Gamma_R\ (R>0)$ starts from $-R$ below the negative real-axis, goes around the origin in the counter-clockwise direction before returning to $x=-R$ 
	\[
		\begin{split}
		\mathfrak{N}_R(J_{-\nu}(x)) &= \dfrac{e^{-i\pi x}} {2i \sin(\pi x)\Gamma(-x)}
		 \int_{Ne^{-i\pi}}^{Ne^{i\pi}} \sum_{k=0}^{\infty} \frac{(-1)^k \lambda^{-\nu+2k}}{2^{\nu+2k}k!\Gamma(-\nu+k+1)}e^\lambda(-\lambda)^{-x-1}\, d\lambda
		 \\
		 &=\sum_{k=0}^{\infty}
		 \dfrac{e^{-i\pi x}} {2i \sin(\pi x)\Gamma(-x)}\int_{Ne^{-i\pi}}^{Ne^{i\pi}} 
		  \frac{(-1)^k \lambda^{-\nu+2k}}{2^{\nu+2k}k!\Gamma(-\nu+k+1)}e^\lambda(-\lambda)^{-x-1}\, d\lambda\\
		  &\to \sum_{k=0}^{\infty}
		 \dfrac{e^{-i\pi x}} {2i \sin(\pi x)\Gamma(-x)}\int_{-\infty}^{(0+)} 
		  \frac{(-1)^k \lambda^{-\nu+2k}}{2^{\nu+2k}k!\Gamma(-\nu+k+1)}e^\lambda(-\lambda)^{-x-1}\, d\lambda\\ 
		  &=\sum_{k=0}^{\infty} \frac{(-1)^k}{2^{-\nu+2k}k!\Gamma(-\nu+k+1)}(x)_{-\nu+2k}
		  =J^\Delta_{-\nu}(x),
		 \end{split}
	\]as $R\to\infty$, where the interchange of the integral and summation signs above is guaranteed by the uniform convergence of the classical Bessel function $J_{\nu}(x)$ in any compact subset of $\mathbb{C}^\dagger=\mathbb{C}\backslash\{x:\, x\le 0\}$.

Let $\nu=n-\frac12$. Apply a truncated  Newton transformation \eqref{E:newton_trans} over the truncated Hankel-type contour $\Gamma_R$ as defined above to both sides of \eqref{E: first half bessel_gf}  with respect to $\lambda$. This yields
\begin{equation}\label{E: gf-Delta_glaisher_cos_2}
	\begin{split}
		\sqrt{\frac{2}{\pi}}\dfrac{e^{-i\pi x}} {2i \sin(\pi x)}
		 &\int_{Ne^{-i\pi}}^{Ne^{i\pi}} 
		\frac{e^{\lambda}
		(-\lambda)^{-x-1}\lambda^{-\frac{1}{2}} 
	 	\cos\sqrt{\lambda ^2-2\lambda t}}{\Gamma(-x)}\,  d\lambda\\
		&
=\sum_{n=0}^{\infty} \Big(\dfrac{e^{-i\pi x}} {2i \sin(\pi x)}
		 \int_{Ne^{-i\pi}}^{Ne^{i\pi}} 
		\frac{e^{\lambda}
		(-\lambda)^{-x-1}J_{n-\frac12}(\lambda)}{\Gamma(-x)}\,  d\lambda\Big)\frac{t^n}{n!},\\
%		&=\sum_{n=0}^{\infty}J^{\Delta}_{n-\frac{1}{2}}(x)\frac{t^n}{n!}
	\end{split}
\end{equation}where the interchange of the integral and summation signs above is allowed because of the uniform convergence of the \eqref{E: first half bessel_gf} in each compact subset of $\mathbb{G}:=\{2|t|<|x|:\ (x,\, t)\in \mathbb{C}^\dagger\times\mathbb{C}\}$ where $\mathbb{C}^\dagger:=\mathbb{C}\backslash\{x:x\le 0\}$ as shown there. We finally obtain the formula \eqref{E: gf-Delta_glaisher_cos} after letting $R\to\infty$ in the truncated contour $\Gamma_R$. We may analytic continue the $x$ to the whole of $\mathbb{C}$ by deforming the Hankel contour when necessary. The proof of the \eqref{E: gf-Delta_glaisher_sin} is similar. This completes the proof.
\end{proof}

We give a direct derivation of the  integral representations of $J^\Delta_{\pm\frac12}$ from the \eqref{E: gf-Delta_glaisher_cos} which are of different type from those given in \eqref{E:dBessel integral rep} and \eqref{E:dBessel integral rep-nu}.

\begin{corollary} Let $x\in \mathbb{C}$. Then
	\begin{enumerate}
		\item 
			\[
				J^\Delta_{-\frac12}(x)=\sqrt{\frac{2}{\pi}}
						\frac{e^{-i\pi x}}{2i \sin\pi x}
		\int_{-\infty}^{(0+)} 
		\dfrac{e^{\lambda}}
		{\Gamma(-x)	(-\lambda)^{x+1}\lambda^\frac{1}{2}}\cos \lambda\, d\lambda
			\]
and
		\item 
			\[
				J^\Delta_{\frac12}(x)=\sqrt{\frac{2}{\pi}}
			\frac{e^{-i\pi x}}{2i \sin\pi x}
		\int_{-\infty}^{(0+)} 
		\dfrac{e^{\lambda}}
		{\Gamma(-x)	(-\lambda)^{x+1}\lambda^\frac{1}{2}}\sin \lambda\, d\lambda.
			\]
	\end{enumerate}
\end{corollary}
\medskip

\begin{proof}Expanding the $\cos\lambda$ into series in 
	\[
		\begin{split}
			&\frac{e^{-i\pi x}}{2i \sin\pi x}
		\int_{-\infty}^{(0+)} 
		\dfrac{e^{\lambda}}
		{\Gamma(-x)	(-\lambda)^{x+1}\lambda^\frac{1}{2}}\cos \lambda\, d\lambda\\
		&=\frac{e^{-i\pi x}}{2i \sin\pi x}\int_{-\infty}^{(0+)} 	\dfrac{e^{\lambda}}
		{\Gamma(-x)	(-\lambda)^{x+1}\lambda^\frac{1}{2}}
		\sum_{k=0}^{\infty} \frac{(-1)^k\lambda^{2k}}{(2k)!}\, d\lambda\\
		&=\sum_{k=0}^\infty \frac{(-1)^k}{(2k)!}\, G_k(x)
		\end{split}
	\]where, with applications of  the generalised Gamma function and Euler's reflection formula,
	\[
		\begin{split}
			G_k(x)&:=\frac{e^{-i\pi x}}{2i \sin\pi x}\frac{e^{-i\frac{\pi}{2}}}{\Gamma(-x)}
			\int_{-\infty}^{(0+)}e^{\lambda}(-\lambda)^{-x-\frac{1}{2}+2k-1}d\lambda\\
			&=\frac{e^{-i\pi x}}{2i \sin\pi x}\frac{e^{-i\frac{\pi}{2}}}{\Gamma(-x)}
			\frac{e^{-i\pi(x+\frac12-2k)}}{e^{-i\pi(x+\frac12-2k)}}
			\frac{\sin\pi(-x-\frac12+2k)}{\sin\pi(-x-\frac12+2k)}
			\frac{\Gamma(-x-\frac12+2k)}{\Gamma(-x-\frac12+2k)}
			\int_{-\infty}^{(0+)} e^\lambda(-\lambda)^{-x-\frac12+2k-1}\, d\lambda
			\\
			&=\frac{e^{-i\pi x}}{\sin\pi x}\frac{e^{-i\frac{\pi}{2}}}{\Gamma(-x)}
			\frac{\sin\pi(-x-\frac12+2k)\Gamma(-x-\frac12+2k)}{e^{-i\pi(x+\frac12-2k)}}\,
			\\
			&=\frac{\Gamma(-x-\frac{1}{2}+2k)}{\Gamma(-x)}
			\frac{\sin(x+\frac{1}{2}-2k)\pi}{\sin (-\pi x)}\\
			&=\frac{\Gamma(x+1)}{\Gamma(x+\frac{3}{2}-2k)}.
		\end{split}
	\]%Clearly $G_k(-\frac12)=0$ for each $k\ge 1$, and $G_0(-\frac12)=1/\sqrt{\pi}$. 
	Hence
	\begin{equation}\label{E:glaisher_cos_expansion}
		\frac{e^{-i\pi x}}{2i \sin\pi x}
		\int_{-\infty}^{(0+)} 
		\dfrac{e^{\lambda}}
		{\Gamma(-x)	(-\lambda)^{x+1}\lambda^\frac{1}{2}}\cos \lambda\, d\lambda\
		=\sum_{k=0}^\infty \frac{(-1)^k}{(2k)!}\, \frac{\Gamma(x+1)}{\Gamma(x+\frac{3}{2}-2k)}=\sqrt{\frac{\pi}{2}}\, J^\Delta_{-\frac12}(x)
	\end{equation}
	since, with applications of the well-known identity $\Gamma(k+\frac12)k!\, 2^{2k}=\sqrt{\pi}(2k)!$, see for example \cite[p. 211]{Copson_1935}, that
	\begin{equation}\label{E:difference_bessel_1/2}
		\begin{split}
			J^{\Delta}_{-\frac12}(x)&:=\sum_{k=0}^{\infty} \frac{(-1)^k}{2^{-\frac12+2k}k!\, \Gamma(\frac12+k)}(x)_{-\frac12+2k}
			=\sum_{k=0}^{\infty} 
			\frac{(-1)^k}{2^{-\frac12}\sqrt{\pi}(2k)!}(x)_{-\frac12+2k}\\
			&=\sqrt{\frac{2}{\pi}}\sum_{k=0}^{\infty} 
			\frac{(-1)^k}{(2k)!} \frac{\Gamma(x+1)}{\Gamma(x+\frac{3}{2}-2k)}.
			\end{split}
		\end{equation}

Similarly, we have
		\begin{equation}\label{E:glaisher_sin_expansion}
		\frac{e^{-i\pi x}}{2i \sin\pi x}
		\int_{-\infty}^{(0+)} 
		\dfrac{e^{\lambda}}
		{\Gamma(-x)	(-\lambda)^{x+1}\lambda^\frac{1}{2}}\sin \lambda\, d\lambda\
		=\sum_{k=0}^\infty \frac{(-1)^k}{(2k+1)!}\, \frac{\Gamma(x+1)}{\Gamma(x+\frac{1}{2}-2k)}=\sqrt{\frac{\pi}{2}}\, J^\Delta_{\frac12}(x).
	\end{equation}
\end{proof}

\section{Discussion}\label{S:discussion}

\subsection*{Truesdell\rq{}s $F$-equation theory}
 It had been observed by Lommel \cite{Lommel_1871},  Nielson \cite{Nielson} and Sonine \cite{Sonine},
 all from the nineteenth century, to characterize Bessel functions from the pair of recurrence formulae, which we call PDEs in this article (Proposition \ref{P:PDE_cylinderical_bessel}),
	\begin{equation}\label{E:any_bessel_recus_0}
		\begin{split}
			& x\mathscr{C}_{\nu}^\prime(x)+\nu\mathscr{C}_{\nu}(x)-
		x\mathscr{C}_{\nu-1}(x)=0,\\
			&x\mathscr{C}^{\prime}_{\nu}(x)-\nu\mathscr{C}_{\nu+n}(x)+
		x\mathscr{C}_{\nu+1}(x)=0,
		\end{split}
	\end{equation}
as analytic functions of two complex variables. They called the functions that satisfy both of the equations \textit{cylinder functions}, see also \cite[pp. 82-84]{Watson1944}. In fact, according to Watson \cite{Watson1944}, Sonine refrained from using Bessel's differential equation \eqref{E:Bessel_eqn} in his study. Truesdell published the book \cite{Truesdell_1948} in 1948 perfecting the viewpoint started by the aforementioned nineteenth century scholars as an analytic theory of differential-difference equation that can be written in the forms
	\begin{equation}\label{E:Truesdell}
		\frac{\partial}{\partial x}f(x,\alpha)=A(x,\alpha)f(x,\alpha)+B(x,\alpha)f(x,\alpha+1).
	\end{equation}If the $f(x,\alpha)$, assumes a specific form that after a suitable change of variables, can be written in an equivalent form
	\begin{equation}\label{E:F_eqn}
		\frac{\partial}{\partial x}F(x,\alpha)=F(x, \alpha+1),
	\end{equation}
	then he demonstrated that derived many known classical special function identities\footnote{Thirty five formulae have been listed in \cite{Truesdell_1947}.}
	 as their generating functions and to discover new ones for different $f$ for all sorts of classical special functions. Truesdell called the equation \eqref{E:F_eqn} an $F$-equation. We quote from \cite[p. 7] {Truesdell_1948}
	\begin{quote}
		``The aim of this essay is to provide a general theory which motivates, discovers, and coordinates such seemingly unconnected relations among familiar special functions as the formulas (1) through (35).\rq\rq{}
	\end{quote}

	 Truesdell's aimed to find ``rational methods of discovery" \footnote{Truesdell, p. 7.} of the numerous classical special functions formulae found in the literature that often left beginners, and even experienced researchers, wondering about their origins and the methods of their derivation. For example,  Truesdell commented that if a textbook has the following question:
	 \begin{quote}
	 	``Find a formula which gives Laguerre polynomials in terms of Bessel functions.\rq\rq{}\footnote{Truesdell, p. 7.}
	\end{quote} 
then, it would not be straightforward to any student unless he/she had a priori knowledge of such formulae. 
However, once the student sees such a formula, then it would be a typical textbook exercise for a student to give a rigorous proof. Thus Truesdell would like to ``coordinate", ``to bring order into the part of collection of known relations concerning special functions by showing that they are simple special cases of about a dozen general formulas and by adding to their number some of the missing analogues which do not seem to have been discovered thus far." \footnote{\cite{Truesdell_1948}.} 	 
	 
	 Truesdell further commented that it is not with the particular special function formulae (and some of them were claimed to be new) that was his focus, since 
	\begin{quote}
		``it is of no great task to construct ad hoc rigorous proofs and ... to any of the formulas we have just listed,"
	\end{quote}
but rather a general property, although it had been  observed by earlier researchers in various literature over a period of several decades in the second half of the nineteenth century, that was largely neglected \footnote{\cite[p. 8]{Truesdell_1948}.} is the result of applying his $F$-equation theory that he wanted to bring about. Indeed Truesdell was not alone that Ehrenpreis, without explicit mention, essentially also took the $F$-equation viewpoint in his study of classical analogues  of the Rogers-Ramanujan identities in \cite{Ehrenpreis_1990, Ehrenpreis_1993} in which a generalised system of \eqref{E:PDE_gf_bessel_0} that is related to multi-variable Bessel functions has been proposed. From the $D$-modules viewpoint employed in this article, although Truesdell and his predecessors certainly had the insight to study the classical Bessel functions $J_{\nu}$ via the recursions \eqref{E:any_bessel_recus_0}, the ``general and neglected property" is replaced by the algebraic property that the corresponding modules, such as the Bessel module studied in this paper, are holonomic in the sense of Bernstein-Kashiwara-Sato holonomic modules theory. Moreover, the essence of the $F$-equation theory which is about the existence and uniqueness of the \eqref{E:F_eqn} having analytic solutions is being replaced by the specific algebraic properties of the $D$-modules determined by their generators. In the case of the Bessel module $\mathcal{B}_\nu$ considered in this paper, the generators are given by \eqref{E:bessel_PDE}. What is definitely new are the difference analogues of \eqref{E:any_bessel_recus_0} given in \eqref{P:bilateral_Delta_PDE} and other formulae obtained in this article.

In addition to the well-known works of \cite{Paule_Schorn_1995, Chyzak_Salvy_1998, Chyzak_2000, Zeilberger_1990, Wilf_Zeilberger_1992} of applying $D$-modules to compute for ($q$-)hypergeometric type identities and the standard works on computations with $D$-modules \cite{SST_2000, IMST_2020}, in fact, rudiments of $D$-modules ideas had been employed by Boole in his seminar work \cite{Boole}\footnote{Boole was awarded the first gold medal in mathematics by the Royal Society in 1844 for this paper.}  in which a very general commutation relation on operational method has been applied to solve  a wide variety of linear differential equations including a variant of confluent hypergeometric equation. \footnote{See \cite{CCC_1} for a more detailed discussion.} Of course, in Boole's time, he could not possibly have the knowledge of the modern $D$-modules theory. However, the use of his generalised commutator in solving for series solutions of several linear differential equations is a vindication of his enormous insights into the nature.\footnote{Indeed, the paper together with other earlier works on differential and difference equations of Boole could be considered as preludes to his celebrated work on mathematical logic \cite{Boole_1847}, see \cite[\S 3]{Laita}.} In addition to the difference in notations used and problems tackled,  a major differences between our work and  those of Wilf and Zeilberger and his successors \cite{Zeilberger_1990, Wilf_Zeilberger_1992, Paule_Schorn_1995, Chyzak_Salvy_1998, Chyzak_2000} is that the way that the generators used to define our  $D$-modules  
which is based on an algebraization of Truesdell's $F$-equation theory. For example, the two generators \eqref{E:2_elements} of the Bessel module $\mathcal{B}_\nu$ are abstractions of the consequences of the gauge transformations on the Bessel operators parametrised by integers derived from the the corresponding transmutation formulae in Proposition \ref{P:bessel_transmutation}. The $\partial_1,\partial_2, X_1, X_2$ in the $\mathcal{B}_\nu$ carry different interpretations (manifestations) after appropriate $\mathcal{A}_2$-linear maps in the construction of Bessel's generating functions. Lemma 4.1 in \cite{Zeilberger_1990}, although not needed in this paper, provides a theoretical basis for the existence of characteristic change of variables made in the proof of Theorem \ref{T:bessel_gen_map_2}, Theorem \ref{T:bessel_poly_gf_map}, Theorem \ref{T:symbol gen yn-1}, Theorem \ref{T:neg_glaisher_cosine} and Theorem \ref{T:pos_glaisher_sine} in order to find (well-defined) $\mathcal{A}_2$-linear maps. The well-defined $\mathcal{A}_2$-linear maps found from these Theorems represent the generating functions derived (in $\mathcal{O}_{dd}$ and  $\mathcal{O}_{d\Delta}$) in the main theorems of this papers.

\subsection*{Umbral calculus}
We also note that the Weyl-exponential $\E(X)$, Weyl-sine $\Sin (X)$ and Weyl-cosine $\Cos (X)$  derived as ``solutions" to the holomorphic 
 $D$-modules $\mathcal{A}_1/\mathcal{A}_1(\partial-a)$ and $\mathcal{A}_1/\mathcal{A}_1(\partial^2+1)$  relative to $\mathcal{A}_1/\mathcal{A}_1\partial$ in Example \ref{Eg:exp} and Example \ref{Eg:trigo}. The manifestation of these ``solutions" in $\mathcal{O}_\Delta$ result in 
  the difference-exponential, difference-sine and cosine introduced in Example \ref{Eg:delta_exp} and Example \ref{Eg:dtrigo} respectively, can be found in various literature under ``umbral calculus", such as \cite{Curtright_Zachos_2013, Gessel_2003,LNOS_2008,Roman_1984}. In this spirit, it is interesting to ask if the 
  	\[
		\E[X_1(X_2-1/X_2)/2]
	\] constructed in the Theorem \ref{T:bessel_gen_map_2} for the generating function of Bessel module $\mathcal{B}_0$, the 
  	\[
		\rho^{-1}\E[X_1(1-\rho)],\qquad \rho^2=1-2X_2
	\] constructed in Theorem \ref{T:bessel_poly_gf_map} for the Bessel polynomials module $\Theta$, and the
	\[
		C_1\Cos\omega_\mp+C_2\Sin\omega_\mp,
		\qquad \omega_\mp^2=W_1^2\mp 2W_1W_2	
	\]constructed in Theorem \ref{T:neg_glaisher_cosine} for the Glaisher modules $\mathcal{G}_\mp$ can be interpreted from umbral calculus's viewpoint.

%\medskip

\section{Conclusion}\label{S:conclusion} This paper revisits system of PDEs satisfied by the generating function of the classical Bessel functions of the first kind $J_n(x)$ from an holonomic $D$-modules viewpoint. The approach unifies the generating functions of the classical Bessel functions and the recently found difference Bessel functions $J^\Delta_n(x)$.  One can view this as an algebraization of the $F$-equation theory, which is analytic in nature, proposed by Truesdell in 1948. The key ingredients of our arguments involve new transmutation formulae for the Bessel functions and Bessel polynomials in certain $\mathcal{A}_1$-modules, the derivation of well-defined $\mathcal{A}_2$-linear maps of the $\mathcal{A}_2$-modules, which includes the Bessel module $\mathcal{B}_\nu$, reverse Bessel polynomial module $\Theta$, Bessel polynomial module $\mathcal{Y}$, and the Glaisher modules $\mathcal{G}_\pm$. Although the transmutation formulae can be viewed as abstractions of the relatedf factorization method proposed by Infeld and Hull in \cite{Infeld_Hull_1951}, the scope of the study here and details involved are very different. 

 We have studied ``solutions" in each $\mathcal{A}_2$-module above (see also Appendix \ref{SS:holo_modules}) in the form of well-defined $\mathcal{A}_2$-linear maps that amount to solving systems of linear PDEs defining the holonomic $D$-modules in ``closed-forms". We have thus extended the classical Bessel's generating function to include $J_\nu(x)$ for arbitrary $\nu\in\mathbb{C}$ and its difference analogue by constructing $\mathcal{A}_2$-linear maps from the Bessel module $\mathcal{B}_\nu$ to analytic function spaces of two variables $\mathcal{O}_{dd}$ and $\mathcal{O}_{\Delta d}$ which are also $D$-modules with different ``manifestations" as indicated by their respective subscripts. The difference Bessel functions $J^{\Delta}_\nu(x)$ that appear in the aforementioned difference manifestation of Bessel's generating function coincide with those found by Bohner and Cuthta in 2017. These new generating functions are divergent  (1-Gevrey) series unless $\nu=0$ which are  Borel-resummable. These new Bessel's generating functions for general $\nu\not=0$ of asymptotic type explains the classical Schl\"afli-Sonine  integral representation  \eqref{E:Sonine} of the Bessel function $J_\nu(x)$ that can be obtained via the usual procedure of ``residue extraction" of $J_{\nu+n}$ for any $n$ from the asymptotic formula \eqref{E:gf_J_nu} instead of series manipulation technique usually used, see \cite[\S 6.2]{Watson1944}. In fact, no knowledge of the series expansion of the Bessel function $J_\nu(x)$ is needed with the ``residue extraction" from the \eqref{E:gf_J_nu}.  This viewpoint conforms with the way of studying the sequence of Bessel functions $(J_{\nu+n})$ as \textit{coefficients} of generating functions discussed in \cite[Chapter 2]{Watson1944}. The ``residue extraction" method has been applied to the generating function formula \eqref{gf-db-2} to obtain a \textit{difference analogue of Schl\"afli-Sonine  integral representation} for the difference Bessel functions $J^\Delta_\nu(x)$ in \eqref{E:dBessel integral rep}. 
The results obtained thus far (\S 3) in this article suggests that one can approach the classical Bessel functions as  (Weyl-)algebraic entities without necessarily recourse to Bessel's series expansions as has been done in most modern literature. 

Before we discuss about modifications of the half Bessel module $\mathcal{B}_{1/2}$, we would like to emphasize that although we have devoted this entire paper to Bessel functions, 
the choice of topics studied and results presented in this paper, are by no means exhaustive, only serves to illustrate our objectives and methodology of studying special functions, with the Bessel function as an example, by solving the holonomic system PDEs in the generalised sense represented by the respective $D$-modules. In fact, it is clear to us that the $D$-modules approach could apply to many other topics about Bessel functions found, for example,  in Watson's book \cite{Watson1944}, and to other type of special functions (see below). Thus the choice of topics studied in this paper is not limited by the range of applicability on the methodology, but is limited rather by confining the paper to within a reasonable scope.

We next turn to the modifications of the Bessel module $\mathcal{B}_\nu$ when $\nu=\frac12$, that would lead to a new generating functions to the new difference Bessel polynomials and new difference Glaisher trigonometric type generating functions in addition to their classical counterparts, which can be found from \S\ref{S:half_bessel_I} and \S\ref{S:half_bessel_II} respectively. Each modification of $\mathcal{B}_{1/2}$ requires certain ``quadratic extensions" in addition to their respective ``change of variables". One important feature of $\mathcal{B}_{1/2}$ comes from the fact that the half Bessel operators can be factorised (because it is \textit{solvable} \cite{Morales_Ruiz_1999}).

 In the case about the Weyl-Bessel polynomials studied in \S\ref{S:half_bessel_I}, the focus on $\mathcal{B}_{1/2}$ is that after a suitable ``change of variables" the reverse polynomial operator possesses the transmutation formula
	\[
		X\partial^2-(2+X)\partial+2n=(X\partial-2n)(\partial-1)=(\partial-1)(X\partial-2n-1)
	\]
that can be further derived from the transmutation formula \eqref{C:bessel_poly_gauge_2}. This factorization of the reverse polynomial operator shows the reason that one can find \textit{closed-form} solutions which can always be transformed to a terminated Frobenius  solution or a polynomial solution if needed. This is the case only when $\nu=1/2$ in $\mathcal{B}_\nu$. Indeed, most of the algebraic results derived for the Weyl-Bessel polynomials in \S\ref{SSS:bessel_poly_oper},  to some extent, can be found in \cite{Burchnall_1953, Burchnall_Chaundy_1931, Grosswald} without the benefit of $D$-modules language. However, the different manifestations of the reverse Bessel polynomial operator \S\ref{SSS:bessel_poly_oper} in the analytic function spaces $\mathcal{O}_{d}$ and $\mathcal{O}_{\Delta}$ give rise to the classical reverse Bessel polynomials $\theta_n(x)$ and the new difference reverse Bessel polynomials $\theta^\Delta_n(x)$ and their respectively recursive formulae. The introduction of the reverse Bessel polynomial module $\Theta$ in the same spirit as the Bessel module $\mathcal{B}_\nu$ allows us to solve the corresponding PDEs that leads to generating functions of the reverse classical Bessel polynomials and their difference analogue respectively. The same has been done for the Bessel polynomial module $\mathcal{Y}$ where the new difference Bessel polynomials which are denoted by ($y^\Delta_n(x)$). 

The consideration of Glaisher modules $\mathcal{G}_\pm$ in  \S\ref{S:half_bessel_II} is achieved after the ``change of variables" 
%	\begin{equation}\label{E:glaisher_char_1}
	\[
		\begin{array}{ll}
			\Xi_1=\partial_1, & W_1=X_1\\
	 		\Xi_2={1}/{X_2}, & W_2=X_2\partial_2 X_2.
		\end{array}
	\]
%	\end{equation}
in \eqref{E:glaisher_char_1} to $\mathcal{B}_{1/2}$  can be considered a kind of `` Borel transform" on $\mathcal{B}_{1/2}$, which is instrumental in solving the PDEs in the Glaisher modules $\mathcal{G}_\pm$ leading to the trigonometric type generating functions in Corollary \ref{C:Graisher_gf}, Theorem \ref{T:difference_glaisher_gf} and Theorem  \ref{T:Delta_glaisher_gf}.

To summarise, it is clear from within the limited scope of study made in this article that there is a close relationships amongst the systems of PDEs (in the generalised sense) for which the generating functions of various type of Bessel functions or polynomials satisfy are holonomic, and related integral representations of Bessel functions. This relationship enables us to derive the first integral representation for the recently discovered difference Bessel functions. It also follows from our study that these difference Bessel functions are on ``equal footing" with the classical Bessel functions. Similar relationship also applies to the newly discovered difference Bessel polynomials in this article and their classical counterparts.  This investigation show that the space of differential operators annihilating 
 some generating functions of Bessel functions have rich algebraic structures and that both the classical results and the new ones for the difference Bessel functions and polynomials can be derived from the same Bessel module $\mathcal{B}_\nu$, Bessel polynomial modules $\Theta, \mathcal{Y}$ and the Glaisher modules $\mathcal{G}_\pm$. 

\section{Outlook}
The findings from this article have brought up more questions than the problems solved. For example, it is very likely that our method could also derive the more sophisicated generating function
	\[
		(x+h)^{-\nu/2}J_\nu (2\sqrt{x+h})=\sum_{k=0}^\infty \frac{(-h)^k}{k!}
		x^{-(\nu+k)/2}J_{\nu+k}(2\sqrt{x}),\quad |h|<|x|,
	\]
found by Lommel \cite[\S5.22]{Watson1944}\footnote{Bessel actually discovered the formula when $\nu=0$ in 1826.}, its difference analogue and the likes. What are mathematical basis behind those systems of PDEs that allows for such generating functions? Since generating functions can be viewed as a natural way to  possess certain combinatorial structures that re-organise the various symbols $\partial_1,\partial_2, X_1, X_2$, and the generating functions are solutions to the specific systems of PDEs studied in this article, so a question is how these combinatorial structures are  described and understood against the algebraic structures hidden in these systems of holonomic PDEs. Another question is about the ``orthogonality" of the new difference Bessel polynomials found in this paper. The orthogonality issue of the new difference Bessel polynomials will be dealt with in a separate article \cite{CCL_2023_2}. It turns out that the classical orthogonal polynomial theory \cite{Ismail_2005} does not seem to apply to  these new  polynomials immediately which are naturally organised as Newton polynomials. Thus a new approach that treats ``orthogonality" as residue pairings has been developed in the forthcoming article \cite{CCL_2023_2}. An advantage of such an algebraic approach is that it could treat the orthogonality for both the classical Bessel polynomials, i.e., the classical theory, and the difference Bessel polynomials simultaneously. Finally, in spite of the lack of understanding of  deeper algebraic structures encoded in these systems of holonomic PDEs ($D$-modules) that lead to the Bessel's generating functions studied in this article, similar strategy has been applied to other types of $D$-modules of special functions in another forthcoming article \cite{CCC_2}.

%\vfill\eject
\appendix
%\chapter{List of tables}
\section{Proofs of some theorems/propositions}\label{A:proofs}
\subsection{Proof of Proposition \ref{P:bessel_convergence}}
\label{SS:bessel_convergence}

\begin{proof}
	We rewrite the series of $J^{\Delta}_{\nu}(x)$ \eqref{E:difference_bessel_fn} as
	\begin{equation}\label{E:difference_bessel_fn_2_a}
		J^{\Delta}_{\nu}(x)=\sum_{k=0}^{\infty}
		\dfrac{(-1)^k \Gamma(x+1)}
		{k! 2^{\nu+2k}\Gamma(\nu+k+1)\Gamma(x+1-\nu-2k)}.
	\end{equation}
	An asymptotic formula of Gamma function \cite[Theorem 1.4.2]{AAR}, \cite{WW}
	for a  complex number $x$ not equal to zero or a negative real number, can be derived from the formula
	\[
	\log \Gamma(x)=\frac{1}{2} \log 2\pi +(x-\frac{1}{2})\log x-x+\sum_{j=1}^{m}
	\frac{B_{2j}(0)}{(2j-1)2j}\frac{1}{x^{2j-1}}-\frac{1}{2m}\int_{0}^{\infty}
	\frac{B_{2m}(t-[t])}{(x+t)^{2m}}dt,
	\]where the $B_{2m}(t)$ is the Bernoulli polynomial of degree $2m$. 
	The value of $\log \Gamma(x)$ is the branch with $\log \Gamma(x)$ being real when $x$ is real and positive.
	
	Set 
	\[
	g(x)=\sum_{j=1}^{m}
	\frac{B_{2j}(0)}{(2j-1)2j}\frac{1}{x^{2j-1}}-\frac{1}{2m}\int_{0}^{\infty}
	\frac{B_{2m}(t-[t])}{(x+t)^{2m}}dt.
	\]
	
	Then
	\[
	\Gamma(x)=\sqrt{2\pi}x^{x-1/2} e^{-x+g(x)}.
	\]
We rewrite those Gamma functions that appear in the denominator of $	J^{\Delta}_{\nu}(x)$ as
	\[
	\begin{split}
	&\Gamma(k+1)=\sqrt{2\pi}\,k^{k+1/2}(1/k+1)^{k+1/2} e^{-k-1+g(k+1)},\\
	&\Gamma(\nu+k+1)=\sqrt{2\pi}\,k^{\nu+k+1/2}(\nu/k+1/k+1)^{\nu+k+1/2} e^{-\nu-k-1+g(\nu+k+1)},\\
	&\Gamma(x-\nu-2k+1)=\sqrt{2\pi}
	\dfrac{	(x/2k-\nu/2k+1/2k-1)^{x-\nu-2k+1/2} e^{g(x-\nu-2k+1)}}
	{(2k)^{-x+\nu+2k-1/2}e^{x-\nu-2k+1}}.
 \end{split}
 	\]
	Set
	\[
	u_k=\dfrac{2^{-\nu-2k}}{k!\Gamma(\nu+k+1)\Gamma(x-\nu-2k+1)}.
	\]
Since $g(x)\to 0$ as $x\to\infty$ with $\Re(x)>0$, so that it is routine to verify that
	\[
		u_k\approx \frac{ (-1)^{2k+\nu-\frac12-x}2^{-\nu-2k}}
		{\sqrt{2}k^{\frac32} (2k)^x2^{-\nu-2k}}
		=\frac{(\textrm{Const.}) (-1)^x}{k^{x+\frac32}}
		=\frac{(\textrm{Const.}) e^{\pi \Im x}}{k^{x+\frac32}}
				\]
	holds as $k\to\infty$. Hence 
		\[
			\sum_{k=1}^\infty u_k
		\]converges uniformly in each compact subset of $\mathbb{C}$ where $\Re (x)>-\frac12$. Therefore, the same holds for the \eqref{E:difference_bessel_fn_2_a}.
%	 which is also valid for replacing $k$ by $\nu+k, x-\nu-2k$. Hence for any  $\varepsilon>0$, there exists enough large $k_0$, such that
%	\textcolor{blue}{
%	\[
%	|u_{k_0}|<\left|\dfrac{(1+\varepsilon)e^{x+3}}
%	{(2\pi)^{3/2} k_0^{2/3+x}2^{x+1/2}}\right|, \quad k>k_0.
%	??\]}
%	Then,  we obtain
%	\textcolor{blue}{
%	\[
%	\sum_{k_0}^{\infty} \left|
%	(-1)^k \Gamma(x+1) u_k\right|
%%	<\sum_{k_0}^{\infty} \left|
%	\frac{(1+\varepsilon)e^{x+3}\Gamma(x+1)}
%	{(2\pi)^{3/2} 2^{x+1/2}k_0^{3/2+x}}\right|.
%	??\]}
%	which implies that $J^{\Delta}_{\nu}(x)$ convergesa for \rk{$\Re(x)>0$}.
\end{proof}
\medskip

\subsection{Proof of Proposition \ref{P:Weyl_Bessel_Poly_Eqn}}\label{A:Weyl_Bessel_Poly_Eqn}
\begin{proof}
Multiplying the expression
	\[
		\partial_1-1+\dfrac{X_1}{\partial_2} 
	\]
 by $\partial_1\partial_2$	 yields
	\[
		\partial_2\partial_1^2 -\partial_1\partial_2+\partial_1X_1.
	\]
Left multiplying the expression $X_1\partial_1-2X_2\partial_2-1-X_1+\partial_2$ by $\partial_1$ and taking into account of the last two relations yield
	\[
		\begin{split}
			\partial_1 X_1 \partial_1-2X_2\partial_1\partial_2-\partial_1-\partial_1X_1+\partial_1\partial_2
			&= X_1\partial_1^2-2X_2\partial_1\partial_2+\partial_2\partial_1^2\\
			&= X_1\partial_1^2-2X_2\partial_1\partial_2 +\partial_1\partial_2-\partial_1X_1\\
			&=X_1\partial_1^2-(X_1+2X_2\partial_2)\partial_1+\partial_1\partial_2-1\\
			&=X_1\partial_1^2-2(X_1+X_2\partial_2)\partial_1+X_1\partial_1+\partial_1\partial_2-1\\
			&=X_1\partial_1^2-2(X_1+X_2\partial_2)\partial_1+X_1\partial_1+\partial_2-X_1-1\\
						&=X_1\partial_1^2-2(X_1+X_2\partial_2)\partial_1+
						1+2X_2\partial_2-1\\
		\end{split}
	\]as desired.	
	
\end{proof}		
%We slightly modify Hankel's contour into the following form:

%\begin{theorem}[(Hankel]\cite{?}, \S12.22)\label{T:hankel}Let $t\not\in\mathbb{Z}$. Then
%	\begin{equation}
%		\label{E:hankel}
%			\Gamma (t)=\frac{1}{2i\sin \pi t} \int_{-\infty}^{(0-)} \lambda^{t-1}e^{\lambda}\, d\lambda,
%	\end{equation}
%	where the integration path starts at $-\infty$ below the real axis, encircles the origin in the negative direction and returns to infinity above the real axis.
%\end{theorem}
\medskip

%Watson attributed the following asymptotic estimates  to Horn (1899) and Nielsen (1899). 
%	\begin{proposition}Watson\cite[p. 44 and p. 225]{Watson1944}\label{P:large-order}
%		Let $\nu\in\mathbb{C}\backslash\{-\mathbb{N}\}$. Then for each fixed $x$, we have
%			\begin{equation}\label{E:large-order-1}
%				J_\nu (x)\sim \frac{1}{\sqrt{2\pi}}\exp\big\{\nu+\nu\log x/2-(\nu+1/2)\log\nu\big\}
%				\cdot \big[1+\frac{c_1}{x}+\frac{c_2}{x^2}+\cdots\big],
%			\end{equation}and
%			\begin{equation}\label{E:large-order-2}
%				J_\nu(x)=\frac{(x/2)^\nu}{\Gamma(\nu+1)} (1+R),\qquad
%				|R|<\exp\big(\frac{|x|^2/4}{|\nu_0+1|}\big)-1,
%			\end{equation}
%			where $|\nu_0+1|=\min_{k\ge 1}\{|\nu+k|\}$.
%	\end{proposition}
%\medskip

\subsection{Proof of Theorem \ref{T:symbol gen yn-1}}\label{A:BP_chara}

\begin{proof}
	It is easy to verify that
	\[
	\partial_{1}\eta^{2}-\eta^{2}\partial_1=-2X_2,\qquad 
	\partial_{2}\eta^{2}-\eta^{2}\partial_2=-2X_1
	\]hold. 	Suppose  that
	\[ 
	\delta_1=\partial_1\eta-\eta\partial_1,
	\qquad
	\delta_2=\partial_2\eta-\eta\partial_2,
	\]
	and $\eta$ commutes with $\delta_i, i=1, 2$, i.e.,
	\[
	\delta_1\eta=\eta	\delta_1, \qquad	\delta_2\eta=\eta	\delta_2.
	\]
	Then we obtain 
	\[
	\partial_1 \eta-\eta\partial_1=-\frac{X_2}{\eta},\qquad
	\partial_2\eta-\eta\partial_{2}=-\frac{X_1}{\eta}.
	\]
	Let
	\[
	\Theta_1=\partial_1-\dfrac{1-X_1X_2}{X_1^2}\partial_2,\qquad Y_1=X_1,\qquad
	\Theta_2=\eta\partial_2,\qquad Y_2=-\dfrac{\eta}{X_1}.
	\]
\medskip

	We deduce from direct computation that
	\[
	[\Theta_1, Y_1]=1,\quad [\Theta_2,Y_2]=1, \quad [Y_1, Y_2]=0,
	\quad [\Theta_1, Y_2]=0,\quad [\Theta_2, Y_1]=0, \quad [\Theta_1,\Theta_2]=0.
	\]
	Then $\Theta_i, Y_i, i=1, 2$ generate the following module
	\[
	\widehat{\mathcal{A}_{2}}=\mathbb{C}\langle \Theta_1, \,\Theta_2, \,Y_1, \,Y_2\rangle.
	\] The two generators  of $\mathcal{Y}(\eta)$ in Theorem \ref{T: iso bp-n-1} become, when written in terms of $\Theta_i, Y_i, i=1, 2$,
	\[
	(\Theta_2-1)(\Theta_2+1), \qquad  \Theta_1+ \dfrac{1}{Y^2_1}.
	\]
	Thus we define a modified $\widehat{\mathcal{A}}_2$-module
	\[
	\widehat{\mathcal{Y}}=\dfrac{\widehat{\mathcal{A}_{2}}}
	{\widehat{\mathcal{A}_{2}}((\Theta_2-1)(\Theta_2+1))
		+\widehat{\mathcal{A}_{2}}(\Theta_1+ \dfrac{1}{Y^2_1})}.
	\]
	It follows from Example \ref{E:prime_integrable_eg_2} that 
	$\widehat{\mathcal{Y}}$ is a holonomic module with dimension and multiplicity both equal to two.
	We apply  Example \ref{Eg:exp} to derive that the map
%	\rk{
	\[
	\widehat{\mathcal{A}}_{2}/\widehat{\mathcal{A}}_{2}(\Theta^2_2-1)
	\xrightarrow[]{\times [\E(\pm Y_2)f(Y_1)]}
	\overline{\widehat{\mathcal{A}}_{2}/\widehat{\mathcal{A}}_{2}\Theta_2}
	\]
	%}
	is well-defined left $\widehat{\mathcal{A}}_{2}$-linear if the unknown relations $f(Y_1), g(Y_1)$ depending only on $Y_1$ can be worked out. So we require
	\[
	\begin{split}
	(\Theta_1+ \dfrac{1}{Y^2_1})\E(Y_2)f(Y_1)
	=&\E(Y_2)\Theta_1f(Y_1)+\dfrac{1}{Y^2_1}\E(Y_2)f(Y_1)\\
	=&\E(Y_2)(\Theta_1+\dfrac{1}{Y^2_1})f(Y_1)\mod \widehat{\mathcal{A}_{2}}\Theta_1\\
	=&0.
	\end{split}
	\]
	This would hold if $f(X_1)=\E({1}/{Y_1})$. Similarly, we can obtain $g(X_1)=\E({1}/{Y_1})$. 
	
	Hence, for each choice of constants $C_1, C_2$, we have
%	\rk{
		\[
	\widehat{\mathcal{Y}}
	\xrightarrow[]{\times [C_1\E(Y_2+\frac{1}{Y_1})+C_2\E(-Y_2+\frac{1}{Y_1})]}
	\overline{\widehat{\mathcal{A}}_{2}/(\widehat{\mathcal{A}}_{2}\Theta_1+\widehat{\mathcal{A}}_{2}\Theta_2)},
	\]%}
	which gives rise to the $\mathcal{A}_2(\eta)$-linear map
	\[
	\mathcal{Y}(\eta)
	\xrightarrow[]{\times [C_1\E(\frac{1-\eta}{X_1})+C_2\E(\frac{1+\eta}{X_1})]}
	\overline{\mathcal{A}_{2}(\eta)/(\mathcal{A}_{2}(\eta)\partial_1+
		\mathcal{A}_{2}(\eta)\partial_2)}.
	\]
\end{proof}
\medskip
 
\subsection{Completion of proof of Theorem \ref{T:Bessel p_gf_1}}\label{A:BP_gf_1}

\begin{proof} We continue the proof that is left unfinished below the Theorem  \ref{T:Bessel p_gf_1}.
	\begin{enumerate}
		\item The $\mathcal{O}^{\mathbb{N}_0}$ being an $\mathcal{A}_2(\eta)$-module is essentially verified from Example \ref{Eg:seq-functions_poisson}. The  composition of the horizontal map  and the Poisson  transform $\mathfrak{p}$ of the top-right path from the commutative diagram \eqref{E:commute-diff bess pol} 
		\[
		\begin{array}{rlcll}
		\mathcal{Y}(\eta) & \xrightarrow[]{\times (y_{n-1})}
		& \mathcal{O}^{\mathbb{N}_0}
		&\xrightarrow[]{\times 1} & \mathcal{O}_{dd},\\
		&&&\\
%		1&\longmapsto & (y_{n-1})_n & \longmapsto
%		&  \displaystyle\sum_{n=0}^\infty y_{n-1}(x)\, \dfrac{t^n}{n!}\\
%		&&&
		\end{array}
		\]is a left $\mathcal{A}_{2}(\eta)$-linear map. Hence the composition map yields
			\[
				1 \longmapsto
				 \displaystyle\sum_{n=0}^\infty y_{n-1}(x)\, \dfrac{t^n}{n!},
			\]which implies that  the sum $\sum_n y_{n-1}(x)\, {t^n}/{n!}$ is a solution to the following system of PDEs \eqref{pde-1} and \eqref{pde-2} below.
		\item The first map from the bottom-left path of the commutative diagram \eqref{E:commute-diff bess pol}  has already been constructed in  Theorem \ref{T:symbol gen yn-1}. On the other hand, the following maps
		\[
		\begin{array}{rlcll}
		\mathfrak{g}_p: \mathcal{Y}(\eta) &\xrightarrow{ C_1\E(\frac{1-\eta}{X_1})+C_2\E(\frac{1+\eta}{X_1})} & \overline{\mathcal{A}_{2}(\eta)/\big[\mathcal{A}_{2}(\eta)\partial_1
			+\mathcal{A}_{2}(\eta)\partial_2\big] }
		&\xrightarrow{ \times 1}	\mathcal{O}_{d d}
		\end{array}
		\]
		and \eqref{E:commute-diff bess pol}  show that the 
		\begin{equation}\label{E:gen_bessel_poly_gf_1}
			C_1\exp\left(\frac{1-(1-2xt)^{1/2}}{x}\right)+C_2
			\exp\left(\frac{1+(1-2xt)^{1/2}}{x}\right)
			=\sum_{k=0}^\infty \frac{y_k(x)}{k!}t^k
		\end{equation}holds for some $C_1,\, C_2\in\mathbb{C}$ and that both sides satisfy the system of PDEs 
		\begin{equation}\label{pde-1}
		(1-2xt)f_{tt}(x, \,t)-f(x, \,t)-xf_t(x, \, t)=0,
		\end{equation}
		\begin{equation}\label{pde-2}
		x^2f_x(x,\, t)-f_{t}(x,\,t)+xtf_t(x, \, t)+f(x,\, t)=0.
		\end{equation}
		Putting $t=0$ in both sides of \eqref{E:gen_bessel_poly_gf_1} implies that $C_1=1$ and $C_2=0$.
%	The fact that 
%		the  $\mathcal{A}_2(\eta)$-module $\widehat{\mathcal{Y}}$ as defined in the proof of Theorem \ref{T:symbol gen yn-1}  has multiplicity two implies that the dimension of the local solution space of the PDEs \eqref{pde-1} and \eqref{pde-2} is also two. However, the Theorem \ref{T:symbol gen yn-1} verifies that there is only one solution map found. Thus it follows that the local solution space  of the PDE system \eqref{pde-1} and  \eqref{pde-2} is uni-dimensional. 
%		Since 	$\exp\left(\frac{1-(1-2xt)^{1/2}}{x}\right)$ is also a solution for \eqref{pde-1} and  \eqref{pde-2}. 
%		It follows that the formula \eqref{E:bessel p_gf-1} holds up to a non-zero constant multiple. This constant must be unity by substituting $t=0$ into both sides of \eqref{E:bessel p_gf-1}. On the other hand, both $\sum_n y_{n-1}(x)\, {t^n}/{n!}$ and 
%		$\exp\left(\frac{1-(1-2xt)^{1/2}}{x}\right)$ satisfy the second order ODE
%		\[
%		(1-2xt)_{tt}f(x, t)-xf_t(x, t)-f(x, t)=0
%%		\]
%		for which $t=\frac{1}{2x}$ is a regular singularity. Hence the infinite sum in \eqref{E:bessel p_gf-1}  converges uniformly in each compact subset of $\mathbb{C}\times \mathbb{C}$ wherever $|t|<|\frac{1}{2x}|$.
	\end{enumerate}
\end{proof}

We recall the well-known theorem that describes which half-plane does a Newton series and its associated series, called \textit{faculty series} by N\"orlund,  would converge. 

\subsection{Proof of Theorem \ref{T:delta Bessel_gf_1} }\label{A:detal_BP_gf_1}
Here we give full details of the proof to Theorem  \ref{T:delta Bessel_gf_1} in addition to the commutative diagram \eqref{E:commute-dbp}.

\begin{proof}
	\begin{enumerate}
		\item The composition of the following maps
			\[
				\mathcal{Y}(\eta)  \xrightarrow[]{\times (y^\Delta_{n-1})}
		 \mathcal{O}_\Delta^{\mathbb{N}_0}
		\xrightarrow[]{\times 1}  \mathcal{O}_{\Delta d}
			\]
		is left $\mathcal{A}(\eta)$-linear. Hence the infinite sum from
			\[
				1 \longmapsto
		  \displaystyle\sum_{n=0}^\infty y_{n-1}^\Delta (x)\, \frac{t^n}{n!}
		  	\]
			is a solution to the following delay-differential equations
		\begin{align}
			&x(x-1)[f(x-1, t)-f(x-2, t)]-f_t(x, t)+xtf_t(x-1, t)+f(x, t)=0,\label{diff-Bessel yn-1 eq1}\\
		&f_{tt}(x, t)-2xtf_{tt}(x-1, t)-f(x, t)-xf_t(x-1, t)=0.\label{diff-Bessel yn-1 eq2}
		\end{align}
		\item To deal with the left-side of the equation \eqref{E:bessel dp_gf-1}, we make use of the commutative diagram \eqref{E:commute-diff bess pol} from the proof of Theorem \ref{T:Bessel p_gf_1} from the last subsection.  Thus to complete the left-half the commutative diagram \eqref{E:commute-dbp}, the final step of the composition map
		\[
			\mathcal{Y}(\eta) \xrightarrow[]{\times \E[\frac{1-\eta}{X_1}]} {\tilde{\mathcal{A}}_2(\eta)}
			\xrightarrow[]{\times 1}\mathcal{O}_{dd}
			\xrightarrow[]{\times \mathfrak{N}} \mathcal{O}_{\Delta d}
		\]that is given by the Newton transformation \eqref{E:newton_trans} yields
		\[
		\mathfrak{N}\circ \E(\frac{1-\eta}{X_1})\cdot 1=
		\displaystyle\frac{e^{-i\pi x}}{2i \sin \pi x\Gamma(-x)}\int_{-\infty}^{(0+)}
		e^{\lambda+\frac{1-\sqrt{1-2\lambda t}}{\lambda} }
		(-\lambda)^{-x-1}d\lambda
		\]as asserted.
	
	Although the $\mathcal{A}_2(\eta)$-linearity  consideration of the above composition map is sufficient to complete the proof, we offer a direct verification that the above $	\mathfrak{N}\circ \E(\frac{1-\eta}{X_1})\cdot 1$ is indeed a solution to the system of PDEs 	\eqref{diff-Bessel yn-1 eq1} and \eqref{diff-Bessel yn-1 eq2} here.
		Let
		\[
		f(x, t)=\displaystyle\frac{e^{-i\pi x}}{2i \sin \pi x\Gamma(-x)}\int_{-\infty}^{(0+)}
		e^{\lambda+\frac{1-\sqrt{1-2\lambda t}}{\lambda} }.
		(-\lambda)^{-x-1}d\lambda.
		\]
		Then we verify 
		\[
		\begin{split}
		&(X_{1}^{2}\partial_1-\partial_2+X_1X_2\partial_2+1) f(x, t)=0,\\
		&(1-2X_1X_2)\partial_2^2-X_1\partial_2-1) f(x, t)=0.
		\end{split}
		\]
		We first compute 
		\[
		\begin{split}
		&\partial_{2}^2\displaystyle\frac{e^{-i\pi x}}{2i \sin \pi x\Gamma(-x)}\int_{-\infty}^{(0+)}
		e^{\lambda+\frac{1-\sqrt{1-2\lambda t}}{\lambda} }
		(-\lambda)^{-x-1}d\lambda\\
		&=\displaystyle\frac{e^{-i\pi x}}{2i \sin \pi x\Gamma(-x)}\int_{-\infty}^{(0+)}
		e^{\lambda+\frac{1-\sqrt{1-2\lambda t}}{\lambda} }
		(-\lambda)^{-x-1} 
		\big[(1-2\lambda t)^{-1}+
		\lambda(1-2\lambda t)^{-3/2}\big]
		d\lambda,
		\end{split}
		\]
		and
		\[
		\begin{split}
		&(1-2X_1X_2)\partial_{2}^2
		\displaystyle\frac{e^{-i\pi x}}{2i \sin \pi x\Gamma(-x)}\int_{-\infty}^{(0+)}
		e^{\lambda+\frac{1-\sqrt{1-2\lambda t}}{\lambda} }
		(-\lambda)^{-x-1}d\lambda\\
		&=\displaystyle\frac{e^{-i\pi x}}{2i \sin \pi x\Gamma(-x)}\int_{-\infty}^{(0+)}
		e^{\lambda+\frac{1-\sqrt{1-2\lambda t}}{\lambda} }
		(-\lambda)^{-x-1} 
		\big[1+\lambda(1-2\lambda t)^{-1/2}\big]
		d\lambda.
		\end{split}
		\]
		Combine the above computations yields
		\[
		((1-2X_1X_2)\partial_2^2-X_1\partial_2-1)
		\displaystyle\frac{e^{-i\pi x}}{2i \sin \pi x\Gamma(-x)}\int_{-\infty}^{(0+)}
		e^{\lambda+\frac{1-\sqrt{1-2\lambda t}}{\lambda} }
		(-\lambda)^{-x-1} d\lambda=0.
		\]
		Similarly, we can compute
		\[
		\begin{split}
		&X_1^2\partial_{1}\displaystyle\frac{1}{\Gamma(-x)}\int_{-\infty}^{0}
		e^{\lambda+\frac{1-\sqrt{1-2\lambda t}}{\lambda} }
		(-\lambda)^{-x-1}  d\lambda\\
		&=\displaystyle\frac{1}{\Gamma(-x)}\int_{-\infty}^{0}
		e^{\lambda+\frac{1-\sqrt{1-2\lambda t}}{\lambda} }
		(-\lambda)^{-x+1}  
		\big[\dfrac{x-1}{\lambda}-1\big]d\lambda\\
		&=-\displaystyle\frac{1}{\Gamma(-x)}\int_{-\infty}^{0}
		e^{\frac{1-\sqrt{1-2\lambda t}}{\lambda} }
		d(e^{\lambda} (-\lambda)^{-x+1}  ).
		\end{split}
		\]
		Integration-by-part yields
		\[
		\begin{split}
		&(X_1^2\partial_{1}+1)\displaystyle\frac{1}{\Gamma(-x)}\int_{-\infty}^{0}
		e^{\lambda+\frac{1-\sqrt{1-2\lambda t}}{\lambda} }
		(-\lambda)^{-x-1}  d\lambda\\
		&=(1-X_1X_2)\partial_2
		\displaystyle\frac{1}{\Gamma(-x)}\int_{-\infty}^{0}
		e^{\lambda+\frac{1-\sqrt{1-2\lambda t}}{\lambda} }
		(-\lambda)^{-x-1}  d\lambda,
		\end{split}
		\]
		which implies
		\[
		(X_1^2\partial_{1}-(1-X_1X_2)\partial_2+1)
		\displaystyle\frac{1}{\Gamma(-x)}\int_{-\infty}^{0}
		e^{\lambda+\frac{1-\sqrt{1-2\lambda t}}{\lambda} }
		(-\lambda)^{-x-1}  d\lambda=0.
		\]
		Hence $f(x, t)$ satisfies the system of differential-difference equations
		\eqref{diff-Bessel yn-1 eq1} and 
		\eqref{diff-Bessel yn-1 eq2}.
		Similar to the argument used in the proof of Theorem \ref{T:Bessel p_gf_1}  that 
		the solution space  of the PDEs \eqref{diff-Bessel yn-1 eq1} and  \eqref{diff-Bessel yn-1 eq2} is two-dimensional. Hence \eqref{E:gen_bessel_poly_gf} from Theorem \ref{T:general_bessel_poly_gf} holds for some constants $C_1, \, C_2$ with $\mathscr{Y}_{n-1}$ replaced by $y_{n-1}$.
	Substitute $t=0$ into both sides of \eqref{E:bessel dp_gf-1} yields $C_1=1,\, C_2=0$.
		Then we determine the region of convergence of \eqref{E:bessel dp_gf-1}.  
		The sum of the two generators yields
		\[
		X_1\partial_1\partial_2+X_1\partial_1-X_2\partial^2_2+X_2\partial_2
		\]
	which gives rise to the equation 
	\[
	x[f_t(x,\, t)-f_t(x-1, \,t)]+x[f(x, \, t)-f(x-1,\, t)]-tf_{tt}(x,\, t)+tf_t(x,\, t)=0.
	\]
		We differentiate above equation with respect to $t$ to obtain
		\begin{equation}\label{diff-Bessel yn-1 eq3}
			x[f_{tt}(x,\, t)-f_{tt}(x-1, \,t)]+x[f_t(x, \, t)-f_t(x-1,\, t)]-tf_{ttt}(x,\, t)+(t-1)f_{tt}(x,\, t)+f_t(x,\,t)=0.
		\end{equation}
		
		Since the manifestation of equation \eqref{bessel_poly_recursion_3 in a A becomes} 
		in $\mathcal{O}_{\Delta d}$ is given by
		
		\begin{equation}\label{diff-Bessel yn-1 eq4}
		f_{tt}(x,\,t)-2xtf_{tt}(x-1,\, t)-xf_t(x-1,\, t)-f(x, \,t)=0.
		\end{equation}
	
		Multiplying the equation \eqref{diff-Bessel yn-1 eq3} by $2t$ and subtract the resulting equation from \eqref{diff-Bessel yn-1 eq4} yield the first equation in \eqref{diff-Bessel yn-1 eq5}. Subtracting the \eqref{diff-Bessel yn-1 eq3} and \eqref{diff-Bessel yn-1 eq4} yields the second equation in \eqref{diff-Bessel yn-1 eq5}:
				\begin{equation}\label{diff-Bessel yn-1 eq5} 
		\begin{split}
	&xf_t(x-1,\, t)=\dfrac{1}{1-2t}[(1-2tx+2t-2t^2)f_{tt}(x, t)+2t^2f_{ttt}-2t(x+1)f_t(x, t)-f(x, t)],\\
	&xf_{tt}(x-1,\, t)=\dfrac{1}{1-2t}[(x-2+t)f_{tt}(x, t)-tf_{ttt}(x, t)+(x+1)f_t(x, t)+f(x, t)].
		\end{split}
		\end{equation}
	
		We  replace $x$ by $x-1$ in the equation \eqref{diff-Bessel yn-1 eq4} and differentiate the resulting equation  with respect to $t$ yield
	\begin{equation}\label{diff-Bessel yn-1 eq100}
			f_{ttt}(x-1,\,t)-2(x-1)tf_{ttt}(x-2,\, t)-3(x-1)f_{tt}(x-2,\, t)-f_t(x-1, \,t)=0.
	\end{equation}
	We multiply $x$ to the equation \eqref{diff-Bessel yn-1 eq100} and subtract the resulting equation to the equation  \eqref{diff-Bessel yn-1 eq1} after it being differentiated three times and multiply by the $2t$ throughout  the whole equation would eliminate the term $f_{ttt}(x-2,\, t)$ from the equation  \eqref{diff-Bessel yn-1 eq100}. Multiply the equation \eqref{diff-Bessel yn-1 eq100} by $x$ and subtract 3 times the \eqref{diff-Bessel yn-1 eq1} after it being differentiated twice would eliminate the term  $f_{tt}(x-2,\, t)$ from the equation \eqref{diff-Bessel yn-1 eq100}. This results in the following fourth order equation
		\begin{equation}\label{diff-Bessel yn-1 eq6}
		\begin{split}
		&(x-2tx^2-7xt)f_{ttt}(x-1, \, t)-2xt^2f_{tttt}(x-1,\, t)+(-3x^2-3x)f_{tt}(x-1,\, t)-xf_t(x-1,\, t)\\
		&+2tf_{tttt}(x,\, t)+(3-2t)f_{ttt}(x,\, t)-3f_{tt}(x,\, t)=0.
		\end{split}
	\end{equation}
We then eliminate the terms $f_{tttt}(x-1,\, t)$ and $f_{ttt}(x-1,\, t)$ from the \eqref{diff-Bessel yn-1 eq6} utilizing  the equation \eqref{diff-Bessel yn-1 eq4} after it being differentiated several times. The substitutions of the terms $f_{tttt}(x-1,\, t)$ and 
	$f_{ttt}(x-1,\, t)$ from the equation \eqref{diff-Bessel yn-1 eq6} yield
		\[
		\begin{split}
		\dfrac{-3x}{2t}f_{tt}(x-1, t)-xf_t(x-1, t)+&tf_{tttt}(x, t)+(\dfrac{4t-4t^2-2xt+1}{2t})
		f_{ttt}(x, t)\\
		&+(t-3)f_{tt}-\dfrac{1-2tx-2t}{2t}f_t(x, t)=0.
		\end{split}
		\]
	
		Substituting \eqref{diff-Bessel yn-1 eq5} into above equation yields
		\begin{equation}\label{diff-Bessel yn-1 eq7}
		\begin{split}
		&tf_{tttt}(x, \,t)+\big(\dfrac{-4t^3+12t^2-5t-4xt^2+2xt-1}{2t(2t-1)}\big)f_{ttt}(x, \,t)\\
		&+\big(\dfrac{-10t^2-4t^2x+11t-6+3x}{2t(2t-1)}\big)f_{tt}(x,\, t)
		+\big(\dfrac{3x+4-4t-2xt}{2t(2t-1)}\big)f_t(x, \,t)\\
		&+\dfrac{3-2t}{2t(2t-1)}f(x, \,t)=0,
		\end{split}
		\end{equation}
		which is a fourth order linear ordinary differential equation in $t$ for each fixed $x$. Since the equation \eqref{diff-Bessel yn-1 eq5} has a regular points at $t=0,\, 1/2$, so that every analytic solution \eqref{diff-Bessel yn-1 eq5} has a radius of convergence $1/2$ in the neighbourhood of $x=0$. Hence 
the \eqref{E:bessel dp_gf-1} converges uniformly in each compact subset of $\mathbb{C}\times B(0; \frac12)$. Substituting a Frobenius solution 
$f(x,\cdot)=\sum_{k=0}^\infty a_kx^{\sigma+k}$ into the equation \eqref{diff-Bessel yn-1 eq5} further confirming that $\sigma=0,\, 1,\, 2$ are the only integer indicial roots. Hence the \eqref{E:bessel dp_gf-1} holds.
		\end{enumerate}
\end{proof}
%\medskip

%\begin{theorem}[(Gelfond]\cite{?}, pp. 163-164; N\"orlund \cite{??}, p. 115)] \label{T:newton} The abscissa of convergence of the Newton series
%	\[
%		\sum_{k=0}^\infty a_k \frac{x(x-1)\cdots (x-k+1)}{k!}
%	\]is given by
%	\begin{enumerate}
%		\item
%			\[
%				\lambda=\limsup_{n\to \infty}\frac{\ln \big|\sum_{k=0}^n(-1)^ka_k\big|}{\ln n}
%			\]if $\lambda\ge 0$, 
%		\item 	
%			\[
%				\lambda=\limsup_{n\to \infty}\frac{\ln \big|\sum_{k=n}^\infty (-1)^ka_k\big|}{\ln n}
%			\]if $\lambda< 0$.
%	\end{enumerate} 
%\end{theorem}		
%\medskip

%\begin{theorem}[(N\"orlund]\cite{?}, p. 177)]\label{T:Norlund} The faculty series
%	\[
%		xF(x)=\sum_{k=0}^\infty \frac{a_k\, k!}{(x+1)\cdots (x+k)}
%	\] and its associated Newton series
%	\[
%		xF_1(x)=\sum_{k=0}^\infty (-1)^k a_{k+1} \frac{x(x-1)\cdots (x-k)}{(k+1)!}
%	\]share the same abscissa of convergence of half-plane. Moreover,  the faculty series converges uniformly in each compact subset of $\mathbb{C}\backslash \{-\mathbb{N}\}$.
%\end{theorem}
%\medskip

\section{Difference operators with arbitrary step sizes}\label{Append:arbitrary} In practical applications, one would probably like to consider those $D$-modules like $\mathcal{O}_\Delta$ from Example~\ref{Eg:O_deleted_d} or $\mathcal{O}_{\Delta d}$ from Example~\ref{Eg:O_delta_d} with any step size $h\in\mathbb{C}\setminus\{0\}$. So we gather some formulae about these $D$-modules and the corresponding realizations in various function spaces.

\begin{example}\label{Eg:O_delta_h}
Let $\mathcal{O}$ be the space of analytic functions. Then $\mathcal{O}$ becomes a left $\mathcal{A}$-module with
	\[
		\begin{split}
		&(\partial f)(x) =\frac1h \Delta_h f(x)=
		\frac1h \big(f(x+h)-f(x)\big), \\
		&(Xf)(x)=xf(x-h).
		\end{split}
	\]
	This left $\mathcal{A}$-module is denoted by $\mathcal{O}_{\Delta_h}$.
\end{example}
%\medskip

Clearly $\mathcal{O}_{\Delta_h}$ is the same as $\mathcal{O}_\Delta$ when $h=1$. Similarly in $\mathcal{O}_{\Delta_h}$,  if we denote 
	\[
		(x)_{n,h}:=x(x-h)\cdots \big(x-(n-1)h\big)=h^n\dfrac{\Gamma(x/h+1)}{\Gamma(x/h-n+1)},
	\] then we have
	\[
		X^n\cdot 1 =x(x-h)\cdots \big(x-(n-1)h\big)= (x)_{n, h}.
	\]
%	\medskip
	
	As a result, the element $L=(X\partial)^2+X^2-\nu^2$ gives the difference Bessel equation
	\begin{equation}\label{E:Bessel_eqn_delta_h}
			\begin{split}
				\Big(\frac{x+2h)^2}{h^2}-\nu^2\Big)f(x+2h) &-\frac{(x+2h)(2x+3h)}{h^2}f(x+h)\\
				&+2\Big(1+\frac{1}{h^2}\Big)(x+2h)(x+h)f(x)=0,
			\end{split}
		\end{equation}which converges to Bohner and Cuchta\rq{}s difference Bessel equation \eqref{E:difference_bessel_eqn}
with step  $h\to 1$. One can also recover the classical Bessel equation \eqref{E:Bessel_eqn} when the step  $h\to 0$ with a more careful analysis.  Classical solutions of \eqref{E:Bessel_eqn_delta_h} are solutions of $\mathcal{D}/\mathcal{D}L$ in $\mathcal{O}_{\Delta_h}$. One of them is
		\begin{equation}%\label{E:composition}
			\mathcal{D}/\mathcal{D}L\stackrel{\times S_\nu}{\longrightarrow} \mathcal{D}/\mathcal{D}(X\partial-\nu)\stackrel{\times (x)_{\nu,h}}{\longrightarrow}
			\mathcal{O}_{\Delta_h},
		\end{equation}
		in which $S_\nu=\displaystyle\sum_{k=0}^{\infty} \frac{(-1)^kX^{2k}}{2^{\nu+2k}k!\Gamma(\nu+k+1)}$, and gives rise to the Newton series
		\begin{equation}\label{E:difference_bessel_fn_h}
			J^{\Delta_h}_{\nu}(x):=\sum_{k=0}^{\infty} \frac{(-1)^kh^{2k}}{2^{\nu+2k}k!\Gamma(\nu+k+1)}(x)_{\nu+2k,h}.
		\end{equation}
	Another solution in $\mathcal{O}_{\Delta_h}$ corresponding to $S_{-\nu}$ yields the Newton series
			\begin{equation}%\label{E:difference_bessel_fn_h}
			J^{\Delta_h}_{-\nu}(x):=\sum_{k=0}^{\infty} \frac{(-1)^kh^{2k}}{2^{-\nu+2k}k!\Gamma(-\nu+k+1)}(x)_{-\nu+2k,h},
		\end{equation}
		which is ``linearly independent" to the first one in case that $2\nu$ is not an integer.
%\medskip

\begin{example}\label{Eg:O_delta_h_d} Let $\mathcal{O}_2$ be the space of  analytic functions in two variables. It becomes a left $\mathcal{A}_2$-module when endowed with the structure
	\begin{equation}\label{E:O_delta_h_d}
		\begin{array}{ll}
			\partial_1f(x,\, t)=\frac1h\big(f(x+h,\, t)-f(x,\, t)\big),   & X_1f(x,\, t)=xf(x-h,\, t);\\
			\partial_2f(x,\, t) =f_t(x,\, t), 		& X_2 f(x,\, t)=tf(x,\, t).
		\end{array}	
	\end{equation}
	Such a left $\mathcal{A}_2$-module is denoted by $\mathcal{O}_{\Delta_h d}$.
\end{example}
%\medskip

\begin{example}\label{Eg:two-seq-functions_2_h} Let $\mathcal{O}^\mathbb{Z}$ be the space of all bilateral sequences of analytic functions with an appropriately imposed growth restriction. Let 
	\[
		(f_n)=(\cdots, \, f_{-1},\, f_0,\, f_1,\,  \cdots)
	\]denote a bilateral infinite sequence of $f_k\in\mathcal{O}$.  Then $\mathcal{O}^\mathbb{Z}$ becomes a left $\mathcal{A}_2$-module by
	\begin{equation}
		\begin{array}{ll}
		(\partial_1 f)_n(x)=\frac1h\big(f_n(x+h)-f_n(x)\big), &  (X_1 f)_n(x)=xf_n(x-h),\\
		(\partial_2 f)_n(x)=(n+1)f_{n+1}(x),& (X_2 f)_n(x)=f_{n-1}(x),
		\end{array}
	\end{equation}for all bilateral sequences of analytic functions $(f_n)$. Such a left $\mathcal{A}_2$-module is sometimes denoted by $\mathcal{O}_{\Delta_h}^\mathbb{Z}$.
\end{example}
%\medskip

We state without proof the following extensions of the Theorem \ref{T:difference_Bessel_gf} with an arbitrary step size $h$:
%\medskip

\begin{theorem}\label{T:difference_h_Bessel_gf} Let $\nu\in\mathbb{C}$. 
	\begin{enumerate}
		\item Let $(\mathscr{C}^\Delta_{n+\nu})_n$ be a bilateral sequence of analytic functions which is a solution of the Bessel module $\mathcal{B}_\nu$ in $\mathcal{O}_{\Delta}^\mathbb{Z}$. Then there exists a $1$-periodic function $C_\nu$ such that
	\begin{equation}\label{gf-db_h-2}
		e^{i\pi\nu}
		\frac{\sin(x/h-\nu)\pi}{\sin(\pi x/h)}\,
		t^{-\nu}\big[\frac{h}{2}(t-\frac{1}{t})+1\big]^{x/h}
		\sim \sum_{n=-\infty}^{\infty}
J^{\Delta_h}_{n+\nu}(x)\,t^n,
\end{equation}			where the symbol $\sim$ means that the left-hand side is the Borel resummation \footnote{Definition \ref{D:borel}} of the right-hand side whenever it diverges.
		\item Moreover,  the manifestation of the holonomic system of PDEs (delay-differential equations) \eqref{E:bessel_PDE} in $\mathcal{O}_{\Delta d}$ in Example \ref{Eg:O_delta_h_d}  		defined by \eqref{E:O_delta_d}  is given by
	\begin{equation}\label{E:PDE_delta_h_d}
		y(x, t)-y(x-h, t)=
		\frac{h}{2}(t-\frac{1}{t}) y(x-h, t),\quad
		\nu y(x,t) +ty_t(x,t)=\frac{x}{2}\big(t+\frac{1}{t}\big)y(x-h, t).
	\end{equation}		
	\end{enumerate}
\end{theorem}
%\medskip

We note that although the above result is valid for each fixed $h\not=0$, one cannot take the limit $h\to 0$ so that the \eqref{E:PDE_delta_h_d} becomes the generating function \eqref{E:bessel_classical_gf_1} for the classical Bessel functions, unless when $\nu=0$, as in the following theorem.

\begin{theorem}\label{T:h_limit} 
	\begin{enumerate}
		\item Let $(J^{\Delta_h}_n(x))$ and $(J_n(x))$ be sequences in $\mathcal{O}_{\Delta_h}^{\mathbb{Z}}$ and ${\mathcal{O}_d}^\mathbb{Z}$. Then we have the following limits
			\begin{equation}\label{E:h_limit}
				\begin{tikzcd} [row sep=large, column sep=large]
				\displaystyle\big[\frac{h}{2}(t-\frac{1}{t})+1\big]^{x/h}\qquad = 
				\arrow[swap]{d}{h\to 0} 
				&  \displaystyle \sum_{n=-\infty}^{\infty}
J^{\Delta_h}_{n}(x)\,t^n \arrow{d}{h\to 0}\\
 			 	\displaystyle\exp\big[\frac{x}{2}(t-\frac{1}{t})\big]
				\qquad= & \displaystyle\sum_{n =-\infty}^{\infty} J_{n}(x)\,t^n
				\end{tikzcd}
			\end{equation}
		\item Moreover, the system of PDEs \eqref{E:O_delta_h_d} with $\nu=0$, that is
			\begin{equation}\label{E:O_delta_h_d_0}
		y(x, t)-y(x-h, t)=
		\frac{h}{2}(t-\frac{1}{t}) y(x-h, t),\quad
		ty_t(x,t)=\frac{x}{2}\big(t+\frac{1}{t}\big)y(x-h, t).
	\end{equation}		
 becomes the system of PDEs \eqref{E:PDE_gf_bessel_0} as $h\to 0$.
	\end{enumerate}
\end{theorem}

The case of the above limits when $\nu\not=0$, while likely to hold, is more delicate that would be dealt with in a separate publication.
\vfill\eject

\section{List of tables} \label{A:list}

\subsection{List of commonly used $D$-modules}\label{SS:common_D_list}
%\listoftables
%In this section, some $D$-modules that are commonly used in this article are listed.

\begin{table}[h!]\label{T:common_D_list}
%\scriptsize{
{\small
%{\footnotesize
\begin{tabular}[h!]{|c|c|c|l|}
\hline
$D$-module & \begin{tabular}{c}``ground"\\ Weyl algebra\end{tabular} & \begin{tabular}{c}underlying\\ space\end{tabular} & Manifestation \\
\hline
$\mathcal{A}/\mathcal{A}L$ & $\mathcal{A}$ & $\mathcal{A}/\mathcal{A}L$ & same as multiplication in $\mathcal{A}$ \\
\hline
$\dfrac{\mathcal{A}_2}{\mathcal{A}_2L_1+\mathcal{A}_2L_2}$ & $\mathcal{A}_2$ & $\dfrac{\mathcal{A}_2}{\mathcal{A}_2L_1+\mathcal{A}_2L_2}$ & same as multiplication in $\mathcal{A}_2$ \\
\hline
$\mathbb{C}^{\mathbb{N}_0}$ & $\mathcal{A}$ & $\mathbb{C}^{\mathbb{N}_0}$ & \begin{tabular}{l}$(\partial a)_n=a_{n+1}$\\ $(Xa)_n=na_{n-1}$\end{tabular}\\
\hline
$\mathbb{C}^\mathbb{Z}$ & $\mathcal{A}$ & $\mathbb{C}^\mathbb{Z}$ & \begin{tabular}{l}$(\partial a)_n=(n+1)a_{n+1},$\\ $(Xa)_n=a_{n-1}$\end{tabular}\\
\hline
$\mathcal{O}_d$ & $\mathcal{A}$ & $\mathcal{O}$ & \begin{tabular}{l}$\partial f(x)=f'(x),$\\ $Xf(x)=xf(x)$\end{tabular}\\
\hline
$\mathcal{O}_\Delta$ & $\mathcal{A}$ & $\mathcal{O}$ & \begin{tabular}{l}$\partial f(x)=f(x+1)-f(x),$\\ $Xf(x)=xf(x-1)$\end{tabular}\\
\hline
$\mathcal{O}_{dd}$ & $\mathcal{A}_2$ & $\mathcal{O}_2$ & \begin{tabular}{l}$\partial_1 f(x,t)=f_x(x,t),$\\ $X_1 f(x,t)=xf(x,t)$\\ $\partial_2 f(x,t)=f_t(x,t),$\\ $X_2 f(x,t)=tf(x,t)$\end{tabular}\\
\hline
$\mathcal{O}_{\Delta d}$ & $\mathcal{A}_2$ & $\mathcal{O}_2$ & \begin{tabular}{l}$\partial_1 f(x,t)=f(x+1,t)-f(x,t),$\\ $X_1 f(x,t)=xf(x-1,t)$\\ $\partial_2 f(x,t)=f_t(x,t),$\\ $X_2 f(x,t)=tf(x,t)$\end{tabular}\\
\hline
$\mathcal{O}_d^{\mathbb{N}_0}$ & $\mathcal{A}_2$ & $\mathcal{O}^{\mathbb{N}_0}$ & \begin{tabular}{l}$(\partial_1 f)_n(x)=f_n'(x),$\\ $(X_1 f)_n(x)=xf_n(x)$\\ $(\partial_2 f)_n(x)=f_{n+1}(x),$\\ $(X_2 f)_n(x)=nf_{n-1}(x)$\end{tabular}\\
\hline
	$\mathcal{O}_d^\mathbb{Z}$  & $\mathcal{A}_2$ & $\mathcal{O}^\mathbb{Z}$ 
	& \begin{tabular}{l}$(\partial_1 f)_n(x)=f_n'(x),$\\ $(X_1 f)_n(x)=xf_n(x)$\\ 
		$(\partial_2 f)_n(x)=(n+1)f_{n+1}(x),$\\ 
		$(X_2 f)_n(x)=f_{n-1}(x)$
	     \end{tabular}
	     \\
\hline
	$\mathcal{O}_\Delta^\mathbb{Z}$ & $\mathcal{A}_2$  & $\mathcal{O}^\mathbb{Z}$ & \begin{tabular}{l} $(\partial_1 f)_n(x)=f_n(x+1)-f_n(x),$\\ $(X_1 f)_n(x)=xf_n(x-1)$;\\ $(\partial_2 f)_n(x)=(n+1)f_{n+1}(x),$\\ $(X_2 f)_n(x)=f_{n-1}(x)$\end{tabular}\\
\hline
	$\mathcal{O}_\Delta^{\mathbb{N}_0}$ & $\mathcal{A}_2$ & $\mathcal{O}^{\mathbb{N}_0}$ &
	 \begin{tabular}{l} $(\partial_1 f)_n(x)=f_n(x+1)-f_n(x),$\\ $(X_1 f)_n(x)=xf_n(x-1)$;\\ $(\partial_2 f)_n(x)=f_{n+1}(x),$\\ $(X_2 f)_n(x)=nf_{n-1}(x)$\end{tabular}\\
\hline
\end{tabular}\\
\begin{center}
\caption{\label{demo-table}List of $D$-modules used.\newline
%Each ground Weyl algebra above is allowed to be slightly modified, e.g. $\mathcal{A}$ can be modified into $\mathcal{A}(1/X)$, and $\mathcal{A}_2$ can be modified into $\mathcal{A}_2(1/X)$, etc.
}\end{center}
}
\end{table}
%\vskip-1.5cm
Each ground Weyl algebra above is allowed to be slightly modified, e.g. $\mathcal{A}$ can be modified into $\mathcal{A}(1/X)$, and $\mathcal{A}_2$ can be modified into $\mathcal{A}_2(1/X)$, etc.
\vfill\eject
%\medskip

\subsection{List of transmutation formulae}
\begin{table}[h!]
%\small{
\scriptsize{
\begin{tabular}[h!]{|c|c|c|}
\hline
 $D$-modules & Transmutation formulae  & No.\\
 \hline 
 $\mathcal{B}_\nu$ & 
  	\begin{tabular}{l}
  	   $[(X\partial)^2+X^2-(\nu-1)^2]\big(\partial+{\nu}/{X}\big)$\\
  	   \hskip4cm$=\big(\partial+(\nu-2)/{X}\big)[(X\partial)^2+X^2-\nu^2]$,
  	   %\label{E:bessel_transmutation_1}
  	   \\
      $[(X\partial)^2+X^2-(\nu+1)^2]\big(\partial-\nu/X\big)$\\
      \hskip4cm $=\big(\partial-(\nu+2)/X\big)[(X\partial)^2+X^2-\nu^2]$
      \end{tabular}
      & 
      \begin{tabular}{c}				\eqref{E:bessel_transmutation_1}\\
      -\eqref{E:bessel_transmutation_2}
      \end{tabular}\\
 \hline
 $\Theta$ & 
 	\begin{tabular}{l}
 				$\Big(X\partial^2-2(\nu+\frac12+X)\partial  +2(\nu +\frac12)\Big)(X\partial-X-2\nu)$\\
 				$=\big(X\partial-X-2(\nu+1)\big)\Big(X\partial^2-2(\nu-\frac12+X)\partial+2(\nu -\frac12)\Big)$, \\
		$\Big(X\partial^2-2(\nu-\frac12+X)\partial +2(\nu-\frac12)\Big)\frac1X(\partial-1)$\\
			$=\frac1X (\partial-1+\frac1X)\Big(X\partial^2-2(\nu+\frac12+X)\partial+2(\nu+\frac12)\Big)$
	\end{tabular}
	&
		\begin{tabular}{c}
	 		\eqref{E:bessel_poly_transmutation_1}\\
	 		-\eqref{E:bessel_poly_transmutation_2}
	   \end{tabular}
	\\
\hline
 \end{tabular}
}
\caption{\label{demo-table}List of transmutation formulae.}
\end{table}
% \hline

%\vfill\eject
%\newpage
\subsection{List of holonomic modules}\label{SS:holo_modules}
%\medskip

%\begin{table}
%\begin{sidewaystable}[h!]
%\small{
\scriptsize{
\begin{tabular}[h!]{|c|c|c|c|}
\hline
\begin{tabular}{l}
Quotient modules\\
 $\mathcal{A}_2/(\mathcal{A}_2L_1+\mathcal{A}_2L_2)$
 \end{tabular} & 
\begin{tabular}{l}$D$-module\\
 terminologies
 \end{tabular}
& Generators 
& Definitions no. \\
\hline
	$\mathcal{B}_\nu$ & Bessel modules &
	\begin{tabular}{c}
		$L_1=X_1\partial_1+(\nu+X_2\partial_2)-X_1X_2$, \\
	         $L_2=X_1\partial_1-(\nu+X_2\partial_2)+{X_1}/{X_2}$
	\end{tabular} &
	 Definition \ref{E:bessel_mod}
	\\
\hline
	$\Theta$ & 
	\begin{tabular}{l}
	Half-Bessel module:\\
	reverse Bessel\\
	 polynomial module 
	 \end{tabular}
	  &
	\begin{tabular}{l} 
		$L_1=\partial_1-1+{X_1}/{\partial_2}$,\\
		$L_2=X_1\partial_1-2X_2\partial_2-1-X_1+\partial_2$
	\end{tabular} &
	Definition \ref{D:Theta}\\
\hline
	$\mathcal{Y}$ & 
	\begin{tabular}{l}
	Half-Bessel module:\\
	Bessel polynomial\\
	module 
	 \end{tabular}
	  &
	\begin{tabular}{l} 
		$L_1=X_{1}^2\partial_{1}\partial_{2}-X_{1}X_{2}\partial^2_{2}+\partial_{2}-1$,\\
		$L_2=X_{1}^2\partial_{1}-\partial_{2}+X_{1}X_{2}\partial_{2}+1$
	\end{tabular} &
	Definition \ref{D:y-1}\\
\hline
	$\mathcal{G}_{-}$ & 
	\begin{tabular}{l}
	Half-Bessel module:\\
	Negative Glaisher\\
	module 
	 \end{tabular}
	  &
	\begin{tabular}{l} 
		$L_1=W_1 \Xi_1+(W_2\Xi_2-\frac{1}{2})-W_1/\Xi_2$,\\
%		\label{E:PDEs_neg_glaisher_1}\\
		$L_2=W_1 \Xi_1-(W_2 \Xi_2-\frac{1}{2})+W_1 \Xi_2$.
		%\label{E:PDEs_neg_glaisher_2}$
	\end{tabular} &
	Definition \ref{D:neg_glaisher_module}\\
\hline
	$\mathcal{G}_{+}$ & 
	\begin{tabular}{l}
	Half-Bessel module:\\
	Positive Glaisher\\
	module 
	 \end{tabular}
	  &
	\begin{tabular}{l} 
		$L_1=W_1 \Xi_1+(W_2\Xi_2+\frac{1}{2})+W_1/\Xi_2$,\\
	        $L_2=W_1 \Xi_1-(W_2 \Xi_2+\frac{1}{2})-W_1 \Xi_2$.
	\end{tabular} &
	Definition \ref{D:pos_glaisher_module}\\
\hline
\end{tabular}
}

%\vfill\eject

\subsection{List of formulae in $\mathcal{O}_d$, $\mathcal{O}_\Delta$}
%\medskip

%\begin{sidewaystable}[h!]
%\scriptsize{
\small{
\begin{tabular}[h!]{|c|c|c|}
\hline
$D$-modules & PDEs. & No. \\
\hline
$\mathcal{B}_\nu$ in $\mathcal{O}_d$ & 
%\label{E:any_bessel_recus_1}
		\begin{tabular}{l}
			$x\mathscr{C}_{\nu}^\prime(x)+\nu\mathscr{C}_{\nu}(x)-x\mathscr{C}_{\nu-1}(x)=0$,\\
		%\begin{equation}\label{E:any_bessel_recus_2}
			$x\mathscr{C}^{\prime}_{\nu}(x)-\nu\mathscr{C}_{\nu}(x)+
		x\mathscr{C}_{\nu+1}(x)=0$,\\
%			\begin{equation}\label{E:any_bessel_3_term}
			$2\nu\mathscr{C}_{\nu}(x)-x\mathscr{C}_{\nu-1}(x)-x\mathscr{C}_{\nu+1}(x)=0$
		\end{tabular}
		& \eqref {E:any_bessel_recus_1}, \eqref {E:any_bessel_recus_2}, \eqref {E:any_bessel_3_term}\\
\hline
$\mathcal{B}_\nu$ in $\mathcal{O}_\Delta$ & 
			\begin{tabular}{l}
				$x\Delta \mathscr{C}^{\Delta}_{\nu}(x-1)+\nu\mathscr{C}^{\Delta}_{\nu}(x)-
		x\mathscr{C}^{\Delta}_{\nu-1}(x-1)=0$,\\
			$x\Delta \mathscr{C}^{\Delta}_{\nu}(x-1)-\nu\mathscr{C}^{\Delta}_{\nu}(x)+
		x\mathscr{C}^{\Delta}_{\nu+1}(x-1)=0$,\\
			$2\nu\mathscr{C}_{\nu}^\Delta(x)-x\mathscr{C}_{\nu-1}^\Delta(x-1)-x\mathscr{C}_{\nu+1}^\Delta(x-1)=0$
		\end{tabular}
		& \eqref {E:delta_bessel_PDE_a}, \eqref {E:delta_bessel_PDE_b}, \eqref{E:delta_bessel_3_term_a}
		\\
\hline
$\Theta$ in $\mathcal{O}_d$ 
	&
	\begin{tabular}{l}  
   	$\theta_n^\prime(x)-\theta_n(x)+x\theta_{n-1}(x)=0$,\\
	$x\theta_n^\prime(x)-(x+2n+1)\theta_n(x)+\theta_{n+1}(x)=0$,\\
      $\theta_{n+2}(x)-(2n+3)\theta_{n+1}(x)-x^2\theta_n(x)=0$.
	\end{tabular}
	&\eqref{E:bessel_poly_trans}, 
		\eqref{E:bessel_poly_3term} \\
\hline
	$\Theta$ in $\mathcal{O}_\Delta$ &
	\begin{tabular}{l}
		$\theta^\Delta_n(x+1)-2\theta^\Delta_n(x)+x\theta^\Delta_{n-1}(x-1)=0$, \\
		$ \theta^\Delta_{n+1}(x)+(x-2n-1)\theta^\Delta_n(x)-2x\, \theta^\Delta_n(x-1)=0$,\\
	$\theta^\Delta_{n+2}(x)-(2n+3)\theta^\Delta_{n+1}(x)-x(x-1)\theta^\Delta_n(x-2)=0$	
	\end{tabular}
		& \eqref{E:delta_bessel_poly_trans},
		\eqref{E:delta_reverse_bessel_poly_3term}
	\\	
\hline
	$\mathcal{Y}$ in  $\mathcal{O}_d$ 
	&
	\begin{tabular}{l}  
		$x^2y'_{n-1}(x)-[(n-1)x-1]y_{n-1}(x)-y_{n-2}(x)=0$,\\
		$x^2y'_{n-1}(x)-y_{n}(x)+(nx+1)y_{n-1}(x)=0$,\\
		$y_{n+1}(x)-(2n+1)xy_{\color{blue}n}(x)-y_{n-1}(x)=0$.
	\end{tabular}
	& \eqref{E:classical_bessel_poly_trans},  \eqref{E:classical_bessel_poly_3term}
	\\
\hline
	$\mathcal{Y}$ in $\mathcal{O}_\Delta$ 
	&
	\begin{tabular}{l}
		$x(x-1) \Delta y^{\Delta}_{n-1}(x-2)-(nx-x)y^{\Delta}_{n-1}(x-1)+y^{\Delta}_{n-1}(x)-
y^{\Delta}_{n-2}(x)=0$, \\
		$x(x-1) \Delta y^{\Delta}_{n-1}(x-2)-y^{\Delta}_{n}(x)+nxy^{\Delta}_{n-1}(x-1)+y^\Delta_{n-1}(x)=0$,\\
		$y^{\Delta}_{n+1}(x)-x(2n+1)y^{\Delta}_{n}(x-1)-y^{\Delta}_{n-1}(x)=0$.
	\end{tabular}
	& \eqref{E:delta_classical_bessel_poly_trans}, \eqref{E:delta_bessel_poly_3term}
\\
\hline
\end{tabular}
}
%\caption{\label{demo-table}List of formulae in $\mathcal{O}_d$, $\mathcal{O}_\Delta$.}
%\end{sidewaystable}

%\vfill\eject
\subsection{List of PDE systems in complex domains}\label{SS:PDE_list}
%In this subsection, the holonomic PDEs systems that appear in this article are listed.
\vskip-.8cm
%\medskip
%\newpage
\begin{sidewaystable}
{\small
%\scriptsize{
\begin{tabular}[h!]{|c|c|c|c|c|}
\hline
$D$-modules 
& \begin{tabular}{c} ``target"\\ Weyl algebras\end{tabular} 
& PDEs & Gen. fns
& Eqn no.
\\
\hline
$\mathcal{B}_\nu$ & $\mathcal{O}_{dd}$ & 
\begin{tabular}{l}
			 $y_x(x,t)+(1/t-t)/2\, y(x,t)=0$,\\
			  $\nu y(x,t)+ty_t(x,t)-{x}/2\,(1/t+t)\, y(x,t)=0$. 
	\end{tabular}
	& $t^{-\nu}e^{\frac{x}{2}(t-\frac1t)}$
	& \eqref{E:PDE_bessel}
	\\
\hline
$\mathcal{B}_\nu$ & $\mathcal{O}_{\Delta d}$ & 
\begin{tabular}{l}
		$y(x+1, t)-y(x, t)=
		\frac{1}{2}(t-\frac{1}{t}) y(x, t)$,\\
		$ \nu y(x,t) +ty_t(x,t)={x}/2(t+{1}/{t})y(x-1, t)$.
			\end{tabular}	
& $t^{-\nu} \big[\frac12(t-\frac1t)+1\big]^x$ &
\eqref{gf-db-2}
		   \\
\hline
$\Theta$ & $\mathcal{O}_{dd}$ & 
	\begin{tabular}{l}
		$f_{xt}(x,\, t)-f_t(x,\, t)+xf(x, \, t)=0$,\\
		$xf_x (x,\, t)+ (1-2t)f_t(x,\, t)-(1+x)f(x,\, t)=0$.
	\end{tabular}
	& $(1-2t)^{-\frac12}\exp[x(1-\sqrt{1-2t)}]$ 
	& \eqref{E:bessel_poly_PDE}\\
\hline 
$\Theta$ & $\mathcal{O}_{\Delta d}$ & 
	\begin{tabular}{l}
		 $f_t(x+1,\, t)-2f_t(x,\, t)+xf(x-1,\, t)=0$,\\
		$(1-2t)f_t(x,\, t)+(x-1)f(x,\, t)-2xf(x-1,\, t)=0$
	\end{tabular}
	& $\mathfrak{N} (1-2t)^{-\frac12}\exp[x(1-\sqrt{1-2t)}]$
	& \eqref{E:delta_bessel_poly_PDE}\\
\hline 
$\mathcal{Y}$ & $\mathcal{O}_{d d}$ & 
	\begin{tabular}{l}
		 $x^2f_{xt}(x,\, t)	-xtf_{tt}(x,\, t)+f_t(x,\,t)-f(x,\,t)=0$,\\
		$x^2f_x(x,\, t)-f_t(x,\, t)+xtf_t(x,\, t)+f(x,\,t)=0$
	\end{tabular}
	& $\exp\left(\frac{1-(1-2xt)^{1/2}}{x}\right)$
	& \eqref{E:classical_bessel_poly_PDE}\\
\hline 
$\mathcal{Y}$ & $\mathcal{O}_{\Delta d}$ & 
	\begin{tabular}{l}
		 $x(x-1)[f_t(x-1,\, t)-f_t(x-2, \, t)]-xtf_{tt}(x-1,\, t)+f_t(x,\, t)-f(x, \, t)=0$,\\
		$x(x-1)[f(x-1,\, t)-f(x-2, \, t)]-f_{t}(x,\, t)+xtf_t(x-1,\, t)+f(x, \, t)=0.$
	\end{tabular}
	& $\mathfrak{N} \exp\left(\frac{1-(1-2xt)^{1/2}}{x}\right)$
	& \eqref{E:delta_classical_bessel_poly_PDE}\\
\hline 
$\mathcal{G}_-$ & $\mathcal{O}_{d d}$ & 
	\begin{tabular}{l}
		 $xf_{xt}(x,\, t)+tf_{tt}(x,\, t)+\frac12 f_t(x,\, t)-xf(x, t)=0$,\\
		$xf_x(x,\, t)+(x-t)f_t(x,\, t)+\frac12 f(x,\, t)=0$
	\end{tabular}
	& $\sqrt{2/\pi x}\cos\sqrt{x^2-2xt}$
	& 
		\begin{tabular}{c}
			\eqref{E:glaisher_PDE_1},\\
			 \eqref{E:glaisher_PDE_2}
		\end{tabular}
		\\
\hline 
$\mathcal{G}_+$ & $\mathcal{O}_{d d}$ & 
	\begin{tabular}{l}
		 $xf_{xt}(x,\, t)+tf_{tt}(x,\, t)+\frac12 f_t(x,\, t)+xf(x, t)=0$,\\
		$xf_x(x,\, t)-(x+t)f_t(x,\, t)+\frac12 f(x,\, t)=0$
	\end{tabular}
	& $\sqrt{2/\pi x}\sin\sqrt{x^2+2xt}$
	& 
		\begin{tabular}{c}
			\eqref{E:glaisher_PDE_6},\\
			 \eqref{E:glaisher_PDE_7}
		\end{tabular}
		\\
\hline 
$\mathcal{G}_-$ & $\mathcal{O}_{\Delta d}$ & 
	\begin{tabular}{l}
		 $tf_{tt}(x,\, t)+\big(x+\frac12\big)f_t(x,\, t)-xf_t(x-1,\, t)-xf(x-1,\, t)=0$,\\
		$tf_{t}(x,\, t)-xf_t(x-1,\, t)+xf(x-1,\, t)-\big(x+\frac12\big) f(x,\, t)=0$
	\end{tabular}
	& $\mathfrak{N}\sqrt{2/\pi x}\cos\sqrt{x^2-2xt}$
	& 
		\begin{tabular}{c}
			\eqref{E:PDeltaE_neg_glaisher_-1},\\
			 \eqref{E:PDeltaE_neg_glaisher_-2}
		\end{tabular}
		\\
\hline 
 $\mathcal{G}_+$ & $\mathcal{O}_{\Delta d}$ & 
	\begin{tabular}{l}
		 $tf_{tt}(x,\, t)+\big(x+\frac12\big)f_t(x,\, t)-xf_t(x-1,\, t)+xf(x-1,\, t)=0$,\\
		$ tf_{t}(x,\, t)+xf_t(x-1,\, t)+xf(x-1,\, t)-\big(x+\frac12\big) f(x,\, t)=0$
	\end{tabular}
	& $\mathfrak{N}\sqrt{2/\pi x}\sin\sqrt{x^2+2xt}$
	& 
		\begin{tabular}{c}
			\eqref{E:PDeltaE_pos_glaisher_+1},\\
			\eqref{E:PDeltaE_pos_glaisher_+2}
		\end{tabular}
		\\
\hline 
\end{tabular}
}\caption{\label{demo-table}List of systems of holonomic PDEs.}
\end{sidewaystable}

\vfill\eject
%\medskip
%\newpage

\subsection{List of generating functions}\label{SS:gf_list}
%In this subsection, the holonomic PDEs systems that appear in this article are listed.
%\vskip-.8cm
%\medskip
%\newpage
%\begin{sidewaystable}
%\small{
{\footnotesize
%{\scriptsize
\begin{tabular}[h!]{|l|c|c|}
\hline
Generating functions (${}^\ast$ denote likely new formulae)&  $D$-mod. & Eqn no.
\\
\hline	
	$\displaystyle
		t^{-\nu}\exp\big[\dfrac{x}{2}\big(t-\dfrac{1}{t})\big]\sim\sum_{n=-\infty}^\infty J_{\nu+n}(x)\, t^n$& $\mathcal{B}_\nu$ & $\eqref{E:gf_J_nu}^\ast$\\
\hline
	$\displaystyle
		\exp\big[\dfrac{x}{2}\big(t-\dfrac{1}{t})\big]=\sum_{n=-\infty}^\infty J_{n}(x)\, t^n$& $\mathcal{B}_0$ & \eqref{E:bessel_classical_gf}\\ 
\hline
	$\displaystyle e^{-i\pi/2}\,  t^{-\nu}\exp\big[\dfrac{x}{2}\big(t-\dfrac{1}{t})\big]\sim \sum_{n=-\infty}^\infty Y_{\nu+n}(x)\, t^n$
	& $\mathcal{B}_\nu$ & $\eqref{E:gf_Y_nu}^\ast$\\
\hline
	$\displaystyle 	 t^{-\nu}\exp\big[\dfrac{x}{2}\big(t+\dfrac{1}{t})\big]\sim \sum_{n=-\infty}^\infty I_{\nu+n}(x)\, t^n$
	& $\mathcal{B}_\nu$ & $\eqref{E:gf_I_nu}^\ast$\\
\hline
	$\displaystyle 	i\pi \,  t^{-\nu}\exp\big[-\dfrac{x}{2}\big(t+\dfrac{1}{t})\big]\sim \sum_{n=-\infty}^\infty K_{\nu+n}(x)\, t^n$	
	& $\mathcal{B}_\nu$ & $\eqref{E:gf_K_nu}^\ast$\\
\hline
	$\displaystyle e^{i\pi\nu} \frac{\sin(x-\nu)\pi}{\sin(\pi x)}\, t^{-\nu}\big[\frac{1}{2}(t-\frac{1}{t})+1\big]^{x} \sim \sum_{n=-\infty}^{\infty} J^{\Delta}_{n+\nu}(x)\,t^n$
	& $\mathcal{B}_\nu$ & $\eqref{gf-db-2}^\ast$\\
\hline
	$\displaystyle \big[\frac{1}{2}(t-\frac{1}{t})+1\big]^{x} = \sum_{n=-\infty}^{\infty} J^{\Delta}_{n}(x)\,t^n$, \quad $|\frac{1}{2}(t-\frac{1}{t})|<1, \ x\in \mathbb{C}$ 
	& $\mathcal{B}_0$ & $\eqref{gf-db-1}^\ast$\\
\hline
	$\displaystyle \frac{1}{\sqrt{1-2t}}\Big\{ C_1\exp\big[x(1-\sqrt{1-2t})\big]
				+C_2 \exp\big[x(1+\sqrt{1-2t})\big]\Big\}
				=\sum_{n=0}^\infty \vartheta_n(x)\, \frac{t^n}{n!},
				\ (\mathbb{C}\times B(0,1/2))$
	& $\Theta$ & $\eqref{E:gen_rev_bessel_poly_gf}^\ast$\\
\hline
	$\displaystyle \frac{1}{\sqrt{1-2t}} \exp\big[x(1-\sqrt{1-2t})\big]=\sum_{n=0}^\infty
			\theta_n(x)\, \frac{t^n}{n!}, \ (\mathbb{C}\times B(0,1/2))$
	& $\Theta$ & \eqref{E:rev_bessel_poly_gf}\\
\hline
	$\displaystyle \frac{e^{-i\pi x}}{2i\sin\pi x\Gamma(-x)}\int_{-\infty}^{(0+)}e^\lambda(-\lambda)^{-x-1} \frac{1}{\sqrt{1-2t}} \exp\big[\lambda(1-\sqrt{1-2t})\big]\,d\lambda=\sum_{n=0}^\infty \theta_n^\Delta (x)\, \frac{t^n}{n!}, \ (\mathbb{C}\times B(0,1/2))$
		& $\Theta$ & $\eqref{E:bessel_difference_poly_gf}^\ast$\\
\hline
	$\displaystyle C_1\exp\left(\frac{1-(1-2xt)^{1/2}}{x}\right)
		+C_2 \exp\left(\frac{1+(1-2xt)^{1/2}}{x}\right)
		=\sum_{n=0}^\infty
		\mathscr{Y}_{n-1}(x)\, \frac{t^n}{n!}, \ (|t|<1/|2x|)$
	& $\mathcal{Y}$ & $\eqref{E:gen_bessel_poly_gf}^\ast$\\
\hline	
	$\displaystyle \exp\left(\frac{1-(1-2xt)^{1/2}}{x}\right)
		=\sum_{n=0}^\infty y_{n-1}(x)\, \frac{t^n}{n!},\ (|t|<1/|2x|)$
	& $\mathcal{Y}$ & \eqref{E:bessel p_gf-1}\\
\hline
	$\displaystyle \frac{e^{-i\pi x}}{2i \sin \pi x\Gamma(-x)}\int_{-\infty}^{(0+)}
		e^{\lambda+\frac{1-\sqrt{1-2\lambda t}}{\lambda} }
		(-\lambda)^{-x-1}d\lambda
		=\sum_{n=0}^\infty y^{\Delta}_{n-1}(x)\, \frac{t^n}{n!}, \ (\mathbb{C}\times B(0,1/2))$
	& $\mathcal{Y}$ & $\eqref{E:bessel dp_gf-1}^\ast$\\
\hline
	$\displaystyle x^{-\frac12} 
			\big[C_1\cos\sqrt{x^2\pm 2xt}+C_2\sin \sqrt{x^2\pm 2xt}
\big]=\sum_{n=0}^\infty \mathscr{C}_{\mp n\pm \frac{1}{2}}(x)\, \frac{t^n}{n!}$,
	\  ($|t|<1/|2x|$) & $\mathcal{G}_\pm$ & $\eqref{E: general half bessel_gf}^\ast$\\
\hline
	$\displaystyle \sqrt{\frac{2}{\pi x}} \cos(x^2-2xt)^{\frac{1}{2}}
		=\sum_{n=0}^\infty J_{n-\frac{1}{2}}(x)\, \frac{t^n}{n!}$,\
		 ($|t|<1/|2x|$) & $\mathcal{G}_-$ & \eqref{E: first half bessel_gf}\\
\hline
	$\displaystyle \sqrt{\frac{2}{\pi x}} \sin(x^2+2xt)^{\frac{1}{2}}
		=\sum_{n=0}^\infty J_{-n+\frac{1}{2}}(x)\, \frac{t^n}{n!}$,
		\  ($|t|<1/|2x|$) & $\mathcal{G}_+$ &\eqref{E: first half bessel_gf}\\
\hline
	$\displaystyle \sqrt{\frac{2}{\pi x}}\sin(x^2-2xt)^{\frac{1}{2}}
	=\sum_{n=0}^{\infty} Y_{n-\frac{1}{2}}(x)\,\frac{t^n}{n!}$,
	\ ($|t|<1/|2x|$) & $\mathcal{G}_-$ & \eqref{E: second half bessel gf}\\
\hline
	$\displaystyle -\sqrt{\frac{2}{\pi x}} \cos(x^2+2xt)^{\frac{1}{2}}
	=\sum_{n=0}^{\infty} Y_{-n+\frac{1}{2}}(x)\,\frac{t^n}{n!}$,
		\  ($|t|<1/|2x|$) & $\mathcal{G}_+$ &\eqref{E: second half bessel gf}\\
\hline
	$\displaystyle \sqrt{\frac{2}{\pi x}}\cos(-x^2-2xt)^{\frac{1}{2}}=
	\sum_{n=0}^{\infty} I_{n-\frac{1}{2}}(x)\, \frac{t^n}{n!}$,
		\  ($|t|<1/|2x|$) & $\mathcal{G}_-$ &\eqref{E: first modify half bessel gf}\\
\hline
	$\displaystyle -i\sqrt{\frac{2}{\pi x}}\sin(-x^2-2xt)^{\frac{1}{2}}=
	\sum_{n=0}^{\infty} I_{-n+\frac{1}{2}}(x)\, \frac{t^n}{n!}$,
	\  ($|t|<1/|2x|$) & $\mathcal{G}_+$ &\eqref{E: second half bessel gf}\\
\hline
	$\displaystyle \sqrt{\frac{\pi}{2 x}}e^{i (-x^2+2xt)^{1/2}}= \sum_{n=0}^{\infty} K_{n-\frac{1}{2}}(x)\, \frac{t^n}{n!}=
\sum_{n=0}^{\infty} K_{-n+\frac{1}{2}}(x)\, \frac{t^n}{n!}$,
	\  ($|t|<1/|2x|$) & $\mathcal{G}_\mp$ &\eqref{E: second modify half bessel gf}\\
\hline
		$\displaystyle {}\int_{-\infty}^{(0+)} 
			\lambda^{-\frac12} \frac{\big[C_1(x)\sin \sqrt{\lambda^2\mp 2\lambda t} +C_2(x)\cos \sqrt{\lambda^2\mp 2\lambda t}\big]}{2i\sin(\pi x)\Gamma(-x)}  e^{-i\pi x+\lambda} (-\lambda)^{-x-1}\, d\lambda
			=\sum_{n=0}^\infty \mathscr{C}^\Delta_{\pm n\mp \frac{1}{2}}(x)\, \frac{t^n}{n!}$,
%			($\mathbb{C}\times\mathbb{C}$)
			($\mathbb{C}^2$)
			 & $\mathcal{G}_\mp$ &$\eqref{E:delta_glaisher_gf}^\ast$\\
\hline
			$\displaystyle \sqrt{\frac{2}{\pi}}\dfrac{e^{-i\pi x}} {2i \sin(\pi x)}
	\int_{-\infty}^{(0+)} 
	\frac{e^{\lambda}
	(-\lambda)^{-x-1}\lambda^{-\frac{1}{2}} 
	 \cos\sqrt{\lambda ^2-2\lambda t}}{\Gamma(-x)}\,  d\lambda
	=\sum_{n=0}^{\infty}J^{\Delta}_{n-\frac{1}{2}}(x)\frac{t^n}{n!}$,
       ($\mathbb{C}\times\mathbb{C}$) & $\mathcal{G}_-$ & $\eqref{E: gf-Delta_glaisher_cos}^\ast$\\
\hline
			$\displaystyle \sqrt{\frac{2}{\pi}}\dfrac{e^{-i\pi x}} {2i \sin(\pi x)}
	\int_{-\infty}^{(0+)} 
	\frac{e^{\lambda}
	(-\lambda)^{-x-1}\lambda^{-\frac{1}{2}} 
	 \sin\sqrt{\lambda ^2+2\lambda t}}{\Gamma(-x)}\,  d\lambda
	=\sum_{n=0}^{\infty}J^{\Delta}_{-n+\frac{1}{2}}(x)\frac{t^n}{n!}$,
	 ($\mathbb{C}\times\mathbb{C}$) & $\mathcal{G}_+$ & $\eqref{E: gf-Delta_glaisher_sin}^\ast$\\
\hline
\end{tabular}
}
%\caption{\label{demo-table}The ``$\sim$" denotes Borel-resummation.}

%\vfill\eject
%\appendix{Appendix B} \label{S:appendix_b}

\noindent{\textbf{Acknowledgement.}} The authors would like to thank Aimo Hinkkanen for pointing out an error found in an earlier draft to the statement of the Theorem \ref{T:Bessel_gf}. We thank Chun-Kong Law for his suggestion to include  tables of results obtained. The authors benefit from the critical and constructive comments from Henry Cheng that improve the readability of the original manuscript. Special thanks also go to Yum Tong Siu who spent time to explain to the first author of the importance of $D$-modules in relation to special functions.  
%\medskip
 
\bibliographystyle{amsplain}
\bibliography{references_21Oct2022}

\end{document}